%% file: Thesis.tex
\documentclass[hidelinks,11pt]{report}
\usepackage[utf8]{inputenc}

\usepackage{geometry} 
\usepackage{graphicx} 

\usepackage{booktabs} 
\usepackage{array} 
\usepackage{paralist} 
\usepackage{verbatim} 
\usepackage{subfig} 
\usepackage{hyperref}
\usepackage{enumitem}
\setlist{noitemsep}
\usepackage{amsmath}
\usepackage{amssymb}
\usepackage{authblk}
\usepackage{bbm}
\usepackage{mathtools}
\usepackage{textcomp}
\usepackage{xcolor}
\usepackage{listings}
\usepackage{float}
\usepackage{centernot}
\usepackage{bm}
\usepackage[mathscr]{euscript}
\usepackage{tikz-cd}
\usepackage{tabularx}
\newcolumntype{Z}{>{\centering\let\newline\\\arraybackslash\hspace{0pt}}X}
\usepackage{afterpage}
\newcommand\blankpage{%
    \null
    \thispagestyle{empty}%
    \addtocounter{page}{-1}%
    \newpage}

\newcommand{\rfrak}{\mathfrak{r}}
\newcommand{\Rfrak}{\mathfrak{R}}
\newcommand{\Rfrakk}{\underline{\Rfrak}}

\newcommand{\Bcal}{\mathcal{B}}
\newcommand{\Ccal}{\mathcal{C}}
\newcommand{\Dcal}{\mathcal{D}}
\newcommand{\Ecal}{\mathcal{E}}
\newcommand{\Fcal}{\mathcal{F}}
\newcommand{\Gcal}{\mathcal{G}}

\newcommand{\Ical}{\mathcal{I}}
\newcommand{\Lcal}{\mathcal{L}}
\newcommand{\Mcal}{\mathcal{M}}
\newcommand{\Ocal}{\mathcal{O}}
\newcommand{\Pcal}{\mathcal{P}}
\newcommand{\Rcal}{\mathcal{R}}
\newcommand{\Scal}{\mathcal{S}}
\newcommand{\Tcal}{\mathcal{T}}
\newcommand{\Ucal}{\mathcal{U}}
\newcommand{\Vcal}{\mathcal{V}}

\newcommand{\Bb}{\mathbb{B}}

\newcommand{\Dbb}{\mathbb{D}}

\newcommand{\Fbb}{\mathbb{F}}
\newcommand{\Ibb}{\mathbb{I}}
\newcommand{\Jbb}{\mathbb{J}}
\newcommand{\Mbb}{\mathbb{M}}
\newcommand{\Nbb}{\mathbb{N}}
\newcommand{\Obb}{\mathbb{O}}
\newcommand{\Qbb}{\mathbb{Q}}
\newcommand{\Rbb}{\mathbb{R}}
\newcommand{\Tbb}{\mathbb{T}}

\newcommand{\Zbb}{\mathbb{Z}}
\newcommand{\yon}{\mathbf{y}}
\newcommand{\Set}{\mathbf{Set}}
\newcommand{\FinSet}{\mathbf{FinSet}}
\newcommand{\Pos}{\mathbf{Pos}}
\newcommand{\Mon}{\mathbf{Mon}}
\newcommand{\Grp}{\mathbf{Grp}}

\newcommand{\Top}{\mathbf{Top}}

\newcommand{\Arch}{\mathbf{Arch}}
\newcommand{\mArch}{\mathbf{mArch}}
\newcommand{\mSpan}{\mathbf{mSpan}}
\newcommand{\Hom}{\mathrm{Hom}}
\newcommand{\HOM}{\mathcal{H}\! \mathit{om}}

\newcommand{\Nat}{\mathrm{Nat}}
\newcommand{\Cat}{\mathbf{Cat}}

\newcommand{\TOP}{\mathfrak{TOP}}
\newcommand{\op}{^{\mathrm{op}}}
\newcommand{\co}{^{\mathrm{co}}}

\newcommand{\lr}[2]{( #1, #2 )}
\newcommand{\fllr}[2]{[ #1, #2 )}

\newcommand{\too}{\twoheadrightarrow}

\mathchardef\mhyphen="2D

\newcommand{\name}[1]{\ulcorner{#1}\urcorner}
\newcommand{\Fraisse}{Fra\"{i}ss\'{e} }
\DeclareMathOperator{\im}{im}
\DeclareMathOperator{\ob}{ob}

\DeclareMathOperator{\id}{id}

\DeclareMathOperator{\Sh}{Sh}

\DeclareMathOperator{\Geom}{Geom}
\DeclareMathOperator{\EssGeom}{EssGeom}

\DeclareMathOperator{\Ind}{Ind-}
\DeclareMathOperator{\Cont}{Cont}
\DeclareMathOperator{\Aut}{Aut}
\DeclareMathOperator{\End}{End}
\DeclareMathOperator{\Flat}{Flat}
\DeclareMathOperator{\Site}{Site}

\DeclareMathOperator{\Sub}{Sub}

\DeclareMathOperator{\Span}{Span}
\DeclareMathOperator{\Fix}{Fix}
\DeclareMathOperator{\colim}{colim}

\tikzset{
  no line/.style={draw=none,
    commutative diagrams/every label/.append style={/tikz/auto=false}},
  from/.style args={#1 to #2}{to path={(#1)--(#2)\tikztonodes}}
	}
\tikzset{symbol/.style={draw=none, every to/.append style={edge node = {node [sloped, allow upside down, auto=false] {$#1$}}}}}
\newcommand{\Setswith}[1]{\mathrm{PSh}(#1)}

\usepackage{amsthm}
\newtheorem{thm}{Theorem}[section]
\newtheorem{theorem}[thm]{Theorem}
\newtheorem{proposition}[thm]{Proposition}
\newtheorem{prop}[thm]{Proposition}
\newtheorem{lemma}[thm]{Lemma}
\newtheorem{fact}[thm]{Fact}
\newtheorem{corollary}[thm]{Corollary}
\newtheorem{crly}[thm]{Corollary}
\newtheorem{scholium}[thm]{Scholium}
\newtheorem{schl}[thm]{Scholium}
\newtheorem{conj}[thm]{Conjecture}
\theoremstyle{definition}
\newtheorem{definition}[thm]{Definition}
\newtheorem{dfn}[thm]{Definition}
\newtheorem{example}[thm]{Example}
\newtheorem{xmpl}[thm]{Example}
\theoremstyle{remark}
\newtheorem{rmk}[thm]{Remark}
\newtheorem{remark}[thm]{Remark}

\usepackage{fancyhdr} 
\usepackage[nottoc,notlof,notlot]{tocbibind} 
\usepackage[titles,subfigure]{tocloft} 

\begin{document}

\pagenumbering{gobble}

\input{The_Titlepage}

\vspace*{7cm}

\small
\begin{center}
\textbf{Abstract}

We study toposes of actions of monoids on sets. We begin with ordinary actions, producing a class of presheaf toposes which we characterize. As groundwork for considering topological monoids, we branch out into a study of supercompactly generated toposes (a class strictly larger than presheaf toposes). This enables us to efficiently study and characterize toposes of continuous actions of topological monoids on sets, where the latter are viewed as discrete spaces. Finally, we refine this characterization into necessary and sufficient conditions for a supercompactly generated topos to be equivalent to a topos of this form.
\end{center}

\normalsize

\afterpage{\blankpage}

\tableofcontents

\newpage

\pagenumbering{arabic}

\input{The_Intro}

\input{The_TDMA}

\input{The_MPaTTI}

\input{The_SGT}

\input{The_SL}

\input{The_TTMA}

\input{The_TSGT}

\input{The_Conclusion}

\bibliographystyle{alpha}
\bibliography{thesisbib}

\end{document}

%% file: The_Titlepage.tex
\begin{titlepage}
   \begin{center}
       \vspace*{1cm}

       \Huge
       \textbf{Toposes of Monoid Actions}

       \vspace{0.5cm}

       \LARGE
       From discrete to topological monoids, through a lens of geometry, logic and category theory

       \vspace{1.2cm}

       \textbf{Morgan Rogers}

       \vspace{0.2cm}

       \normalsize
       Marie Sklodowska-Curie fellow of the Istituto Nazionale di Alta Matematica\\
       \vspace{0.2cm}
       \textbf{Thesis advisor: Professor Olivia Caramello}
            
       \vspace{1.1cm}
       
       \includegraphics[width=0.5\textwidth]{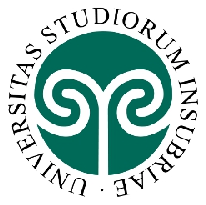}

       \vfill
       \large
       A thesis presented for the degree of\\
       Doctor of Philosophy in Computer Science and Mathematics of Computation

       \vspace{0.8cm}
            
       Dipartimento di Scienza e Alta Tecnologia\\
       Universit\`a degli Studi dell{'}Insubria\\
       Italia

       \vspace{0.2cm}

       October 2021
       
       \afterpage{\blankpage}
            
   \end{center}
\end{titlepage}

%% file: The_Intro.tex
\chapter*{Introduction}
\addcontentsline{toc}{chapter}{Introduction}

The most easily described examples of Grothendieck toposes are presheaf toposes: categories of the form $[\Ccal\op,\Set]$, whose objects are contravariant functors from a small category $\Ccal$ to the category of sets and whose morphisms are natural transformations between these. The initial motivation for this thesis was asking the question, `What does such a topos look like for the simplest choices of $\Ccal$?' In particular, what happens when $\Ccal$ is a preorder or a monoid?

It turns out that the preorder case is already very well understood: only a little theory is needed to verify that the topos of presheaves on a preorder is equivalent to the category of sheaves on the corresponding \textit{Alexandroff space}, and so is a special case of a localic topos. Localic toposes are extensively studied due to their direct connection with topology and geometry; this case is how sheaf toposes are motivated in \cite[Chapter II]{MLM}, for example, and most of the properties of toposes and geometric morphisms described in \cite[Part C]{Ele} are derived from this special case. As such, there are a variety of tools readily available to analyze toposes of presheaves on preorders from a geometric point of view.

On the other hand, while toposes of the form $\Setswith{M}$ feature as illustrative examples in introductory texts such as \cite{MLM} and \cite{TTT}, if one looks to the most comprehensive topos theory references to date, notably P.T. Johnstone's work \cite{Ele}, there is no systematic treatment of presheaves on monoids which parallels the one for presheaves on preorders. Chapter \ref{chap:TDMA} of this thesis addresses this disparity, mostly by recovering results from semigroup theory literature, and arrives at a characterization of such presheaf toposes. We take this opportunity to examine the extent to which monoids can be recovered from their toposes of right actions and the geometric morphisms induced by monoid (and semigroup) homomorphisms and biactions of monoids.

With the characterization in hand, we put the class of toposes of monoid actions to work as a conduit between topos theory and semigroup theory. More precisely, in Chapter \ref{chap:mpatti} we systematically investigate the correspondence between algebraic properties of a monoid $M$ and categorical properties of its topos of actions $\Setswith{M}$. This conduit carries insight in both directions, however: there are coincidences of properties which occur in localic toposes which do not occur for toposes of monoid actions, so this investigation yields some counterexamples to generalizations which one might consider making from the localic case.

Both in \cite[Chapter III.9]{MLM} and as a recurring example through \cite{Ele}, continuous actions of topological groups on sets (viewed as discrete spaces) are investigated. Given our success up to this point, a natural next step in the investigation of toposes of monoid actions would be to investigate continuous actions of topological monoids. This is indeed one of our major goals, but we put off the exposition of it until Chapter \ref{chap:TTMA}. The reason for this is that the arguments needed to characterize the properties of categories of topological monoid actions and extract canonical sites for them can be developed and applied in the much more general setting of the class of \textit{supercompactly generated} Grothendieck toposes. Seeking greater depth of understanding of these toposes and their relatives is the subject of Chapter \ref{chap:sgt}.

In order to set the stage for the final chapter, we give a brief exposition of the well-known theory of classifying toposes for geometric theories in Chapter \ref{chap:logic}. We ground this exposition in the goal of gaining a basic characterization of theories classified by supercompactly generated toposes, and we obtain two variants of such a characterization. We also examine how examples of such theories can be derived from known theories of presheaf type.

In Chapter \ref{chap:TTMA}, we reach the anticipated study of toposes of topological monoid actions. We start by characterizing continuity in terms of both `necessary clopens' which are comparable to the isotropy subgroups in the topological group case, and open congruences. We subsequently show that the continuous actions of a monoid are coreflective in the category of all actions (of the underlying discrete monoid), which is enough to show that this category forms a topos, which we ultimately characterize in terms of this coreflection, which is the inverse image part of a hyperconnected morphism. In the process, we pursue natural questions about the relationship between topological monoids and their toposes of actions, including the extent to which a monoid can be recovered from its topos of continuous actions. We ultimately arrive at classes of `powder monoids' having the right topological properties to present such toposes and then the class of `complete monoids' which are the canonical representatives of these toposes. We also examine how continuous homomorphisms between topological monoids induce geometric morphisms between the corresponding toposes, and how these interact with some of the factorization systems on geometric morphisms.

Our final goal continues in the vein of extending the theory for groups. Toposes associated to groups arise in topos-theoretic treatments of Galois theory such as \cite{TGT} and \cite{LGT}, and a selection of such results have been extended to the more general context of toposes associated to monoids. Notably, \cite{SGC} studies the actions of pro-finite topological monoids. In Chapter \ref{chap:TSGT}, we refine the characterization of toposes of topological monoid actions to the point of being able to construct some non-trivial examples of sites whose toposes of sheaves have this form. We then use the approach taken by Caramello in \cite{TGT}: by considering the supercompact objects (as a generalization of atoms in the group case) on either side of the resulting equivalence, we recover in the first instance a characterization of geometric theories whose classifying toposes are equivalent to toposes of actions of topological monoids, and in the second instance an analogue of Galois theory in the form of an anafunctor from a category underlying a suitable site to a category of \textit{open right congruences} on a topological monoid presenting the topos, which is faithful, fully faithful or an equivalence in good cases. Arguably the most important results of the thesis appear in this chapter, including a generalization of the \Fraisse construction from model theory, which demonstrates that the abstract framework has vast numbers of applications, of which we present a small selection.

In the Conclusion, we identify some prospective research directions and make some conjectures based on all of the material in the preceding chapters.

\vspace{0.5cm}

With the exception of Chapters \ref{chap:logic} and \ref{chap:TSGT}, all of the material in this thesis comes from existing articles and preprints, with duplicated definitions and results removed and some further abridgement. Chapter \ref{chap:TDMA} is a refinement and expansion of \cite{TDMA}. As we explain there, there are only a few original aspects of this chapter; it serves primarily as a foundation for the subsequent developments. Section \ref{sec:biact} of that chapter as well as the entirety of Chapter \ref{chap:mpatti} are reproduced (with permission) with minor modifications from collaborative work with Jens Hemelaer, \cite{MPaTP}. The material of Chapter \ref{chap:sgt} comes from \cite{SGT}, while the material of Chapter \ref{chap:TTMA} is from \cite{TTMA}. In a couple of places we mention some existing related completed work which there was not space to include here.

Unless otherwise stated, the unqualified noun `topos' refers to a Grothendieck topos. We have tried to keep the exposition relatively self-contained, but we regularly defer to more comprehensive texts for details, especially Mac Lane and Moerdijk's \cite{MLM} and Johnstone's reference works on topos theory, \cite{TT} and \cite{Ele}.

\newpage

\section*{Acknowledgements}
\addcontentsline{toc}{subsection}{Acknowledgements}

This work was supported by INdAM and the Marie Sklodowska-Curie Actions as a part of the \textit{INdAM Doctoral Programme in Mathematics and/or Applications Cofunded by Marie Sklodowska-Curie Actions}.

I would like to thank my thesis advisor Olivia Caramello for introducing me to topos theory. Whether or not I reach the end of this path, you gave me a glimpse of how the dreams I have had since I first encountered category theory might be achieved.

Deep gratitude goes to Jens Hemelaer for our hundreds of hours of collaboration, a mere fraction of which is represented here, because it was through that collaboration that I came to appreciate what our work could eventually mean to others.

Thanks to my Insubria colleagues Riccardo Zanfa and Joshua Wrigley for participating in this place and this topic with me, providing complementary perspectives and making life more bearable in general. A special mention goes to Riccardo's attention to detail which has caught many a typographical error of mine. 

Ivan Di Liberti and Axel Osmond get a joint mention as invaluable (and exceptionally efficient) navigators in the more turbulent waters of the category theory literature in their respective areas of expertise. 

More formal thanks go to the respective organizers and participants in the PSSL105 and SYCO6 conferences and the Firbush Theta Meetup in 2019; the Topics in Category Theory School, the Junior Mathematician Research Archive and the University of Augsburg Algebra Colloquium in 2020, and the Cambridge Category Theory Seminar, the Global Noncommutative Geometry Seminar, the York Semigroup Seminar, the Categories and Companions Symposium and Toposes online in 2021 for giving myself and my collaborator the opportunity to present our work. Additional thanks go to the founders of the Category Theory Zulip platform and, even though I didn't present any work there, to the organizers of the Category Theory $20{\rightarrow}21$ conference, for bringing the community together in invaluable ways during the pandemic\footnote{It was bound to get a mention somewhere..!}.

Most importantly, because none of this would matter without the ones I love, I send my gratitude to my partner \"{O}yk\"{u} Ko\c{c}, who has cherished and supported me for more than eighteen months now; to my parents James and Caroline who have been doing the same for over twenty-five years; to my sisters Hazel and Joanna who have always been proud of me; to \"Oyk\"u's parents Meral and B\"ulent for the warm hospitality, and to all of our extended families for the good times when I needed them.

%% file: The_TDMA.tex
\chapter{Toposes of Discrete Monoid Actions}
\label{chap:TDMA}

It is often observed that the category of presheaves on a monoid $M$ is precisely the category of (right) \textbf{actions of $M$} or \textbf{$M$-sets}. Indeed, considering a monoid as a one-object category, the data of a functor $M\op \to \Set$ can be identified with a set $X$ equipped with a monoid homomorphism $M\op \to \End(X)$, which can in turn be presented as a function $X \times M \to X$ which interacts appropriately with the monoid multiplication.

Categories of monoid actions have been studied by semigroup theorists and ring theorists in analogy with the representation theory of rings. For example, both Knauer in \cite{MEM} and Banaschewski in \cite{FCM} independently solved the Morita equivalence problem for (left) actions of discrete monoids which we re-derive in Section \ref{sec:disc}. Banaschewski presented some of the equivalent conditions characterizing these categories which we give in Theorem \ref{thm:point} and Knauer identified adjoint pairs of functors between the categories of $M$-sets and $N$-sets with $(M,N)$-biactions as we do in Section \ref{sec:biact}. Their results subsequently featured in a reference text \cite{MAC} on categories of monoid actions. In short, many of the essential results of this chapter are not original. Our main contributions, then, are presenting these categories from the perspective of topos-theory (the word `topos' appears exactly once amongst the three references above), and in the process integrating existing results into a pair of $2$-categorical equivalence results, Theorems \ref{thm:2equiv0} and \ref{thm:2equiv1}.

\subsection*{Overview}

We begin in Section \ref{sec:disc} by addressing the Morita equivalence problem for monoids, of identifying when two monoids have equivalent toposes of (right) actions, via idempotent completions. We then examine the topos-theoretic properties of these categories in Section \ref{sec:properties0} and present a characterization of them (as categories) in Section \ref{sec:characterize0}.

In Section \ref{sec:morphisms0} we augment the initial Morita equivalence result by showing that the mapping $M \mapsto \Setswith{M}$ is part of a 2-categorical equivalence between a 2-category of monoids and a 2-category of Grothendieck toposes and essential geometric morphisms; in Section \ref{sec:biact} we extend this to a 2-equivalence involving arbitrary geometric morphisms. As a prelude to the next chapter, in Section we reproduce some examples, again featured in \cite{MEM} and \cite{FCM}, demonstrating that Morita equivalence does not in general reduce to isomorphism and conversely giving some examples of properties which are Morita-invariant by virtue of forcing the Morita-equivalence classes to collapse to isomorphism classes.

\section{Monoids and their idempotent completions}
\label{sec:disc}

As was implicit above, we treat a monoid $M$ as a (small) category with a single object; the identity shall be denoted $1$. Our analysis takes place at three levels: the level of the monoids themselves, the level of their associated presheaf toposes, and the intermediate level of their idempotent completions.

\begin{rmk}
\label{rmk:sgrp0}
Even a priori, these considerations can easily be extended to semigroups, since any semigroup $S$ has a category of right $S$-sets, which is to say sets $X$ equipped with a semigroup homomorphism $S\op \to \End(X)$. By freely adding an identity element to $S$, it becomes a monoid $S_1$ such that $\Setswith{S_1}$ is equivalent to the category of such right $S$-sets, since a monoid homomorphism $S_1\op \to \End(X)$ necessarily sends the new identity to the identity of $\End(X)$, and is therefore determined by a semigroup homomorphism $S\op \to \End(X)$. It follows that for the purposes of a classification of toposes of this form there is no difference. However, we shall show later that semigroup homomorphisms, rather than monoid homomorphisms, are the right morphisms to consider in order to understand essential geometric morphisms between toposes and to properly describe Morita equivalence.
\end{rmk}

Recall that a category $\Ccal$ is \textbf{idempotent complete} (also called \textit{Cauchy complete} or \textit{Karoubi complete}) if all idempotent morphisms in $\Ccal$ split. Recall also that any given category $\Ccal$ has an \textbf{idempotent completion}, denoted $\check{\Ccal}$, equipped with a full and faithful functor $\Ccal \to \check{\Ccal}$ universal amongst functors from $\Ccal$ to idempotent complete categories up to equivalence. For a more detailed reminder and a construction of the idempotent completion in general, see the discussion in \cite{Ele} which begins just before Lemma A1.1.8.

For a monoid $M$, $\check{M}$ can be identified up to equivalence with a category whose objects are idempotents of $M$, and this is the definition of $\check{M}$ we shall use since the resulting idempotent splittings in this category are uniquely defined. Where necessary for clarity, we shall denote by $\underline{e}$ the object of $\check{M}$ corresponding to an idempotent $e$. The morphisms $\underline{e} \to \underline{d}$ in this category are morphisms $f$ of $M$ such that $fe = f = df$; composition is inherited from $M$. $M$ is included in $\check{M}$ as the full subcategory on the object $\underline{1}$; meanwhile, the monoid of endomorphisms of $\underline{e}$ is the monoid $eMe := \{eme \mid m \in M\}$ with identity element $e$.

\begin{dfn}
\label{dfn:indecproj}
Recall that an object $C$ of a category $\Ccal$ is called \textbf{projective} if whenever there exists a morphism $f: C \to B$ and an epimorphism $g:A \too B$, there is a lifting $f':C \to A$ with $f = gf'$.

Independently, an object $C$ of $\Ccal$ is called \textbf{indecomposable} (or \textbf{connected}) if $C$ is not initial and whenever $C \cong A \sqcup B$, one of the coproduct inclusions is an isomorphism.
\end{dfn}

To justify the introduction of idempotent completions, we point to \cite[Lemmas A1.1.9, A1.1.10]{Ele} and their natural corollary:
\begin{lemma}
\label{lem:idempotent}
For any category $\Ccal$, $\Setswith{\Ccal} \simeq \Setswith{\check{\Ccal}}$, and $\check{\Ccal}$ is equivalent to the full subcategory of $\Setswith{\Ccal}$ on the indecomposable projective objects. Thus, since the properties of Definition \ref{dfn:indecproj} are stable under equivalence, $\Setswith{\Ccal} \simeq \Setswith{\Dcal}$ if and only if $\check{\Ccal} \simeq \check{\Dcal}$.
\end{lemma}

Thus $\Setswith{M} \simeq \Setswith{N}$ if and only if $\check{M} \simeq \check{N}$. We can directly pass from Lemma \ref{lem:idempotent} to the Morita equivalence result:
\begin{thm}
\label{thm:discMorita}
Two monoids $M$ and $N$ are Morita equivalent (that is, $\Setswith{M} \simeq \Setswith{N}$) if and only if there is an idempotent $e$ of $M$ with $N \cong eMe$ and $\beta, \beta' \in M$ such that $\beta \beta' = 1$, $\beta e = \beta$.
\end{thm}
\begin{proof}
Given $\check{M} \simeq \check{N}$, there exists some object $\underline{e} \in \check{M}$ corresponding to $\underline{1}' \in \check{N}$, and hence $N$ is isomorphic to the monoid of endomorphisms of this $\underline{e}$, which is precisely $eMe$. Since the idempotent completion of the full subcategory on $\underline{e}$ is identified via the equivalence with an essentially wide subcategory of $\check{M}$, it must be that $\underline{1}$ is a retract of $\underline{e}$, whence we obtain the elements $\beta, \beta'$ satisfying the given properties.

Conversely, given $e, \beta, \beta'$ in $M$ with the given properties, note that replacing $\beta'$ with $e \beta'$ if necessary, one obtains elements with the additional property that $e \beta' = \beta'$. Consider the mapping $g: M \to eMe$ given by $m \mapsto \beta' m \beta$. We see that $g$ is a semigroup homomorphism since $\beta' mn \beta = \beta' m \beta \beta' n \beta$, it has the correct codomain since $e \beta' m \beta e = \beta' m \beta$, and it is injective since we may cancel $\beta'$ on the left and $\beta$ on the right. Moreover, the image of $g$ is precisely the set of endomorphisms of \underline{$\beta'\beta$} in $\check{N}$, whence $M$ is isomorphic to the endomorphism monoid of this object, and it follows that $\check{N}$ is the whole subcategory: the idempotent completions are equivalent, as claimed.
\end{proof}

By the proof, we see that the Morita equivalence condition is self-dual; alternatively, it is easily shown that $(M\op)^{\vee} \simeq \check{M}\op$. Either way, we immediately deduce a result which is not at all obvious from the algebraic description of the category of $M$-sets:
\begin{crly}
$\Setswith{M} \simeq \Setswith{N}$ if and only if $[M,\Set] \simeq [N,\Set]$; there is no need to distinguish between `left' and `right' Morita equivalence of monoids.
\end{crly}

It is worth mentioning that the Morita equivalence presented here is distinct from the `Topos Morita Equivalence' for inverse semigroups discussed by Funk et al. in \cite{MEiS} (although the `Semigroup Morita equivalence' described there is the one introduced by Talwar in \cite{MES} based on the work of Knauer in \cite{MEM}). Indeed, the toposes considered there have as objects actions of an inverse semi-group $S$ on sets by \textit{partial isomorphisms}, which they show is equivalent to the topos of presheaves on the full subcategory $\check{S} \hookrightarrow \check{S}_1$ on the non-identity elements. Rather than constructing a detailed example to demonstrate the distinction, we point out that the Morita equivalences of \cite{MEiS} are non-trivial, whereas the extension of Morita equivalence for monoids to semigroups described in Section \ref{sec:disc} is trivial by Corollary \ref{crly:trivial}.3 below, a fact which appears as Proposition 5 in \cite{FCM}.

\section{Topos-theoretic properties}
\label{sec:properties0}

Recall that the forgetful functor $U:\Setswith{M} \to \Set$ sending a right $M$-set to its underlying set is both monadic and comonadic. In particular, it has left and right adjoints,
\[\begin{tikzcd}
\Set \ar[r, "(-) \times M", bend left = 50] \ar[r, "\Hom_{\Set}(M{,}-)"', bend right= 50] \ar[r, symbol = \bot, shift right = 5, near end] \ar[r, symbol = \bot, shift left = 5, near end] & {\Setswith{M}} \ar[l, "U"],
\end{tikzcd}\]
where the action of $M$ on $X \times M$ is simply multiplication on the right on the $M$ component and the action of $m \in M$ on $\Hom_{\Set}(M,-)$ sends $f \in \Hom_{\Set}(M,X)$ to $f\cdot m := (x \mapsto f(mx))$.

Monadicity is intuitive, since $\Setswith{M}$ is easily seen to be (equivalent to) the category of algebras for a free-forgetful adjunction: an algebra is a set $A$ equipped with a morphism $\alpha: A \times M \to A$ satisfying identities that correspond to those for an $M$-action. Meanwhile, transposing $\alpha$ and the identities it satisfies across the product-exponential adjunction in $\Set$ gives a presentation of this action as a coalgebra for the adjoint comonad.

\begin{dfn}
\label{dfn:gm}
Recall that for toposes $\Ecal$ and $\Fcal$, a \textbf{geometric morphism} $\phi:\Fcal \to \Ecal$ consists of an adjunction
\[\begin{tikzcd}
\Fcal \ar[r, "\phi_*"', shift right = 2] \ar[r, phantom, "\bot"] & \Ecal \ar[l, "\phi^*"', shift right = 2],
\end{tikzcd}\]
where $\phi_*$ is called the \textbf{direct image functor}, and $\phi^*$ is called the \textbf{inverse image functor}, which is required to preserve finite limits.

The geometric morphism $\phi$ is said to be \textbf{essential} if $\phi^*$ admits a further left adjoint, denoted $\phi_!$. A \textbf{point} of a Grothendieck topos $\Ecal$ is simply a geometric morphism $\Set \to \Ecal$. Finally, a geometric morphism is \textit{surjective} if its inverse image functor is comonadic.
\end{dfn}

Therefore we can summarize the preceding observations from a topos-theoretic perspective by saying that $U$ is the inverse image of an \textbf{essential surjective point} of $\Setswith{M}$; the existence of this point is the first property of note. We shall call this point the \textbf{canonical point} of $\Setswith{M}$, although we emphasize that the canonicity is relative to the presenting monoid $M$. Recall that a topos is said to have \textbf{enough points} if the inverse images functors of all its points are jointly conservative; it follows that every topos of monoid actions has enough points, since the canonical point is enough on its own!

There is another (essential) geometric morphism associated with the topos $\Setswith{M}$, but which does not depend on the presenting monoid: the \textbf{global sections} geometric morphism,
\begin{equation} \label{eq:can-point-and-global-sections}
\begin{tikzcd}
{\Setswith{M}} \ar[r, "C", bend left = 50] \ar[r, "\Gamma"', bend right= 40]
\ar[r, symbol = \bot, shift right = 4, near start] \ar[r, symbol = \bot, shift left = 6, near start] &
\Set \ar[l, "\Delta"', near end],
\end{tikzcd}
\end{equation}
where the global sections functor $\Gamma$ sends an $M$-set $X$ to its set $\Fix_M(X)= \HOM_M(1,X)$	of fixed points under the action of $M$, the constant sheaf functor $\Delta$ sends a set $Y$ to the same set with trivial $M$-action, and the connected components functor $C$ sends an $M$-set $X$ to its set of components under the action of $M$ (that is, to its quotient under the equivalence relation generated by $x \sim x \cdot m$ for $x \in X$, $m \in M$).

The fact that this geometric morphism is essential makes $\Setswith{M}$ a \textbf{locally connected topos}, which can equivalently be characterized by the fact that every object can be decomposed as a coproduct of indecomposable objects (see \cite[Lemma C3.3.6]{Ele}); this is a property of all presheaf toposes. Indeed, the indecomposable $M$-sets are precisely the equivalence classes of the equivalence relation generated by $x \sim y$ when $\exists m$ with $y = xm$; when $M$ is a group, these are simply the orbits of the action. Strengthening local connectedness is the fact that any presheaf topos is \textbf{supercompactly generated}, a property which we shall define, discuss and examine at length in its own right in Chapter \ref{chap:sgt}.

Next, note that the terminal object $1$ of $\Setswith{M}$ is the trivial action of $M$ on the one-element set. In particular, the only subobjects of $1$ (the \textbf{subterminal objects}) are $1$ itself and the empty $M$-action, which is to say that $\Setswith{M}$ is \textbf{two-valued}. We shall observe in Proposition \ref{prop:hype2} in Chapter \ref{chap:sgt} that two-valuedness of a Grothendieck topos is equivalent to the global sections geometric morphism being \textbf{hyperconnected}: the inverse image is full and faithful and its image is closed under subquotients. This can be deduced from the fact that $M$ is a `strongly connected category' (a category in which there is at least one morphism $A \to B$ for every ordered pair of objects $A,B$); see the discussion following \cite[Example A4.6.9]{Ele}.

This conclude that:
\begin{lemma}
\label{lem:orthog}
For a locale $X$, the localic topos $\Sh(X)$ is equivalent to a topos of the form $\Setswith{M}$ if and only if both $X$ and $M$ are trivial. Similarly, for any preorder $P$, $\Setswith{P} \simeq \Setswith{M}$ if and only if $P$ is equivalent to the one-element poset and $M$ is trivial.
\end{lemma}
\begin{proof}
The frame $\Ocal(X)$ of the locale $X$ is isomorphic to the frame of subterminal objects of $\Sh(X)$, but for any $M$, $\Setswith{M}$ is two-valued, so $\Ocal(X)$ is the initial frame, making $X$ the terminal locale, so $\Sh(X) \simeq \Set$. There is a unique geometric morphism $\Set \to \Set$ which must coincide with the canonical point described above, but the induced comonad therefore sends any object $A$ to $\Hom_{\Set}(M,A) \cong A$, which forces $M$ to have exactly one element and hence be trivial.

The subterminal objects of $\Setswith{P}$ can be identified with the downward closed sets, and it is easily seen that if any element is not a top element, the principal downset generated by that element gives a non-trivial subterminal object, and the topos fails to be two-valued; it follows that to be two valued, every element of $P$ must be a top element (and $P$ must be non-empty), which gives an equivalence with the one-element poset. The remainder of the argument is as above.
\end{proof}

Lemma \ref{lem:orthog} illustrates that the `conceptual orthogonality' between preorders and monoids as contrasting families of small categories extends in a concrete way to the topos-theoretic setting, and is represented in another sense in the orthogonality of hyperconnected and localic geometric morphisms.

\section{Characterization of presheaves on monoids}
\label{sec:characterize0}

Our next task is to characterize toposes of discrete monoid actions. The first problem, given a topos of the form $\Setswith{M}$, is to identify whether $M$, or at least some presenting monoid, can be recovered from the structure of the topos. By the Yoneda Lemma, we know that $M$ is the full subcategory on the representable corresponding to its unique object, which is indecomposable and projective; indeed, this representable object is precisely $M$ viewed as a right $M$-set. Does every indecomposable projective give a valid representation of the topos?

Certainly every such object provides us with an essential point. In \cite[Ex. 7.3]{TT}, one can find the following result for an arbitrary Grothendieck topos $\Ecal$ (in fact, it is stated there for a base topos possibly distinct from $\Set$), which provides a connection between essential points and indecomposable projective objects. The source cited there is not especially accessible, so we reprove it here.

\begin{lemma}
\label{lem:proj}
A functor $\phi: \Ecal \to \Set$ is the inverse image of an essential point if and only if it has the form $\Hom_{\Ecal}(Q,-)$ for $Q$ a projective indecomposable object.
\end{lemma}
\begin{proof}
First, we observe that $\phi$ has a left adjoint if and only if it is representable. If $\phi = \Hom(Q,-)$, then $\phi$ certainly preserves all limits by their universal properties, so it has a left adjoint by the special adjoint functor theorem, say. Conversely, if $\phi$ has a left adjoint $\phi_!$, then for an object $E$ of $\Ecal$, it must be that $\phi(E) \cong \Hom_\Set(1,\phi(E)) \cong \Hom_\Ecal(\phi_!(1),E)$, so $\phi$ is represented by $Q:=\phi_!(1)$. Indeed, it follows that $\phi_!(A) = \coprod_{a \in A}Q$.

To demonstrate the existence of the right adjoint, we again invoke the special adjoint functor theorem, whence it suffices to check preservation of colimits.

Since the initial object is strict in a topos, $\Hom_\Ecal(Q,0) = \emptyset$ holds if and only if $Q \not\cong 0$.

To preserve coproducts, it is required that $\Hom_\Ecal(Q,\coprod_{i\in I} A_i) = \coprod_{i\in I}\Hom_\Ecal(Q,A_i)$; that is, every arrow from $Q$ to a coproduct must factor uniquely through one of the coproduct inclusions. If this is so and $Q \cong Q_1 \sqcup Q_2$ then the identity on $Q$ without loss of generality factors through the inclusion of $Q_1$, and since coproducts are disjoint in $\Ecal$, this forces $Q_2 \cong 0$, so $Q$ is necessarily indecomposable. Conversely, if $Q$ is indecomposable and $f \in \Hom_\Ecal(Q,\coprod_{i\in I} A_i)$ then consider $B_i = f^*(A_i)$. Since coproducts are stable under pullback, these form disjoint subobjects of $Q$ and $Q \cong \coprod_{i \in I} B_i$. Indecomposability of $Q$ forces $B_i \cong Q$ for some $i$, and hence one can uniquely identify $f$ with a member of $\Hom_\Ecal(Q,A_i)$.

Finally, $Q$ being projective is equivalent to $\Hom(Q,-)$ preserving epis, which we claim is equivalent to preserving coequalizers given the preservation of coproducts.

All epis in $\Ecal$ are regular, so preservation of coequalizers certainly implies preservation of epimorphisms. Conversely, given a parallel pair $f,g: A \rightrightarrows B$ in $\Ecal$, consider its factorization through the kernel pair of its coequalizer:
\[\begin{tikzcd}
A \ar[r, "\exists ! e"] & B' \ar[r, shift left, "f'"] \ar[r, shift right, "g'"'] & B \ar[r, "c", two heads] & C.
\end{tikzcd}\]
$\Hom(Q,-)$ preserving epis and monos ensures that it preserves image factorizations, so without loss of generality $R = \langle f,g \rangle$ is a relation on $B$ (else take its image in $B \times B$). For $n > 1$, $R^n$ is computed via pullbacks and images, so is also preserved by $\Hom(Q,-)$, as is the diagonal subobject $R^0$. Now, $c$ is precisely the quotient of $B$ by the equivalence relation generated by $R$, which is computed as the image of the coproduct of $R^n$ for $n \geq 0$, also preserved. Hence the coequalizer of $\Hom(Q,f)$ and $\Hom(Q,g)$ is the quotient of $\Hom(Q,B)$ by the generated equivalence relation, and is precisely $\Hom(Q,C)$. We conclude that $\Hom(Q,-)$ preserves all coequalizers.
\end{proof}

Considering the construction of $\check{M}$ described earlier and Lemma \ref{lem:idempotent}, it follows that each idempotent of a monoid $M$ induces an essential point of $\Setswith{M}$; we shall see a more precise correspondence in \ref{crly:idem} later. In particular, there are typically many essential points of $\Setswith{M}$, but not every such is a candidate for an essential \textit{surjective} point. Let us return to the more general setting briefly.

\begin{lemma}
\label{lem:monad}
Let $\phi$ be the essential point of a Grothendieck topos $\Ecal$ induced by an indecomposable projective object $Q$. Then the following are equivalent:
\begin{enumerate}
	\item $\phi^*$ is comonadic (equivalently, $\phi$ is surjective).
	\item $\phi^*$ is faithful.
	\item $\phi^*$ is conservative.
	\item $\phi^*$ is monadic.
	\item $Q$ is a \textbf{separator} (or \textbf{generator}).
\end{enumerate}
\end{lemma}
\begin{proof}
(1 $\Leftrightarrow$ 2 $\Leftrightarrow$ 3) is a special case of Lemma A4.2.6 in \cite{Ele}. Being faithful, $\phi^*$ reflects monos and epis. Since $\Ecal$ is balanced, this is sufficient to reflect isomorphisms.

(2 $\Leftrightarrow$ 5) By definition.

(3 $\Leftrightarrow$ 4) Certainly $\Ecal$ has and $\phi^*$ preserves coequalizers of $\phi^*$-split pairs (and even coequalizers of reflexive pairs), since it has a left and right adjoint. Thus $\phi^*$ is monadic by Beck's monadicity theorem if and only if it is conservative.
\end{proof}

Applied to $\Setswith{M}$, the statement that the object $Q$ corresponding to the canonical point should be a separator is not especially surprising, since the objects of a topos coming from a site representing it always form a separating family, and in this instance there is just one object. More generally, we find that such separators are related very strongly to one another.

\begin{lemma}
\label{lem:retracts}
In an infinitary extensive, locally small category (and in particular in any Grothendieck topos) any pair of indecomposable projective separators are retracts of one another, and conversely if $Q, Q'$ are retracts of one another and $Q$ is an indecomposable projective separator, so is $Q'$.
\end{lemma}
\begin{proof}
Let $\Ccal$ be an extensive category and suppose $Q, Q'$ are indecomposable projective separators. The collection of all morphisms $Q \to Q'$ is jointly epic, which is to say that the composite morphism $\coprod Q \too Q'$ is epic. Since $Q'$ is projective, this epimorphism splits; there is some $Q' \hookrightarrow \coprod Q$. But $Q'$ being indecomposable forces this morphism to factor through one of the coproduct inclusions, making $Q'$ a retract of $Q$. A symmetric argument makes $Q$ a retract of $Q'$.

Now suppose $Q,Q'$ are retracts of one another and $Q$ is an indecomposable projective separator. $Q'$ is indecomposable projective since any retract of such an object also is, and $Q'$ is a separator since $\Hom(Q',-)$ surjects onto $\Hom(Q,-)$ by composition with the epi $Q' \too Q$, so the former functor is conservative when the latter is.
\end{proof}

Now that we have established strong constraints on the surjective essential points of any topos, we show in this section that any such point gives a canonical representation of the topos as the category of presheaves on a monoid.

Given an indecomposable projective separator $Q$, $\phi^* = \Hom_{\Ecal}(Q,-)$ has left adjoint $\phi_! : \Set \to \Ecal$ given by $\phi_!(A) = \coprod_{a \in A} Q \cong A \times Q$, since $\phi_!$ must preserve coproducts and $\phi_!(1) \cong Q$ from the proof of Lemma \ref{lem:proj}.

\begin{lemma}
\label{lem:monadoid}
Let $\Phi := \phi^* \phi_!$ be the functor part of the monad induced by the essential surjective point $\phi$ as above. Then $\Phi(1) = \Hom_{\Ecal}(Q,Q)$, and the unit and multiplication morphisms exactly identify $\Phi(1)$ with the monoid of $\Ecal$-endomorphisms of $Q$.
\end{lemma}
\begin{proof}
The first part is immediate by the preceding comments. By a similar argument to the above, since $\Phi$ has a right adjoint, it preserves coproducts, and hence $\Phi(X) \cong X \times \Phi(1)$; in particular, $\Phi^2(1) \cong \Phi(1) \times \Phi(1)$. The remainder follows by direct computation. \footnote{We must thank Todd Trimble for a valuable discussion on MathOverflow and via email in which he suggested this short proof.}
\end{proof}

\begin{thm}
\label{thm:point}
Let $\Ecal$ be any category. The following are equivalent:
\begin{enumerate}
	\item $\Ecal$ is equivalent to $\Setswith{M}$ for some monoid $M$.
	\item There exists a functor $\Ecal \to \Set$ which is monadic and comonadic.
	\item There exists a functor $\Ecal \to \Set$ which is monadic such that the free algebra on $1$ is indecomposable and projective.
	\item $\Ecal$ is a Grothendieck topos with at least one indecomposable projective separator.
	\item $\Ecal$ is a topos admitting an essential surjective point, $\Set \to \Ecal$.
\end{enumerate}
In particular, such an $M$ is recovered as the free algebra on the terminal object of $\Set$ for the monad $\Phi$ induced by the essential surjective point.
\end{thm}
\begin{proof}
Most of the proof is already established; the third point (a Corollary in \cite{FCM}) is equivalent to the fourth since monadic functors are faithful.

It remains to observe that if $M:= \Phi(1)$ is the monoid obtained from the monad as in Lemma \ref{lem:monadoid}, then an algebra for $\Phi$ is exactly a right $M$-set, and hence by monadicity, $\Ecal \simeq \Setswith{M}$.
\end{proof}

While we shall not explore this in any detail, we should note that Theorem \ref{thm:point} depended on the properties of $\Set$ which guaranteed that the monad $\Phi$ should have the particular form it does; more care would be needed when working over a more general topos.

\begin{rmk}
\label{rmk:endopt}
Before we conclude this section, we record that there is another approach to recovering $M$ from the canonical essential surjective point of $\Ecal \simeq \Setswith{M}$ that is somewhat easier to generalize, variants of it having appeared in \cite{TGT} and \cite{LGT} to respectively recover topological and localic group representations of toposes from their points. This approach shall generalize to the method we employ in later chapters on topological monoids.

Since the inverse image functor of the point is representable, by the usual Yoneda argument there is an isomorphism of monoids:
\[\End(U) := \Nat(\Hom_{\Ecal}(Q,-),\Hom_{\Ecal}(Q,-)) \cong \Hom_{\Ecal}(Q,Q)\op \cong M\op\]
and this provides another way of recovering $M$.
\end{rmk}

\section{Morphisms between monoids}
\label{sec:morphisms0}

A monoid homomorphism $f:N \to M$ induces an essential geometric morphism $\Setswith{N} \to \Setswith{M}$ whose inverse image is the restriction of $M$-actions along $N$. This morphism is always a surjection, being induced by a functor which is surjective on objects (see \cite[Example A4.2.7(b)]{Ele}). Notably, the canonical point studied in Section \ref{sec:disc} is induced by the inclusion of the trivial monoid into a given monoid $M$. This is not the only possible source of essential surjections, however: any equivalence is an essential surjection and we have already seen in Theorem \ref{thm:discMorita} that a monoid Morita-equivalent to $M$ is merely a sub-\textit{semigroup} of $M$. Let us examine inclusions of subsemigroups of this kind.

\begin{lemma}
\label{lem:includes}
Each (semigroup homomorphism) inclusion of $N = eMe$ into $M$ produces a fully faithful inclusion $\check{N} \hookrightarrow \check{M}$ of the respective idempotent completions. Hence the induced essential geometric morphism $\Setswith{eMe} \to \Setswith{M}$ is an \textbf{inclusion} (its direct image is full and faithful).
\end{lemma}
\begin{proof}
Observe that $eMe$ consists precisely of those elements $m \in M$ such that $eme = m$; in particular the idempotents of $eMe$ are indexed by idempotents $f \in M$ with $ef = f = fe$. In the idempotent completion $\check{M}$, recall that the morphisms $\underline{f} \to \underline{f}'$ (with $f,f' \in eMe$) are those $m \in M$ such that $f'mf = m$. But then $eme = ef'mfe = f'mf = m$. Hence $m$ lies in $N$ and $\check{N}$ is precisely the full subcategory of $\check{M}$ on the objects corresponding to the idempotents $f$ with $ef = f = fe$.

The proof that this makes the resulting geometric morphism an inclusion is described in \cite[Example A4.2.12(b)]{Ele}.
\end{proof}

More generally, \textit{any} semigroup homomorphisms $f:N \to M$ factors canonically as a monoid homomorphism to $f(1)Mf(1)$ followed by an inclusion of the above form. The equivalence in Theorem \ref{thm:2equiv0} below lifts this canonical factorization to the topos level, where it is a special case of the surjection--inclusion factorization of geometric morphisms described in \cite[Theorem A4.2.10]{Ele}.

\begin{dfn}
\label{dfn:conjugation}
Let $f,g:N \to M$ be semigroup homomorphisms. A \textbf{conjugation}\footnote{This is the author's own terminology.} $\alpha$ from $f$ to $g$, denoted $\alpha:f \Rightarrow g$ is an element $\alpha \in M$ such that $\alpha f(1') = \alpha = g(1') \alpha$ and for every $n \in N$, $\alpha f(n) = g(n) \alpha$. The conjugation $\alpha$ is said to be \textbf{invertible} if there exists a conjugation $\alpha': g \Rightarrow f$ with $\alpha' \alpha = f(1)$ and $\alpha \alpha' = g(1)$; note that $\alpha$ need not be a unit of $M$ to be invertible as a conjugation.
\end{dfn}

\begin{prop}
\label{prop:extend}
Let $M, N$ be monoids. Then functors $\check{f},\check{g}:\check{N} \to \check{M}$ correspond uniquely to semigroup homomorphisms $f,g: N \to M$, and any natural transformation $\check{\alpha}: \check{f} \rightarrow \check{g}$ is determined by the conjugation $\alpha = \check{\alpha}_{1'}:f \Rightarrow g$. A conjugation is invertible if and only if it corresponds to a natural isomorphism.
\end{prop}
\begin{proof}
Of course, $f$ is the restriction of $\check{f}$ to $N$ (that is, to the full subcategory on $\underline{1}'$). This produces a semigroup homomorphism $N \to M$, since it gives a monoid homomorphism from $N$ to $eMe$, where $e$ is the idempotent such that $\underline{e} = f(\underline{1}')$; this monoid then includes into $M$ via a semigroup homomorphism as in Lemma \ref{lem:includes}.

Conversely, any semigroup homomorphism $f$ extends uniquely to a functor $\check{f}: \check{N} \to \check{M}$, since the splittings of the idempotents of $M$ must be mapped to the splittings of their images, which forces $\check{f}(\underline{e}') := \underline{f(e')}$, and a morphism $m': \underline{e}' \to \underline{d}'$ must be sent to the conjugate of $f(m'): \underline{f(1')} \to \underline{f(1')}$ by the splitting components $\underline{f(e')} \hookrightarrow \underline{f(1')}$ and $\underline{f(1')} \too \underline{f(d')}$.

Similarly, $\check{\alpha}$ determines and is determined by $\alpha := \check{\alpha}_{\underline{1}'}$ because the horizontal morphisms in the naturality square
\[\begin{tikzcd}
f(\underline{1}') \ar[r, two heads, shift left] \ar[d, "\alpha"]
& f(\underline{e}') \ar[d, "\check{\alpha}_{\underline{e}'}"] \ar[l, hook, shift left]\\
g(\underline{1}') \ar[r, two heads, shift left]
& g(\underline{e}') \ar[l, hook, shift left]
\end{tikzcd}\]
splits and $\alpha$ defined in this way is a conjugation by the definition of the morphisms in $\check{M}$ and by the conditions imposed by the naturality square.

Finally, $\check{\alpha}$ is a natural isomorphism if and only if $\alpha$ is an isomorphism in $\check{M}$, which by inspection corresponds to the condition in Definition \ref{dfn:equiv}.
\end{proof}

By introducing 2-cells, we have constructed a 2-category $\Mon_s$ of monoids, semigroup homomorphisms between them, and conjugations between those. In this setting it is appropriate to explicitly state the relevant notion of equivalence imposed by the 2-cells.

\begin{dfn}
\label{dfn:equiv}
A semigroup homomorphism $f: N \to M$ is an \textbf{equivalence} if there exists a further homomorphism $g:M \to N$, called its \textbf{pseudo-inverse}, along with invertible conjugations $\alpha: \id_{N} \Rightarrow gf$ and $\beta: fg \Rightarrow \id_{M}$.
\end{dfn}

Let $\TOP^*_{\mathrm{ess}}$ be the 2-category whose objects are Grothendieck toposes equipped with an essential surjective point, whose morphisms are essential geometric morphisms (not required to commute with the designated point), and whose 2-cells are geometric transformations, which are natural transformations between inverse image functors.

\begin{thm}
\label{thm:2equiv0}
The functor $M \mapsto \Setswith{M}$ is the object part of a 2-equivalence
\[\Mon_s\co \simeq \TOP^*_{\mathrm{ess}}.\]
\end{thm}
\begin{proof}
Directly, Proposition \ref{prop:extend} shows that the mapping $M \mapsto \check{M}$ is not only functorial but also full and faithful, and by \cite[Lemma A4.1.5]{Ele} the mapping $\Ccal \mapsto \Setswith{\Ccal}$ is a full and faithful (but 2-cell reversing) functor from the sub-2-category of $\Cat$ on the idempotent-complete small categories to the 2-category of Grothendieck toposes, essential geometric morphisms and natural transformations. Therefore it suffices to show that the image of the composite is the stated subcategory.

That the composite lands inside $\TOP^*_{\mathrm{ess}}$ follows from the observations in Section \ref{sec:disc}. Conversely, given an object $\Ecal$ of $\TOP^*_{\mathrm{ess}}$, the essential surjective point provides an $M$ with $\Setswith{M} \simeq \Ecal$ by Theorem \ref{thm:point}.
\end{proof}

This result can be compared directly with the 2-equivalence between the category $\Pos$ of posets, order-preserving functions and identity 2-cells and the corresponding 2-category of localic toposes with enough essential points, essential geometric morphisms between these and geometric transformations as 2-cells, which arises as a consequence of the fact that posets are Cauchy complete. It can also be thought of as a first step towards a parallel of the results in \cite[Section C1.4]{Ele} which gives a full equivalence of 2-categories between locales and localic toposes.

\begin{crly}
\label{crly:idem}
There is an contravariant equivalence of categories between $\check{M}$ and the category of essential points of $\Setswith{M}$.
\end{crly}
\begin{proof}
This is just a restriction of the 2-equivalence to the category of semigroup homomorphisms from the trivial monoid to the monoid $M$; observe that this is also true more generally, providing another way to construct $\check{\Ccal}$ for any small category $\Ccal$.
\end{proof}

The 2-equivalence has the usual Morita equivalence result as a direct consequence, so this provides an alternative route to that result:
\begin{crly}
\label{crly:discMorita}
Two monoids $M$ and $N$ are Morita equivalent (that is, $\Setswith{M} \simeq \Setswith{N}$) if and only if they are equivalent in the sense of Definition \ref{dfn:equiv} (which can be reduced to the conditions of Theorem \ref{thm:discMorita}).
\end{crly}
\begin{proof}
This is a trivial consequence of Theorem \ref{thm:2equiv0}, since all equivalences can be expressed as essential geometric morphisms.
\end{proof}

\section{Biactions}
\label{sec:biact}

It is natural to wonder whether we can extend the equivalence of Theorem \ref{thm:2equiv0} to include arbitrary geometric morphisms at the topos level. In order to achieve this, we introduce biactions.

\begin{dfn}
Let $M,N$ be monoids and let $B$ be a set equipped with a right $N$-action and a \textbf{compatible} left $M$-action, so that,
\begin{equation} \label{eq:compatible}
(m \cdot b) \cdot n = m \cdot (b\cdot n)
\end{equation} 
for all $m \in M$, $b \in B$ and $n \in N$. We call such a $B$ a \textbf{left-$M$-right-$N$-set}, or more concisely an $\lr{M}{N}$-set or (if the monoids are implicitly given) a \textbf{biaction}.
\end{dfn}

Given a right $M$-set $A$ and a $\lr{M}{N}$-set $B$, recall that the \textbf{tensor product} of $A$ and $B$ over $M$ is defined as the set:
\begin{equation*}
A \otimes_M B ~=~ (A \times B)/\!\sim\,
\end{equation*}
where ${\sim}$ is the equivalence relation generated by
\begin{equation*}
(am,b) \sim (a,mb)
\end{equation*}
for all $a \in A$, $b \in B$ and $m \in M$. The equivalence class of a pair $(a,b)$ is denoted by $a \otimes b$. The set $A \otimes B$ inherits a right $N$-action from $B$, defined by $(a \otimes b) \cdot n = a \otimes (b \cdot n)$, and this construction is functorial, so that $- \otimes_M B$ defines a functor $\Setswith{M} \to \Setswith{N}$.

Dually, given a right $N$-sets $A$ and a $\lr{M}{N}$-set $B$, we can consider the set
\begin{equation*}
\HOM_N(B,A)
\end{equation*}
of right $N$-set homomorphisms from $B$ to $A$. This set inherits a right $M$-action by letting
\begin{equation}
\label{eq:homaction}
	f\cdot m(b) := f(m \cdot b),
\end{equation}
for homomorphisms $f$ and elements $m \in M$ and $b \in B$. As one might expect, $\HOM_N(B,-)$ defines a functor $\Setswith{N} \to \Setswith{M}$.

Just as in ring theory, the functor $-\otimes_M B$ is left adjoint to $\HOM_N(B,-)$.
\begin{proposition}
\label{prop:adjunction}
Let $M$ and $N$ be monoids. Any $\lr{M}{N}$-set $B$ induces an adjunction:
\[\begin{tikzcd}[column sep=large]
\Setswith{N} \ar[r, bend right = 15, "{\HOM_N(B,-)}"'] \ar[r, phantom, "\bot"] & \Setswith{M} \ar[l, bend right = 15, "- \otimes_M B"']
\end{tikzcd}\]
Conversely, any adjunction between these categories (in this direction) has this form for some $\lr{M}{N}$-set $B$.
\end{proposition}
\begin{proof}
For $X$ a right $M$-set and $Y$ a right $N$-set, the isomorphism $\HOM_M(Y \otimes_M B, X) \cong \HOM_N(Y, \HOM_N(B,X))$ sends $h$ in the former to the mapping $y \mapsto (b \mapsto h(y \otimes b))$. It is easy to check that this is well-defined, and has inverse mapping $k$ to $y \otimes b \mapsto k(y)(b)$.

Conversely, given $F:\Setswith{M} \to \Setswith{N}$ having a left adjoint, $B$ is given by $F(M)$, with left $M$-action induced by the left $M$-action by endomorphisms on $M$.
\end{proof}

Thus biactions are \textit{enough} to express all geometric morphisms; we need only cut down to those biactions such that the tensoring functor preserves finite limits.

\begin{dfn}[{\cite[VII.6, Definition 2]{MLM}}]
\label{dfn:filtering}
Recall that a functor $F: \Ccal \to \Set$ is called \textbf{flat} (or \textit{filtering}) if:
\begin{enumerate}[label = ({\alph*})]
\item $F(C)\neq \varnothing$ for some object $C$ of $\Ccal$;
\item for elements $a \in F(A)$ and $b \in F(B)$ there is an object $C$, morphisms $f : C \to A$ and $g : C \to B$, and an element $c \in F(C)$ such that $F(f)(c) = a$ and $F(f)(c) = b$;
\item for morphisms $f,g : B \to A$ in $\Ccal$ and $b \in F(B)$ such that $F(f)(b) = F(g)(b)$, there is a morphism $h : C \to B$ and an element $c \in F(C)$ such that $fh = gh$ and $F(h)(c) = b$. 
\end{enumerate}
\end{dfn}

\begin{proposition}
\label{prop:flat}
Consider monoids $M,N$ and a $\lr{M}{N}$-set $B$. Then $- \otimes_M B : \Setswith{M} \to \Setswith{N}$ preserves finite limits if and only if, ignoring the $N$-action, $B$ is flat as a functor $M \to \Set$, which is to say:
\begin{enumerate}[label = ({\alph*})]
\item $B$ is non-empty;
\item for elements $b,b' \in B$ there exists $a \in B$ and $m,m' \in M$ with $m \cdot a = b$ and $m' \cdot a = b'$; and
\item whenever $c \in B$ and $m,m' \in M$ with $m \cdot c = m' \cdot c$, there exists $d \in B$, $n \in M$ with $n \cdot d = c$ and $mp = m'p$.
\end{enumerate}
We call $B$ with these properties a \textbf{flat left-$M$-right-$N$-set}, or $\fllr{M}{N}$-set for short.
\end{proposition}
\begin{proof}
This appears as \cite[Theorem VII.7.2]{MLM}, an extension of which is also referred to as Diaconescu's theorem by Johnstone in \cite[Theorem B3.2.7]{Ele}. It should not be surprising that the $N$-action is irrelevant here, since the underlying set functor on $\Setswith{N}$ creates limits, so flatness is determined at the level of underlying sets.
\end{proof}

\begin{dfn}
\label{dfn:fllrmor}
A \textbf{morphism of $\lr{N}{M}$-sets} $B \to B'$ is a function $B \to B'$ which is both a left-$M$-set homomorphism and a right-$N$-set homomorphism. For fixed $M$ and $N$, these morphisms make the collection of $\lr{N}{M}$-sets into a category. The $\fllr{N}{M}$-sets form a full subcategory of the category of $\lr{N}{M}$-sets.
\end{dfn}

Let $\Mon_{bi}$ be the bicategory\footnote{This is a bicategory rather than a $2$-category, since the composition of $\fllr{N}{M}$-sets is has identities and associativity defined only up to isomorphism.} whose objects are monoids, whose 1-cells $M \to N$ are the $\fllr{M}{N}$-sets and whose 2-cells are $\lr{M}{N}$-homomorphisms between these. Also, let $\TOP^*$ be the $2$-category whose objects are Grothendieck toposes equipped with an essential surjective point, whose morphisms are geometric morphisms (still not required to commute with the designated point), and whose 2-cells are geometric transformations.

\begin{thm} \label{thm:2equiv1}
The functor $M \mapsto \Setswith{M}$ is the object part of a biequivalence,
\[\Mon_{bi} \simeq \TOP^*.\]
\end{thm}
\begin{proof}
The object level is established by Theorem \ref{thm:point}. The 1-cell level is established by Propositions \ref{prop:adjunction} and \ref{prop:flat}, once we have observed that, if $B$ is a $\fllr{M}{N}$-set and $A$ is a $\fllr{L}{M}$-set, then $A \otimes_M B$ is the $\fllr{L}{N}$-set corresponding to the inverse image functor $(- \otimes_L A) \otimes_M B$; this is simply the observation that the tensor product operation is associative up to isomorphism.

Given two parallel adjunction $f^* \dashv f_*$ and $g^* \dashv g_*$, a natural transformation $f^* \Rightarrow g^*$ is determined by its component $f^*(M) \to g^*(M)$, which is automatically a right-$N$-set homomorphism; it is also a left-$M$-set homomorphism by naturality with respect to the endomorphisms of $M$. Conversely, a $\lr{M}{N}$-set homomorphism $B \to B'$ induces a natural transformation $- \otimes_M B \to - \otimes_M B'$ by composition on the second component; commutation with the respective actions ensures that this is well-defined and an $M$-set homomorphism at each object $X$.
\end{proof}

\begin{rmk}
More generally, a (Lawvere) \textbf{distribution} $f : \Fcal \to \Ecal$ between toposes is any adjoint pair $f^* \dashv f_*$, where $f^*$ does not necessarily preserve finite limits. A morphism $f \Rightarrow g$ between distributions $f,g : \Fcal \to \Ecal$ is still a natural transformation $f^* \Rightarrow g^*$. The proof of Theorem \ref{thm:2equiv1} also establishes a biequivalence between the larger bicategories where we include all biactions at the monoid level and all distributions at the topos level.
\end{rmk}

Thus we have an algebraic characterization of arbitrary geometric morphisms between toposes of discrete monoid actions, as well as an alternative perspective on the extra adjunction $(f_! \dashv f^*)$ in an essential geometric morphism $f$. Explicitly, by direct calculation:

\begin{lemma}
\label{lem:esstensor}
Let $f: \Setswith{N} \to \Setswith{M}$ be an essential geometric morphism induced by a semigroup homomorphism $\phi:M \to N$. Then the $\fllr{M}{N}$-set corresponding to $(f^* \dashv f_*)$ is the left ideal $N \phi(1)$ equipped with left-$N$-action by multiplication and right-$M$-action by multiplication after applying $\phi$. In particular, when $\phi$ is a monoid homomorphism, the $\fllr{N}{M}$-set is simply $N$ equipped with the respective actions.

Meanwhile, the $\lr{M}{N}$-set corresponding to the extra adjunction $(f_! \dashv f^*)$ is the right ideal $\phi(1) N$ of $N$, similarly equipped with respective multiplication actions but with the handedness reversed.
\end{lemma}

In analogy with Corollary \ref{crly:idem}, we can deduce the well-known result that the category of points of $\Setswith{M}$ is equivalent to the category of flat left $M$-sets and left $M$-set homomorphisms.

In principle, Theorem \ref{thm:2equiv1} could be used to re-derive the Morita equivalence characterization of Theorem \ref{thm:discMorita} from the perspective of biactions, which would reflect classical Morita theory for rings originating in \cite{Morita}. For the purposes of the present thesis, however, beyond some discussion and examination of `trivial' biactions in the next chapter, our investigation of biactions ends here.

\section{Examples and corollaries}
\label{sec:xmpls0}

To end this chapter, we turn to a concrete examination of Morita equivalence. To begin, here is an example demonstrating that Morita equivalence is (in general) strictly weaker than isomorphism.

\begin{xmpl}
\label{xmpl:Schein}
The `Schein monoids' were described by Knauer in \cite{MEM}. Consider the monoid $M'$ of partial endomorphisms of $[0,1]$; that is, of those functions $A \to [0,1]$ where $A$ is some subset of $[0,1]$. The composite of two such morphisms $f: A \to [0,1]$ and $g:B \to [0,1]$ is defined to be the function $g \circ f: f^{-1}(B) \to [0,1]$.

Let $M$ be the submonoid of $M'$ generated by the inclusions $e_x: [0,x] \hookrightarrow [0,1]$ for $3/4 \leq x \leq 1$, the halving map $\beta':[0,1] \to [0,1]$ sending $a \mapsto a/2$ and the doubling map $\beta:[0,1/2] \to [0,1]$ which is a left inverse to $\beta'$. By inspection $e_{3/4}, \beta, \beta'$ satisfy the required conditions to generate a Morita equivalence; let $N = e_{3/4}Me_{3/4}$.

To see that $M$ and $N$ are not isomorphic, observe that the idempotents of $M$ are all of the form $e_x$ for some $x \in [0,1]$; a more detailed case analysis demonstrates that the idempotents are precisely $e_x$ with $x \in [3/2^{n+2},1/2^n]$ for some $n \geq 0$. The idempotents come with a canonical order given by $e_x < e_y$ if $x<y$, or equivalently if $e_x e_y = e_x$; this order is thus preserved by isomorphism. The non-identity idempotents of $M$ have no maximal element. However, the non-identity idempotents of $N$ do have a maximum (specifically $e_{1/2}$). Thus $M \not\cong N$. 
\end{xmpl}

This and further examples are collected in \cite{MAC}. It should be clear, however, that the conditions in Corollary \ref{thm:discMorita} force Morita equivalence to reduce to isomorphism in many important cases.

\begin{crly}
\label{crly:trivial}
Let $M$ be a monoid. Then for equivalence to coincide with isomorphism at $M$, any of the following conditions suffices:
\begin{enumerate}
	\item $M$ is commutative.
	\item $M$ is a group.
	\item Every right (or every left) invertible element of $M$ is invertible; equivalently, the non-units of $M$ are closed under multiplication (such as when $M=S_1$ for a semigroup $S$, or $M$ is finite).
	\item $M$ is left (or right) cancellative.
	\item The idempotents of $M$ satisfy the descending chain condition with respect to absorption on the right (or left).
	\item The left (or right) ideals of $M$ satisfy the descending chain condition.
\end{enumerate}
\end{crly}
\begin{proof}
It suffices to examine the condition for equivalence in Corollary \ref{thm:discMorita}. We obtain an equivalence with $M$ whenever $M$ contains elements $\beta$, $\beta'$ and an idempotent $e$ with $\beta \beta' = 1$ and $\beta e = \beta$ (the equivalence is with $eMe$); if such a $\beta$ is necessarily an isomorphism, this forces $e=1$, so the Morita equivalence class is trivial and the equivalence collapses to an inner automorphism of $M$. In the first three cases, the equation $\beta \beta' = 1$ indeed forces $\beta$ to be an isomorphism, while in the fourth case $\beta e = \beta$ forces $e=1$ so there is nothing further to do.

For the last two conditions, note that $e_n:= {\beta'}^n {\beta}^n$ is an idempotent for every $n$, with the property that $e_n e_m = e_n = e_m e_n$ whenever $n\geq m$; if it is ever the case that $e_{n+1} = e_n$, then by multiplying on the left by ${\beta}^n$ and on the right by ${\beta'}^n$ it is again the case that $\beta' \beta = 1$. Thus for equivalence to be non-trivial $M$ must have an infinite descending chain of idempotents. By instead considering the ideals $Me_n$ we reach a similar conclusion for ideals.
\end{proof}

These conditions are variants of those which appear in \cite{MEM} and \cite{FCM}. They can also be interpreted as properties which are invariant under Morita equivalence. \textit{Any such property necessarily corresponds to an invariant at the topos-theoretic level.} If these can be identified, each gives its own immediate Corollary of Theorem \ref{thm:2equiv0}. For example:

\begin{crly}
\label{crly:Gequiv}
The mapping $G \mapsto \Setswith{G}$ is the object part of an equivalence between the 2-category $\Grp \simeq \Grp\co$ of groups, group homomorphisms and conjugations and the 2-category $\TOP^*_{\mathrm{at, ess}}$ of \textbf{atomic} Grothendieck toposes with an essential surjective point, essential geometric morphisms and natural transformations.
\end{crly}
\begin{proof}
Note that any semigroup homomorphism between groups is automatically a group homomorphism. Thus this equivalence is simply a restriction of the earlier one, and it suffices to show that the essential image is what we claim it is. We shall see both the definition of atomicity and this result in Theorem \ref{thm:atomic}.
\end{proof}

In the next chapter, we exhibit many instances of Morita-invariant properties and their corresponding topos-theoretic invariants, although this shall include only a couple of those mentioned in Corollary \ref{crly:trivial}. The reader may deduce that each such result provides, by restriction of the equivalences of Theorems \ref{thm:2equiv0} and \ref{thm:2equiv1}, a $2$-equivalence and a biequivalence. For reasons of conciseness, we shall not make these explicit, but we hope that they will eventually be usefully applied to provide insight into the 2-categorical structure at the topos level in terms of the more algebraic structure at the monoid level.

%% file: The_MPaTTI.tex
\chapter{Monoid Properties as Topos-theoretic Invariants}
\label{chap:mpatti}

The topos $\Setswith{M}$ holds a great deal more structure than the monoid $M$ alone. In particular, it is the natural setting in which to define a great variety of constructions and tools for examining the subtler properties of monoids. Perhaps more significantly, being a Grothendieck topos, the category $\Setswith{M}$ can be compared, either indirectly through its properties or directly via equivalences or geometric morphisms, to other toposes. In this chapter, we take the indirect approach, investigating correspondences between properties of the representing monoids and well-understood properties of the corresponding toposes from the topos-theoretic literature. These happen to include some of the classes identified in Corollary \ref{crly:trivial}, of monoids whose Morita equivalence classes are unique up to isomorphism.

While we saw in the last chapter that the canonical point is determined by the choice of the presenting monoid $M$, the global sections geometric morphism is uniquely determined by the topos, so that its properties are automatically Morita-invariant. We shall primarily concern ourselves with analysis of the properties of this geometric morphism in this chapter. We have already seen in Section \ref{sec:properties0} that the global sections morphism is hyperconnected and locally connected, so we examine properties which are supplementary to these.

Several of the properties of toposes we examine are geometric, in the sense that they inherit their names from properties of toposes $\Sh(X)$ of sheaves on a topological space (or more generally a locale), $X$. Accordingly, we supplement most of the definitions in this chapter with an illustration of what they mean for toposes of this form. In doing so, it will occasionally benefit us to exploit the equivalence, demonstrated in \cite[Corollary II.6.3]{MLM}, between $\Sh(X)$ and the category $\mathbf{LH}/X$ of \textbf{local homeomorphisms over $X$}, whose objects are \textit{local homeomorphisms} or \textit{\'etale maps} $E \to X$, and whose morphisms are continuous maps making the resulting triangle over $X$ commute. Passing through this equivalence, the components of the global sections morphism $f:\Sh(X) \to \Set$ acquire new interpretations. $f_*(\pi: E \to X)$ is the set of global sections of $\pi$, i.e.\ the set of continuous maps $s : X \to E$ such that $\pi \circ s = \mathrm{id}_X$ (this is where the name of this functor comes from for a general topos). $f^*$ sends a set $A$ to the $A$-fold cover $\pi_1:X \times A \to X$ of $X$. The inverse image functor $f^*$ has a left adjoint $f_!$ if and only if $X$ is a locally connected space, in which case $f_!(\pi:E \to X)$ is the set of connected components of $E$, which is where the name of locally connected geometric morphisms originates. Finally, $X$ is connected if and only if $f^*$ is full and faithful, which is why a geometric morphism with this property is called connected. Since hyperconnectedness of a topos implies connectedness, toposes of monoid actions are connected and locally connected over $\Set$, whence toposes of sheaves over connected, locally connected spaces are a good source of intuition for their properties. It should be stressed, however, that the global sections morphism of $\Sh(X)$ is only hyperconnected if $X$ is the one-point space, so these comparisons can never be realized as equivalences of toposes outside of the case where both the space and monoid are trivial. This is a strength of our property-oriented approach: it allows us to draw formal comparisons between classes of objects even when their corresponding toposes do not coincide.

In our examples we will always talk about the topos $\Sh(X)$ of sheaves on a \textit{sober topological space} $X$, which means that the points of $X$ correspond bijectively with the points of $\Sh(X)$, and that this topos has enough points. In instances where the requirement of having enough points is not explicitly mentioned, the results can be extended to encompass toposes of sheaves on suitable locales, should the reader desire it.

There has been recent interest in the toposes $\Setswith{M}$ from a geometrical point of view. Connes and Consani, in their construction of the Arithmetic Site \cite{connes-consani} \cite{connes-consani-geometry-as}, considered the special case where $M$ is the monoid of nonzero natural numbers under multiplication. In this case, the points of $\Setswith{M}$ are related to the finite ad\`{e}les in number theory. Related toposes are studied in \cite{sagnier}, \cite{arithmtop} and \cite{llb-three}. As mentioned in \cite{TGRM}, this geometric study of monoids is inspired by the idea of ``algebraic geometry over $\mathbb{F}_1$'' \cite{manin}. In this philosophy, commutative monoids are thought of as dual to affine $\mathbb{F}_1$-schemes, while the topos $\Setswith{M}$ is seen as the topos of quasi-coherent sheaves on the space corresponding to the monoid $M$, see \cite{pirashvili}.

Since the paper that was the basis of the present chapter was written, we have shown in joint work with Jens Hemelaer \cite{Fitzgerald} that categories of the form $\Setswith{M}$ can also be useful in resolving questions regarding endomorphism monoids arising in universal algebra and beyond.

The ``purely semigroup-theoretic content'' of many of the results presented in this chapter have turned out to be known results, in that after deriving them we discovered references for them in existing literature. However, this is typical when establishing a category-theoretic approach to any area of mathematics: reproving elementary results in context is a necessary first step in applying topos-theoretic machinery, since it illuminates the efficacy and potential for generalization of this approach. We believe that topos theory will ultimately be a fruitful source of new results in semigroup theory. Reciprocally, monoids shall provide a useful source of examples, properties, constructions and intuition for toposes distinct from the usual geometric and logical perspectives, and simpler than the larger context of presheaf toposes over categories with several objects. We share the goal of Funk and Hofstra in \cite{FunkHofstra} of drawing together the research communities in semigroup theory and topos theory, and hope this chapter's contents represent a contribution towards that goal.

\subsection*{Overview}

This chapter is organized as follows. In Section \ref{sec:bg} we expose how some features common to all toposes emerge in the special case of toposes of the form $\Setswith{M}$, since many of the properties we present will be characterized by how the functors constituting the global sections morphism interact with these features.

In the main body of the chapter, Section \ref{sec:1down}, we examine the consequences of additional properties of $\Setswith{M}$ on the monoid $M$. A great number of properties of geometric morphisms in the topos theory literature are inspired from geometry, since any topological space or locale has an associated Grothendieck topos (its category of sheaves), and the global sections functor of such a topos has properties determined by the properties of the original space; we therefore accompany our exposition for monoids with the parallel analysis for topological spaces, for which results are readily available in the topos theory literature. Our arguments are guided in some cases by the examples of toposes of group actions periodically presented in Johnstone's reference text \cite{Ele}.

In the Conclusion we explain directions in which the results accumulated in this chapter could be extended in future. We are aware of some properties of monoids which are expressible in terms of categorical properties of $\Setswith{M}$ but which we have not (yet) been able to express in terms of the global sections morphism; we outline these and some further properties of toposes which we shall not be able to explore in the present text.

\section{Background}
\label{sec:bg}

\subsection{Features of our toposes}
\label{ssec:features}

In \cite{MLM}, one can find exercises (at the end of Chapter I, for example) for identifying topos-theoretic structures in toposes of the form $\Setswith{M}$, and more generally in presheaf toposes; we recall some of this structure here.

Being a Grothendieck topos, the category of right $M$-sets has all limits and colimits. Since the functor $U$ from \eqref{eq:can-point-and-global-sections} preserves both limits and colimits, it follows that colimits and limits can be computed on underlying sets. We will need the following notation:
\begin{itemize}
\item $0$ for the initial object, i.e.\ the empty right $M$-set;
\item $1$ for the terminal object, i.e.\ the right $M$-set with one element;
\item $A \sqcup B$ for the coproduct (disjoint union) of two right $M$-sets $A$ and $B$;
\item $\bigsqcup_{i \in I} X_i$ for the coproduct (disjoint union) of a family of right $M$-sets $\{X_i\}_{i \in I}$;
\item $\colim_{i \in I} X_i$ for the colimit of a diagram $\{X_i\}_{i \in I}$.
\end{itemize}

For two right $M$-sets $X$ and $Y$, we write
\begin{equation*}
\HOM_M(X,Y) = \Hom_{\Setswith{M}}(X,Y)
\end{equation*}
to denote the set of morphisms from $X$ to $Y$; this is consistent with the earlier notation when $X$ happens to also be equipped with a compatible left action of some other monoid.

Toposes are cartesian closed, which is to say that for each pair of objects $P,Q$ in a topos $\Ecal$ there is an \textbf{exponential object} $Q^P$ such that for any third object $X$, we have an isomorphism 
\begin{equation} \label{eq:exponential-adjunction}
\Hom_{\Ecal}(X\times P,Q) \cong \Hom_{\Ecal}(X,Q^P) 
\end{equation}
natural in $X$ and $Q$, i.e.\ the functor $(-)^P$ is right adjoint to $- \times P$. In particular, for $X = Q^P$, the identity map on the right hand side corresponds on the left hand side to the \textbf{evaluation map} 
\begin{equation} \label{eq:evaluation-map}
\mathrm{ev}: Q^P \times P \longrightarrow Q.
\end{equation}

In $\Ecal = \Setswith{M}$, for two $M$-sets $P,Q$, the exponential $Q^P$ has as underlying set
\begin{equation}
\label{eq:exp1}
\HOM_M(M\times P,Q),
\end{equation}
with right $M$-action defined by $(f \cdot m)(n,p) = f(mn,p)$, for $f \in Q^P$, $m,n \in M$, $p \in P$, similarly to \eqref{eq:homaction} above (compare \cite[Proposition I.6.1]{MLM}). The evaluation map is then given by
\begin{align*} 
\HOM_M(M\times P,Q) \times P & \to \phantom{()}Q \\
(f,p) \phantom{()}& \mapsto f(1,p).
\end{align*}

If $F:\Fcal \to \Ecal$ is a functor preserving binary products, then there is a natural comparison morphism 
\begin{equation}\label{eq:theta}
\theta_{P,Q}: F(Q^P) \to F(Q)^{F(P)},
\end{equation}
obtained by applying $F$ to the evaluation map (\ref{eq:evaluation-map}), and then transposing back across the product-exponential adjunction in $\Ecal$. If $\theta_{P,Q}$ is an isomorphism for every pair of objects $P,Q$, then we say that $F$ is \textbf{cartesian-closed}, or that $F$ \textbf{preserves exponentials}. Note that the inverse image functor of a locally connected geometric morphism is always cartesian closed, by \cite[Proposition C3.3.1]{Ele}, so that in particular $\Delta$ is cartesian closed. The condition of cartesian-closedness can be weakened in two directions: either by restricting to a smaller collection of pairs $P,Q$ of objects on which $\theta_{P,Q}$ is required to be an isomorphism, or by asking that $\theta_{P,Q}$ have a property weaker than being an isomorphism. We discuss some cases of the former, applied to the connected components functor $C$, in Section \ref{ssec:exponential} and cases of the latter in Section \ref{ssec:trivial}.

Toposes also have subobject classifiers. That is, there is an object $\Omega_{\Ecal}$ in a topos $\Ecal$ equipped with a subobject $\top:1 \hookrightarrow \Omega_{\Ecal}$ such that every subobject $S \hookrightarrow A$, for $A$ an object of the topos, is the pullback of $\top$ along a unique `classifying morphism' $s: A \to \Omega_{\Ecal}$.
\begin{equation}
\begin{tikzcd}
S \ar[r] \ar[dr, phantom, "\lrcorner", very near start] \ar[d, hook] & 1 \ar[d,"{\top}", hook] \\
A \ar[r, "s"'] & \Omega_{\mathcal{E}}
\end{tikzcd}
\end{equation}

The subobject classifier of $\Ecal = \Setswith{M}$ is the set $\Omega$ of right ideals of $M$ equipped with the inverse image action for left multiplication, which for a right ideal $I$ and $m \in M$ is defined by $I\cdot m = \{m'\in M \mid mm' \in I\}$. The morphism $\top: 1 \to \Omega$ identifies $M \in \Omega$ as the largest ideal of $M$, so that the subobject $S$ of an object $A$ classified by a morphism $s: A \to \Omega$ is the subset of $A$ on the elements $a$ with $s(a) = M$.

If $F:\Fcal \to \Ecal$ is a functor preserving monomorphisms, then there is a canonical morphism
\begin{equation} \label{eq:chi}
\chi:F(\Omega_{\Fcal}) \to \Omega_{\Ecal}
\end{equation}
classifying the subobject $F(\top)$ of $F(\Omega_{\Fcal})$. $F$ is said to \textbf{preserve the subobject classifier} if $\chi$ is an isomorphism. The direct image functor of a geometric morphism preserves the subobject classifier if and only if it is \textbf{hyperconnected}, by \cite[Proposition A4.6.6(v)]{Ele}, whence we deduce that $\Gamma$ preserves the subobject classifier when $\Ecal$ is $\Setswith{M}$ for any monoid $M$.

Finally, suppose $F$ preserves monomorphisms \textit{and} finite products. We say $F$ is \textbf{logical} if it preserves the subobject classifier and exponentials. A geometric morphism whose direct image is logical is automatically an equivalence (\cite[Remark A4.6.7]{Ele}), so $\Gamma$ is logical if and only if $M$ is trivial; this appears as a condition in Theorem \ref{thm:trivial}. We shall see what happens when $\Delta$ is logical in Theorem \ref{thm:atomic}. However, since $\Delta$ already has so many strong preservation properties for an arbitrary monoid $M$, always having a left and right adjoint and preserving exponentials, that theorem is the only one we identify in this chapter expressed in terms of properties of $\Delta$.

\subsection{Properties of actions}

Observe that the global sections functor $\Gamma$ of $\Setswith{M}$ can be expressed as $\HOM_M(1,-)$, for $1$ the trivial right $M$-set, so that properties of $\Gamma$ can be expressed as properties of this $M$-set. As such, we examine properties of right $M$-sets which relate to this functor. These definitions work in any topos, but for clarity we formulate them only in our special case of a topos $\Setswith{M}$ with $M$ a monoid.

\begin{dfn}
\label{dfn:projectives}
Let $B$ be a right $M$-set, and consider the functor $\HOM_M(B,-): A \mapsto \HOM_M(B,A)$. Then we say that $B$ is: 
\begin{itemize}
\item \textbf{connected} or \textbf{indecomposable} if $\HOM_M(B,-)$ preserves arbitrary (small/set-indexed) coproducts;
\item \textbf{projective} if $\HOM_M(B,-)$ preserves epimorphisms;
\item \textbf{finitely presentable} if $\HOM_M(B,-)$ preserves filtered colimits.
\end{itemize}
The definitions for left $M$-sets are analogous.
\end{dfn}

Every right $M$-set can be written as the disjoint union of its connected components. Further:
\begin{proposition} \label{prop:connected-objects}
Let $X$ be a right $M$-set. Then the following are equivalent:
\begin{enumerate}
\item $X$ is connected/indecomposable;
\item $\HOM_M(X,-)$ preserves binary coproducts and the initial object;
\item $X$ is non-empty, and $X \cong X_1 \sqcup X_2$ implies that either $X_1$ is empty or $X_2$ is empty (compare Definition \ref{dfn:indecproj} above);
\item $C(X)=1$, with $C$ the connected components functor of (\ref{eq:can-point-and-global-sections}).
\end{enumerate}
\end{proposition}
\begin{proof}
The equivalences ($1 \Leftrightarrow 2 \Leftrightarrow 3$) hold in any infinitary extensive category by \cite[Theorem 2.1]{janelidze}. Grothendieck toposes are infinitary extensive, see e.g.\ \cite[\S4.3]{carboni-vitale}. The equivalence $(3 \Leftrightarrow 4)$ holds for locally connected Grothendieck toposes, and is discussed in \cite[around Lemma C3.3.6]{Ele}.
\end{proof}

In particular, the initial object is not indecomposable. Note that in the semigroup literature, $M$-sets are sometimes assumed to be non-empty by definition, see e.g.\ \cite{MAC}. We have not followed this convention, because it prevents the category of right $M$-sets from being a topos.

A right $M$-set will be called \textbf{free} if each connected component is isomorphic to $M$ (with right $M$-action given by multiplication). Free and projective right $M$-sets have a relationship which is familiar from the corresponding properties of modules for a ring:
\begin{proposition}
For $M$ a monoid and $P$ a right $M$-set, the following are equivalent:
\begin{enumerate}
\item $P$ is projective;
\item $P$ is a retract of a free right $M$-set;
\item $P \cong \bigsqcup_{i \in I} e_i M$ for some family $\{e_i\}_{i \in I}$ of idempotents in $M$. 
\end{enumerate}
\end{proposition}
\begin{proof}
See e.g.\ \cite[III, 17]{MAC}.
\end{proof}

Meanwhile, the definition of finite presentability given above agrees with the one from universal algebra:
\begin{proposition}[{Cf. \cite[Corollary 3.13]{adamek-rosicky}}] \label{prop:finitely-presentable}
A right $M$-set $X$ is finitely presentable if and only if we can write $X$ as the colimit of a diagram
\begin{equation*}
\begin{tikzcd}
F \ar[r,shift left=1,"{a}"] \ar[r,shift right=1,"{b}"'] & F'
\end{tikzcd}
\end{equation*}
for some finitely generated free right $M$-sets $F$ and $F'$ and morphisms $a,b$.
\end{proposition}

Next, observe that an alternative expression for the connected components functor is $C(X) = X \otimes_M 1$, where $1$ is the trivial left $M$-set. Thus properties of $C$ can be expressed as properties of this trivial left $M$-set. We therefore present some properties of left $M$-sets relating to the functor they induce via tensoring.

\begin{dfn}
\label{dfn:flats}
Let $B$ be a non-empty left $M$-set, and let $F : \Setswith{M} \to \Set$ be the functor $A \mapsto A \otimes_M B$. By Proposition \ref{prop:connected-objects}, $F$ preserves the terminal object if and only if $B$ is indecomposable. $F$ preserves arbitrary limits if and only if it is indecomposable projective (this follows from Lemma \ref{lem:proj}). More generally, we say that $B$ is:
\begin{itemize}
\item \textbf{monomorphism-flat} if $F$ preserves monomorphisms.
\item \textbf{finitely product-flat} if $F$ preserves finite products.
\item \textbf{product-flat} if $F$ preserves products.
\item \textbf{equalizer-flat} if $F$ preserves equalizers.
\item \textbf{pullback-flat} if $F$ preserves pullbacks.
\item \textbf{flat} if $F$ preserves finite limits.
\end{itemize}
The definitions for right $M$-sets are analogous.
\end{dfn}

It is \textit{very important} to note that this naming system differs from the naming conventions in semigroup theory literature, notably that of Bulman-Fleming and Laan in \cite{flatness}. Our terminology is the same when it comes to `finitely product-flat', `equalizer-flat' and `pullback-flat'. However, what we call `monomorphism-flat' is called `flat' in their paper. Our justification for this departure is that our naming system aligns more closely with that in the category theory literature, as we have already seen in the last chapter.

It follows from the definitions that, for example, flat left $M$-sets are pullback-flat. However, other interactions between the different notions of flatness are not so clear. We present some general facts which simplify the situation. These are (by now) well-known in category theory literature thanks to  authors such as Freyd and Scedrov in \cite{freyd-scedrov}, but were reached independently by semigroup theorists such as Bulman-Fleming in \cite{bulman-fleming}. We reproduce proofs here anyway.
\begin{proposition} \label{prop:flatness-properties}
Let $F: \Setswith{M} \to \Setswith{N}$ be a functor. Then:
\begin{enumerate}
\item Suppose that $F$ is nontrivial (i.e.\ $F(A) \neq 0$ for at least one $A$).  If $F$ preserves binary products, then it also preserves the terminal object; thus in order for a left $M$-set $B$ to be finitely product-flat, it suffices that $B \neq 0$ and $- \otimes_M B$ preserves binary products.
\item If $F$ preserves pullbacks, then it also preserves equalizers, so a pullback-flat object is equalizer-flat.
\item If $F$ preserves pullbacks and the terminal object, then it preserves all finite limits, so an object is flat if and only if it is indecomposable and pullback-flat.
\end{enumerate}
\end{proposition}
\begin{proof} \ 
\begin{enumerate} 
\item If $F$ preserves binary products, then in particular the natural map
\begin{gather*}
F(1 \times 1) \longrightarrow F(1)\times F(1) \\
x \mapsto (x,x)
\end{gather*}
is an isomorphism, and hence has at most one element.
\item This is a special case of \cite[Chapter 1, \S 1.439]{freyd-scedrov}. Suppose that $F$ preserves pullbacks, and consider a diagram
\begin{equation*}
\begin{tikzcd}
A \ar[r,"{f}",shift left=1] \ar[r,"{g}"',shift right=1] & B
\end{tikzcd}
\end{equation*}
Then we can rewrite the equalizer $E$ of this diagram as a pullback:
\begin{equation*}
\begin{tikzcd}
E \ar[r] \ar[d] \ar[dr, phantom, "\lrcorner", very near start] & B \ar[d,"\delta"] \\
A \ar[r,"{(f,g)}"'] & B \times B 
\end{tikzcd}
\end{equation*}
where $\delta$ is the diagonal map. Pullbacks are preserved by $F$, so 
\begin{equation*}
\begin{tikzcd}
F(E) \ar[r] \ar[d] \ar[dr, phantom, "\lrcorner", very near start] & F(B) \ar[d,"F\delta"'] \ar[r, "\sim"'] \ar[dr, phantom, "\lrcorner", very near start] & F(B) \ar[d, "\delta"'] \\
F(A) \ar[r,"{(Ff,Fg)}"] & F(B) \times_{F(1)} F(B) \ar[r, hook] & F(B) \times F(B)
\end{tikzcd}
\end{equation*}
is a composite of pullback squares; the right hand one being a pullback is a consequence of the fact that the diagonal $\delta$ of $F(B)$ in $\Setswith{N}$ factors through $F\delta$ by the universal property of $F(B) \times F(B)$. It follows that $F(E)$ can be identified with the equalizer of the diagram
\begin{equation*}
\begin{tikzcd}
F(A) \ar[r,"{F(f)}",shift left=1] \ar[r,"{F(g)}"',shift right=1] & F(B)
\end{tikzcd},
\end{equation*}
so $F$ preserves equalizers.
\item Suppose that $F$ preserves pullbacks and the terminal object. A binary product can be seen as a pullback over the terminal object, so $F$ preserves binary products (and in fact all finite products). Moreover, $F$ preserves equalizers by the above. As is well-known, any functor preserving finite products and equalizers preserves all finite limits. 
\end{enumerate}
\end{proof}

In the category of sets, arbitrary coproducts are pseudo-filtered in the sense of \cite{bjerrum}: they commute with connected limits, including pullbacks and equalizers. The same holds in the category of $M$-sets for $M$ a monoid, since colimits and limits are computed on underlying sets. So we get the following:
\begin{corollary}[{cf.~\cite[III.3.9]{MAC}}]
\label{crly:eqpb-flat}
Let $B = \bigsqcup_{i \in I} B_i$ be a right $M$-set, with each $B_i$ indecomposable. Then:
\begin{enumerate}
\item $B$ is equalizer-flat if and only if $B_i$ is equalizer-flat for all $i \in I$;
\item $B$ is pullback-flat if and only if $B_i$ is pullback-flat for all $i \in I$.
\end{enumerate}
Since each of the $B_i$ is indecomposable, it follows that $B$ is pullback-flat if and only if $B_i$ is flat for all $i \in I$, so that pullback-flat $M$-sets are precisely the $M$-sets that can be written as disjoint unions of flat $M$-sets.
\end{corollary}

\section{Properties of the global sections morphism}
\label{sec:1down}

\subsection{Groups and atomicity}

The first property we study is expressed in terms of the logical structure of the toposes involved.
\begin{dfn}[{Cf. \cite[C3.5]{Ele}}]
A geometric morphism is \textbf{atomic} if its inverse image functor is logical. For a general geometric morphism this implies local connectedness. We say a Grothendieck topos is \textbf{atomic} if its global sections geometric morphism to $\Set$ is.
\end{dfn}

For Grothendieck toposes with enough points, atomicity coincides with a property of the internal logic of the topos.
\begin{dfn}
\label{dfn:bool}
A topos $\Ecal$ is \textbf{Boolean} if its subobject classifier is an internal Boolean algebra, or equivalently if every subobject of an object of $\Ecal$ has a complement.
\end{dfn}

\begin{xmpl}
For a sober topological space $X$, $\Sh(X)$ has enough points, so we have by \cite[Lemma C3.5.3]{Ele} that $\Sh(X)$ is atomic if and only if it is Boolean, if and only if $X$ is a discrete topological space.
\end{xmpl}

Completing the proof of Corollary \ref{crly:Gequiv} from the last chapter, we have the following.
\begin{thm}[Conditions for $\Setswith{M}$ to be Boolean]
\label{thm:atomic}
Let $M$ be a monoid. The following are equivalent.
\begin{enumerate}
\item $\Setswith{M}$ is atomic.
\item $\Delta$ preserves the subobject classifier
\item $M$ is a group.
\item $\Setswith{M}$ is Boolean.
\end{enumerate}
\end{thm}
\begin{proof}
($1 \Leftrightarrow 2$) By definition, since we have already seen that $\Delta$ preserves exponentials.

($2 \Rightarrow 3$) The subobject classifier in $\Set$ is the two element set, often denoted $\Omega = \{\top,\bot\}$ so that the canonical subobject is the inclusion of the singleton $\{\top\}$. But then $\Omega \cong 1\sqcup1$, and so the subobject classifier is preserved by $\Delta$ if and only if $\Omega_{\Setswith{M}}$ has an underlying set with two elements, which forces every principal ideal in $M$ to contain the identity element (since the ideal must be all of $M$), and hence every element of $M$ has a right inverse, making $M$ a group.

($3 \Rightarrow 4$) If $M$ is a group, every subobject of an $M$-set is a union of orbits; the remaining orbits form the complementary subobject.

($4 \Leftrightarrow 1$) Since $\Setswith{M}$ has enough points (having a canonical surjective point), by \cite[Corollary C3.5.2]{Ele} being Boolean is equivalent to being atomic over $\Set$.
\end{proof}

It should be noted that many of the remaining properties explored in this paper either trivially hold for groups or are incompatible with the property of being a group (for example, a group with any kind of absorbing element is automatically trivial).

\subsection{Right-factorable finite generation and strong compactness}

A Grothendieck topos is called \textbf{compact} if its geometric morphism to $\Set$ is \textbf{proper}, which is to say that it preserves filtered colimits of subobjects; we shall see a formal definition of this property in Definition \ref{dfn:polished} in the next chapter. Since any hyperconnected geometric morphism is proper (cf. \cite[C3.2.13]{Ele}), $\Setswith{M}$ is always compact. However, the stronger notion of being \textbf{strongly compact} is not always satisfied.

\begin{dfn}
A Grothendieck topos $\Ecal$ is called \textbf{strongly compact} if its global sections functor $\Gamma = \Hom_{\Ecal}(1,-)$ preserves filtered colimits. More generally, a geometric morphism $f : \Fcal \to \Ecal$ is called \textbf{tidy} if $f_*$ preserves \textit{$\Ecal$-indexed} filtered colimits (which coincide with ordinary filtered colimits when $\Ecal = \Set$). Thus a Grothendieck topos is strongly compact if and only if the global sections geometric morphism is tidy.
\end{dfn}

\begin{example}
For a topological space $X$, $\Sh(X)$ is a compact topos if and only if $X$ is a compact space in the usual sense. In \cite[Example C3.4.1(a)]{Ele}, Johnstone gives an example of a compact space $X$ such that $\Sh(X)$ is not strongly compact: $X$ consists of two copies of the unit interval $[0,1]$ with the respective copies of the open interval $(0,1)$ identified, as sketched below.
\begin{center}
\begin{tikzpicture}
\coordinate (A) at (0,0);
\coordinate (B) at (2,0);
\draw [-,thick] (A) -- (B);
\node[draw,circle,inner sep=.5pt,fill] at (-.05,.1) {};
\node[draw,circle,inner sep=.5pt,fill] at (-.05,-.1) {};
\node[draw,circle,inner sep=.5pt,fill] at (2.05,.1) {};
\node[draw,circle,inner sep=.5pt,fill] at (2.05,-.1) {};
\end{tikzpicture}
\end{center}
On the other hand, if $X$ is a spectral space or coherent space\footnote{Not to be confused with the coherence spaces appearing in linear logic.} (in the terminology of Hochster \cite{hochster} or Johnstone \cite{Ele} respectively), having a base of compact open sets stable under finite intersections, then $\Sh(X)$ is strongly compact by \cite[Proposition C4.1.13 and Corollary C4.1.14]{Ele}. In particular, $\Sh(X)$ is strongly compact if $X$ is a zero-dimensional compact Hausdorff space or if $X$ is the Zariski spectrum of a commutative ring. More generally, Johnstone shows in \cite[Corollary C4.1.14]{Ele} that any compact Hausdorff space is strongly compact.
\end{example}

In the case where $M$ is a group, we have the following result, which appears in \cite[Example C3.4.1(b)]{Ele}:
\begin{prop} \label{prop:strongly-compact-groups}
For $G$ a group, the topos $\Setswith{G}$ is strongly compact if and only if $G$ is finitely generated.
\end{prop}

We will need the following definitions in our characterization of the monoids $M$ such that $\Setswith{M}$ is strongly compact.

\begin{definition}
Let $M$ be a monoid, and let $S \subseteq M$ be a subset. We say that $S$ is \textbf{right-factorable} if whenever $x \in S$ and $y \in M$ with $xy \in S$, we have $y \in S$. Dually, we may call a subset $S \subseteq M$ left-factorable if whenever $x \in M$ and $y \in S$ with $xy \in S$, we have $x \in S$.
\end{definition}

The above definitions are related to the two-out-of-three property in category theory, in the sense that if we view $M$ as a category with one object, then a submonoid of $M$ has the two-out-of-three property precisely if it is both right-factorable and left-factorable. If $M$ is a commutative, cancellative monoid, then a submonoid $S \subseteq M$ is right-factorable (or equivalently, left-factorable) if and only if $S$ is a \emph{saturated monoid} in the sense of Geroldinger and Halter-Koch \cite{geroldinger-survey}, i.e.\ $S = \mathbf{q}(S) \cap M$, where $\mathbf{q}(S)$ denotes the groupification of $S$.

If $(M_i)_{i \in I}$ is a family of right-factorable subsets of $M$, then the intersection $\bigcap_{i \in I} M_i$ is also right-factorable (and similarly for families of left-factorable subsets). So we can define the following:

\begin{definition}[{See also \cite[Section 2]{kobayashi}}]
Let $M$ be a monoid, and $S \subseteq M$ a nonempty subset. As is standard, we write $\langle S \rangle_M$ for the submonoid of $M$ generated by $S$. We define $\langle{S}\rangle\rangle_M$ to be the smallest right-factorable submonoid of $M$ that contains $S$; we call this the \textbf{right-factorable submonoid generated by $S$}; the extra bracket on the right is intended to evoke the asymmetric extra property this submonoid satisfies compared with $\langle S \rangle_M$. We say that $S$ \textbf{right-factorably generates} $M$ if $\langle{S}\rangle\rangle_M = M$. We call $M$ \textbf{right-factorably finitely generated} if there is a finite subset $S \subseteq M$ such that $S$ right-factorably generates $M$. Dually, we can define $\langle\langle{S}\rangle_M$ to be the left-factorable submonoid generated by $S$, i.e.\ the smallest left-factorable submonoid containing $S$.
\end{definition}

The properties of being right-factorable and right-factorably finitely generated appear in the semigroup literature under the name \textit{left unitary} and \textit{left unitarily finitely generated} see e.g.\ \cite[p.63]{howie-95}, \cite[Definition 4.38]{MAC}. We prefer to employ terms which convey the elementary notion in terms of elements of the monoid, viewed as morphisms.

The submonoid right-factorably generated by any subset can be computed inductively; a functionally identical construction is given by Kobayashi in \cite[before Proposition 2.4]{kobayashi}.
\begin{lemma}
\label{lem:construct}
Given a non-empty subset $S = S_0 \subseteq M$, inductively define for $i \in \Nbb$ the subset $S_{i+1}\subseteq M$ to be $\{m \in M \mid \exists t \in \langle S_i \rangle_M, \, tm \in \langle S_i \rangle_M\}$. Then $\langle S \rangle\rangle_M$ is precisely the subset $\bigcup_{i \in \Nbb} S_i$.
\end{lemma}
\begin{proof}
We show first inductively that $S_i \subseteq \langle S \rangle\rangle_M$ for each $i$. By definition, $S_0 \subseteq \langle S \rangle\rangle_M$. Given that $S_{i-1} \subseteq \langle S \rangle\rangle_M$, to be closed under composition, $\langle S \rangle\rangle_M$ must certainly contain $\langle S_{i-1} \rangle_M$. It follows from the definition of $S_i$ that any right-factorable subset containing $S_{i-1}$ must contain $S_i$, as claimed.

Conversely, $\bigcup_{i \in \Nbb} S_i$ is right-factorable: given $x,y$ in the union, they are both contained in some $S_i$ for some sufficiently large index $i$, and hence $xy \in S_{i+1} \subseteq \bigcup_{i \in \Nbb} S_i$. Similarly, given $t \in S_i$, $m \in M$, with $tm \in \bigcup_{i \in \Nbb} S_i$, we have $tm \in S_j$ for some $j \geq i$, and hence $m \in S_{j+1}$. Thus we are done.
\end{proof}

\begin{rmk}
\label{rmk:leftOre}
If $S_0$ is a subset, or more particularly a submonoid, satisfying the left Ore condition (the dual of Definition \ref{dfn:rOre} below), we find that $S_2 = S_1$: for $m \in S_2$ we have $t = t_1 \cdots t_k$ and $s = s_1 \cdots s_l$ in $\langle S_1 \rangle_M$ such that $tm = s$ and moreover $x_1, \dotsc, x_k$, $y_1, \dotsc, y_l$ in $S_0$ with $x_it_i$ and $y_js_j$ in $\langle S_0 \rangle_M$ for each $i$ and $j$. The left Ore condition allows us to inductively construct from these an element $z \in S_0$ such that $zs$ and $zt$ are both members of $\langle S_0 \rangle_M$, whence $m \in S_1$. Thus the construction described in Lemma \ref{lem:construct} terminates after a single step.
\end{rmk}

\begin{lemma} \label{lem:quotient-rffg}
Let $S \subseteq M$ be a non-empty subset and let $\sim_S$ be the right congruence generated by the relations $x \sim 1$ for $x \in S$. Then:
\begin{equation*}
\langle{S}\rangle\rangle_M = \{ m \in M : m \sim_S 1 \}.
\end{equation*}
\end{lemma}
\begin{proof}
$\subseteq$. It is enough to show that $\{ m \in M : m \sim_S 1 \}$ is a right-factorable submonoid. If $m \sim_S 1$ and $m' \sim_S 1$ then $mm' \sim_S m'$, and by transitivity $mm' \sim_S 1$. Further, if $m \sim_S 1$ and $mm' \sim_S 1$ then $mm' \sim_S m'$ so it follows that $m' \sim_S 1$.

$\supseteq$. Suppose $m \in M$ with $m \sim_S 1$ but $m \notin \langle{S}\rangle\rangle_M$. Then since we can decompose the relation $m \sim_S 1$ into a zigzag of one-step relations of the form $m' \sim xm'$ with $x \in S$, there must be some such relation with $xm' \in \langle{S}\rangle\rangle_M$ and $m' \notin \langle{S}\rangle\rangle_M$ or $xm' \notin \langle{S}\rangle\rangle_M$ and $m' \in \langle{S}\rangle\rangle_M$. Both cases lead to a contradiction.
\end{proof}

With the above definitions, it is possible to generalize Proposition \ref{prop:strongly-compact-groups} to arbitrary monoids $M$, by adapting Johnstone's proof \cite[C3.4.1(b)]{Ele}.
\begin{theorem}[Conditions for $\Setswith{M}$ to be strongly compact] \label{thm:strongly-compact}
For a monoid $M$, the following are equivalent:
\begin{enumerate}
\item $\Gamma$ preserves filtered colimits, i.e.\ $\Setswith{M}$ is strongly compact.
\item The terminal right $M$-set $1$ is finitely presentable.
\item $M$ is right-factorably finitely generated.
\end{enumerate}
\end{theorem}
\begin{proof}
($1 \Leftrightarrow 2$) Recalling that $\Gamma = \HOM_M(1,-)$, this equivalence is definitional (Definition \ref{dfn:projectives}); see also Proposition \ref{prop:finitely-presentable}.

($1 \Rightarrow 3$) For each finite subset $S \subseteq M$, let $\sim_S$ be the right congruence defined in Lemma \ref{lem:quotient-rffg}. Clearly $\colim_S M/{\sim}_S \cong 1$, where this is a filtered colimit over all finite subsets of $M$. Now consider the comparison map
\begin{equation*}
\beta : \colim_S \,\Gamma(M/{\sim}_S) \longrightarrow 1.
\end{equation*}
If $\Gamma$ preserves filtered colimits, then in particular $\beta$ is an isomorphism. By surjectivity of $\beta$, $M/\sim_S$ has a fixed point for some finite subset $S$ of $M$, we will denote it by $[a]$ for some representative $a \in M$. Its projection in $M/{\sim}_{S'}$ is again a fixed point, for $S' = S\cup \{a\}$. Now $[1]$ is a fixed point in $M/{\sim}_{S'}$. But this means that $m \sim_{S'} 1$ for all $m \in M$. By Lemma \ref{lem:quotient-rffg}, $S'$ right-factorably generates $M$.

($3 \Rightarrow 1$) Take a finite subset $S \subseteq M$ such that $\langle{S}\rangle\rangle_M = M$. Consider a filtered colimit $\colim_i S_i$ of right $M$-sets $S_i$. We have to show that the natural map 
\begin{equation*}
\colim_i \Gamma(S_i) \longrightarrow \Gamma(\colim_i S_i)
\end{equation*}
is a bijection. Injectivity is immediate from filteredness, so we prove surjectivity. Take an element $x$ in $\Gamma(\colim_i S_i)$. Let $x$ be represented by an element $x_i \in S_i$. For each $s \in S$, we can find an index $j$ and a structure morphism $\phi_{ji} : S_i \to S_j$ such that $\phi_{ji}(xs) = \phi_{ji}(x)$. Since $S$ is finite we can find a common index $k$ such that $\phi_{ki}(xs) = \phi_{ki}(x)$ for all $s \in S$. Since $\phi_{ki}(x)$ is fixed by all $s \in S$, it is also fixed by all elements of $\langle{S}\rangle\rangle_M = M$: indeed, employing the inductive construction of Lemma \ref{lem:construct}, if $x$ is fixed by $S_i$, then it is fixed by $\langle S_i \rangle_M$, and given $m \in M$, $t \in \langle S_i \rangle_M$ with $tm \in \langle S_i \rangle_M$, we have $xm = xtm = x$, so $x$ is fixed by $S_{i+1}$, and hence inductively by $\bigcup_{i \in \Nbb} S_i = M$. So $x$ is represented by an element of $\Gamma(S_k)$.
\end{proof}
The (dual of the) equivalence ($2 \Leftrightarrow 3$) appears in the semigroup theory work of Dandan, Gould, Quinn--Gregson and Zenab \cite[Theorem 3.10]{dandan-gould-quinn-gregson-zenab}, amongst some other equivalent conditions.

\begin{example} \label{xmpl:rwg} \ 
If $M$ has a right-absorbing element $r$ in the sense of Definition \ref{dfn:absorb} below, then $M$ is right-factorably generated by $\{r\}$ (since $r\cdot m = r$ for all $m \in M$). In particular $M$ is right-factorably finitely generated, so $\Setswith{M}$ is strongly compact. Given a non-empty set $S$, the monoid $M=\End(S)$ of (total) functions $S \to S$ has a right-absorbing element for each element $s \in S$ given by the constant function at $s$, so $M$ is right-factorably finitely generated by the above.
\end{example}

\begin{example}
We already saw that for a group $G$, the topos $\Setswith{G}$ is strongly compact if and only if $G$ is finitely generated. We can use this to produce some more examples of monoids $M$ such that $\Setswith{M}$ is \textit{not} strongly compact. Let $G$ be a group and let $M \subseteq G$ be a submonoid such that
\begin{equation*}
G = \{ a^{-1}b : a, b \in M \}.
\end{equation*}
Note that $\langle{M}\rangle\rangle_G = G$. Now if $S \subseteq M$ is a finite set right-factorably generating $M$, then $\langle{S}\rangle\rangle_G$ contains $M$ and is closed under right factors in $G$, so $\langle{S}\rangle\rangle_G = G$. It follows that $G$ is finitely generated. Therefore, for the following monoids $M$ the topos $\Setswith{M}$ is not strongly compact:  
\begin{itemize}
\item The monoid $\mathbb{N}^{\times}$ of nonzero natural numbers under multiplication;
\item The monoid $R-\{0\}$ for $R$ an infinite commutative ring without zero-divisors (noting that finitely generated fields are finite);
\item The monoid of non-singular $n\times n$ matrices over $R$, for $R$ an infinite commutative ring without zero divisors.
\end{itemize}
\end{example}

The topos $\Setswith{\mathbb{N}^{\times}}$ is (the underlying topos of) the Arithmetic Site by Connes and Consani \cite{connes-consani}. It follows that this topos is not strongly compact. In \cite{connes-consani-geometry-as}, the monoid $\mathbb{N}^{\times}_0 = \mathbb{N}^{\times} \cup \{0\}$ is considered as well, and in this case the topos $\Setswith{\mathbb{N}^{\times}_0}$ is strongly compact, by Example \ref{xmpl:rwg}.

Note that the equivalence ($1 \Leftrightarrow 2$) in Theorem \ref{thm:strongly-compact} directly generalizes to geometric morphisms $f: \Setswith{M} \to \Setswith{N}$ induced by a tensor-hom adjunction as in Proposition \ref{prop:flat}, so we have the following equivalence:

\begin{schl}
\label{schl:tidiness}
Let $M$ and $N$ be monoids, and let $B$ be a left-$M$-right-$N$-set. Suppose that the action of $M$ is flat, so that the adjunction induced by $B$ is a geometric morphism $f : \Setswith{N} \to \Setswith{M}$. Then $f$ is tidy if and only if $B$ is finitely presented as right $N$-set.
\end{schl}

\subsection{Right absorbing elements and localness}

The properties that we explore in this subsection are derived from the concept of localness.
\begin{dfn}[cf. {\cite[C3.6]{Ele}}]
\label{dfn:local}
A Grothendieck topos is called \textbf{local} if its global sections functor has a right adjoint. More generally, a geometric morphism $f: \Fcal \to \Ecal$ is \textbf{local} if its direct image functor $f_*$ has an $\Ecal$-indexed right adjoint.
\end{dfn}

To get some geometrical intuition for this concept, we mention the following criterion for localness for the topos of sheaves on a topological space.

\begin{proposition}[{cf. \cite[C1.5, pp.~523]{Ele}}]
\label{prop:localspace}
Let $X$ be a sober topological space. Then the following are equivalent:
\begin{enumerate}
\item $X$ has a \textbf{focal point}, i.e.\ a point such that the only open set containing it is $X$ itself; being sober makes such a point unique if it exists.
\item $\Sh(X)$ is a local topos.
\end{enumerate}
\end{proposition}

In particular, if $R$ is a commutative ring then the topos of sheaves on $\mathrm{Spec}(R)$ (with the Zariski topology) is a local topos if and only if $R$ has a unique maximal ideal, or in other words if and only if $R$ is a local ring.

\begin{dfn}
\label{dfn:absorb}
An element $m$ of a monoid $M$ is called \textbf{right absorbing} if it absorbs anything on its right, so $mn=m$ for all elements $n \in M$; \textbf{left absorbing} elements are defined dually. An element which is both left and right absorbing is called a \textbf{zero element}, because the $0$ of a commutative ring has this property in the ring's multiplicative monoid. This final case is of the broadest interest, but since left and right absorbing elements manifest themselves very differently in $\Setswith{M}$, we study them independently. Note that if a monoid has both left absorbing element $l$ and a right absorbing element $r$, then $l=r$ is a zero element (which is automatically unique).
\end{dfn}

This convention of handedness of absorbing elements is somewhat arbitrary, since one could alternatively take `right absorbing' to mean `absorbs all elements when multiplied on the right'. Both conventions appear in semigroup and semiring literature; we follow \cite{golan}.

A result allowing us to compare $\Gamma$ and $C$ which we have not yet had cause to introduce is the following, appearing in a more general form in Johnstone's work, \cite[Corollary 2.2(a)]{PLC}:
\begin{lemma}
\label{lem:alpha}
Let $f:\Fcal \to \Ecal$ be a connected essential geometric morphism, so the unit $\eta$ of $(f^* \dashv f_*)$ and the counit $\delta$ of $(f_! \dashv f^*)$ are isomorphisms. Write $\epsilon$ for the counit of the former and $\nu$ for the unit of the latter adjunction. Then there is a canonical comparison natural transformation $\alpha:f_* \to f_!$ which can be expressed as either $f_!\epsilon \circ (\delta_{f_*})^{-1}: f_* \to f_!f^*f_* \to f_!$ or as $(\eta_{f_!})^{-1} \circ f_*\nu: f_* \to f_*f^*f_! \to f_!$. 
\end{lemma}
A more concrete description of this transformation $\alpha:\Gamma \to C$ for $\Setswith{M}$ is that it sends a fixed point of a right $M$-set $X$ to the connected component of $X$ containing it.

From the same paper, we obtain the following term.
\begin{dfn}
\label{dfn:PLC}
A connected, locally connected geometric morphism is said to be \textbf{punctually locally connected} if the natural transformation of Lemma \ref{lem:alpha} is epic. A connected, locally connected topos is called \textbf{punctually locally connected} if its global sections geometric morphism has this property, which can be interpreted in the case of $\Setswith{M}$ as the statement `every component has at least one fixed point'. 
\end{dfn}

In \cite{lawvere-menni}, Lawvere and Menni call a geometric morphism \textbf{pre-cohesive} if it is local, hyperconnected and essential, such that $f_!$ preserves finite products; it was shown in \cite{PLC} that the global sections geometric morphism of a Grothendieck topos satisfies these properties if and only if it is punctually locally connected.

\begin{rmk}
Unlike many of the other properties in this paper, punctual local connectedness is not a geometric property: as shown in \cite[Proposition 1.6]{PLC}, it forces a connected, locally connected geometric morphism to be hyperconnected, which means that the only sober space $X$ with $\Sh(X)$ punctually locally connected is the one-point space. 
\end{rmk}

\begin{lemma}
\label{lem:rfixed}
Suppose $M$ is a monoid with a right absorbing element $r$. Then for any right $M$-set $A$, we have $\Gamma(A) = Ar$.
\end{lemma}
\begin{proof}
Clearly any element of $A$ of the form $ar$ with $r \in R$ is fixed by the action of $M$. Conversely, if $a$ is fixed by the action of $M$ then $ar = a$ for any $r \in R$. Thus every fixed point is in $Ar$.
\end{proof}

\begin{thm}[Conditions for $\Setswith{M}$ to be local]
\label{thm:rabsorb}
Let $M$ be a monoid, $\Gamma:\Setswith{M} \to \Set$ the global sections functor, $C:\Setswith{M} \to \Set$ the connected components functor and $\alpha:\Gamma \to C$ the natural transformation of Lemma \ref{lem:alpha}. The following are equivalent:
\begin{enumerate}
	\item $M$ has a right absorbing element.
	\item $\Gamma$ preserves epimorphisms, i.e.\ the right $M$-set $1$ is projective.
	\item $\Gamma$ has a right adjoint, $\gamma$. That is, $\Setswith{M}$ is local over $\Set$.
	\item $\alpha:\Gamma \to C$ is epic. That is, $\Setswith{M}$ is punctually locally connected (equivalently, pre-cohesive). \label{item:rab6}
	\item $C$ is full. \label{item:rab7}
	\item $\Gamma$ reflects the initial object, i.e.\ every non-empty right $M$-set has a fixed point.
	\item The category of points of $\Setswith{M}$ has an initial object.
\end{enumerate}
\end{thm}
\begin{proof}
($1 \Leftrightarrow 2$) Since $1$ is indecomposable, it is projective if and only if there is an idempotent $e \in M$ such that $1 = eM$, but the latter equality holds if and only if $e$ is a right absorbing element.

($2 \Leftrightarrow 3$) $\Gamma = \HOM_M(1,-)$ has a right adjoint if and only if it preserves colimits, by the Special Adjoint Functor Theorem. This is the case if and only if $1$ is (indecomposable) projective.

($2 \Rightarrow 4$) Since this geometric morphism is hyperconnected and locally connected, it is punctually locally connected if (and only if) $\Gamma$ preserves epimorphisms, by \cite[Lemma 3.1(ii)]{PLC}. Alternatively, for $X$ a right $M$-set, observe that the unit $X \to \Delta C(X)$ is epic. If $\Gamma$ preserves epimorphisms, then in particular $\Gamma(X) \to \Gamma \Delta C(X) \cong C(X)$ is an epimorphism.

($4 \Leftrightarrow 5$)${}^*$ By the axiom of choice, any epimorphism in $\Set$ (in particular any component of $\alpha$) splits. Given a function $g:C(X) \to C(Y)$, we therefore obtain a morphism $X \to Y$ lifting $g$ by sending every $x \in X$ to the fixed point $\alpha_Y^{-1}\circ g([x])$, where $\alpha_Y^{-1}:C(Y) \to \Gamma(Y)$ is a splitting for $\alpha_Y$ and $[x]$ is the connected component of $X$ containing $x$, so $C$ is full. Conversely, by standard adjunction arguments, $C$ is full if and only if the unit $\nu$ of the adjunction $(C \dashv \Delta)$ is component-wise a split epimorphism. $C$ being full therefore makes each component of $\nu$, and hence of $\Gamma \nu$ and $\alpha = (\eta_{C})^{-1} \circ \Gamma\nu$ a split epimorphism.

($4 \Rightarrow 6$) Since $C$ reflects the initial object by inspection, $\alpha:\Gamma \to C$ being an epimorphism ensures that $\Gamma(X)$ is non-empty whenever $C(X)$ is.

($6 \Rightarrow 1$) Consider $M$ as a right $M$-set under multiplication. 

($3 \Rightarrow 7$) This holds for any topos by \cite[C3.6, Theorem 3.6.1]{Ele}.

($7 \Rightarrow 1$) Let $A$ be an initial object in the category of points of $\Setswith{M}$, i.e.\ in the category of flat left $M$-sets. Let $f : A \to M$ be the unique morphism to $M$. Note that flatness implies indecomposability, so in particular $A$ is non-empty. This means we can take a morphism $g: M \to A$. Because $A$ is initial, $gf$ is the identity, so $A$ is a retract of $M$, i.e.\ $A = Me$ for some idempotent $e \in M$. The morphisms of left $M$-sets $Me \to M$ correspond to the elements of $eM$, so if $A$ is initial, then $e$ is right absorbing.
\end{proof}

The equivalence between conditions $(1)$ and $(2)$ of Theorem \ref{thm:rabsorb} is shown by Bulman-Fleming and Laan in \cite[Corollary 3.6]{flatness}, and by Kilp \textit{et al.} in \cite[Proposition III.17.2(4)]{MAC}. As the proof suggests, several of the conditions are shown to be general consequences of one another in \cite{PLC}.

\begin{remark}
In the above theorem, we could replace $(2)$ with the statement $\Gamma$ preserves pushouts, or coequalizers, or reflexive coequalizers. Indeed, because $1$ is indecomposable, $\Gamma = \HOM_M(1,-)$ preserves coproducts, so if it preserves pushouts then it also preserves coequalizers (in particular reflexive coequalizers). Conversely, any epimorphism can be written as a reflexive coequalizer, so if $\Gamma$ preserves reflexive coequalizers then it preserves epimorphisms as well.
\end{remark}

\subsection{The right Ore condition and de Morgan toposes}
\label{ssec:rightOre}

We can view the preceding sections as investigations of `projectivity' properties of the terminal right $M$-set, which generally correspond to properties of $\Gamma$. We move on in this section to the examination of the `flatness' properties of terminal right $M$-set, corresponding to properties of the connected components functor, $C$. The first and weakest of these is monomorphism-flatness of the terminal left $M$-set; one equivalent property to this is the dual of one from the last section:
\begin{dfn}
\label{dfn:coPLC}
Dualizing Definition \ref{dfn:PLC}, we say a connected, locally connected geometric morphism $f$ is \textbf{co\-punctually locally connected} if the natural transformation $\alpha$ of Lemma \ref{lem:alpha} is a monomorphism. A Grothendieck topos is called \textbf{co\-punctually locally connected} if the global sections geometric morphism is copunctually locally connected, which can be interpreted in the case of $\Setswith{M}$ as `every component has at most one fixed point'.
\end{dfn}

\begin{example}
Like punctual local connectedness, copunctual local connectedness is not a geometric property. Suppose $X$ is a connected, locally connected sober space. Consider the global sections geometric morphism $f : \Sh(X) \to \Set$. Viewing the objects of $\Sh(X)$ as local homeomorphisms, the map $\alpha_E : f_*(E) \to f_!(E)$ sends each global section $s$ to the unique connected component of $E$ that contains the image $s(X)$. If $X$ has an open subset $U$ not equal to the empty set or all of $X$, we can construct a local homeomorphism $\pi:E \to X$ by taking $E$ to be the quotient of the disjoint union of two copies of $X$ which identifies the two copies of $U$, and take $\pi$ to be the natural projection map. This $E$ is connected since $X$ is, but has exactly $2$ global sections, so $\alpha_E$ fails to be monic. Thus $\Sh(X)$ is colocally punctually connected if and only if $X$ is the one point space.
\end{example}

Another equivalent property relates to the internal logic of $\Setswith{M}$.
\begin{dfn}
A topos $\Ecal$ is said to be \textbf{de Morgan} if its subobject classifier is de Morgan as an internal Heyting algebra. Equivalently, this says that for any subobject $A$ of an object $B$ of $\Ecal$, we have $(A \rightarrow 0) \cup ((A \rightarrow 0) \rightarrow 0) = B$ in the Heyting algebra of subobjects of $B$.
\end{dfn}

\begin{example}
Let $X$ be a topological space. Then from \cite[Example D4.6.3(b)]{Ele} we know that $\Sh(X)$ is de Morgan if and only if $X$ is \textbf{extremally disconnected}, i.e.\ the closure of every open subset of $X$ is again open.

The name `extremally disconnected' is a bit misleading, since the existence of a dense point in $X$ (see Proposition \ref{prop:veryconnected} below) makes $X$ both connected and extremally disconnected. An explanation for this confusion is that extremal disconnectedness was originally only defined by Gleason for Hausdorff spaces in \cite{gleason}. In that situation, extremal disconnectedness is strictly stronger than total disconnectedness.

More generally, if $X = \bigsqcup_{i \in I} X_i$ is a disjoint union of a family $\{X_i\}_{i \in I}$ of irreducible topological spaces in the sense to be defined in Proposition \ref{prop:veryconnected}, then it follows that $X$ is extremally disconnected. If $X$ is a variety (with the Zariski topology), then $X$ is extremally disconnected if and only if each of its connected components is irreducible.
\end{example}

\begin{dfn}
\label{dfn:rOre}
A monoid $M$ is said to satisfy the \textbf{right Ore condition} or is \textit{left reversible} if for every $m_1,m_2 \in M$ there exists $m_1',m_2' \in M$ with $m_1m_1'=m_2m_2'$, or equivalently $m_1M \cap m_2M \neq \emptyset$. We employ the former terminology, as used by Johnstone in \cite[Example A2.1.11]{Ele}; the latter is employed by Sedaghatjoo and Khaksari in \cite{sedaghatjoo-khaksari} and by Kilp \textit{et al.} in \cite{MAC}.
\end{dfn}

\begin{thm}[Conditions for {$\Setswith{M}$} to be de Morgan]
\label{thm:deMorgan}
Let $M$ be a monoid, $\Omega$ the subobject classifier of $\Setswith{M}$, $\Gamma, C: \Setswith{M} \to \Set$ the usual functors and $\alpha:\Gamma \to C$ the natural transformation from Lemma \ref{lem:alpha}. The following are equivalent:
\begin{enumerate}
	\item $C$ preserves monomorphisms, so the terminal left $M$-set is monomorphism-flat.
	\item Any two non-empty sub-$M$-sets of an indecomposable right $M$-set have non-empty intersection. \label{item:deMorgan-7}
	\item $M$ satisfies the right Ore condition.
	\item Any sub-$M$-set of an indecomposable right $M$-set is either empty or indecomposable. 
	\item $C(\Omega)$ is a two element set (combined with condition $1$, this means that $C$ preserves the subobject classifier). \label{item:deMorgan-5}
	\item $\Setswith{M}$ is de Morgan.
	\item $\alpha:\Gamma \to C$ is a monomorphism; that is, $\Setswith{M}$ is copunctually locally connected.
\end{enumerate}
\end{thm}
\begin{proof}
($1 \Rightarrow 2$) Let $A$ be indecomposable and let $A_1,A_2 \subseteq A$ be two non-empty sub-$M$-sets of $A$. Since $C(A_1\cup A_2) \hookrightarrow C(A)=1$ (and the first expression has at least one element), the two subsets must intersect, by the inclusion-exclusion principle.

($2 \Rightarrow 3$) Applying ($2$) to the principal $M$-sets generated by elements $m_1,m_2 \in M$ gives the right Ore condition.

($3 \Leftrightarrow 4$) Informally, given a finite zigzag connecting elements $a_1,a_2$ in an indecomposable right $M$-set $A$, the right Ore condition allows us to inductively `push out' the spans of this zigzag to obtain elements $m_1,m_2$ with $a_1\cdot m_1 = a_2 \cdot m_2$, so any sub-$M$-set containing $a_1$ intersects any containing $a_2$; it follows that any non-empty sub-$M$-set is indecomposable. Conversely, the union of a pair of principal ideals in $M$ being indecomposable means they must intersect, else we could not construct a connecting zigzag.

($4 \Rightarrow 1$) A subobject of an $M$-set $A$ is a coproduct of subobjects of the indecomposable components of $A$. In particular, ($4$) ensures that the image under $C$ of an inclusion of sub-$M$-sets is monic.

($3 \Leftrightarrow 5$) Recall from Section \ref{ssec:features} that the subobject classifier $\Omega$ of $\Setswith{M}$ is the collection of right ideals of $M$. Given a non-empty ideal $I$ of $M$ and $m \in I$, we have $I \cdot m = M$, so the connected component of $\Omega$ containing $M$ also contains all of the non-empty ideals. Since $\emptyset$ is a fixed point of $\Omega$, $C(\Omega)$ has two elements if and only if $I \cdot m$ is non-empty whenever $I$ is for every $m \in M$ (otherwise $C(\Omega)$ is a singleton). Since the action of $M$ on $\Omega$ is order-preserving, it is necessary and sufficient that this is true for the principal ideals, and $m_1M \cdot m \neq \emptyset$ if and only if $m_1M \cap mM \neq \emptyset$.

($3 \Leftrightarrow 6$) By \cite[Example D4.6.3(a)]{Ele}, the topos of presheaves on a category $\Ccal$ is de Morgan if and only if $\Ccal$ satisfies the right Ore condition, which reduces to Definition \ref{dfn:rOre} when $\Ccal$ has a single object. 

($1 \Rightarrow 7$) Since the counit of a hyperconnected geometric morphism is monic \cite[Proposition A4.6.6]{Ele}, we have $\alpha_X = C\epsilon_X \circ \left(\delta_{\Gamma(X)}\right)^{-1}$ is monic when $C$ preserves monomorphisms.

($7 \Rightarrow 3$) Let $m_1,m_2 \in M$. Generate a right congruence of $M$ from the relations $m_1n \sim m_1n'$ and $m_2n \sim m_2n'$ for all $n,n' \in M$. The quotient of $M$ by this congruence is indecomposable, so has at most one fixed point since $\alpha$ is monic; in particular, the equivalence classes represented by $m_1$ and $m_2$ are fixed and so must be equal. That is, $m_1M \cap m_2M \neq \emptyset$.
\end{proof}
The equivalences ($1 \Leftrightarrow 3$) and a partial form of ($3 \Leftrightarrow 4$) appear respectively as \cite[Exercise 12.2(2) and Exercise 11.2(2)]{MAC}.

\begin{remark}
\label{rmk:sufficiently}
Following \cite{lawvere-menni}, an object $X$ of a topos $\Ecal$ is called \textbf{contractible} if $X^A$ is connected for all objects $A$ of $\Ecal$. Lawvere and Menni say a pre-cohesive topos is \textbf{sufficiently cohesive} if every object can be embedded in a contractible object. Equivalently, a pre-cohesive topos is sufficiently cohesive if and only if the subobject classifier is connected, see \cite[Proposition 4]{cohesion}. As we observed in the proof of ($3 \Leftrightarrow 5$) above, the subobject classifier of $\Setswith{M}$ is connected if and only if $M$ \textit{does not} satisfy the right Ore condition. It follows easily that $\Setswith{M}$ is sufficiently cohesive if and only if $M$ has at least two right-absorbing elements.
\end{remark}

Since any group satisfies the right Ore condition, the equivalent properties of Theorem \ref{thm:deMorgan} are satisfied when $M$ is a group; we shall see further special cases while investigating stronger properties in subsequent sections. Since every monomorphism in a topos is regular, $C$ preserving equalizers implies $C$ preserves monomorphisms. However, preserving equalizers is a strictly stronger condition:
\begin{example}[Connected components functor can preserve monomorphisms but not equalizers]
Consider the commutative monoid $M = \mathbb{N}$ of natural numbers under addition. Clearly, $\mathbb{N}$ satisfies the right Ore condition, so $\Setswith{\mathbb{N}}$ is de Morgan. We prove that the functor $C$ does not preserve equalizers in this case. Consider the diagram of right $\mathbb{N}$-sets
\begin{equation} \label{eq:example-dm-not-tc}
\begin{tikzcd}
\mathbb{N} \ar[r,shift left=1,"{\id}"] \ar[r,shift right=1,"{s}"'] & \mathbb{N},
\end{tikzcd}
\end{equation}
with $s$ the successor map, i.e.\ $s(n) = n+1$. Then the equalizer of this diagram is empty. But applying the connected components functor to (\ref{eq:example-dm-not-tc}) gives
\begin{equation*}
\begin{tikzcd}
\{\ast \} \ar[r,shift left=1,"{C(\id)}"] \ar[r,shift right=1,"{C(s)}"'] & \{\ast \}
\end{tikzcd}
\end{equation*}
with $C(\id)$ and $C(s)$ both the identity map. So the equalizer of this diagram is $\{\ast\}$, which does not agree with $C(0) = 0$.
\end{example}

We shall see in Theorem \ref{thm:totally} the extent to which preservation of equalizers is stronger than preservation of monomorphisms for $C$. We first examine the conditions under which the connected components functor preserves finite products.

\subsection{Spans and strong connectedness}
\label{ssec:strong-connectedness}

\begin{dfn}
A geometric morphism $f:\Fcal \to \Ecal$ is called \textbf{strongly connected} if it is locally connected and its left adjoint $f_!$ preserves finite products.
\end{dfn}

If $f$ is strongly connected, then in particular $f_!(1)=1$, so $f$ is connected. This means that $f$ is strongly connected if and only if it is connected, locally connected, such that $f_!$ preserves binary products.

One justification for this name is a geometric one, but as we shall see in Proposition \ref{prop:veryconnected}, it coincides with a stronger property, called total connectedness, for toposes of the form $\Sh(X)$. Therefore we give some justification in terms of presheaf toposes in Example \ref{xmpl:strongconnect}.

Let $\Dcal$ be a small category. Recall that colimits of shape $\Dcal$ commute with finite products in $\Set$ if and only if $\Dcal$ is a \textbf{sifted category}; we refer the reader to the survey on sifted colimits \cite{Sifted} for background on these. Concretely, siftedness may be expressed as the requirement that for each pair of objects $D_1,D_2$ in $\Dcal$, the category $\mathrm{Cospan}(D_1,D_2)$ of cospans on this pair of objects is non-empty and connected. More abstractly, siftedness may be characterized by the diagonal functor $\Dcal \to \Dcal \times \Dcal$ being a final functor.

\begin{example}
\label{xmpl:strongconnect}
The presheaf topos $\Setswith{\Dcal}$ is always locally connected. The left adjoint $f_!$ in its global sections geometric morphism sends $X:\Dcal\op \to \Set$ to its colimit. It follows by considering the terminal object of $\Setswith{\Dcal}$ that $\Setswith{\Dcal}$ is connected if and only if $\Dcal$ is connected as a category, and that $f_!$ preserves finite products if and only if $\Dcal\op$ is sifted, which is to say that the categories of spans $\mathrm{Span}(D_1,D_2)$ with $D_1,D_2 \in \Dcal$ are connected too.
\end{example}

Applying this to the case where $\Dcal$ has just one object, we arrive at the following:
\begin{proposition}[Conditions for $\Setswith{M}$ to be strongly connected]
\label{prop:strongcon}
Let $M$ be a monoid and $C:\Setswith{M} \to \Set$ the connected components functor. The following are equivalent:
\begin{enumerate}
	\item $C$ preserves finite products, i.e.\ $\Setswith{M}$ is strongly connected over $\Set$, or the terminal left $M$-set is finitely product-flat.
	\item $C$ preserves binary products.
	\item The diagonal monoid homomorphism $D: M \to M \times M$ is initial as a functor.
	\item The category $\mathrm{Span}(M)$ of spans on $M$ is connected.
	\item Any product of two indecomposable $M$-sets is indecomposable.
	\item As a right $M$-set, $M \times M$ is indecomposable. \label{item:MxMindec}
\end{enumerate}
\end{proposition}
\begin{proof}
The equivalence of $(1)$, $(2)$, $(3)$ and $(4)$ follows from the preceding discussion.

($2 \Rightarrow 5 \Rightarrow 6$) Given indecomposable $M$-sets $X,Y$, we have $C(X \times Y) \cong C(X) \times C(Y) \cong 1$, so $X \times Y$ is indecomposable; the special case $X = Y = M$ gives ($6$).

($6 \Rightarrow 4$) A morphism of spans $m: (x,y) \to (xm,ym)$ is, by inspection, the action by right multiplication of $m \in M$ on the corresponding element $(x,y) \in M \times M$.
\end{proof}

\begin{example}
\label{xmpl:groupnotstrong}
Suppose $M$ is a non-trivial group. Then $C$ fails to preserve products since $M \times M$ decomposes into orbits indexed by the elements of $M$. This can also be deduced from Proposition \ref{prop:totally2} below combined with the result \cite[Scholium A2.3.9]{Ele} that a logical functor with a left adjoint which preserves finite limits must be an equivalence.
\end{example}

\begin{lemma}
\label{lem:localstrong}
Suppose $M$ is a monoid with a right absorbing element. Then $M \times M$ is indecomposable as a right $M$-set.
\end{lemma}
\begin{proof}
At the level of toposes, one can prove that any punctually locally connected geometric morphism $f$ has leftmost adjoint $f_!$ preserving products, as Johnstone does in \cite[Proposition 2.7]{PLC}. More concretely, let $r$ be a right absorbing element and $(x,y) \in M \times M$. Then $(x,y)r = (xr,yr) = (1,yr)xr$ and $(1,1)yr = (yr,yr) = (1,yr)yr$, which gives a zigzag connecting $(x,y)$ and $(1,1)$, making $M \times M$ indecomposable as required.
\end{proof}

\subsection{Preserving exponentials}
\label{ssec:exponential}

In the context of $\Setswith{M}$ being a strongly connected topos, it makes sense to ask under which extra conditions the connected components functor is cartesian-closed; we explore this and some related conditions now for the sake of curiosity. First, we recall from \eqref{eq:exp1} in Section \ref{ssec:features} that for right $M$-sets $P$ and $Q$ we have $Q^P = \HOM_M(M \times P,Q)$ in $\Setswith{M}$. The comparison morphism $\theta_{P,Q}$ for $C$ sends the connected component $[g] \in C(Q^P)$ to the function sending a component $[p] \in C(P)$ to the component $[g(1,p)] \in C(Q)$, where $g:M \times P \to Q$, $p \in P$ and $g(1,p) \in Q$ are representative elements of their respective components.

For a general geometric morphism $f:\Fcal \to \Ecal$ to be locally connected, the inverse image functor not only needs a left adjoint $f_!$, but this adjoint must be \textbf{$\Ecal$-indexed}, which can be paraphrased as the condition that transposition from $\Fcal$ to $\Ecal$ across the adjunction $(f_! \dashv f^*)$ should preserve pullback squares of a suitable form. We refer the reader to \cite[Definition 1.2.1]{Cover} for a more precise statement of this, but we will only use the following fact.
\begin{fact}
Let $f:\Fcal \to \Ecal$ be a connected, locally connected geometric morphism. Then for every $Y$ in $\Fcal$, $X$ in $\Ecal$ we have:
\[f_!(Y \times f^*(X)) \cong f_!(Y) \times X \cong f_!(Y) \times f_!f^*(X),\]
naturally in $X$ and $Y$. Thus even when $f$ is not strongly connected, we can construct the canonical morphisms $\theta_{f^*(X),Y}: f_!(Y^{f^*(X)}) \to f_!(Y)^{f_!f^*(X)}$.
\end{fact}

\begin{dfn} \label{dfn:powers}
Let $f:\Fcal \to \Ecal$ be an essential geometric morphism satisfying $f_!(1) = 1$. Observe that since colimits are stable under pullback in a topos, we have $Y \times \coprod_{i=1}^n 1 = \coprod_{i=1}^n Y$ for any object $Y$, so that products of this form are preserved by $f_!$. Thus the canonical morphisms $\theta_{\left(\coprod_{i = 1}^n 1 \right),Y}: f_!(Y^{\left(\coprod_{i=1}^n 1\right)}) \to f_!(Y)^{\left(\coprod_{i=1}^n 1\right)}$ are well-defined. When they are isomorphisms for every $n$, we say $f$ is \textbf{finitely power-connected}.

Now suppose $f$ is a connected, locally connected geometric morphism. Then we say $f_!$ \textbf{preserves $\Ecal$-indexed powers}, or that $f$ is \textbf{power-connected} if $\theta_{f^*(X),Y}$ is an isomorphism for all objects $X \in \Ecal$ and $Y \in \Fcal$. Since $f^*(\coprod_{i \in I} 1) \cong \coprod_{i \in I} 1$ in $\Fcal$, this implies finite power-connectedness.

Finally, suppose $f$ is strongly connected. We say $f$ is \textbf{cartesian-closed-connected}\footnote{For want of a better name: given the negative results of this section, we have been unable to obtain enough useful intuition about this condition to inform a suitable choice of name.} if $f_!$ is cartesian-closed. By inspection, this implies being power-connected.
\end{dfn}

The first of these definitions is clearly implied by strong connectedness, since powers correspond to products in which all entries are equal. For the global sections morphisms of the toposes studied in this paper, however, it is equally strong:
\begin{crly}
The global sections morphism of $\Setswith{M}$ is finitely power-connected if and only if it is strongly connected, since $\Setswith{M}$ is strongly connected if and only if $M \times M = M^2$ is indecomposable by Proposition \ref{prop:strongcon}. See Proposition \ref{prop:veryconnected} below for this result in the case of toposes of the form $\Sh(X)$.
\end{crly}

For our investigation of power-connectedness, we begin with the case where $M$ fails to satisfy the right Ore condition. In \cite[Proposition 3.4]{sedaghatjoo-khaksari} it is noted that a further condition coinciding with the right Ore condition is the property that every indecomposable $M$-set has finite width, in the sense that there exists an upper bound on the length of the zigzag needed to connect any pair of elements.
This is formalized as follows:
\begin{dfn}[{\cite[Section 1]{sedaghatjoo-khaksari}}] 
\label{dfn:width}
Let $M$ be a monoid and $A$ a right $M$-set. We say that two elements $a,b \in A$ can be \textbf{connected by a scheme of length $n$} if we can find $s_1,\dots,s_n,t_1,\dots,t_n \in M$, $a_1,\dots,a_n \in A$ such that
\begin{equation*}
a = a_1 s_1,~ a_1t_1=a_2s_2,~\dots,~a_{n-1}t_n = a_n s_n,~ a_n t_n = b.
\end{equation*}
\end{dfn}

\begin{proposition}
\label{prop:cc2}
Suppose that $M$ \textbf{does not} satisfy the right Ore condition. Then the connected components functor $C$ fails to preserve $\Set$-indexed powers; in particular, $C$ is not cartesian-closed.
\end{proposition}
\begin{proof}
We shall construct an indecomposable right $M$-set $X$ such that $X^{\Delta(\Nbb)}$ is not indecomposable.

Since $M$ does not satisfy the right Ore condition, there is some pair of elements $a,b$ with $aM \cap bM = \emptyset$. Construct the $M$-set $S$ by quotienting by the equivalence relation $m \sim m'$ if and only if $m = m'$ or $m,m' \in aM$ or $m,m' \in bM$. The quotient map $M \too S$ sends the ideals $aM$ and $bM$ to distinct fixed points of $S$; abusing notation we call these fixed points $a$ and $b$ respectively.

Now let $X$ be the quotient of $\bigsqcup_{n \in \Nbb} S$ by the equivalence relation identifying the element $b$ of the $n$th copy of $S$ with the element $a$ of the $(n+1)$th copy of $S$. Denoting the image of $a$ from the $n$th copy by $a_n$, we observe that for $k \in \Nbb$, $a_0$ and $a_k \in X$ cannot be connected by a scheme of length less than $k$.

Elements of $X^{\Delta(\Nbb)}$ can be identified with $\Nbb$-indexed tuples of elements of $X$, which we notate as vectors. Take two elements $\vec{x},\vec{y} \in X^{\Delta(\Nbb)}$. If $\vec{x}$ and $\vec{y}$ can be connected by a scheme of length $1$, then this means that there is some $\vec{z} \in X^{\Delta(\Nbb)}$ and elements $s,t \in M$ such that $\vec{x} = \vec{z}s$ and $\vec{y} = \vec{z}t$. In particular, for each index $n$ we have $x_n = z_ns$ and $y_n = z_nt$ so these are connected by a scheme of length $1$ in $X$. Analogously, if $\vec{x}$ and $\vec{y}$ can be connected by a scheme of length $k$, then $x_n$ and $y_n$ can be connected by a scheme of length $k$ in $X$ for all indices $n \in \Nbb$.

Now let $x_n = a_n$ and $y_n = a_0$ for all $n \in \Nbb$. The resulting elements $\vec{x}$ and $\vec{y}$ are in separate components of $X^{\Delta(\Nbb)}$, since for any $k \in \Nbb$, an element $\vec{z}$ connected to $\vec{y}$ by a scheme of length less than $k$ cannot have $k$th component equal to $a_k$. Thus $X^{\Delta(\Nbb)}$ is not indecomposable, as claimed.
\end{proof}

For the case where $M$ satisfies the right Ore condition and $\Setswith{M}$ is power-connected, we have Theorem \ref{thm:labsorb}.\ref{item:powers} below. Finally, we examine the still stronger condition of cartesian-closed-connectedness.

\begin{proposition}
\label{prop:cc1}
Suppose that $M$ satisfies the right Ore condition and that the connected components functor $C$ is cartesian-closed. Then $M$ is trivial.
\end{proposition}
\begin{proof}
Being cartesian-closed requires that $C(M^M) \cong C(M)^{C(M)} \cong 1$, so $M^M$ is indecomposable. In particular, the projection maps $\pi_1,\pi_2:M \times M \rightrightarrows M$ in $M^M = \HOM_M(M \times M, M)$ must be in the same component under the action described above. But since $M$ has the right Ore condition, they are in the same component if and only if there are $m_1,m_2$ with $\pi_1 \cdot m_1 = \pi_2 \cdot m_2$. By inspection of the action, $\pi_2 \cdot m_2 = \pi_2$ for all $m_2 \in M$, while $\pi_1 \cdot m_1$ is independent of the second argument for any $m_1 \in M$, so the same must also be true of $\pi_2$, which forces $M$ to be trivial.
\end{proof}

\subsection{Right collapsibility and total connectedness}

We can express localness of a connected geometric morphism $f$ (Definition \ref{dfn:local}) as the existence of a right adjoint to $f$ in the 2-category of Grothendieck toposes and geometric morphisms; see \cite[Theorem C3.6.1]{Ele}. From this perspective, the dual property to localness, appearing in \cite[Theorem C3.6.14]{Ele}, is total connectedness.

\begin{dfn}
A geometric morphism $f:\Fcal \to \Ecal$ is called \textbf{totally connected} if it is locally connected and the left adjoint $f_!$ preserves finite limits.
\end{dfn}

The following is adapted from Johnstone's results, \cite[Example C3.6.17(a)]{Ele} and \cite[Lemma 1.1]{PLC}. 
\begin{proposition}
\label{prop:veryconnected}
Let $X$ be a sober topological space. Then the following are equivalent:
\begin{enumerate}
\item $\Sh(X)$ is totally connected.
\item $\Sh(X)$ is strongly connected.
\item $\Sh(X)$ is finitely power-connected.
\item Any non-empty open subset of $X$ is connected and a finite intersection of these is nonempty.
\item $X$ is \textbf{irreducible}: if $X = X_1 \cup X_2$ for closed subsets $X_1$ and $X_2$, then $X_1 = X$ or $X_2 = X$.
\item $X$ has a \textbf{dense point}, i.e.\ a point that is contained in all non-empty open sets.
\end{enumerate}
\end{proposition}
\begin{proof}
($1 \Rightarrow 2 \Rightarrow 3$) These implications are trivial.
 
($3 \Rightarrow 4$) Let $X$ be a connected, locally connected topological space such that $\Sh(X)$ is finitely power-connected with global sections geometric morphism $f : \Sh(X) \to \Set$. For two open subsets $U$ and $V$, their product as subterminal objects in $\Sh(X)$ is given by the intersection $U \cap V$. Since $f_!$ preserves finite powers, we know that the natural map $f_!(U) \to f_!(U) \times f_!(U)$ coincides with the diagonal and is an isomorphism, which shows that $U$ is connected whenever it is non-empty (since $U$ has at least one connected component). Moreover, for non-empty $U,V$, if we had $U \cap V$ empty, $U \cup V$ would be a disconnected open set, a contradiction.

($4 \Leftrightarrow 5$) Given closed subsets $X_1,X_2 \subset X$ with $X_1 \cup X_2 = X$, we have $(X - X_1) \cap (X - X_2) = \emptyset$. Given that intersections of non-empty open sets are non-empty, this means that one of $(X - X_1)$ or $(X - X_2)$ must be empty. Hence $X$ is irreducible. Conversely, given two disjoint open subsets, their complements are closed and cover $X$, so one of them must be empty. 

($4 \Rightarrow 6$) Condition ($4$) ensures that the non-empty open sets of $X$ form a completely prime filter, and hence correspond to a point contained in every open set.

($6 \Rightarrow 1$) Having a dense point ensures that $X$ is connected and locally connected. Expressing $\Sh(X)$ as the category of local homeomorphisms over $X$, each connected component of an object $E \to X$ must meet the fibre over the dense point in exactly one point. The inverse image functor of (the geometric morphism corresponding to) this point, which gives the set of points of $E$ lying in the fibre over it, is therefore isomorphic to the connected components functor $f_!$. Being the inverse image functor of a geometric morphism means that $f_!$ preserves finite limits, as required.
\end{proof}

For a ring $R$ without non-zero nilpotent elements, $\mathrm{Spec}(R)$ is irreducible (for the Zariski topology) if and only if $R$ does not have zero divisors (i.e.\ $R$ is a domain).

As in the last section, we can express the connected components functor as sending an $M$-set to its colimit as a functor. Thus $C$ preserves finite limits if and only if colimits of shape $M\op$ commute with finite limits in $\Set$. A well-known result regarding commuting limits and colimits is that colimits of shape $\Dcal$ commute with finite limits in a topos if and only if $\Dcal$ is \textbf{filtered}, which means concretely that:
\begin{itemize}
	\item $\Dcal$ is non-empty,
	\item For any pair of objects $P,Q$ of $\Dcal$, there is a cospan from $P$ to $Q$.
	\item For any pair of parallel morphisms $f,g:P \rightrightarrows Q$ there is some $h:Q \to R$ in $\Dcal$ with $hf = hg$.
\end{itemize}
One might compare these conditions to those for flat functors in Definition \ref{dfn:filtering}. We correspondingly say that $\Dcal$ is \textbf{cofiltered} if $\Dcal\op$ satisfies these conditions. Applying this to $M$ as a one-object category as in the last section, the first two conditions are trivial, so we arrive at the following definition.
\begin{definition}[see {\cite{sedaghatjoo-khaksari}, \cite[III, Definition 14.1]{MAC}}]
We say that $M$ is \textbf{right collapsible} if for any pair $m_1,m_2$ of elements of $M$, there exists $m \in M$ with $m_1m=m_2m$, that is, if $M$ is cofiltered as a category.
\end{definition}

\begin{theorem}[Conditions for $\Setswith{M}$ to be totally connected]
\label{thm:totally}
The following are equivalent, for $C$ the connected components functor:
\begin{enumerate}
	\item $C$ preserves finite limits, i.e.\ $\Setswith{M}$ is totally connected, or the terminal left $M$-set $1$ is flat.
	\item $M$ is cofiltered as a category.
	\item $M$ is right collapsible.
	\item $C$ preserves pullbacks. 
	\item $C$ preserves equalizers. 
	\item The category of points of $\Setswith{M}$ has a terminal object.
\end{enumerate}
\end{theorem}
\begin{proof}
($1 \Leftrightarrow 2 \Leftrightarrow 3$) follows from the discussion above.

($1 \Rightarrow 4 \Rightarrow 5$) The first implication is trivial, the latter is Proposition \ref{prop:flatness-properties}.

($5 \Rightarrow 3$) For $m_1,m_2 \in M$, consider the diagram
\begin{equation*}
\begin{tikzcd}
M \ar[r,shift left=1,"{m_1 \cdot}"] \ar[r,shift right=1,"{m_2 \cdot}"'] & M
\end{tikzcd}.
\end{equation*}
Because $C$ preserves equalizers, the equalizer of this diagram is non-empty, so there is an $m \in M$ with $m_1m = m_2m$.

($1 \Leftrightarrow 6$) The category of points of $\Setswith{M}$ can be identified with the category of flat left $M$-sets and homomorphisms of $M$-sets between them. Since in particular $M$ (with left action given by multiplication) is flat, the category of points has a terminal object if and only if $1$ is flat (on the left), since any other non-empty $M$-set admits more than one homomorphism from $M$.
\end{proof}

\begin{rmk}
Note that a functor into $\Set$ is flat in the sense of Definition \ref{dfn:filtering} if and only if its category of elements is filtered in the above sense; since the general definition of the category of elements is rather involved, we mention only in passing that the category of elements for the terminal left $M$-set is precisely $M\op$, which gives an alternative proof of the equivalence ($1 \Leftrightarrow 2$) in the above. 
\end{rmk}

\begin{remark}
In the above we recovered Sedaghatjoo and Khaksari's result \cite[Lemma 3.7]{sedaghatjoo-khaksari}, which is the equivalence ($1 \Leftrightarrow 3$). The equivalence ($4 \Leftrightarrow 5$) can also be seen as the statement that the left $M$-set with one element is pullback-flat if and only if it is equalizer-flat. In the semigroup literature, it is shown more generally that equalizer-flatness and pullback-flatness coincide for cyclic $M$-sets, see Kilp \textit{et al.} \cite[III,Theorem 16.7]{MAC}.
\end{remark}

In \cite[Corollary 3.8]{sedaghatjoo-khaksari}, Sedaghatjoo and Khaksari show that a monoid $M$ is right collapsible if and only if it is right Ore and $M \times M$ is indecomposable as a right $M$-set with the diagonal action. By the characterizations in preceding sections, this means that $\Setswith{M}$ is totally connected if and only if it is de Morgan and strongly connected. We give an alternative proof of this fact:
\begin{proposition}
\label{prop:totally2}
Let $M$ be a monoid. Then $\Setswith{M}$ is totally connected if and only if it is de Morgan and strongly connected.
\end{proposition}
\begin{proof}
If the connected components functor preserves finite limits, then it certainly also preserves finite products and monomorphisms. Conversely, suppose that $C$ preserves finite products and monomorphisms. To show that $C$ preserves all finite limits, it is enough to show that $C$ preserves equalizers, so suppose we are given an equalizer diagram,
\begin{equation*}
\begin{tikzcd}
E \ar[r, hook] & X \ar[r,shift left=1,"{f}"] \ar[r,shift right=1,"{g}"'] & Y.
\end{tikzcd}
\end{equation*}
Since $C(E)$ is a subobject of $C(X)$ and $C$ preserves coproducts, it is enough to show that $E$ is non-empty whenever $X$ and $Y$ are indecomposable. Because $C$ preserves finite products, $X \times Y$ is indecomposable. Therefore, consider the two (non-empty) subobjects
\begin{equation*}
\{ (x,y) \in X \times Y : f(x) = y \} \qquad \{ (x,y) \in X \times Y : g(x) = y \}.
\end{equation*}
Their intersection is isomorphic to $E$, and is non-empty by part \ref{item:deMorgan-7} of Theorem \ref{thm:deMorgan}.
\end{proof}

Proposition \ref{prop:totally2} enables us to show that when $M$ is a monoid with $\Setswith{M}$ strongly connected, $C$ need not preserve all monomorphisms. This demonstrates (for example) that the properties of a topos of being de Morgan and strongly connected are independent.
\begin{example}[Strongly connected $\not\Rightarrow$ de Morgan]
Let $M = \{1,a,b\}$ be the three-element monoid with $a$ and $b$ right-absorbing. In this case, $\Setswith{M}$ can be identified with the topos of reflexive graphs. This topos also appears in the work of Connes and Consani \cite{connes-consani-gromov}, where the objects of this topos are seen as sets equipped with a certain notion of reflexive relation. It follows from Lemma \ref{lem:localstrong} that $\Setswith{M}$ is strongly connected. However, $aM \cap bM = 0$ so $\Setswith{M}$ is not de Morgan.
\end{example}

\subsection{Left absorbing elements and colocalness}

\begin{dfn}
\label{dfn:colocal}
A more direct dual of Definition \ref{dfn:local} in the context of essential geometric morphisms is the existence of an `extra left adjoint'. We say that a locally connected Grothendieck topos $\Ecal$ with global sections morphism $f$ is \textbf{colocal} if the left adjoint $f_!$ has a further left adjoint. More generally, we might say a locally connected geometric morphism $f : \Fcal \to \Ecal$ is \textbf{colocal} if its left adjoint $f_!$ has a further $\Ecal$-indexed left adjoint.
\end{dfn}

There is a characterization of topological spaces $X$ such that $\Sh(X)$ is colocal comparable to Proposition \ref{prop:localspace}.
\begin{proposition}
\label{prop:colocalspace}
Let $X$ be a locally connected sober topological space. Then $\Sh(X)$ is a colocal topos if and only if $X$ has a (necessarily unique) \textbf{dense open point}, i.e.\ a point $x_0 \in X$ such that $\{x_0\}$ is an open set that is contained in all other open sets.
\end{proposition}
\begin{proof}
Suppose that $X$ has a dense open point $x_0$. From the proof of Proposition \ref{prop:veryconnected}, we know that the connected components functor $f_!$ of $\Sh(X)$ coincides with the inverse image functor for the geometric morphism corresponding to $x_0$. Moreover, we can construct a left adjoint to $f_!$ which maps a set $S$ to the sheaf $F_S$ defined as
\begin{equation*}
F_S(U) = \begin{cases}
S \quad & \text{if }U = \{x_0\} \\
0 \quad & \text{otherwise.}
\end{cases},
\end{equation*}
So $\Sh(X)$ is a colocal topos.

Conversely, suppose that $\Sh(X)$ is a colocal topos. Then the connected components functor preserves arbitrary limits. In particular, it preserves the terminal object and monomorphisms, which shows that each non-empty open subset of $X$ is connected. Now consider the diagram $\{U_i\}_{i \in I}$ of all non-empty open subsets of $X$. We have $C(\bigwedge_{U \in \Ocal(X)} U) = \bigwedge_{U \in \Ocal(X)} C(U) = 1$, so $\bigwedge_{U \in \Ocal(X)} U$ is a minimal non-empty open subset. Because $X$ is sober, this minimal open subset contains exactly one point.
\end{proof}

For a commutative ring $R$ without zero-divisors, we find that the topos of sheaves on $\mathrm{Spec}(R)$ (with the Zariski topology) is colocal if and only if there is an $f \in R$ such that $R[f^{-1}]$ is a field (necessarily the field of fractions of $R$). In this case, $R$ is called a Goldman domain. If we assume that $R$ is noetherian, then $R$ is a Goldman domain if and only if $R$ has only finitely many prime ideals; see \cite[Theorem 12.4]{PeteLClarkNotes}.

\begin{lemma}
\label{lem:lconnected}
Suppose $M$ is a monoid with a left absorbing element $l$. Then for any right $M$-set $A$, we have $C(A) \cong Al$, for $C$ the connected components functor.
\end{lemma}
\begin{proof}
Recall that we can express $C(A)$ as the set of equivalence classes of $A$ under the equivalence relation generated by $a \sim a\cdot m$ for $a \in A$, $m \in M$. Clearly every equivalence class has a representative of the form $a \cdot l$, so it suffices to show that if $a \sim b$ then $a \cdot l = b \cdot l$ (so this representative is unique). Indeed, for $a \sim b$ to hold there must be a finite sequence of elements of $A$, $a=a_0,a_1,\dotsc,a_k=b$ and elements $m_0,\dotsc,m_{k-1}$ and $n_1,\dotsc,n_k$ with $a_i \cdot m_i = a_{i+1} \cdot n_{i+1}$ for $i=0,\dotsc,k-1$. Then we have $a_i \cdot l = a_i \cdot m_il = a_{i+1} \cdot n_{i+1}l = a_{i+1} \cdot l$ and so inductively $a \cdot l = b \cdot l$, as required.
\end{proof}

The above lemma is the dual of Lemma \ref{lem:rfixed}. We will use it to prove $(1 \Rightarrow 6)$ in the following theorem.

\begin{thm}[Conditions for $\Setswith{M}$ to be colocal]
\label{thm:labsorb}
Let $M$ be a monoid, $C:\Setswith{M} \to \Set$ its connected components functor and $\Gamma:\Setswith{M} \to \Set$ its global sections functor. The following are equivalent:
\begin{enumerate}
	\item $M$ has a left absorbing element.
	\item $M$ has a minimal non-empty right ideal whose monoid of endomorphisms as a right $M$-set is trivial. \label{item:minimal}
	\item The category of indecomposable right $M$-sets has an initial object.
	\item Every indecomposable right $M$-set has a minimal non-empty subobject admitting no non-trivial endomorphisms. \label{item:minimal2}
	\item The category of essential points of $M$ has a terminal object.
	\item $C$ has a left adjoint $c$ which preserves connected limits.
	\item $C$ has a left adjoint $c$; that is, $\Setswith{M}$ is colocal over $\Set$.
	\item $C$ preserves arbitrary products, so the terminal left $M$-set is product-flat.
	\item $C$ preserves $\Set$-indexed powers. \label{item:powers}
\end{enumerate}
\end{thm}
\begin{proof}
($1 \Rightarrow 2$) Any non-empty right ideal of $M$ must contain the collection of left-absorbing elements, and the subset of left-absorbing elements is a right ideal of $M$. Moreover, any $M$-endomorphism of an ideal of $M$ must fix the set of left-absorbing elements (since $f(l) = f(l) \cdot l = l$ for any left-absorbing $l$), so this minimal ideal has no non-trivial endomorphisms.

($2 \Rightarrow 3$) Let $l$ be a left-absorbing element. Then the set of left-absorbing elements can be written as $lM$, in particular it is indecomposable projective as a right $M$-set. We claim that $lM$ is an initial object in the category of indecomposable right $M$-sets. Take an indecomposable right $M$-set $X$. Note that morphisms $lM \to X$ correspond to elements of $Xl \cong X \otimes_M Ml$. Because $Ml = 1$ as a left $M$-set, we have $Xl = C(X)=1$, so there is a unique morphism $lM \to X$.

($3 \Rightarrow 4$) Let $A$ be the initial object in the category of indecomposable right $M$-sets. Then for every indecomposable right $M$-set $X$, there is a unique morphism $f: A \to X$. The image $Q$ of $f$ is necessarily contained in every indecomposable subobject of $X$, which in turn implies that it is contained in every subobject. Further, the unique morphism $\pi : A \to Q$ is an epimorphism, so for any endomorphism $g : Q \to Q$ we have $g \circ \pi = \pi$, which shows $g = 1$. 

($4 \Rightarrow 1$) Consider the minimal non-empty subobject $A$ of $M$. For arbitrary $m \in M$, take the morphism $f: A \to M$, $a \mapsto ma$. Then $f(A)$ contains $A$ and $f^{-1}(A) = A$ by minimality of $A$. So $f$ defines an endomorphism of $A$, which is trivial by assumption. This shows that $ma=a$, for all $m \in M$ and all $a \in A$. In other words, each element of $A$ is a left-absorbing element. 

($1 \Leftrightarrow 5$) Recall that the category of essential points can be identified with the category of indecomposable projective left $M$-sets. If $l$ is a left absorbing element, then the terminal left $M$-set is projective, since we can write it as $1 = Ml$. Conversely, if the category of essential points has a terminal object, then $1$ is projective, so there is an idempotent $e \in M$ such that $1 = Me$. But then $e$ is a left absorbing element.

($1 \Rightarrow 6$) By Lemma \ref{lem:lconnected}, if $l$ is a left-absorbing element of $M$ then $C = \HOM_M(lM,-)$, so $C$ preserves limits. It follows from the Special Adjoint Functor Theorem that $C$ has a left adjoint $c$. Using Proposition \ref{prop:adjunction}, we know that $c(X) \cong X \times lM$, whence it preserves connected limits. Note that it preserves products if and only if $lM$ has a single element (so $l$ is a zero element).

($6 \Rightarrow 7 \Rightarrow 8 \Rightarrow 9$) These implications are trivial.

($9 \Rightarrow 1$) By Proposition \ref{prop:cc2}, if $\Setswith{M}$ is power-connected, then $M$ must satisfy the right Ore condition. Consider the product of $|M|$ copies of $M$ in $\Setswith{M}$. This is indecomposable by assumption, so there are elements $s,t \in M$ such that $(m)_{m \in M} \cdot s = (1)_{m \in M} \cdot t$, which is to say such that $ms = t$ for all $m \in M$. Taking $m=1$ we have that $s=t$ is a left-absorbing element, as required.
\end{proof}

The equivalence of conditions 1 and 7 appears as Proposition 3.9 of \cite{sedaghatjoo-khaksari}; we underline once again that their `right zero' elements are our `left absorbing' elements.

\begin{remark}
Between the preservation of finite products by $C$ in Proposition \ref{prop:strongcon} and the preservation of arbitrary products in Theorem \ref{thm:labsorb}, we can also investigate intermediate sizes of products, which can be equivalently stated as the requirement that colimits over $M\op$ commute with such products (see the discussion in Section \ref{ssec:strong-connectedness}). Surprisingly, by \cite[Theorem 3.1]{adamek-boubek-velebil} for an arbitrary small category $\Dcal$, commuting with even countably infinite products forces commutation with equalizers and hence with all equally large limits. That is, we may as well expand our considerations from products to arbitrary limits. We can also conclude that the equivalence ($7 \Leftrightarrow 8$) in Theorem \ref{thm:labsorb} is true in general for presheaf toposes.

We recall some classical terminology. Let $\kappa$ be a regular cardinal. Then a \textbf{$\kappa$-small category} is a category for which the collection of morphisms is a set of cardinality \textit{strictly smaller} than $\kappa$. Further, \textbf{$\kappa$-small limits} are limits of diagrams indexed by $\kappa$-small categories. For example, $\omega$-small limits are finite limits, for $\omega = |\Nbb|$. We say that a category $\Dcal$ is \textbf{$\kappa$-filtered} if every diagram $F : I \to \Dcal$, with $I$ a $\kappa$-small category, has a cone over it; this is equivalent to $\Dcal$-colimits commuting with $\kappa$-small limits. Dually, $\Dcal$ is \textbf{$\kappa$-cofiltered} if $\Dcal\op$ is $\kappa$-filtered.

It follows that the connected components functor preserves $\kappa$-small limits if and only if $M$ has the property that for any family $\{m_i\}_{i \in I}$ with $|I|<\kappa$, there is an $m$ such that:
\begin{equation*}
m_i m = m_j m
\end{equation*}
for all $i,j \in I$. We may call monoids with this property \textbf{right $\kappa$-collapsible}.

Suppose that $M$ is right $\kappa$-collapsible for some $\kappa > |M|$. Then there is some $l \in M$ such that $ml = m'l$ for all $m,m' \in M$, so by taking $m' = 1$ we see that $l$ is left absorbing (see also the proof of Theorem \ref{thm:labsorb}). However, it is possible that $M$ is $|M|$-collapsible but does not have a left-absorbing element. We can construct examples as follows. Let $\kappa$ be a regular cardinal. Then we can identify $\kappa$ with the set of ordinals of cardinality strictly smaller than $\kappa$. The union of two ordinals in $\kappa$ is still in $\kappa$, so the union defines a (commutative idempotent) monoid structure on $\kappa$. Now let $\{\alpha_i\}_{i \in I}$ be a family of ordinals in $\kappa$, with $|I|<\kappa$. Then the union $\alpha = \bigcup_{i \in I} \alpha_i$ is again in $\kappa$, because $\kappa$ is regular. Clearly, $\alpha_i \cup \alpha = \alpha \cup \alpha_j$ for all $i,j \in I$. So this monoid is right $\kappa$-collapsible, but it does not have a left absorbing element.
\end{remark}

\subsection{Zero elements}

\begin{dfn}
Let $f:\Fcal \to \Ecal$ be a connected, locally connected geometric morphism. We say $f$ is \textbf{bilocal} if it is both local and colocal. We say $f$ is \textbf{bipunctually locally connected} if it is both punctually and copunctually locally connected. As usual, we say a Grothendieck topos has these properties if its global sections morphism does.
\end{dfn}

In \cite[Definitions 1 and 2]{cohesion}, Lawvere introduced the terms \textbf{of quality type} and \textbf{category of cohesion} over a base category. For Grothendieck toposes over $\Set$, these respectively coincide with the bipunctual local connectedness presented above and condition \ref{item:cohesion} of Theorem \ref{thm:zero} below, so that in the case of toposes of the form $\Setswith{M}$ over $\Set$ they coincide.

\begin{thm}[Conditions for $\Setswith{M}$ to be bilocal]
\label{thm:zero}
Let $M$ be a monoid, $\Gamma$ and $C$ the usual functors and $\alpha:\Gamma \to C$ the natural transformation of Lemma \ref{lem:alpha}. The following are equivalent:
\begin{enumerate}
	\item $M$ has a zero element.
	\item $\Gamma$ is full.
	\item $\alpha$ is a split monomorphism.
	\item $\alpha$ is an isomorphism: $\Setswith{M}$ is bilocally punctually connected.
	\item $\Gamma$ has a right adjoint and $C$ has a left adjoint: $\Setswith{M}$ is bilocal.
	\item $\alpha$ is epic and $C$ preserves $\Set$-indexed powers. \label{item:cohesion}
	\item $\Gamma$ has a right adjoint and $C$ preserves monomorphisms.
	\item $C$ has a left adjoint which preserves the terminal object of $\Set$.
\end{enumerate}
\end{thm}
\begin{proof}
($1 \Rightarrow 2$) Each component of each $M$-set has a unique fixed point, obtained by acting by the unique zero element. Given a morphism $g: \Gamma(X) \to \Gamma(Y)$ we can map each element of $X$ to the fixed point in the same component and then apply $g$ to obtain an $M$-set homomorphism $X \to Y$ whose image under $\Gamma$ is $g$.

($2\Rightarrow 3 \Rightarrow 4$) Let $\epsilon$ be the counit of $(\Delta \dashv \Gamma)$. Then $\Gamma$ is full if and only if $\epsilon$ is a split monomorphism, which makes $C\epsilon$ and hence $\alpha$ a split monomorphism. In particular, $\alpha_M: \Gamma(M) \to C(M) = 1$ is a split monomorphism, meaning there is a morphism $1 \to \Gamma(M)$ in $\Set$ and so $\Gamma(M)$ is non-empty, which means $M$ has a right absorbing element and $\alpha$ must be epic and hence an isomorphism.

($4 \Rightarrow 5$) Since $\Gamma$ and $C$ are naturally isomorphic, $\Delta$ is a right and left adjoint to both of them.

($5 \Leftrightarrow 6$) By Theorem \ref{thm:rabsorb}, we have that $\alpha$ is epic if and only if $\Gamma$ has a right adjoint and by Theorem \ref{thm:labsorb} we have that $C$ has a left adjoint if and only if $C$ preserves $\Set$-indexed powers.

($5 \Rightarrow 7$) This is immediate.

($7 \Rightarrow 1$) By Lemma \ref{lem:localstrong}, $M$ being local makes $\Setswith{M}$ strongly connected, and so $C$ also preserving monomorphisms makes $\Setswith{M}$ totally connected and hence $M$ is right collapsible by Proposition \ref{prop:totally2} and Theorem \ref{thm:totally}. Applying right collapsibility to any pair of right absorbing elements shows that they must be equal, so there is a unique right absorbing element, which is thus a zero element.

($5 \Leftrightarrow 8$) One direction follows from the fact that $\Delta$ preserves $1$. Conversely, any left adjoint functor whose domain is $\Set$ is determined by the image of $1$. In particular, if the left adjoint of $C$ preserves $1$ it is naturally isomorphic to $\Delta$, and hence their right adjoints are naturally isomorphic too, so $\Delta$ is also a right adjoint to $\Gamma$.
\end{proof}

Note that ($4 \Leftrightarrow 7$) appears in a more general form as \cite[Proposition 3.7]{PLC}. Observe also that $\Gamma$ being full implies that $C$ is full, but now in an entirely constructive way!

\begin{remark}
Since any monoid with a right-absorbing element has either exactly one, which is necessarily a zero element, or at least two, we have a dichotomy between Lawvere's sufficiently cohesive toposes from Remark \ref{rmk:sufficiently} above and their toposes of quality type. This dichotomy is shown by Menni in \cite[Corollary 4.6]{menni-continuous-cohesion} to hold for arbitrary pre-cohesive presheaf toposes.
\end{remark}

\subsection{Trivializing conditions}
\label{ssec:trivial}

Many of the conditions encountered in this chapter suggest lines of investigation for further properties. As it turns out, many of these directions turn out to be dead ends, in the sense that they force the monoid to be trivial. In this section we present a variety of these conditions. As promised in an earlier section, we include some alternative weakenings of the concept of cartesian-closedness in this list.

\begin{dfn}
Suppose $F:\Fcal \to \Ecal$ is a functor between toposes which preserves products. $F$ is \textbf{sub-cartesian-closed} if the comparison morphisms $\theta_{P,Q}$ of \eqref{eq:theta} in Section \ref{ssec:features} are monomorphisms for every pair $P,Q$ of objects of $\Fcal$. If $F$ moreover preserves monomorphisms and the comparison morphism $\chi$ of \eqref{eq:chi} is a monomorphism, we say $F$ is \textbf{sublogical}.
\end{dfn}

Sublogical functors appear in the definition of \textbf{open} geometric morphisms: a geometric morphism is called open if its inverse image functor is sub-cartesian-closed. We therefore refer the reader once again to Johnstone \cite[Section C3.1]{Ele} for background on this concept, where a different but equivalent definition is given. Since any hyperconnected morphism is open, \cite[Corollary C3.1.9]{Ele}, we have that $\Delta$ is sub-logical for any monoid $M$.

\begin{thm}[Conditions for $\Setswith{M}$ to be equivalent to $\Set$]
\label{thm:trivial}
Let $M$ be a monoid, $\Setswith{M}$ its topos of right actions, and $\Gamma, \Delta, C$ the usual functors. The following are equivalent:
\begin{enumerate}
	\item $M$ is the trivial monoid.
	\item The geometric morphism $(\Delta \dashv \Gamma)$ is an equivalence.
	\item $\Gamma$ is full and faithful, or the above geometric morphism is an inclusion of toposes.
	\item $\Gamma$ is faithful.
	\item The geometric morphism $(\Delta \dashv \Gamma)$ is localic.
	\item $C$ is full and faithful.
	\item $C$ is faithful.
	\item $\Gamma$ is cartesian-closed or logical.
	\item $\Gamma$ is sub-cartesian-closed or sub-logical.
	\item $\Delta$ is logical and $C$ preserves products.
	\item $C$ is logical, cartesian-closed, sub-logical or sub-cartesian-closed. \label{item:trivial10}
	\item $\Gamma$ reflects binary coproducts or binary products or the terminal object or monomorphisms.
\end{enumerate}
\end{thm}
\begin{proof}
($1 \Leftrightarrow 2$) This is immediate after noting that the trivial monoid is the only monoid which can represent $\Set$ as a presheaf topos.

($2 \Leftrightarrow 3 \Rightarrow 4 \Rightarrow 5 \Rightarrow 2$) Since any equivalence is an inclusion and $\Gamma$ is faithful if and only if the counit of $(\Delta \dashv \Gamma)$ is epic, which is sufficient to make the geometric morphism localic. But a geometric morphism which is hyperconnected and localic is automatically an equivalence.

($2 \Rightarrow 6 \Rightarrow 7 \Rightarrow 5$) The components of an equivalence are always full and faithful. The second implication is trivial, and $C$ is faithful if and only if the unit of $(C \dashv \Delta)$ is a monomorphism, which is again sufficient to make the geometric morphism localic.

($3 \Leftrightarrow 8 \Rightarrow 9 \Leftrightarrow 3$) Since $\Gamma: \Setswith{M} \to \Set$ always preserves the subobject classifier by \cite[Proposition A4.6.6(v)]{Ele}, it is logical if and only if it is cartesian-closed, and the latter is equivalent to the global sections morphism being full and faithful by Lemma A4.2.9 there. Being sub-logical (or equivalently sub-cartesian-closed) is an apparently weaker condition, but is still equivalent to the geometric morphism being an inclusion by \cite[Proposition C3.1.8]{Ele}.

($1 \Leftrightarrow 10$) This is the content of Example \ref{xmpl:groupnotstrong}.

($2 \Rightarrow 11 \Rightarrow 1$) The components of an equivalence are always logical. Conversely, we observed in the proof of Theorem \ref{thm:deMorgan} that $C(\Omega)$ always has one or two elements, so that the comparison morphism $\chi$ for $C$ is always monic. Being sublogical is therefore equivalent to being sub-cartesian-closed. The remaining implication is contained in the proofs of Propositions \ref{prop:cc2} and \ref{prop:cc1}, since in both cases we actually showed that one of the comparison morphisms failed to be monic.

($2 \Rightarrow 12 \Rightarrow 1$) All of the functors involved in an equivalence preserve and reflect all limits and colimits. If $\Gamma$ reflects coproducts, consider two cases. First, if $\Gamma(M) = \emptyset$, then $\Gamma(1 \sqcup M) = \Gamma(1) \sqcup \Gamma(0)$ forces $1\sqcup M \cong 1\sqcup 0$, a contradiction. On the other hand, if $\Gamma(M)$ is non-empty then $M$ has some right absorbing elements; consider the $M$-set $X$ obtained by identifying them all. Then in particular $\Gamma(X) = 1 = \Gamma(1)$ and hence $\Gamma(X\sqcup 1) \cong \Gamma(1) \sqcup  \Gamma(1)$ forces $X\sqcup 1 \cong 1\sqcup 1$ and hence $X \cong 1$; it follows that every element of $M$ was right-absorbing and so $M \cong 1$ as required. The arguments for binary products, the terminal object or monomorphisms are similar: replace the reflected coproducts with the reflections of $\Gamma(M \times 1) = \Gamma(0) \times \Gamma(1)$ or $\Gamma(X \times 1) = \Gamma(1) \times \Gamma(1)$ in the first case; $\Gamma(M\sqcup 1) = \Gamma(1)$ and $\Gamma(X) = \Gamma(1)$ in the second case and with the morphisms $M \too 1$ and $X \too 1$ in the final case.
\end{proof}

\section{Conclusion}
\label{sec:conclusion}

\subsection{Summary table}
\label{ssec:table}

We summarize some of the properties and results obtained in this paper regarding the global sections functor and connected components functor in Table \ref{table:results1} and Table \ref{table:results2}.
\begin{table}[!h] 
\caption{Summary of results regarding $\Gamma$} \label{table:results1}
\begin{tabularx}{\linewidth}{Z|ZZZZ} 
\toprule[0.1em]
Topos property     & Topological property        & Monoid property                   & $\Gamma$ preserves                                          & Examples                                                                   \\ \midrule[0.1em]
Local              & $\exists$ focal point       & $\exists$ right \mbox{absorbing} \mbox{element} & all colimits (equiv. \mbox{epimorphisms})  & multiplicative monoid of a ring, $\End(S)$                        \\ \hline
Strongly compact   & e.g.\ spectral or compact Hausdorff              & right-weakly finitely generated   & filtered colimits                               & as in the cell above, as well as finitely generated monoids                \\ 
\bottomrule[0.1em]
\end{tabularx}
\end{table}

\begin{table}[h!] 
\caption{Summary of results regarding $C$} \label{table:results2}
\begin{tabularx}{\linewidth}{Z|ZZZZ} 
\toprule[0.1em]
Topos property     & Topological property        & Monoid property    & $C$ preserves                         & Examples                                                                   \\ \midrule[0.1em]
Colocal            & $\exists$ open dense point  & $\exists$ left absorbing element                                   & all limits (equiv. products)          & $\End(S)\op$                                                      \\ \hline
Totally connected  & irreducible                 & right collapsible                                     & finite limits (equiv. equalizers)     & $\End(S)\op$, $(\mathbb{N},\mathrm{max})$                         \\ \hline
Strongly connected & irreducible                 & $M \times M$ is indecomposable                            & finite products                       & $\End(S)\op$, $\End(S)$                                  \\ \hline
de Morgan          & extremally disconnected     & right Ore                                                & monomorphisms                         & $\End(S)\op$, any commutative monoid \\ \bottomrule[0.1em]
\end{tabularx}
\end{table}

\subsection{Notable omissions}
\label{ssec:omish}

This chapter is far from an exhaustive presentation of what topos-theoretic properties mean for toposes of the form $\Setswith{M}$ and the monoids presenting them: we have focused exclusively on those properties expressible in terms of the functors constituting the global sections geometric morphism. Before closing out the chapter, we give a couple of examples of other topos-theoretic properties and their translations into monoid properties.

A Grothendieck topos $\Ecal$ is said to be an \textbf{\'etendue} if there is an object $X$ of $\Ecal$ such that the slice topos $\Ecal/X$ is localic, i.e.\ such that $\Ecal/X$ is equivalent to the category of sheaves on a locale. This can alternatively be stated as the existence of an atomic geometric morphism to $\Ecal$ from a localic topos. By \cite[Lemma C5.2.4]{Ele}, a presheaf topos $\Setswith{\Ccal}$ is an \'etendue if and only if every morphism in the category $\Ccal$ is a monomorphism. For a monoid $M$, it follows that $\Setswith{M}$ is an \'etendue if and only if for all $a,b,m \in M$, the equality $ma=mb$ implies that $a = b$. Monoids with this property are usually called \textbf{left cancellative}.

An object $A$ of a topos is called \textbf{decidable} if the diagonal subobject $A \hookrightarrow A \times A$ has a complement. In particular, if $A$ is a right $M$-set, then $A$ is decidable if for two distinct elements $a,b \in A$ we have $a \cdot m \neq b \cdot m$ for all $m \in M$. Subobjects of decidable objects are again decidable. We say a topos is \textbf{locally decidable} if every object is a quotient of a decidable object. By \cite[Remark C5.4.3]{Ele}, $\Setswith{M}$ is locally decidable if and only if $\Setswith{M\op}$ is an \'etendue. So $\Setswith{M}$ is locally decidable if and only if for all $a,b,m \in M$, the equality $am=bm$ implies that $a = b$, which is to say that $M$ is \textbf{right cancellative}, dually to the above.

We shall discuss some broader directions for further investigation in Section \ref{sec:disc} of the Conclusion.

%% file: The_SGT.tex
\chapter{Supercompactly Generated Toposes}
\label{chap:sgt}

In this chapter we shall present and thoroughly investigate the class of \textit{supercompactly generated toposes}. By a supercompactly generated topos, we mean a topos with a separating set of supercompact objects (see Definition \ref{dfn:scompact}). This class includes all regular toposes (Example \ref{xmpl:regular}), as well as all presheaf toposes and some other important classes described in Proposition \ref{prop:xmpls}, and so is of general interest in topos theory. Our motivation for examining this broad class of toposes is that, as we shall eventually see in Theorem \ref{thm:hype2}, any topos admitting a \textit{hyperconnected} morphism from a supercompactly generated Grothendieck topos (such as a presheaf topos) will also be such a topos. This notably includes our motivating example of toposes of continuous actions of a topological monoid on sets, which we shall present in Chapter \ref{chap:TTMA}; such toposes typically do not fall into any of the other aforementioned classes (see Example \ref{xmpl:profin}), whence our desire to develop a general theory of these toposes ahead of time.

Since it is convenient to do so, we also study \textit{compactly generated toposes}, which are conceptually similar enough that we can prove analogous results about them with little extra work. For brevity, some of the section headings refer only to the supercompact naming conventions.

Some relevant results appear in Section 4.1 of the thesis of Bridge, \cite{TAC}\footnote{Bridge refers to our supercompactly generated toposes as `\textit{locally supercompact}'; the reason for the terminology choice of the present author is that the adverb ``locally'' is already overloaded in topos theory literature.}; this is a theoretical excursion from the main topics of that thesis, but one yield of their results is that the notion of Krull-Gabriel dimension for regular toposes respects inclusions of regular subtoposes. Where there are overlaps between the basic results of that paper and our work, the main distinction of the present chapter is the emphasis on topos-theoretic machinery: the present author avoids reasoning directly with sheaves as far as possible, which allows for more concise categorical proofs.

Our developments of the relevant classes of geometric morphisms for studying these toposes draw from and generalize the definitions and results for proper and relatively proper geometric morphisms in \cite[Chapters I and V]{Compact}. We show how those notions of compactness for geometric morphisms interact with the notion for objects inside a topos, and ultimately show that the definitions we arrive at are the `right' notion of morphisms between (super)compactly generated toposes, since they are induced by morphisms between the natural sites for these toposes.

For the site-theoretic material, we rely on the monograph \cite{Dense}, which contains general results about site representations of toposes; we quote some of those results without proof here. As such, the principal and finitely generated sites of Section \ref{sec:principal} can be appreciated as illustrative applications of the abstract techniques from that monograph.

In \cite{SPM}, one can find `\textit{$B$-sites}', which are a restricted class of the principal sites found here. These sites are subsequently constrained further so as to yield toposes of actions of `locally prodiscrete' monoids\footnote{We were unable to find a satisfactory definition of the property of local prodiscreteness in the literature.}, but their developments raise the question of which toposes are generated from less restricted sites. The present chapter provides some answers.

\subsection*{Overview}

The structure of this chapter is as follows. In Section \ref{ssec:super}, we recall the definitions of supercompact and compact objects in a topos. This leads us to formally define supercompactly and compactly generated toposes in Section \ref{ssec:scompgen}.

In Section \ref{ssec:Cs}, we turn to the full subcategories of supercompact and compact objects in a topos, presenting the structure they inherit from their ambient toposes, with a focus on monomorphisms, epimorphisms, and the classes of \textit{funneling} and \textit{multifunneling} colimits, which we introduce in Definitions \ref{dfn:funnel} and \ref{dfn:multifun}. The purpose of this investigation is to present these subcategories as canonical sites for supercompactly and compactly generated toposes in Section \ref{ssec:cansite}, which we do in Theorem \ref{thm:canon}. In Section \ref{ssec:twoval}, we investigate some extra conditions on a supercompactly generated topos which guarantee further properties of its category of supercompact objects.

In Section \ref{ssec:proper}, we examine some classes of geometric morphism whose inverse image functors preserve supercompact or compact objects, introducing the notions of \textit{pristine} and \textit{polished} geometric morphisms in analogy with proper geometric morphisms (Definition \ref{dfn:polished}), before focusing on relative versions of these properties (Definition \ref{dfn:relpolished}) which are more directly useful in our analysis. Having established these definitions, we examine how some more familiar classes of geometric morphism interact with supercompactly and compactly generated toposes: surjections and inclusions in Section \ref{ssec:surj}, then hyperconnected morphisms in Section \ref{ssec:hype}. This exploration gives us several tools for constructing such toposes, which are summarized in Theorem \ref{thm:closure}.

The focus of Section \ref{sec:principal} is a broader site-theoretic investigation. In Section \ref{ssec:stable}, we exhibit the categorical data of \textit{principal} and \textit{finitely generated} sites, which are natural classes of sites whose categories of sheaves are supercompactly and compactly generated toposes, respectively. In Section \ref{ssec:representable}, we examine the morphisms between the representable sheaves on these sites, and then show in Section \ref{ssec:quotient} how a general such site may be reduced via a canonical congruence without changing the resulting topos. We use what we have learned about these sites in Section \ref{ssec:redcat} to characterize the categories of supercompact and compact objects which were the subject of Section \ref{ssec:cansite} as \textit{reductive} and \textit{coalescent} categories, respectively (Definition \ref{dfn:reductive}), satisfying additional technical conditions. We make the correspondence between such categories and the toposes they generate explicit in Theorem \ref{thm:correspondence}, briefly examining the special cases of these categories which produce localic toposes in Section \ref{ssec:localic}. It is natural to compare reductive and coalescent categories to the well-known classes of (locally) regular and coherent categories, which we do in Section \ref{ssec:regcoh}.

Moving onto morphisms, we recall the definition of \textit{morphisms of sites} in Section \ref{ssec:morsites}, showing that, according to the class of sites under consideration, these induce the relatively pristine, polished or proper geometric morphisms introduced in Section \ref{ssec:proper}. More significantly, restricting to canonical sites, we are able to extend the correspondences of Theorem \ref{thm:correspondence} to some $2$-equivalences between $2$-categories of sites and $2$-categories of toposes. We also to examine comorphisms of sites in Section \ref{ssec:comorsite}, which provide some results about points of various classes of topos.

Finally, to ground the discussion, we present some more examples and counterexamples of reductive and coalescent sites and supercompactly and compactly generated toposes in Section \ref{ssec:xmpl}.

\section{Supercompact and compact objects}
\label{sec:supercompact}

Throughout, when $(\Ccal,J)$ is a site, we write $\ell:\Ccal \to \Sh(\Ccal,J)$ for the composite of the Yoneda embedding and the sheafification functor, assuming the Grothendieck topology is clear in context, and call the images $\ell(C)$ of the objects $C \in \Ccal$ the \textbf{representable sheaves}.

\subsection{Supercompact objects}
\label{ssec:super}

The following definitions can be found in
\cite[Definition 2.1.14]{TST}:

\begin{definition}
\label{dfn:scompact}
An object $C$ of a category $\Ecal$ is \textbf{supercompact} (resp. \textbf{compact}) if any jointly epic family of morphisms $\{A_i \to C \mid i \in I\}$ contains an epimorphism (resp. a finite jointly epic sub-family). 
\end{definition}

Clearly every supercompact object is compact. Compact objects are more widely studied, notably in the lifting of the concept of compactness from topological spaces to toposes reviewed in \cite{Compact}. Since the two classes of objects behave very similarly, we treat them in parallel. 

In a topos, we may re-express the definitions of supercompact and compact objects in terms of their subobjects. As is standard, we can further convert any statement about subobjects of an object $X$ in a topos $\Ecal$ into a statement about the subterminal object in the slice topos $\Ecal/X$.

\begin{lemma}
\label{lem:scompact}
An object $C$ of a Grothendieck topos $\Ecal$ is supercompact (resp. compact) if and only if every covering of $C$ by a family (resp. a directed family) of subobjects $A_i \hookrightarrow C$ contains an isomorphism. This occurs if and only if the global sections functor $\Gamma: \Ecal/C \to \Set$ preserves arbitrary (resp. directed) unions of subobjects.
\end{lemma}
\begin{proof}
For the first part, suppose that we are given a family (resp. a directed family) of subobjects covering a supercompact (resp. compact) object $C$. Then one of the monomorphisms involved must be epic and hence an isomorphism (resp. this collection contains a finite covering family, but the union of these subobjects is also a member of the family and must be covering). In the opposite direction it suffices to consider images of the morphisms in an arbitrary covering family resp. finite unions of these images). 

The remainder of the proof is modeled after that in \cite[C1.5.5]{Ele}. Given an object $f: A \to C$ of $\Ecal/C$ which is a union (resp. a directed union) of subobjects $A_i \hookrightarrow A \to C$ and given a global section $x: C \to A$ of $f$, we may consider the pullbacks:
\[\begin{tikzcd}
C_i \ar[r, hook] \ar[d]
\ar[dr, phantom, "\lrcorner", very near start] &
C \ar[d, "x"]\\
A_i \ar[r, hook] & A.
\end{tikzcd}\]
By extensivity, $C$ is the union of the $C_i$, and by the above one of the $C_i \hookrightarrow C$ must be an isomorphism, so that $x$ factors through one of the $A_i$, which gives the result.

Conversely, given a jointly epic family (resp. directed family) of subobjects $C_i \hookrightarrow C$ considered as subterminals in $\Ecal/C$, we may apply $\Gamma$ to see that one of them must be an isomorphism, as required.
\end{proof}

\subsection{Supercompactly generated toposes}
\label{ssec:scompgen}

The collections of supercompact and compact objects in \textit{any} Grothendieck topos are conveniently tractable:

\begin{lemma}
\label{lem:scsite}
Let $\Ecal \simeq \Sh(\Ccal,J)$ be a Grothendieck topos of sheaves on a small site $(\Ccal,J)$. Then the supercompact objects are quotients of the representable sheaves $\ell(C)$ for $C \in \Ccal$. In particular, they are indexed (up to isomorphism) by a set. Similarly, the compact objects are quotients of finite coproducts of the images
$\ell(C)$, and so (up to isomorphism) also form a set.
\end{lemma}
\begin{proof}
Given a supercompact object $Q$, since the objects $\ell(C)$ are separating in $\Ecal$, the collection of morphisms $\ell(C) \to Q$ (is inhabited and) jointly epimorphic. It follows that one such must be epimorphic.

Given a compact object $Q$, the above argument instead yields a (possibly empty) finite jointly epimorphic family of morphisms $\ell(C_i)\to Q$, which corresponds to an epimorphism $\coprod_{i\in I} \ell(C_i) \too Q$, as claimed. 
\end{proof}

Lemma \ref{lem:scsite} ensures that we can always consider the full subcategories on the supercompact (resp. compact) objects, equipped with the canonical topology induced by the topos (which shall be recalled in Definition \ref{dfn:effective} below) as an essentially small site $(\Ccal_s,J_{can}^{\Ecal}|_{\Ccal_s})$ (resp. $(\Ccal_c,J_{can}^{\Ecal}|_{\Ccal_c})$). By the Comparison Lemma, the induced canonical comparison morphism $\Ecal \to\Sh(\Ccal_s, J_{can}^{\Ecal}|_{\Ccal_s})$ is an equivalence if and only if the collection of supercompact objects is separating, and similarly for the compact case. We shall continue to use $\Ccal_s$ and $\Ccal_c$ to denote these categories in the remainder; for simplicity, we shall actually assume that we have chosen a representative set of the supercompact or compact objects, such as the quotients of representables in Lemma \ref{lem:scsite}, so that we are working with small sites.

\begin{definition}
We say a topos is \textbf{supercompactly generated} (resp. \textbf{compactly generated}) if its collection of supercompact (resp. compact) objects is separating.
\end{definition}

\begin{example}
\label{xmpl:regular}
The syntactic category of a regular theory can be recovered up to effective completion (cf. Definition \ref{dfn:regeff} below) from its classifying topos as the full category of regular objects. In general, we say that an object $X$ in a topos is \textbf{regular} if $X$ is supercompact and for any cospan
\[\begin{tikzcd}
Y \ar[r, "f"] & X & Z \ar[l, "g"']
\end{tikzcd}\]
with $Y$ and $Z$ supercompact, the pullback $Y \times_X Z$ is also supercompact. In particular, classifying toposes of regular theories are special cases of supercompactly generated toposes. The same can be said when `supercompact' is replaced by `compact' and `regular' is replaced by `coherent'. See \cite{SCCT} for this result and a more detailed discussion (note that Caramello refers to regular objects as \textit{supercoherent} objects). We shall return to examination of categories of regular and coherent objects in Section \ref{ssec:regcoh}.
\end{example}

Supercompactly generated Grothendieck toposes include several other important established classes of Grothendieck topos.

\begin{proposition}
\label{prop:xmpls}
\begin{enumerate}[label = ({\roman*})]
	\item Every atomic topos is supercompactly generated.
	\item Every supercompactly generated topos is compactly generated and locally connected.
	\item Every presheaf topos is supercompactly generated.
\end{enumerate}
\end{proposition}
\begin{proof}
For (i), recall that a Grothendieck topos is atomic if and only if it has a separating set of atoms, and these are easily seen to be supercompact. For (ii), we can similarly observe that any supercompact object is compact and indecomposable, then recall (by Theorem 2.7 of \cite{SCCT}, say) that a topos is locally connected if and only if it has a separating set of indecomposable objects.

For (iii), note that the representable presheaves are irreducible projectives so they are in particular supercompact (every jointly epic family over a representable presheaf contains a \textit{split} epimorphism).
\end{proof}

\begin{example}
\label{xmpl:profin}
Let $(M,\tau)$ be a monoid equipped with a topology. We may consider the full subcategory of the topos $[M\op,\Set]$ of right actions of $M$ on sets on those actions which are continuous with respect to the product of $\tau$ and the discrete topology on the underlying set; call this category $\Cont(M,\tau)$. In Chapter \ref{chap:TTMA}, the results of the present chapter will be used to show that this category is a supercompactly generated Grothendieck topos.

Given that it is a topos, however, we can show directly that it is supercompactly generated: the supercompact objects are the \textit{continuous principal $M$-sets}, which is to say those generated by a single element. Since we can clearly cover any continuous $M$-set with the principal sub-$M$-sets generated by its elements, these objects form a separating collection in the topos.

We can produce examples of such toposes which do not fall into any of the classes described in Proposition \ref{prop:xmpls}; see Remark \ref{rmk:sgtxmpl} in a Chapter \ref{chap:TTMA}.
\end{example}

\subsection{Categories of supercompact objects}
\label{ssec:Cs}

We now examine properties of the categories $\Ccal_s$ and $\Ccal_c$ in a general topos $\Ecal$.

\begin{lemma}
\label{lem:closed}
Let $\Ecal$ be a Grothendieck topos and let $\Ccal_s$, $\Ccal_c$ be the categories of supercompact and compact objects of $\Ecal$ respectively. These categories are closed in $\Ecal$ under quotients.
\end{lemma}
\begin{proof}
Given an epimorphism $k: D \too C$ with $D$ supercompact and a covering family over $C$, pulling back this family along $k$ we immediately conclude that one of the constituent morphisms must be an epimorphism. Thus $C$ is a member of $\Ccal_s$. The argument for $\Ccal_c$ is analogous, except that we end up with a finite family of morphisms.
\end{proof}

Lemma \ref{lem:closed} has as a consequence that when considering a covering family of supercompact or compact objects over an object $X$ of $\Ecal$, we may without loss of generality assume that the morphisms in the family are monomorphisms. That is, we may restrict attention to covering families of (super)compact \textit{subobjects} when we so choose, because a family of morphisms with common codomain in a topos is jointly epic if and only if the union of their images is the maximal subobject. 

The subcategories inherit some further structure from $\Ecal$.

\begin{corollary}
\label{crly:images}
For $\Ecal$, $\Ccal_s$, $\Ccal_c$ as in Lemma \ref{lem:closed}, $\Ccal_s$ and $\Ccal_c$ are closed under image factorizations in $\Ecal$, so that in particular they have image factorizations.
\end{corollary}
\begin{proof}
Given a morphism $C \to C'$ between supercompact (resp. compact) objects, the image object $C''$ in the factorization $C \too C'' \hookrightarrow C'$ is also supercompact (resp. compact) by Lemma \ref{lem:closed}, whence the factoring morphisms lie in $\Ccal_s$ (resp. $\Ccal_c$) since it is a full subcategory.
\end{proof}

Note that the resulting orthogonal factorization systems on $\Ccal_s$ and $\Ccal_c$ are not between all monomorphisms and all epimorphisms; only between those inherited from $\Ecal$. We spend the rest of this section deriving an intrinsic characterization of these morphisms.

\begin{definition}
\label{dfn:funnel}
We say a small indexing category $\Dcal$ is a \textbf{funnel} if it has a weakly terminal object. A \textbf{funneling diagram} in an arbitrary category $\Ccal$ is a functor $F: \Dcal \to \Ccal$ with $\Dcal$ a funnel. We call the colimit of such a diagram, if it exists, a \textbf{funneling colimit}.
\end{definition}

Given a funneling diagram $F:\Dcal \to \Ccal$, its colimit is determined by an object $C$ of $\Ccal$ equipped with an epimorphism $f: F(D_0) \too C$, where $D_0$ is a weakly terminal object of $\Dcal$, since all legs of the colimit cone factor through this one.

\begin{xmpl}
An example of a funnel is the following:
\[\begin{tikzcd}[row sep = small]
A_i \ar[dr, "f_i", shift left] \ar[dr, "f'_i"', shift right] & \\
\vdots & D. \\
A_j \ar[ur, "f_j", shift left] \ar[ur, "f'_j"', shift right] &
\end{tikzcd}\]

Consider the topos $\Setswith{M}$. We can present any right $M$-set $X$ as the colimit of a funneling diagram of the above general shape, where the weakly terminal object is sent to a coproduct of copies of $M$, and all other objects are sent to copies of $M$:
\[\begin{tikzcd}[row sep = small]
M \ar[dr, "f_i", shift left] \ar[dr, "f'_i"', shift right] & \\
\vdots & \coprod_{i \in I} M. \\
M \ar[ur, "f_j", shift left] \ar[ur, "f'_j"', shift right] &
\end{tikzcd}\]
Here, the set $I$ indexes a set of generators of $X$ and the morphisms $(f_i,f'_i)$ identify pairs of elements which are to be identified.
\end{xmpl}

Recall that a morphism $h:D \to C$ in a category $\Ccal$ is called a \textbf{strict epimorphism} if whenever another morphism $k: D \to E$ satisfies the condition that for each parallel pair $p,q: B \rightrightarrows D$ with $h \circ p = h \circ q$ we have $k \circ p = k \circ q$, it follows that $k$ factors uniquely through $h$. The dual concept appears in \cite[Theorem 4.1]{TGT}.

\begin{lemma}
In a small category $\Ccal$, a morphism $h:C' \to C$ is a strict epimorphism if and only if there exists a funneling diagram $F:\Dcal \to \Ccal$ with weakly terminal object $C'$ whose colimit is expressed by $h$.
\end{lemma}
\begin{proof}
By definition, if $h$ is a strict epimorphism, it is a colimit for the diagram consisting of all pairs of morphisms with domain $C'$ which $h$ coequalizes. Conversely, if $h$ expresses the colimit of any funneling diagram, and $k$ coequalizes all of the same parallel pairs that $h$ does, then it clearly induces a cone by composition with the morphisms of the funneling diagram, whence it has a universal factorization through $h$, as required.
\end{proof}

Notably, strict epimorphisms include isomorphisms and regular epimorphisms. Continuing with the parallel treatment of compactly generated toposes, we arrive at the following definitions.

\begin{definition}
\label{dfn:multifun}
A small indexing category $\Dcal$ is a \textbf{multifunnel} if it has a (possibly empty) finite collection of objects $D_1, \dotsc, D_n$ to which all other objects admit morphisms\footnote{A reader interested in the obvious generalizations of this concept to higher cardinalities might prefer to employ a name such as `finitely funneled' to emphasize the finitary aspect.}. A colimit of a multifunneling diagram (a diagram indexed by a multifunnel) in a category $\Ccal$ shall be called a \textbf{multifunneling colimit}, and is defined by a finite jointly epic family from the images of the objects $D_1, \dotsc, D_n$. A finite jointly epic family obtained in this way will be called a \textbf{strictly epic finite family}.
\end{definition}

\begin{lemma}
\label{lem:multi}
A category has multifunneling colimits if and only if it has finite coproducts and funneling colimits.
\end{lemma}
\begin{proof}
Clearly finite coproducts and funneling colimits are special cases of multifunneling colimits. Conversely, given a multifunneling diagram $F: \Dcal \to \Ccal$ with weakly terminal objects $F(D_1),\dotsc,F(D_n)$, consider the coproduct of these objects. Composing the morphisms in the diagram $F$ with the coproduct inclusions, we get a funneling diagram. The universal property of the coproduct ensures that the colimit of this diagram coincides with the colimit of $F$.
\end{proof}

\begin{remark}
\label{rmk:simple}
Note that we can make the further simplification, implicit in the diagram of Definition \ref{dfn:funnel}, that all of the non-identity morphisms in a funneling diagram have the weakly terminal object as their codomains, since given $t:A_i \to A_j$, there exists some morphism $f_j:A_j \to D$, and the cocone commutativity conditions for $f_j \circ t$ and $f_j$ ensure that $\lambda_i = \lambda_D \circ (f_j \circ t) = \lambda_j \circ t$ is automatically satisfied, so we may omit $t$ from the diagram. 
\end{remark}

\begin{lemma}
\label{lem:closed2}
Let $\Ecal$ be a Grothendieck topos and let $\Ccal_s$, $\Ccal_c$ be the categories of supercompact and compact objects of $\Ecal$ respectively. Then $\Ccal_s$ is closed in $\Ecal$ under funneling colimits and $\Ccal_c$ is closed in $\Ecal$ under multifunneling colimits.
\end{lemma}
\begin{proof}
Let $F:\Dcal \to \Ccal_s$ be a funneling diagram with weakly terminal object $F(D)$. In $\Ecal$, this diagram has a colimit determined by an epimorphism $f: F(D) \too C$; the colimit $C$ is supercompact by Lemma \ref{lem:closed}, as required. The argument for multifunneling colimits in $\Ccal_c$ is analogous, except that we must pull back along each member of a finite family in the proof of Lemma \ref{lem:closed} to obtain the finite covering subfamily of a given covering family over the colimit.
\end{proof}

\begin{corollary}
\label{crly:strict}
Let $\Ecal$ be a supercompactly (resp. compactly) generated topos and $\Ccal_s$, $\Ccal_c$ the usual subcategories. Then a morphism of $\Ccal_s$ is an epimorphism in $\Ecal$ if and only if it is a strict epimorphism in $\Ccal_s$, and a finite family of morphisms into $C$ in $\Ccal_c$ is jointly epic in $\Ecal$ if and only if it is a strictly epic finite family in $\Ccal_c$.
\end{corollary}
\begin{proof}
This could be deduced by checking the conditions of the dual of \cite[Proposition 4.9]{TGT}, but rather than reproducing that result, we give a direct proof.

Suppose $h: C' \too C$ is epic in $\Ecal$ with $C',C$ in $\Ccal_s$. Any epimorphism in $\Ecal$ is regular, so is the coequalizer of some pair $p,q:D \rightrightarrows C'$. Since $D$ is covered by supercompact objects, composing $p$ and $q$ with the monomorphisms $C_i \hookrightarrow D$ such that $C_i$ is in $\Ccal_s$ we obtain a funneling diagram in $\Ccal_s$ whose colimit is still $C$, as required. The argument for $\Ccal_c$ is analogous.

Conversely, the inclusion of $\Ccal_s$ and $\Ccal_c$ into $\Ecal$ preserves funneling (resp. multifunneling) colimits by Lemma \ref{lem:closed2}, whence the strict epimorphisms (resp. strictly epic finite families) from these categories are still epic in $\Ecal$.
\end{proof}

In fact, epimorphic families in $\Ccal_c$ are better behaved than those in $\Ccal_s$ in general:
\begin{lemma}
\label{lem:compactepi}
Let $\Ecal$ be any Grothendieck topos. Then a family of morphisms in $\Ccal_c$ with common codomain is jointly epic in $\Ccal_c$ if and only if it is so in $\Ecal$. In particular, when $\Ecal$ is compactly generated, every jointly epimorphic family (including every epimorphism) in $\Ccal_c$ is strict.
\end{lemma}
\begin{proof}
Observe that a family of morphisms $f_i:C_i \to C$ in a category is jointly epimorphic if and only if the diagram:
\[\begin{tikzcd}
C_i \ar[dr, phantom, "\ddots"] \ar[drr, "f_i", bend left]
\ar[ddr, "f_i"', bend right] & & \\
& C_j \ar[r, "f_j"] \ar[d, "f_j"'] & C \ar[d, equal] \\
& C \ar[r, equal] & C
\end{tikzcd}\]
is a colimit diagram. But the diagram (after removing the copy of $C$ in the lower right corner) is clearly an instance of a multifunneling colimit, so by Lemma \ref{lem:closed2} its colimit is created by the inclusion of $\Ccal_c$ into $\Ecal$, so a family is jointly epic in $\Ccal_c$ if and only if it is so in $\Ecal$, as required.
\end{proof}

Having extensively discussed the epimorphisms, we should also discuss monomorphisms in the subcategories under investigation.

\begin{lemma}
\label{lem:monocoincide}
Monomorphisms in $\Ccal_s$ and $\Ccal_c$ coincide with monomorphisms in $\Ecal$ when $\Ecal$ is supercompactly or compactly generated, respectively.
\end{lemma}
\begin{proof}
Certainly a monomorphism of $\Ecal$ lying in $\Ccal_s$ or $\Ccal_c$ is still monic, since there are fewer morphisms which it needs to distinguish in general.

Suppose $\Ecal$ is supercompactly generated and let $s : A \hookrightarrow B$ be a monomorphism in $\Ccal_s$, $Q$ an object of $\Ecal$ and $f, g : Q \rightrightarrows A$ such that $sf = sg$. Covering $Q$ with supercompact subobjects $q_i: Q_i \hookrightarrow Q$ , consider $f q_i, g q_i: Q_i \rightrightarrows A$, which are morphisms in $\Ccal_s$. These are equalized by $s$ and hence are equal for every $i$. The $q_i$ being jointly epic then forces $f = g$. Thus $s$ is monic in $\Ecal$, as claimed. The argument for $\Ccal_c$ is analogous, replacing supercompact subobjects with compact ones.
\end{proof}

Once again, we can immediately strengthen this result for $\Ccal_c$.

\begin{lemma}
\label{lem:compactmono}
For any Grothendieck topos $\Ecal$, the monomorphisms of $\Ecal$ lying in $\Ccal_c$ are regular monomorphisms there. In particular, when $\Ecal$ is compactly generated, every monomorphism in $\Ccal_c$ is regular.
\end{lemma}
\begin{proof}
If $e:C' \hookrightarrow C$ is a morphism in $\Ccal_s$ which is monic in $\Ecal$, consider its cokernel pair in $\Ecal$:
\[\begin{tikzcd}
C' \ar[r, "e", hook] \ar[d, "e"', hook] \ar[dr, phantom, "\lrcorner", very near start] &
C \ar[d, hook, "s"] \\
C \ar[r, hook, "t"'] & D.
\ar[ul, phantom, "\ulcorner", very near start]
\end{tikzcd}\]
Since a topos is an adhesive category, this is also a pullback square and so $e$ is the equalizer of $s$ and $t$. Since pushouts are multifunnel colimits, $D$ lies in $\Ccal_s$, so the same is true there.
\end{proof}

Thus we can make Corollary \ref{crly:images} more precise.
\begin{corollary}
\label{crly:orthog}
If $\Ecal$ is supercompactly generated, then $\Ccal_s$ has an orthogonal (strict epi,mono)-factorization system. More generally, if $\Ecal$ is merely compactly generated, $\Ccal_c$ has an orthogonal (epi,mono)-factorization system.
\end{corollary}

For later reference, we observe that even though the categories $\Ccal_s$ and $\Ccal_c$ need not have finite products (see Example \ref{xmpl:nonreg}), we can still extend Corollaries \ref{crly:images} and \ref{crly:orthog} with factorizations of spans through jointly monic spans. Corollary \ref{crly:orthog} is the case $I = 1$ of the following Lemma.

\begin{lemma}
\label{lem:jmonic}
Let $\{f_j: B \to A_j \mid j \in I\}$ be a collection of morphisms with common domain in $\Ccal_s$ or $\Ccal_c$. Then there exists a strict epimorphism $e: B \too R$ and morphisms $\{r_j: R \to A_j \mid j \in I\}$ which are jointly monic, such that $f_j = r_j \circ e$.
\end{lemma}
\begin{proof}
Let $e$ be the strict epimorphism obtained from the funneling colimit of the collection of all parallel pairs of morphism which are coequalized by all of the $f_j$. By definition, all of the $f_j$ factorize through it, and by construction the factors form a jointly monic family.
\end{proof}

\subsection{Canonical sites of supercompact objects}
\label{ssec:cansite}

We have now done enough work to usefully apply the proof of Giraud's theorem and obtain a canonical site of definition for a supercompactly or compactly generated Grothendieck topos.

\begin{definition}
\label{dfn:effective}
Recall that a sieve $S$ on an object $C$ of a category $\Ccal$ is \textbf{effective-epimorphic} if, when $S$ is viewed as a full subcategory of $\Ccal/C$, $C$ is the colimit of the (possibly large) diagram $D_S: S \hookrightarrow \Ccal/C \to \Ccal$ obtained by composing with the forgetful functor. A sieve generated by a single morphism $f$ is effective-epimorphic if and only if the morphism is a strict epimorphism. Such a sieve $S$ is \textbf{universally} effective-epimorphic if its pullback along the functor $\Ccal/D \to \Ccal/C$ induced by a morphism $f: D \to C$ is effective-epimorphic for any $f$.

The \textbf{canonical Grothendieck topology} $J_{can}^{\Ccal}$ on $\Ccal$ is the topology whose covering sieves are precisely the universally effective-epimorphic ones. If $\Ccal$ is a Grothendieck topos, this coincides with the Grothendieck topology whose covering sieves are those containing small jointly epic families.
\end{definition}

We recall the following result, which appears as \cite[Lemma 4.35]{Dense}:

\begin{lemma}
\label{lem:coincide}
Let $\Ecal$ be a Grothendieck topos and $\Ccal$ a small full separating subcategory of $\Ecal$. Let $S$ be a sieve in $\Ccal$ on an object $C$ and let $D_S$ be the diagram in $\Ccal$ described in Definition \ref{dfn:effective}. Suppose that the colimit of $D_S$ in $\Ecal$ lies in $\Ccal$. Then $S$ is universally effective-epimorphic in $\Ccal$ if and only if it is the restriction to $\Ccal$ of a sieve containing a small jointly epic family in $\Ecal$.
\end{lemma}

\begin{theorem}
\label{thm:canon}
Suppose $\Ecal$ is supercompactly generated. Let $J_r$ be the Grothendieck topology on $\Ccal_s$ whose covering sieves are those containing strict epimorphisms. Then $\Ecal \simeq \Sh(\Ccal_s,J_r)$.

Similarly, if $\Ecal$ is compactly generated, and $J_c$ is the Grothendieck topology on $\Ccal_c$ whose covering sieves are those containing strictly epic (equivalently, jointly epic) finite families. Then $\Ecal \simeq \Sh(\Ccal_c,J_c)$.
\end{theorem}
\begin{proof}
By Giraud's theorem, given a (small, full) separating subcategory $\Ccal$ of objects in a Grothendieck topos $\Ecal$, we have an equivalence of toposes $\Ecal \simeq \Sh(\Ccal, J_{can}^{\Ecal}|_{\Ccal})$, where $J_{can}^{\Ecal}|_{\Ccal}$ is the restriction of the canonical topology on $\Ecal$ to $\Ccal$, whose covering sieves are the intersections of $J_{can}^{\Ecal}$-sieves with $\Ccal$. Thus it suffices to show in each case that the restriction of the canonical topology is the topology described in the statement.

By Lemma \ref{lem:closed2}, the principal (resp. finitely generated) sieves $S$ on $\Ccal_s$ (resp. $\Ccal_c$), whose corresponding diagrams $D_S$ are funneling (resp. multifunneling) diagrams, have colimits contained in $\Ccal_s$ (resp. $\Ccal_c$), so Lemma \ref{lem:coincide} applies. Thus these are effective epimorphic sieves if and only if the generating morphism is a strict epimorphism in $\Ccal_s$ (resp. the generating morphisms form a strictly epimorphic finite family in $\Ccal_c$).

Now given any sieve $S$ containing a jointly epic family on an object of $\Ccal_s$ (resp. $\Ccal_c$) in $\Ecal$, by the definition of supercompact (resp. compact) objects, $S$ must contain an epimorphism (resp. a finite covering family). In particular, every $J_{can}^{\Ecal}|_{\Ccal_s}$-covering sieve contains a $J_{can}^{\Ecal}|_{\Ccal_s}$-covering \textit{principal} sieve. Similarly, every $J_{can}^{\Ecal}|_{\Ccal_c}$-covering sieve contains a $J_{can}^{\Ecal}|_{\Ccal_c}$-covering \textit{finitely generated} sieve.

It follows that the strict epimorphisms in $\Ccal_s$ are precisely the morphisms generating universally effective-epimorphic families and similarly for strict jointly epimorphic families in $\Ccal_c$), as required. Alternatively, see the proof of Proposition \ref{prop:representable} below for a direct argument showing that the strict epimorphisms (resp. strictly epic finite families) are stable.
\end{proof}

\begin{remark}
\label{rmk:Pcompact}
It should be clear by now from our joint treatment of
supercompactness and compactness that much of our analysis can be applied to more general notions of compactness. Indeed, Theorem \ref{thm:canon} is an explicit special case of \cite[Proposition 4.36]{Dense}.

Suppose $P$ is some property of pre-sieves (families of morphisms with common codomain); then we may define \textit{$P$-compact} objects in a topos $\Ecal$ as those for which every jointly epic covering family contains a jointly epic presieve satisfying $P$. If $P$ satisfies suitable composition and stability criteria, which Caramello specifies, and the full subcategory $\Ccal_P$ of $\Ecal$ on the $P$-compact objects is separating and closed under certain colimits, then $\Ecal$ is equivalent to the category of sheaves on $\Ccal_P$ for the topology generated by the effective-epimorphic $P$-presieves in $\Ccal_P$. For supercompactness, $P$ is the property `is a singleton', while for ordinary compactness, $P$ is the property `is finite', and our earlier results show that these do satisfy Caramello's criteria. More generally, any property which descends along epimorphisms can yield an interesting class of toposes; this is the methodology of \cite{SCGI}.

While we shall not attempt to extend the present chapter to this most general case, we encourage the reader to explore whether any given topos of interest to them is $P$-compactly generated for some suitable property $P$, and if so to compute the corresponding site produced by Caramello's result.
\end{remark}

The advantage of the intrinsic expressions for the Grothendieck topologies in Theorem \ref{thm:canon} is that it guarantees that the categories of (super)compact objects contain enough information to completely reconstruct the toposes \textit{by themselves}. This immediately gives us results such as the following:
\begin{corollary}
\label{crly:MoritaCs}
Suppose $\Ecal$ and $\Ecal'$ are supercompactly generated toposes and $\Ccal_s$, $\Ccal'_s$ are their respective categories of supercompact objects. Then $\Ecal \simeq \Ecal'$ if and only if $\Ccal_s \simeq \Ccal'_s$.
\end{corollary}

\subsection{Cokernels, well-supported objects and two-valued toposes}
\label{ssec:twoval}

We saw in Lemmas \ref{lem:compactepi} and \ref{lem:compactmono} that when $\Ecal$ is compactly generated, every epimorphism in $\Ccal_c$ is strict and every monomorphism in $\Ccal_c$ is regular. This leads us to wonder under what extra conditions these facts hold true in $\Ccal_s$, given that $\Ecal$ is supercompactly generated.

\begin{example}
To properly motivate this section, we show that epimorphisms in $\Ccal_s$ need not coincide with those in $\Ecal$. Let $\Dcal$ be the category
\[\begin{tikzcd}
A & B \ar[l, "l"'] \ar[r, "r"] & C.
\end{tikzcd}\]
In the topos $\Ecal$ of presheaves on $\Dcal$ it is easily calculated that the supercompact objects are precisely the representables, so $\Dcal$ coincides with $\Ccal_s$ (we shall see in Proposition \ref{prop:localic} that this argument is valid for all posets). The morphisms $l$ and $r$ are trivially epic in $\Dcal$ but are not epic in $\Ecal$.
\end{example}

Our main tool in this section is the following definition.
\begin{definition}
Given a morphism $f:A \to B$, its \textbf{cokernel}\footnote{Not to be confused with the \textit{cokernel pairs} mentioned in the proof of Lemma \ref{lem:compactmono}.} $B \to B/f$ is the pushout:
\[\begin{tikzcd}
A \ar[r, "f"] \ar[d, "!"'] & B \ar[d] \\
1 \ar[r, "x"] & B/f \ar[ul, "\ulcorner", phantom, very near start].
\end{tikzcd}\]
\end{definition}

Cokernels are useful for understanding epimorphisms thanks to the following result.
\begin{lemma}
\label{lem:isokernel}
A morphism $f: A \to B$ of a topos $\Ecal$ is an epimorphism if and only if the lower morphism $x:1 \to B/f$ of its cokernel is an isomorphism (or equivalently, an epimorphism).
\end{lemma}
\begin{proof}
If $f$ is an epimorphism we have:
\[\begin{tikzcd}
A \ar[r, "f", two heads] \ar[d, "!"'] & B \ar[d] \\
1 \ar[r, "x", two heads] & B/f
\ar[ul, "\ulcorner", phantom, very near start],
\end{tikzcd}\] 
since the pushout of an epimorphism is epic. But any quotient of $1$ in a topos is an isomorphism, as required.

Conversely, if $x:1 \to B/f$ is an isomorphism, we can consider the (epi,mono) factorization $f = m \circ e$:
\[\begin{tikzcd}
A \ar[rr, bend left, "f"] \ar[r, "e", two heads] \ar[d, "!"'] &
A' \ar[r, "m", hook] \ar[d, "!"'] & B \ar[d] \\
1 \ar[rr, bend right, "x"'] \ar[r, "\sim"] &
1 \ar[r, "\sim"] \ar[ul, "\ulcorner", phantom, very near start] &
B/f.
\end{tikzcd}\]
By the first part, the left hand square and outside rectangle are both pushouts, which makes the right hand square a pushout. But in a topos (or any adhesive category), a pushout square in which the upper horizontal morphism is monic is also a pullback square. Thus $m$ is an isomorphism, and $f$ is epic.
\end{proof}

Recall that an object $A$ of a category with a terminal object is \textbf{well-supported} if the unique morphism $!_A:A \to 1$ is an epimorphism in $\Ecal$. Using cokernels, we obtain a partial dual to Lemma \ref{lem:monocoincide} even without requiring $\Ecal$ to be (super)compactly generated.
\begin{lemma}
\label{lem:presic}
Let $\Ecal$ be a topos and $\Ccal_s$ the usual subcategory. Let $A$ be an object of $\Ccal_s$ which is well-supported as an object of $\Ecal$. Then a morphism $A \to B$ of $\Ccal_s$ is an epimorphism in that category if and only if it is an epimorphism in $\Ecal$.
\end{lemma}
\begin{proof}
An epimorphism of $\Ecal$ lying in $\Ccal_s$ clearly remains epic there.

Conversely, if $e : A \too B$ is epic in $\Ccal_s$, consider the cokernel $B \to B/e$:
\[\begin{tikzcd}
A \ar[r, "e"] \ar[d, "!_{A}"', two heads] &
B \ar[d, "q", two heads] \\
1 \ar[r, "x"] &
B/e \ar[ul, "\ulcorner", phantom, very near start].
\end{tikzcd}\]
By assumption $B/e$ is supercompact as $q$ is epic in $\Ecal$. The unique morphism $!_{A}:A \to 1$ factors through $e$ via the unique morphism $!_{B}:B \to 1$. Thus we have $qe = x!_{A} = x!_{B}e$, and since these expressions are composed of morphisms lying in $\Ccal_s$ where $e$ is epic, it follows that $q = x!_{B}$, whence $x$ is epic and hence an isomorphism. Hence $e$ is epic in $\Ecal$, by Lemma \ref{lem:isokernel}.
\end{proof}

It is worth noting that the existence of any well-supported supercompact object forces the terminal object of $\Ecal$ to be supercompact. Moreover, this proof concertedly fails when the object $A$ is not well-supported:

\begin{lemma}
\label{lem:coker}
Let $\Ecal$, $\Ccal_s$ be as above. Then $\Ccal_s$ is closed under cokernels if and only if every supercompact object is well-supported.
\end{lemma}
\begin{proof}
Given a supercompact object $A$, consider its \textbf{support}, which is the subterminal object $U \hookrightarrow 1$ in the factorization of the morphism $!_A$. By Lemma \ref{lem:closed}, $U$ is supercompact, so if $\Ccal_s$ is closed under cokernels, the pushout of $U \hookrightarrow 1$ along itself must be in $\Ccal_s$. But the colimit morphisms $1 \rightrightarrows 1 +_U 1$ are jointly epic, so one of them must be an epimorphism and hence an isomorphism, which forces $U \cong 1$. Thus $A$ is well-supported, as required.

Conversely, if every object $A$ of $\Ccal_s$ is well-supported then the cokernel (in $\Ecal$) of a morphism $A \to B$ in $\Ccal_s$ is a quotient of $B$ and so is supercompact. Thus $\Ccal_s$ is closed under cokernels, as required.
\end{proof}

We shall see a relevant sufficient condition under which the set-up of Lemma \ref{lem:coker} arises in Proposition \ref{prop:hype2}. In this setting we can also strengthen Lemma \ref{lem:monocoincide}.

\begin{scholium}
\label{schl:regularmono}
Let $\Ecal$ be a topos such that every object of $\Ccal_s$ is well-supported. Then monomorphisms in $\Ccal_s$ inherited from $\Ecal$ are regular. In particular, if $\Ecal$ is also supercompactly generated, then all monomorphisms in $\Ccal_s$ are regular.
\end{scholium}
\begin{proof}
By Proposition \ref{prop:hype2}, the hypotheses guarantee that the cokernel of a morphism in $\Ccal_s$ also lies in that category.

As remarked in the proof of Lemma \ref{lem:isokernel}, the pushout square defining the cokernel of an inclusion of supercompact objects $i: A \hookrightarrow B$ is also a pullback in $\Ecal$. It follows easily that $i$ is the equalizer in $\Ecal$ of the morphisms $q,x!_B: B \rightrightarrows B/i$ described in the proof of Lemma \ref{lem:presic}, and is consequently also their equalizer in $\Ccal_s$.

If $\Ecal$ is supercompactly (resp. compactly) generated then we may apply Lemma \ref{lem:monocoincide} to conclude that the above applies to all monomorphisms of $\Ccal_s$ (resp. $\Ccal_c$).
\end{proof}

On the other hand, we shall see in Example \ref{xmpl:freemon} that it is not in general possible to strengthen the properties of epimorphisms in $\Ccal_s$ beyond the consequence of the proof of Theorem \ref{thm:canon} that they are strict. In particular, they are not regular in general.

Finally, we explicitly record how Lemma \ref{lem:presic} simplifies the expression for the Grothendieck topology $J_r$ induced on $\Ccal_s$ from Theorem \ref{thm:canon}.
\begin{corollary}
\label{crly:site}
Let $\Ecal$ be a supercompactly generated topos and $\Ccal_s$ its full subcategory of supercompact objects. Suppose every object of $\Ccal_s$ is well-supported. Let $J_r$ be the topology on $\Ccal_s$ whose covering sieves are precisely those containing epimorphisms. Then $\Ecal \simeq \Sh(\Ccal_s,J_r)$.
\end{corollary}

\subsection{Proper, polished and pristine morphisms}
\label{ssec:proper}

In this section we present some classes of geometric morphism whose inverse image functors interact well with supercompact and compact objects. The material in this section contains more technical topos theoretic concepts than that in the earlier sections, so may be skipped on a first reading; the key result is Proposition \ref{prop:relpres}. 

The latter part of Lemma \ref{lem:scompact}, regarding compact objects, states precisely that the geometric morphism $\Ecal/C \to \Set$ is proper in the sense of \cite[Definition I.1.8]{Compact} (see also \cite[C3.2.12 and C1.5.5]{Ele}), or equivalently that $\Ecal/C$ is a compact Grothendieck topos. We can generalize Moerdijk and Vermeulen's definition of proper morphisms to the arbitrary union case in order to capture the idea of supercompactness. In order to do this, we recall some classic topos-theoretic constructions, from \cite[Chapter 2]{TT}.

Recall from \cite[Definition 2.11]{TT} that an \textbf{internal category} $\Ibb$ in a topos (or, more generally, a cartesian category) $\Ecal$ consists of:
\begin{enumerate}[label = ({\roman*})]
	\item An \textit{objects of objects} $I_0$ and an \textit{object of morphisms} $I_1$ in $\Ecal$, and
	\item Morphisms $i:I_0 \to I_1$, $d,c:I_1 \rightrightarrows I_0$ and $m: I_2 \to I_1$,
\end{enumerate}
where $I_2 := c \times_{I_0} d$ is the \textit{object of composable pairs}, defined as the pullback:
\[\begin{tikzcd}
I_2 \ar[r, "\pi_2"] \ar[d,"\pi_1"']
\ar[dr, phantom, "\lrcorner", very near start] &
I_1 \ar[d, "d"]\\
I_1 \ar[r, "c"'] &
I_0.
\end{tikzcd}\]
The morphisms define the \textit{identity morphisms}, \textit{domains}, \textit{codomains} and \textit{composition}, respectively, and are required to satisfy the equations $di = ci = \id_{I_0}$, $dm = d\pi_1$, $cm = c\pi_2$, $m(\id \times m) = m(m \times \id)$ and $m(\id \times i) = m(i \times \id) = \id_{I_1}$. These data and equations are a diagrammatic translation of the axioms for ordinary categories in $\Set$. An \textbf{internal functor} between internal categories is a pair of morphisms between the respective objects of objects and objects of morphisms commuting with the respective structure morphisms.

\begin{definition}
\label{dfn:diagram}
We say an internal category $\Ibb$ in $\Ecal$ is \textbf{filtered} if the usual definition of filteredness, cf. \cite[\S VII.6]{MLM}, is satisfied internally. Since this internalization is rarely made explicit, we explain it in full here: $\Ibb$ is filtered if and only if the following three conditions are satisfied,
\begin{itemize}
	\item $\Ibb$ is internally \textbf{inhabited}\footnote{This is the 	constructive term for `non-emptiness', meaning `has an element'.}, which is to say that $I_0$ is a well-supported object.
	\item The morphism
	$\begin{tikzcd}
	c \times_{I_0} c \ar[r, hook] & I_1 \times I_1
	\ar[r, "d \times d"] & I_0 \times I_0,
	\end{tikzcd}$
	is an epimorphism; the domain is the `object of pairs of morphisms with common codomain'.
	\item Let $A \hookrightarrow I_1 \times I_1 \times I_1$ be the subobject $\left(d \times_{I_0} d \times \id_{I_1}\right) \cap \left(c \times_{I_0} c \times_{I_0} d\right)$, which is the internalization of $\{f,g,h \in I_1 \mid d(f) = d(g) \wedge c(f) = c(g) = d(h)\}$. Let $E$ be the equalizer of the morphisms $m \circ \pi_1, m \circ \pi_2: A \rightrightarrows I_1$, the subobject consisting of the triples satisfying $h \circ f = h \circ g$. Finally, let $B \hookrightarrow I_1 \times I_1$ be $(c \times_{I_0} c) \times_{I_1 \times I_1} (d \times_{I_0} d)$, the `object of parallel pairs of morphisms'. We require that the projection $\pi_1 \times \pi_2: E \to B$ be an epimorphism.
\end{itemize}
\end{definition}
While the constructions of Definition \ref{dfn:diagram} are a little technical, reasoning about filtered internal categories always amounts to a categorical re-expression of the usual external reasoning for such categories.

Given an internal category $\Ibb$ in $\Ecal$, we recall from \cite[Definition 2.14]{TT} that an (internal) \textbf{diagram} of shape $\Ibb$ consists of:
\begin{enumerate}
 	\item An object $a: F_0 \to I_0$ of $\Ecal/I_0$, and
 	\item A morphism $b: F_1 \to F_0$, 
\end{enumerate}
where $F_1$ is the defined as the pullback
\[\begin{tikzcd}
F_1 \ar[r, "\pi_2"] \ar[d,"\pi_1"']
\ar[dr, phantom, "\lrcorner", very near start] &
F_0 \ar[d, "a"]\\
I_1 \ar[r, "d"'] &
I_0,
\end{tikzcd}\]
such that $ab = c\pi_2$, $b(\id \times i) = \id_{F_0}$, and $e(e \times \id) = e(\id \times m)$. In $\Set$, this data captures the encoding of a functor into $\Set$ via the Grothendieck construction.

There is an accompanying notion of natural transformation, and hence we obtain the diagram category $[\Ibb,\Ecal]$. This is a topos over $\Ecal$, which is proved using the comonadicity theorem in \cite[Corollary 2.33]{TT}.

For any geometric morphism $f:\Fcal \to \Ecal$, since $f^*$ preserves finite limits, we can apply $f^*$ to the data of an internal category $\Ibb$ in $\Ecal$ to obtain an internal category $f^*(\Ibb)$ in $\Fcal$. There is an induced pullback square of diagram toposes:\footnote{See \cite[Corollary B3.2.12]{Ele} for an explanation of why this square is a pullback.}
\begin{equation}
\label{eq:pbdiagram}
\begin{tikzcd}[column sep = large]
{[f^*(\Ibb){,}\Fcal]} \ar[r, "f^{\Ibb}"] \ar[d,"\pi"']
\ar[dr, phantom, "\lrcorner", very near start] &
{[\Ibb{,}\Ecal]} \ar[d, "\pi"]\\
\Fcal \ar[r, "f"] &
\Ecal,
\end{tikzcd}
\end{equation}
Each vertical morphism labelled $\pi$ has inverse image functor sending an object to the `constant diagram of shape $\Ibb$'. As well as a right adjoint $\pi_*$, this functor always has an $\Ecal$-indexed left adjoint $\pi_!$. The functors $\pi_*$ and $\pi_!$ send an internal diagram to its (internal) limit and colimit, respectively.

\begin{definition}
\label{dfn:polished}
Let $E$ be an object of $\Ecal$ and $\Ibb$ an internal category (resp. inhabited internal category $\Ibb$; filtered internal category $\Ibb$) in $\Ecal/E$. Let $f:\Fcal \to \Ecal$ be a geometric morphism. Consider the special case of \eqref{eq:pbdiagram} where we take the lower geometric morphism to be $f/E$:
\begin{equation}
\label{eq:polished}
\begin{tikzcd}[column sep = large]
{[f^*(\Ibb){,}\Fcal/f^*(E)]} \ar[r, "(f/E)^{\Ibb}"] \ar[d,"\pi"'] \ar[dr, phantom, "\lrcorner", very near start] &
{[\Ibb{,}\Ecal/E]} \ar[d, "\pi"]\\
\Fcal/f^*(E) \ar[r, "f/E"] &
\Ecal/E.
\end{tikzcd}
\end{equation}

We call a geometric morphism $f:\Fcal\to \Ecal$ \textbf{pristine} (resp. \textbf{polished}; \textbf{proper}) if the square above satisfies the condition $\pi_! \circ (f/E)^{\Ibb}_*(V) \cong (f/E)_* \circ \pi_!(V)$ for every subterminal object $V$ of $[f^*(\Ibb),\Fcal/f^*(E)]$, for every choice of $E$ and $\Ibb$. This can be understood as stating that the direct image of $f$ preserves $\Ecal$-indexed (resp. $\Ecal$-indexed inhabited; $\Ecal$-indexed directed) unions of subobjects.

We call a Grothendieck topos \textbf{supercompact} (resp. \textbf{compact}) if its unique geometric morphism to $\Set$ is pristine (resp. proper), which by Lemma \ref{lem:scompact} occurs if and only if the terminal object has the corresponding property. The global sections morphism of a Grothendieck topos is polished if and only if the topos is supercompact or degenerate (that is, the terminal object is either supercompact or initial).
\end{definition}

Note that Moerdijk and Vermeulen denote $\pi_!$ by $\infty^*$ because, in the proper case, $\pi_!$ preserves finite limits and hence is the inverse image functor of a geometric morphism in the opposite direction. In the pristine and polished cases the left adjoint $\pi_!$ is in general not left exact, so this notation no longer makes sense.

\begin{remark}
A filtered \textit{preorder} is ordinarily called \textbf{directed}; we have of course already seen this notion in Lemma \ref{lem:scompact}. It is intuitive that a filtered diagram of subobjects can be re-expressed as a directed preorder-indexed diagram by identifying parallel indexing morphisms. We do this informally in Definition \ref{dfn:polished}, but it is formally justified: we shall see in Corollary \ref{crly:hypepres} that hyperconnected morphisms are proper, recovering the fact (\cite[Corollary I.2.5]{Compact}) that a geometric morphism is proper if and only if its localic part is, and the localic part of the canonical geometric morphism $[\Ibb,\Ecal] \to \Ecal$ is the corresponding morphism $[\Ibb',\Ecal] \to \Ecal$, where $\Ibb'$ is the internal preorder reflection of $\Ibb$.
\end{remark}

\begin{example}
\label{xmpl:compshf}
The presheaf topos $[\Ccal\op,\Set]$ is supercompact if and only if $\Ccal$ is a funnel in the sense of Definition \ref{dfn:funnel}. Indeed, if $\Ccal$ has a weakly terminal object $C_0$ then there is by inspection an epimorphism $\yon(C_0) \too 1$ in $[\Ccal\op,\Set]$, and conversely if $1$ is supercompact then one of the morphisms $\yon(C) \to 1$ must be an epimorphism, since these are jointly epic, whence every object of $\Ccal$ admits at least one morphism to the corresponding object of $\Ccal$. More generally, $[\Ccal\op,\Set]$ is compact if and only if $\Ccal$ is a multifunnel in the sense of Definition \ref{dfn:multifun}.
\end{example}

An immediate consequence of Definition \ref{dfn:polished} is that we can relativize the concepts of supercompactness and compactness to depend on the base topos over which we work (so far we have been implicitly working over $\Set$). Viewing a geometric morphism $\Fcal \to \Ecal$ as expressing $\Fcal$ as a topos over $\Ecal$, an object $X$ of $\Fcal$ is \textbf{$\Ecal$-compact} if the composite geometric morphism $\Fcal/X \to \Fcal \to \Ecal$ is proper, for example. Conversely, since in this thesis we will only be concerned with objects which are supercompact relative to some fixed base topos $\Scal$, it makes sense to employ the broader classes of geometric morphism introduced in \cite[Chapter V]{Compact}.

\begin{definition}
\label{dfn:relpolished}
Let $p:\Ecal \to \Scal$ and $q:\Fcal \to \Scal$ be toposes over $\Scal$. A geometric morphism $f:\Fcal \to \Ecal$ over $\Scal$ is \textbf{$\Scal$-relatively pristine} (resp. \textbf{$\Scal$-relatively polished}; \textbf{$\Scal$-relatively proper}) if its direct image preserves arbitrary $\Scal$-indexed unions (resp. inhabited $\Scal$-indexed unions; directed $\Scal$-indexed unions) of subobjects. Explicitly, this requires that the respective conditions of Definition \ref{dfn:polished} hold for diagrams $\Ibb$ in $\Ecal$ of the form $p^*(\Ibb')$, where $\Ibb'$ is a diagram category of the appropriate type in $\Scal$. We shall assume $\Scal$ is $\Set$ in the remainder, and so we drop the `$\Scal$-' prefixes.
\end{definition}

All of these definitions appear hard to work with in general for the simple reason that internal diagram categories take a significant amount of computation to express and work with concretely (which is to say externally) in any given case. However, the $\Scal$-relative notions conveniently coincide with their external counterparts, in a sense made precise in Lemma \ref{lem:relprop} below. Also, all of the notions are clearly stable under slicing and composition, from which we can extract general consequences which are sufficient for our purposes.

By inspection, we have the following relationships between Definitions \ref{dfn:polished} and \ref{dfn:relpolished}.
\begin{lemma}
\label{lem:relvsnorel}
Consider a commuting triangle of geometric morphisms:
\[\begin{tikzcd}
\Fcal \ar[rr, "f"] \ar[dr, "q"'] & & \Ecal \ar[dl, "p"]\\
& \Set. &
\end{tikzcd}\]
Then:
\begin{enumerate}
	\item If $f$ is pristine, it is relatively pristine.
	\item $p$ and $q$ are relatively pristine morphism if and only if they are pristine.
	\item If $f$ is relatively pristine and $p$ is pristine, then $q$ is pristine.
\end{enumerate}
The same statements hold when `pristine' is replaced by `polished' or `proper'.
\end{lemma}

A handy consequence of this for our objects of interest is the following:
\begin{corollary}
\label{crly:polished}
Let $f: \Fcal \to \Ecal$ be a geometric morphism between Grothendieck toposes. If $f$ is relatively pristine (resp. relatively polished, relatively proper), then $f^*$ preserves supercompact (resp. `supercompact or initial', compact) objects.
\end{corollary}
\begin{proof}
Given an object $E$ of $\Ecal$ and a relatively pristine (resp. relatively polished, relatively proper) morphism $f$, consider the triangle:
\[\begin{tikzcd}
\Fcal/f^*(E) \ar[rr, "f/E"] \ar[dr] & & \Ecal/E \ar[dl]\\
& \Set. &
\end{tikzcd}\]
If $E$ is supercompact (resp. `supercompact or initial', compact), then the global sections morphism of $\Ecal/E$ is also pristine (resp. polished, proper), so $f^*(E)$ must be supercompact (resp. supercompact or initial, compact) by Lemma \ref{lem:relvsnorel}.3.
\end{proof}

In order to make more explicit arguments, we now extend the characterization of relatively proper geometric morphisms in \cite[Proposition V.3.7(i)]{Compact}.
\begin{lemma}
\label{lem:relprop}
A geometric morphism $f:\Fcal \to \Ecal$ over $\Scal$ is relatively pristine (resp. relatively polished, relatively proper) if and only if for any $\Scal$-indexed\footnote{Taking $\Scal$ to be $\Set$, $\Scal$-indexed just means set-indexed, or small.} (resp. $\Scal$-inhabited, $\Scal$-directed) jointly epimorphic family $\{g_i: X_i \hookrightarrow f^*(Y)\}$ of subobjects in $\Fcal$ there exists a jointly epimorphic family $\{h_j: Y_j \to Y\}$ in $\Ecal$ such that each $f^*(h_j)$ factors through some $g_i$.
\end{lemma}
\begin{proof}
Suppose $f$ has one of the relative properties and we are given a collection of subobjects of the relevant type. By assumption, their union is preserved by $(f/Y)_*$, whence there are subobjects $h_j: Y_j \hookrightarrow Y$ (the images under $(f/Y)_*$ of the subobjects) which are also jointly epic. By construction, each of these must have image under $f^*$ which factors through one or more of the $X_i$.

Conversely, given an $\Scal$-indexed diagram (of the relevant type) of subterminal objects $\{g_i: X_i \to f^*(Y)\}$ in $\Fcal/f^*(Y)$, let $U \hookrightarrow f^*(Y)$ be the union of the $g_i$, and let $Y' \hookrightarrow Y$ be its image under $(f/Y)_*$ in $\Ecal/Y$. Applying $f^*$, we have a monomorphism $f^*(Y') \hookrightarrow U$, so we can pull back to get a jointly epimorphic family of the appropriate shape $\{g'_i: X'_i \to f^*(Y')\}$ over $f^*(Y')$. This data is of the required form to apply the hypotheses, and the covering family of $Y'$ thus provided ensures that the union of the images of the $g_i$ under $(f/Y)_*$ is precisely $Y'$.
\end{proof}

It is an indirect consequence of Lemma \ref{lem:relprop} that the converse of Corollary \ref{crly:polished} cannot hold in general. Indeed, if the only compact object of $\Ecal$ is the initial object, such as in Example \ref{xmpl:Nlocalic} below, then the preservation of compact or supercompact objects by $f^*$ is a vacuous condition, but the required properties in the characterization of Lemma \ref{lem:relprop} are clearly non-trivial. However, $\Ecal$ failing to have enough supercompact (resp. compact) objects is the only obstacle.

\begin{proposition}
\label{prop:relpres}
Let $f:\Fcal \to \Ecal$ be a geometric morphism, and suppose $\Ecal$ is supercompactly generated. Then $f$ is relatively pristine if and only if $f^*$ preserves supercompact objects, and relatively polished if and only if $f^*$ preserves `supercompact or initial' objects. If $\Ecal$ is merely compactly generated, then $f$ is relatively proper if and only if $f^*$ preserves compact objects.
\end{proposition}
\begin{proof}
Given an arbitrary (resp. inhabited) jointly epic collection of subobjects $\{ g_i: X_i \hookrightarrow f^*(Y) \}$ in $\Fcal$, consider a (possibly empty) covering of $Y$ by supercompact subobjects $h_j : Y_j \hookrightarrow Y$ in $\Ecal$. Since each $f^*(Y_j)$ is supercompact (resp. supercompact or initial), pulling back the inclusions $g_i$ along $f^*(h_j)$, we conclude that one of the resulting inclusions $f^*(Y_j)$ (if there are any) must be an isomorphism by Lemma \ref{lem:scompact}; in the inhabited case, this is trivially true when $f^*(Y_j)$ is initial. Hence the $f^*(h_j)$ each factor through one of the $g_i$, whence the criteria of Lemma \ref{lem:relprop} are fulfilled. The compactly generated case, with a directed family of subobjects, is analogous.
\end{proof}

Since the initial object in any topos is strict, the distinction between relatively pristine and relatively polished morphisms is indeed as small as this proposition makes it seem.

\begin{lemma}
\label{lem:0refl}
Given a geometric morphism $f: \Fcal \to \Ecal$, $f^*$ reflects the initial object if and only if $f_*$ preserves it.
\end{lemma}
\begin{proof}
If $f^*$ reflects $0$, then considering the counit $f^*f_*(0) \to 0$, we conclude that $f^*f_*(0)$ is initial (by strictness of the initial object), whence $f_*(0) \cong 0$ so $f_*$ preserves the initial object. Conversely, if $f_*$ preserves $0$, then given $C$ with $f^*(C) \cong 0$ the unit $C \to f_*f^*(C) \cong 0$ is a morphism to $0$, whence $C$ is initial.
\end{proof}

\begin{corollary}
\label{crly:pristine}
Let $f: \Fcal \to \Ecal$ be a geometric morphism, and suppose $\Ecal$ is supercompactly generated. Then $f$ is relatively pristine if and only if it is relatively polished and $f^*$ reflects the initial object.
\end{corollary}
\begin{proof}
A (relatively) pristine morphism is (relatively) polished. Considering $0$ as an empty union of subobjects of any given object, it is preserved by $f_*$ by relative pristineness, so $f^*$ reflects $0$ by Lemma \ref{lem:0refl}, as required. Note that this implication holds even when $\Ecal$ is not supercompactly generated.

Conversely, given that $f$ is relatively polished and $f^*$ reflects $0$, we have that $f^*$ preserves `supercompact or initial' objects by Proposition \ref{prop:relpres}, but a supercompact object $X$ of $\Ecal$ is sent to an initial object if and only if it is initial, which is impossible, so $f^*(X)$ is supercompact, as required.
\end{proof}

We caution the reader that this ostensibly small difference between relatively polished and relatively pristine morphisms greatly impacts the frequency with which these classes of morphisms occur, as witnessed in Corollary \ref{crly:totconn} and Lemma \ref{lem:point} below.

\subsection{Inclusions and surjections}
\label{ssec:surj}

Recall, as we saw in Chapter \ref{chap:TDMA}, that a geometric morphism $f:\Fcal \to \Ecal$ is a \textbf{surjection} if $f^*$ is faithful, or equivalently if $f^*$ is a comonadic functor. Meanwhile, a geometric morphism $f:\Fcal \to \Ecal$ is an \textbf{inclusion} (or \textit{embedding}) if its direct image $f_*$ is full and faithful. See \cite[\S VII.4]{MLM} or \cite[A4.2]{Ele} for general results regarding these, which we shall assume familiarity with. In this section we examine how these two types of geometric morphism interact with supercompact and compact objects, as well as some of the classes of geometric morphism introduced in the last section.

Recall that a geometric inclusion $f: \Fcal \to \Ecal$ is \textbf{closed} if there is some subterminal object $U$ in $\Ecal$ such that $f_*f^*$ sends an object $X$ to the pushout of the product projections from $X \times U$. The following result illustrates why the pristine, polished and proper morphisms are too restrictive for analyzing supercompactly generated subtoposes.

\begin{lemma}
\label{lem:incl}
An inclusion of toposes $f: \Fcal \to \Ecal$ is proper if and only if it is closed, if and only if it is polished. An inclusion is pristine if and only if it is an equivalence.
\end{lemma}
\begin{proof}
The first part is the conclusion of \cite[Remark C3.2.9]{Ele}, where it is observed that a closed inclusion has a direct image functor preserving arbitrary inhabited ($\Ecal$-indexed) unions of subobjects and that closed inclusions are stable under slicing, and conversely that any proper inclusion is closed.

Given that $f$ is a closed inclusion, any nontrivial subobject of the corresponding subterminal object $U$ in $\Ecal$ is sent by $f^*$ to the initial object in $\Fcal$, so the initial object is reflected if and only if $U$ is initial, in which case $f$ is an equivalence.
\end{proof}

The relative versions of these properties are well-behaved with respect to surjections and inclusions.

\begin{proposition}
\label{prop:surjinc}
Consider a factorization of a geometric morphism $f$,
\begin{equation}
\label{eq:factors}
\begin{tikzcd}
\Fcal \ar[rr, "f"] \ar[dr, "q"'] & & \Ecal\\
& \Gcal \ar[ur, "p"']. &
\end{tikzcd}
\end{equation}
\begin{enumerate}
	\item Suppose that $p$ is an inclusion. Then if $f$ is relatively pristine (resp. relatively polished, relatively proper), so is $q$.
	\item Suppose that $q$ is a surjection. Then if $f$ is relatively pristine (resp. relatively polished, relatively proper), so is $p$.
\end{enumerate}
It follows that a geometric morphism is relatively pristine (resp. relatively polished, relatively proper) if and only if both parts of its surjection-inclusion factorization are.
\end{proposition}
\begin{proof}
We use the characterization of these properties from Lemma \ref{lem:relprop}.

1. Given a jointly epic family (resp. inhabited family, directed family) $\{g_i: X_i \hookrightarrow q^*(Z)\}$ in $\Ecal$, we may express $Z$ up to isomorphism as $p^*(Y)$ (taking $Y = p_*(Z)$, say), so this can without loss of generality be seen as a family $\{g_i: X_i \hookrightarrow f^*(Y)\}$. Since $f$ is relatively pristine (resp. relatively polished, relatively proper), we have a jointly epic family $\{h_j:Y_j \to Y\}$ such that each $f^*(h_j)$ factors through some $g_i$, and hence $\{p^*(h_j): p^*(Y_j) \to Z\}$ is the required family to fulfill the criterion of Lemma \ref{lem:relprop}.

2. Given a jointly epic family (resp. inhabited family, directed family) $\{g_i: X_i \hookrightarrow p^*(Y)\}$ in $\Ecal$, we have $\{q^*(g_i): q^*(X_i) \hookrightarrow f^*(Y)\}$ in $\Fcal$ being of the desired form to ensure that there is a covering family $\{h_j: Y_j \to Y\}$ such that each $f^*(h_j)$ factors through some $q^*(g_i)$. But then $q^*$ being conservative forces the $p^*(h_j)$ to factor through $g_i$, as required. Indeed, the intersection of $g_i$ with the image of $p^*(h_j)$ is preserved by $q^*$, and one of the sides of the resulting pullback square is sent to an isomorphism, which is reflected by $q^*$.
\end{proof}

\begin{corollary}
\label{crly:ptcompact}
Any topos $\Ecal$ with a surjective point is supercompact. Any topos with a finite jointly surjective collection of points is compact.
\end{corollary}
\begin{proof}
The hypotheses correspond to the existence of a surjection $\Set \to \Ecal$, or a surjection $\Set/K \to \Ecal$ with $K$ finite. The unique morphism $\Set \to \Set$ is (an equivalence, and hence) pristine, while the morphism $\Set/K \to \Set$ is proper. Thus by Proposition \ref{prop:surjinc}.2 and Lemma \ref{lem:relvsnorel}, we conclude that $\Ecal$ is supercompact (resp. compact) over $\Set$.
\end{proof}

\begin{xmpl}
We saw in Chapter \ref{chap:TDMA} that the first part of Corollary \ref{crly:ptcompact} applies to toposes of the form $\Setswith{M}$ for $M$ a monoid; the latter part similarly applies to $\Setswith{\Ccal}$ for $\Ccal$ any small category with a finite number of objects, even if their idempotent completions may have many more objects.
\end{xmpl}

When the codomain of the geometric morphism is supercompactly (resp. compactly) generated, the simpler characterization of Proposition \ref{prop:relpres} comes to our aid.

\begin{lemma}
\label{lem:incl2}
Suppose that $\Ecal$ is supercompactly (resp. compactly) generated and $f: \Fcal \to \Ecal$ is an inclusion. If $f$ is relatively polished (resp. relatively proper), $\Fcal$ is also supercompactly (resp. compactly) generated.
\end{lemma}
\begin{proof}
By Proposition \ref{prop:relpres}, the separating collection of supercompact (resp. compact) objects in $\Ecal$ is mapped by the inverse image of the relatively polished (resp. relatively proper) inclusion $f$ to a separating collection of supercompact or initial objects (resp. compact objects) in $\Fcal$. 
\end{proof}

We shall see in Corollary \ref{crly:incl} that Lemma \ref{lem:incl2} is optimal, in the sense that a topos is supercompactly (resp. compactly) generated if and only if it admits a relatively polished (resp. relatively proper) inclusion into a presheaf topos.

\begin{corollary}
\label{crly:factor2}
Suppose that $\Ecal$ is supercompactly (resp. compactly) generated. Then a geometric morphism $f: \Fcal \to \Ecal$ is relatively polished (resp. relatively proper) if and only if both parts of its surjection-inclusion factorization have inverse images preserving supercompact or initial (resp. compact) objects.
\end{corollary}
\begin{proof}
This follows from applying Proposition \ref{prop:relpres} to the factorization in Proposition \ref{prop:surjinc}, using Lemma \ref{lem:incl2} to conclude that the intermediate topos must be supercompactly (resp. compactly) generated.
\end{proof}

More generally, surjections interact well with supercompact and compact objects in another way.

\begin{lemma}
\label{lem:surjreflect}
Suppose $f:\Fcal \to \Ecal$ is a surjective geometric morphism. Then $f^*$ reflects supercompact, compact and initial objects. 
\end{lemma}
\begin{proof}
Since the inverse image functor of $f$ is comonadic, we have an equivalence between $\Ecal$ and the topos of coalgebras for the comonad on $\Fcal$ induced by $f$. Without loss of generality we work with coalgebras.

Given a coalgebra $(X,\alpha: X \to f^*f_*(X))$ and a jointly epic family $g_i: (U_i,\beta_i) \to (X,\alpha)$, since $f^*$ preserves arbitrary colimits, the underlying family of morphisms $g_i:U_i \to X$ in $\Fcal$ must be jointly epic. Thus if $X = f^*(X,\alpha)$ is supercompact (resp. compact), one of the $g_i$ must be an epimorphism (resp. there is a finite jointly epic subfamily of the $g_i$). Since $f^*$ moreover creates colimits, we conclude (via the same colimit diagram employed in Lemma \ref{lem:compactepi}) that $g_i$ is an epimorphism in $\Ecal$ too (resp. that the finite subfamily lifts to a jointly epic finite subfamily). Thus $(X,\alpha)$ is supercompact, as required.

Preservation of the initial object by $f_*$, which is equivalent to reflection of $0$ by $f^*$ by Lemma \ref{lem:0refl}, is due to strictness of the initial object forcing $f^*f_*(0) \cong 0$.
\end{proof}

It follows from Lemma \ref{lem:surjreflect} and Corollary
\ref{crly:pristine} that a geometric surjection is relatively pristine
if and only if it is relatively polished.

\subsection{Hyperconnected morphisms}
\label{ssec:hype}

Recall that a geometric morphism $f:\Fcal \to \Ecal$ is \textbf{hyperconnected} if $f^*$ is full and faithful and its image is closed in $\Fcal$ under subobjects and quotients (up to isomorphism); we have briefly encountered such morphisms in previous chapters. Recall also that $f$ is \textbf{localic} if every object in $\Fcal$ is a subquotient of one of the form $f^*(X)$. See \cite[Section A4.6]{Ele} for some background on hyperconnected and localic morphisms, including the hyperconnected--localic factorization of a general geometric morphism.

When it comes to hyperconnected morphisms into $\Set$, there are various alternative characterizations; note that these rely on properties of $\Set$, so are not true constructively.

\begin{proposition}
\label{prop:hype2}
Let $\Ecal$ be a Grothendieck topos. Then the following are equivalent:
\begin{enumerate}
	\item The unique geometric morphism $\Ecal \to \Set$ is hyperconnected.
	\item $\Ecal$ is \textbf{two-valued}: the only subterminal objects are the initial and terminal objects.
	\item Every object of $\Ecal$ is either well-supported (the unique morphism $X \to 1$ is an epimorphism) or initial, but not both.
	\item $\Ecal$ is non-degenerate and has a separating set of well-supported objects.
\end{enumerate}
\end{proposition}
\begin{proof}
($1 \Leftrightarrow 2$) The inverse image of a hyperconnected geometric morphism is full and faithful and closed under subobjects, so in particular the only subobjects of $1$ in $\Ecal$ is $0$. Conversely, if $\Ecal$ is two-valued, consider the hyperconnected-localic factorization of the unique geometric morphism $\Ecal \to \Set$; the intermediate topos is the localic reflection of $\Ecal$, equivalent to the topos of sheaves on the locale of subterminal objects of $\Ecal$, so is equivalent to $\Set$. Thus the morphism $\Ecal \to \Set$ is hyperconnected.

($2 \Leftrightarrow 3$) If $\Ecal$ is two-valued, the monic part of the (epi,mono) factorization of $X\to 1$ is non-trivial if and only if $X \cong 0$, so $X$ is either initial or well-supported. Conversely, any non-trivial subterminal object fails to be well-supported, so if $3$ holds there can be no proper subterminals and $\Ecal$ is two-valued.

($3 \Leftrightarrow 4$) Since $\Ecal$ is a Grothendieck topos, it has some generating set of objects; any such set is still generating after excluding the initial object, and since $0$ is distinct from $1$, any generating set contains a non-initial object, so we obtain a generating set of well-supported objects as required. Conversely, an inhabited colimit of well-supported objects is well-supported, so any non-initial object is well-supported if there is a separating set of well-supported objects. 
\end{proof}

In particular, since supercompact objects are not initial, we obtain a necessary and sufficient condition for the hypotheses in Section \ref{ssec:twoval} to hold:
\begin{corollary}
\label{crly:epic}
If $\Ecal$ is a two-valued (Grothendieck) topos, every supercompact object in $\Ecal$ is well-supported. In particular, a morphism $A \to B$ in the subcategory $\Ccal_s$ of supercompact objects objects of $\Ecal$ is an epimorphism if and only if it is epic in $\Ecal$, so every epimorphism in $\Ccal_s$ is strict. Conversely, if $\Ecal$ is supercompactly generated, then $\Ccal_s$ is closed under cokernels if and only if $\Ecal$ is two-valued.
\end{corollary}

Thus we can give the counterexample, promised earlier, to the hypothesis that epimorphisms in $\Ccal_s$ are regular when every supercompact object of $\Ecal$ is well-supported.
\begin{example}
\label{xmpl:freemon}
Let $M$ be the free monoid on two generators, viewed as a category, and let $\Ecal = [M\op,\Set]$, where the objects are viewed as right $M$-sets. From the results in this chapter and observations in Chapter \ref{chap:TDMA}, we know that this topos is two-valued and supercompactly generated; the supercompact objects in this topos are precisely the principal right $M$-sets.

Given a principal $M$-set $N$ generated by $n$, the quotient of $N$ by a relation generated by a pair of morphisms $f,g: N \rightrightarrows M$ is obtained by identifying $f(n)k$ with $g(n)k$ for each $k \in M$. A case-by-case analysis of the possible pairs of elements $f(n),g(n)$ demonstrates that there is no pair of which the epimorphism $M \too 1$ is a coequalizer, so this epimorphism is not regular in the category of supercompact objects of $\Ecal$.
\end{example}

Returning to a more general setting, the main reason for our interest in hyperconnected morphisms, however, is that they create the structure of supercompactly (and compactly) generated Grothendieck toposes which we study in this chapter.

\begin{corollary}
\label{crly:hypepres}
If $f: \Fcal \to \Ecal$ is a hyperconnected geometric morphism between Grothendieck toposes, then it is pristine. Thus (being surjective) $f^*$ preserves and reflects supercompact, compact and initial objects.
\end{corollary}
\begin{proof}
We extend the proof that hyperconnected morphisms are proper, \cite[Proposition I.2.4]{Compact}, replacing $\infty^*$ with $\pi_!$.

Suppose $f : \Fcal \to \Ecal$ is hyperconnected, and consider a diagram of the form
\begin{equation}
\label{eq:polished2}
\begin{tikzcd}[column sep = large]
{[f^*(\Ibb){,}\Fcal]} \ar[r, "f^{\Ibb}"] \ar[d,"\pi"']
\ar[dr, phantom, "\lrcorner", very near start] &
{[\Ibb{,}\Ecal]} \ar[d, "\pi"]\\
\Fcal \ar[r, "f"] &
\Ecal.
\end{tikzcd}
\end{equation}
Since $f^{\Ibb}$ is a pullback of $f$, it is hyperconnected too, so that any $V \hookrightarrow 1$ in $[f^*(\Ibb){,}\Fcal]$ is of the form $(f^{\Ibb})^*(U)$ for some $U \hookrightarrow 1$ in $[\Ibb{,}\Ecal]$ (the restriction of a hyperconnected morphism to the subterminal objects is an equivalence). Thus,
\[f_*\pi_!(V) = f_*\pi_!(f^{\Ibb})^*(U) = f_*f^*\pi_!(U) = \pi_!(U),\]
where the last equality holds since $f^*$ is full and faithful. But $U = (f^{\Ibb})_*(f^{\Ibb})^*(U) = (f^{\Ibb})_*(V)$, so $f_*\pi_!(V) = \pi_!(f^{\Ibb})_*(V)$, as required. The same argument applied to slices gives the result.

Preservation of supercompact and compact objects by $f^*$ then follows from Proposition \ref{prop:relpres} and Corollary \ref{crly:pristine} (preservation of the initial object is automatic), while reflection follows from Lemma \ref{lem:surjreflect}.
\end{proof}

\begin{theorem}
\label{thm:hype2}
Let $f:\Fcal \to \Ecal$ be a hyperconnected geometric morphism between elementary toposes. If $\Fcal$ is a Grothendieck topos, so is $\Ecal$. Assuming this is so, if $\Fcal$:
\begin{enumerate}[label = ({\roman*})]
	\item is supercompactly generated, or
	\item is compactly generated, or
	\item has enough points, or
	\item is two-valued,
\end{enumerate}
then $\Ecal$ inherits that property.
\end{theorem}
\begin{proof}
Let $\Ccal$ be a (full subcategory on a) small separating set of objects in $\Fcal$. Then every object of $\Fcal$ is a colimit of a diagram in $\Ccal$. Given an object $Q$ of $\Ecal$, we can express $f^*(Q)$ as such a colimit; write $\lambda_i: C_i \to f^*(Q)$ with $C_i \in \Ccal$ for the legs of the colimit cone. By taking image factorizations of the $\lambda_i$ we obtain an expression for $f^*(Q)$ as a colimit where the legs of the colimit cone are all monomorphisms. Since the image of $f^*$ is closed under subobjects, we obtain an expression for $f^*(Q)$ as a colimit of objects of the form $f^*(D_i)$ with $D_i$ in $\Ecal$, which moreover are quotients of objects in the separating subcategory $\Ccal$ of $\Fcal$.

Thus, since $f^*$ creates all small colimits, the quotients of objects in $\Ccal$ lying in $\Ecal$ form a separating set. Also, $\Ecal$ is locally small since $\Fcal$ is. Thus by the version of Giraud's Theorem that appears in \cite[C2.2.8(v)]{Ele}, say, $\Ecal$ is a Grothendieck topos.

The inheritance of property (i) (resp. (ii)) follows from Corollary \ref{crly:hypepres}, taking $\Ccal$ in the above to be $\Ccal_s$ (resp. $\Ccal_c$) and noting that the objects $f^*(D_i)$ in the argument above are supercompact (resp. compact) in $\Fcal$ by Lemma \ref{lem:closed}, whence the $D_i$ are so in $\Ecal$ by Corollary \ref{crly:hypepres}.

For (iii), if $\Fcal$ has enough points, which is to say that there is a collection of geometric morphisms $\Set \to \Fcal$ whose inverse images are jointly faithful, then composing each point with the hyperconnected morphism to $\Ecal$ gives such a collection for $\Ecal$. Finally, for (iv), note once again that the restriction of $f$ to subterminal objects is an equivalence.
\end{proof}
In spite of our reliance on Corollary \ref{crly:hypepres} here, we shall see in Example \ref{xmpl:surjprec} that we cannot extend Theorem \ref{thm:hype2}(i) or (ii) to relatively pristine or relatively proper surjections, although parts (iii) and (iv) do apply in that situation.

For use in a later chapter, we record the following general result which has the supercompact and compact cases of Theorem \ref{thm:hype2} as special cases:
\begin{schl}
\label{schl:descend}
Suppose $\Fcal \to \Ecal$ is a hyperconnected geometric morphism. Let $P$ be a property of objects of a topos which descends along epimorphisms, in the sense that given an epimorphism $A \too B$, if $A$ satisfies $P$ then $B$ must also. Suppose moreover that objects with property $P$ are preserved and reflected by the inverse image of a hyperconnected geometric morphism $f:\Fcal \to \Ecal$. Then if $\Fcal$ has a separating set of objects with property $P$, so does $\Ecal$; explicitly, the latter set may be taken to be the collection of quotients of objects in the separating set for $\Fcal$ which lie in $\Ecal$.
\end{schl}
Several examples of properties which descend along epimorphisms appear in \cite[\S 4]{SCGI}, including the property of being an atom (having no non-trivial subobjects). It is easily checked that atoms are preserved and reflected by a fully faithful functor which is closed under subobjects, so Scholium \ref{schl:descend} applies to atomic toposes, for example.

We can summarize the results from the last two sections as stability results for supercompactly and compactly generated toposes.

\begin{theorem}
\label{thm:closure}
Suppose $\Ecal$ is a supercompactly (resp. compactly) generated Grothendieck topos. If $\Fcal$ is:
\begin{enumerate}
	\item The domain of a closed inclusion $f:\Fcal \to \Ecal$, or 	more generally, the domain of a relatively polished (resp. relatively proper) inclusion into $\Ecal$,
	\item The domain of a local homeomorphism $g: \Fcal \simeq \Ecal/X \to \Ecal$, or
	\item The codomain of a hyperconnected morphism $h: \Ecal \to \Fcal$,
\end{enumerate}
then $\Fcal$ is also a supercompactly (resp. compactly) generated Grothendieck topos.
\end{theorem}
\begin{proof}
Let $\Ccal_s$ and $\Ccal_c$ be the (separating) subcategories of supercompact and compact objects in $\Fcal$, respectively.

1. For any inclusion, the images of objects in $\Ccal_s$ (resp. $\Ccal_c$) under $f^*$ form a separating set for $\Ecal$. The stated properties ensure that these objects are all supercompact or initial (resp. compact), so that $\Ecal$ is supercompactly (resp. compactly) generated, by Corollary \ref{crly:polished}.

2. The objects with domain in $\Ccal_s$ (resp. $\Ccal_c$) in any slice $\Fcal/X$ form a separating set. These lifted objects inherit the property of being supercompact (resp. compact), by Lemma \ref{lem:scompact} and the standard result $(\Fcal/X)/(Y \to X) \simeq\Fcal/Y$.

3. This is immediate from Theorem \ref{thm:hype2}.
\end{proof}

\section{Principal Sites}
\label{sec:principal}

So far, we have established properties of `canonical' sites for supercompactly and compactly generated toposes. In the spirit of Caramello's work, \cite{SCGI}, we obtain in this section a broader class of sites whose toposes of sheaves have these properties.

As in Section \ref{sec:supercompact}, we write $\ell:\Ccal \to \Sh(\Ccal,J)$ for the composite of the Yoneda embedding and the sheafification functor. We will occasionally make use of the following fact regarding representable sheaves, which is easily derived from the fact that $\ell$ is a dense morphism of small-generated sites, in the sense described in \cite{Dense}:
\begin{fact}
\label{fact1}
A sieve $T$ on $\ell(C)$ in $\Sh(\Ccal,J)$ is jointly epimorphic if and only if the sieve $\{f:D \to C \mid \ell(f)\in T\}$ is $J$-covering.
\end{fact}

\subsection{Stable classes}
\label{ssec:stable}

\begin{definition}
\label{dfn:stable}
Let $\Ccal$ be a small category. A class $\Tcal$ of \textit{morphisms} in $\Ccal$ is called \textbf{stable} if it satisfies the following three conditions:
\begin{enumerate}
	\item $\Tcal$ contains all identities.
	\item $\Tcal$ is closed under composition.
	\item For any $f:C\to D$ in $\Tcal$ and any morphism $g$ in $\Ccal$ with codomain $D$, there exists a commutative square,
	\begin{equation}
	\begin{tikzcd}
	A \ar[r, "f'"] \ar[d, "g'"'] & B \ar[d, "g"]\\
	C \ar[r, "f"'] & D
	\end{tikzcd}
	\label{eq:stable}
	\end{equation}
	in $\Ccal$ with $f'\in \Tcal$.
\end{enumerate}
These correspond to the necessary and sufficient conditions for $\Tcal$-morphisms to be singleton presieves generating a Grothendieck topology, as presented in \cite[Exercise III.3]{MLM}. We call the resulting Grothendieck topology the \textbf{principal topology generated by $\Tcal$}.
\end{definition}

In \cite{SPM}, a stable class of morphisms is called
\textit{semi-localizing}, in reference to a related definition in \cite{GZ}. The authors call a principal topology an \textit{$A$-topology}, presumably because the atomic topology is an example of a principal topology; see Example \ref{xmpl:atomic}. We have chosen a naming convention that we believe to be more evocative in this context.

Continuing the parallel investigation of compactness, we obtain a related concept by replacing individual morphisms by finite families of morphisms.
\begin{definition}
\label{dfn:stable2}
Let $\Tcal'$ be a class of \textit{finite families} of morphisms with specified common codomain in $\Ccal$. We say $\Tcal'$ is \textbf{stable} if
\begin{enumerate}[label = {\arabic*}'.]
	\item $\Tcal'$ contains the families whose only member is the identity.
	\item $\Tcal'$ is closed under multicomposition, in that if $\{f_i:D_i \to C \mid i=1,\dotsc,n\}$ is in $\Tcal'$ and so are $\{g_{j,i}: E_j \to D_i \mid j=1,\dotsc,m_i\}$ for each $i = 1,\dotsc,n$, then $\{f_i \circ g_{j,i}\}$ is also a member of $\Tcal'$.
	\item For any $\{f_i:D_i \to C \mid i=1,\dotsc,n\}$ in $\Tcal'$ and any morphism $g:B \to C$ in $\Ccal$, there is a $\Tcal'$-family $\{h_j:A_j \to B \mid j=1,\dotsc,m\}$ such that each $g \circ h_j$ factors through one of the $f_i$.
\end{enumerate}
These are necessary and sufficient conditions for $\Tcal'$-families to generate a Grothendieck topology, which we call the \textbf{topology (finitely) generated by $\Tcal'$}.
\end{definition}

Note that we do not require $\Ccal$ to have pullbacks in the above definitions, so it is sensible to compare them with the usual notion of stability with respect to pullbacks.
\begin{lemma}
\label{lem:pbstable}
Let $\Tcal$ be a stable class of morphisms in $\Ccal$ with the additional `push-forward' property:
\begin{enumerate}
	\setcounter{enumi}{3}
	\item Given any morphism $f$ of $\Ccal$ such that $f \circ g \in \Tcal$ for some morphism $g$ of $\Ccal$, we have $f \in \Tcal$.
\end{enumerate}
This in particular is true of the class of epimorphisms, for example. Then morphisms in $\Tcal$ are stable under any pullbacks which exist in $\Ccal$.
\end{lemma}
\begin{proof}
Given a pullback of a $\Tcal$ morphism, the comparison between any square provided by \eqref{eq:stable} and this pullback provides a factorization of a $\Tcal$ morphism through the pullback, so by the assumed property the pullback is also in $\Tcal$, as required.
\end{proof}

As remarked in \cite{SPM}, we can extend a stable class of morphisms $\Tcal$ to the stable class $\hat{\Tcal}$ of morphisms whose principal sieves contain a member of $\Tcal$ (that is, morphisms $f$ such that $f \circ g \in \Tcal$ for some $g$) without changing the resulting principal topology. Thus we can safely assume that stable classes satisfy axiom 4 of Lemma \ref{lem:pbstable} if we so choose. The advantage of enforcing this assumption is that it gives a one-to-one correspondence between stable classes of morphisms in a category and the principal Grothendieck topologies on that category, since we can recover the classes as those morphisms which generate covering sieves.

\begin{remark}
\label{rmk:split}
The class of identity morphisms in any category satisfies axioms 1,2 and 3, but in order to satisfy axiom 4 it must be extended to the class of split epimorphisms, which is easily verified to satisfy all four axioms in any category. Thus the class of split epimorphisms corresponds to the trivial Grothendieck topology. It is worth noting also that every split epimorphism is regular and hence strict.
\end{remark}

We may similarly extend a class $\Tcal'$ to a maximal class $\hat{\Tcal}'$. However, for a class of finite families to be maximal, it must be closed under supersets as well as under the equivalent of the push-forward property of Lemma \ref{lem:pbstable}, so we have two extra axioms:
\begin{enumerate}[label = {\arabic*}'.]
	\setcounter{enumi}{3}
	\item Given a finite family $\mathfrak{f} = \{f_i:D_i \to C\}$ of morphisms in $\Ccal$ such that every morphism in some $\Tcal'$-family over $C$ factors through one of the $f_i$, we have $\mathfrak{f} \in \Tcal'$.
	\item Any finite family $\mathfrak{f} = \{f_i:D_i \to C\}$ of morphisms in $\Ccal$ containing a $\Tcal'$-family $\mathfrak{f}'$ is also a member of $\Tcal'$.
\end{enumerate}
The equivalent of the pullback stability statement of Lemma \ref{lem:pbstable} is as follows.

\begin{scholium}
\label{schl:finpb}
Suppose that $\Tcal'$ is a stable class of finite families of morphisms in a category $\Ccal$ satisfying the additional axioms 4' and 5'. Then $\Tcal'$ is stable under any pullbacks which exist in $\Ccal$, in that given a family $\{g_i:E_i \to C\}$ in $\Tcal'$ and $h:C'\to C$ such that the pullback of $g_i$ along $h$ exists for each $i$, the family $\{h^*(g_i):E'_i \to C'\}$ is in $\Tcal'$.
\end{scholium}
We leave the proof, and the verification that enforcing axioms 4' and 5' gives a one-to-one correspondence between stable classes of finite families and finitely generated Grothendieck topologies, to the reader.

\begin{definition}
\label{dfn:sites}
Let $\Ccal$ be a small category, $\Tcal$ a stable class of its morphisms; we denote the corresponding \textbf{principal} (Grothendieck) \textbf{topology} by $J_{\Tcal}$. We call a site $(\Ccal,J_{\Tcal})$ constructed in this way a \textbf{principal site}. Similarly, for a stable class of finite families $\Tcal'$ on $\Ccal$, we have a corresponding \textbf{finitely generated} (Grothendieck) \textbf{topology} denoted $J_{\Tcal'}$; a site of the form $(\Ccal,J_{\Tcal'})$ shall be called a \textbf{finitely generated site}.\footnote{We fear this terminology may result in some confusion if adopted more widely, since the `finitely generated' condition refers to the Grothendieck topology and not the underlying category, but it should cause no problems in the present thesis.}
\end{definition}

\begin{proposition}
\label{prop:representable}
Let $\Ccal$ be a small category and $J$ a Grothendieck topology on it. Then the representable sheaves are all supercompact if and only if $J = J_{\Tcal}$ is a principal topology for some stable class $\Tcal$ of morphisms in $\Ccal$. In particular, the topos of sheaves on a principal site $(\Ccal,J_{\Tcal})$ is supercompactly generated.

Similarly, the representable sheaves are all compact in $\Sh(\Ccal,J)$ if and only if $J = J_{\Tcal'}$ for a stable class $\Tcal'$ of finite families of morphisms in $\Ccal$, and the topos of sheaves on a finitely generated site $(\Ccal,J_{\Tcal'})$ is compactly generated.
\end{proposition}
\begin{proof}
By Fact \ref{fact1}, for $J = J_{\Tcal}$, given a covering family on $\ell(C)$, the sieve it generates must contain an epimorphism which is the image of a $\Tcal$-morphism. Since this epimorphism factors through some member of the covering family, that member must also be an epimorphism. Thus $\ell(C)$ is supercompact, as required.

Conversely, given that $\ell(C)$ is supercompact for every $C$, let $\Tcal$ be the class of morphisms $f$ such that $\ell(f)$ is epimorphic. We first claim that $\Tcal$ is a stable class. Indeed, axioms 1, 2 and 4 are immediate; to see that axiom 3 holds, suppose that $f: C \to D$ is in $\Tcal$ and $g: B \to D$ is any $\Ccal$ morphism. Then we may consider the pullback of $\ell(f)$ along $\ell(g)$ in $\Sh(\Ccal,J)$:
\[\begin{tikzcd}
A \ar[r, "f'", two heads] \ar[d, "g'"']
\ar[dr, phantom, "\lrcorner", very near start] &
\ell(B) \ar[d, "\ell(g)"]\\
\ell(C) \ar[r, "\ell(f)"', two heads] & \ell(D).
\end{tikzcd}\]
Since $A$ is covered by objects of the form $\ell(C')$, by supercompactness of $\ell(B)$ there must be an epimorphism $\ell(C')\too \ell(B)$ factoring through the pullback, and in turn the sieve it generates must contain an epimorphism in the image of $\ell$, so that we ultimately recover the square \eqref{eq:stable} required for axiom 3.

Now we show that $J = J_{\Tcal}$. Given a $J$-covering sieve $S$ on $C$ (generated by a family of morphisms, for example), the sieve generated by $\ell(S) := \{\ell(g) \mid g\in S\}$ necessarily covers $\ell(C)$, and therefore $\ell(S)$ contains an epimorphism by supercompactness, so that the original sieve must have contained a member of $\Tcal$, which gives $J \subseteq J_{\Tcal}$. Conversely, $J_{\Tcal}$-covering sieves are certainly $J$-covering, so $J_{\Tcal} \subseteq J$. Thus $J$ is a principal topology, as claimed.

The argument for finitely generated sites is almost identical; we need only replace $\Tcal$-morphisms with finite $\Tcal'$-families in the first part, and define $\Tcal'$ families to be those finite families which are mapped by $\ell$ to jointly epic families in the second part.
\end{proof}

In particular, we may extend a stable class of morphisms $\Tcal$ to a stable class of families of morphisms $\Tcal'$ by viewing the morphisms in $\Tcal$ as one-element families. Then it is clear that $J_{\Tcal} = J_{\Tcal'}$.

Intermediate between the two classes of sites discussed so far are a class which we call \textbf{quasi-principal sites}: these are sites $(\Ccal,J)$ such that for every object $C \in \Ccal$, either the empty sieve is a covering sieve on $C$ or every covering sieve on $C$ contains a principal sieve. Observe that if $\Ccal'$ is the full subcategory of $\Ccal$ on the latter class of objects (which we can always construct over $\Set$), then $\Sh(\Ccal,J) \cong \Sh(\Ccal',J_{\Tcal})$, where $\Tcal$ is the class of morphisms generating principal covering sieves.

The following result subsumes Lemma 4.11 of \cite{TAC}; our work up to this point allows us to avoid any direct manipulation with sheaves in the proof.

\begin{corollary}
\label{crly:incl}
Let $(\Ccal,J)$ be a small site. Then the inclusion $\Sh(\Ccal,J) \to [\Ccal\op,\Set]$ is relatively pristine (resp. relatively proper) if and only if $J$ is a principal (resp. finitely generated) topology. The inclusion is relatively polished if and only if $(\Ccal,J)$ is a quasi-principal site.
\end{corollary}
\begin{proof}
Since $[\Ccal\op,\Set]$ is supercompactly generated, by Proposition \ref{prop:relpres}, the inclusion $\Sh(\Ccal,J) \to [\Ccal\op,\Set]$ is relatively pristine if and only if the sheafification functor preserves supercompact objects; this in particular requires all of the objects $\ell(C)$ to be supercompact, which occurs if and only if $J$ is principal by Proposition \ref{prop:representable}. But any other supercompact objects are quotients of representables, so $\ell(C)$ being supercompact for every $C$ is also sufficient. As usual, the relatively proper case is analogous.

For relative polishedness, we relax the conditions above to requiring that each representable is sent to a supercompact or initial object, and note that the empty sieve is covering on $C$ if and only if $\ell(C)$ is initial.
\end{proof}

\subsection{Morphisms between sheaves on principal sites}
\label{ssec:representable}

In order to better understand the relationship between a principal site and the topos it generates, we employ some results from \cite{Dense}, which enable us to describe morphisms in a Grothendieck topos $\Sh(\Ccal,J)$ in terms of those in a presenting site $(\Ccal,J)$.

For a general site $(\Ccal,J)$, the functor $\ell:\Ccal \to \Sh(\Ccal,J)$ is neither full nor faithful. To describe the full collection of morphisms in the sheaf topos, several notions are introduced in \cite[\S 2]{Dense}, of which we introduce the relevant special cases here.

For morphisms $h,k:A \rightrightarrows B$, we say $h$ and $k$ are \textbf{$J$-locally equal} (written $h \equiv_J k$) if there is a $J$-covering sieve $S$ on $A$ such that $h\circ f = k \circ f$ for every $f \in S$. If $J$ is principal (resp. finitely generated) then this is equivalent to saying that there is some $\Tcal$-morphism which equalizes $h$ and $k$ (resp. a $\Tcal'$-family whose members all equalize $h$ and $k$). This leads naturally to the following moderately technical definitions:

\begin{definition}
\label{dfn:Tspan}
Let $\Ccal$ be a small category and $\Tcal$ a stable class of morphisms in $\Ccal$. Then for objects $A,B$ in $\Ccal$, a \textbf{$\Tcal$-span} from $A$ to $B$ is a span
\begin{equation}
\begin{tikzcd}
& E \ar[dl,"f"'] \ar[dr,"g"] &\\
A & & B,
\end{tikzcd}
\label{eq:span}
\end{equation}
such that $f$ is in $\Tcal$. A \textbf{$\Tcal$-arch} is a $\Tcal$-span such that for any $h,k:D\rightrightarrows E$ with $f \circ h = f \circ k$ we have $g \circ h \equiv_{J_{\Tcal}} g \circ k$.

Similarly, for $\Tcal'$ a stable class of finite families of morphisms on $\Ccal$, a \textbf{$\Tcal'$-span} is a finite (possibly empty) family of spans:
\begin{equation}
\begin{tikzcd}
& E_i \ar[dl,"f_i"'] \ar[dr,"g_i"] &\\
A & & B,
\end{tikzcd}
\label{eq:multispan}
\end{equation}
such that $\{f_1, \dotsc, f_n\}$ is in $\Tcal'$. A \textbf{$\Tcal'$-multiarch} is a $\Tcal'$-multispan such that for any $h:D \to E_i$, $k:D\to E_{i'}$ with $f_i \circ h = f_{i'}\circ k$ we have $g_i \circ h \equiv_{J_{\Tcal'}} g_{i'} \circ k$.

The constituent morphisms in any span or multispan will be referred to as their \textbf{legs}.
\end{definition}

\begin{lemma}
\label{lem:Trel}
Let $(\Ccal,J_\Tcal)$ be a principal site. Let $\Arch_\Tcal(A,B)$ be the collection of $\Tcal$-arches from $A$ to $B$ in $\Ccal$. For each $\Tcal$-arch $(t,g) \in \Arch_\Tcal(A,B)$, there is a (necessarily unique) morphism $\ell(t,g): \ell(A) \to \ell(B)$ in $\Sh(\Ccal, J_{\Tcal})$ such that $\ell(t,g) \circ \ell(t) = \ell(g)$. The mapping $\ell$ so defined is a surjection from $\Arch_\Tcal(A,B)$ to the set of morphisms from $\ell(A)$ to $\ell(B)$ in $\Sh(\Ccal,J_{\Tcal})$.

Similarly, letting $\mArch_{\Tcal'}(A,B)$ be the set of $\Tcal'$-multiarches from $A$ to $B$, $\ell$ induces a surjection from $\mArch_{\Tcal'}(A,B)$ to $\Hom_{\Sh(\Ccal,J_{\Tcal'})}(A,B)$.
\end{lemma}
\begin{proof}
This is immediate from \cite[Proposition 2.5]{Dense}.
\end{proof}

Intuitively it seems that the collections $\Arch_\Tcal(A,B)$ ``should'' be the morphisms of a category. However, while \cite[Proposition 2.5(iv)]{Dense} suggests a composition of arches coming from covering families that generate sieves, this composition produces a maximal family of arches presenting the composite rather than a single $\Tcal$-arch; there is a similar problem for multiarches. We therefore examine what structure exists in general, and identify some sufficient conditions under which arches and multiarches admit a composition operation.

\begin{lemma}
\label{lem:archcat}
Let $(\Ccal,J_{\Tcal})$ be a principal site. For each pair of objects $A$ and $B$ in $\Ccal$, let $\Span_\Tcal(A,B)$ be the collection of $\Tcal$-spans from $A$ to $B$. Then $\Span_\Tcal(A,B)$ admits a canonical categorical structure, where a morphism $x: (t:E \to A, g:E \to B) \to (t':E' \to A, g':E' \to B)$ is a morphism $x: E \to E'$ with $t = t' \circ x$ and $g = g' \circ x$. This restricts to give a category structure on $\Arch_\Tcal(A,B)$ too.

Expanding upon this, if $(\Ccal,J_{\Tcal'})$ is a finitely generated site, there is a canonical categorical structure on each collection of $\Tcal'$-multispans $\mSpan_{\Tcal'}(A,B)$, where $\vec{x}: (t_i:E_i \to A, g_i:E_i \to B) \to (t'_j:E'_j \to A, g'_j:E'_j \to B)$ consists of an index $j$ for each index $i$, and a morphism $x_i:E_i \to E'_j$ with $t_i = t'_j \circ x_i$ and $g_i = g'_j \circ x_i$. Note that any permutation of the spans forming a given $\Tcal'$-multispan form a $\Tcal'$-multispan which is isomorphic in this category. Once again, this structure restricts to the collections of multiarches.
\end{lemma}

\begin{proposition}
\label{prop:composition}
Suppose that $\Tcal$ is a stable class of morphisms in a small category $\Ccal$ such that axiom 3 of Definition \ref{dfn:stable} provides stability squares \textbf{weakly functorially}. That is, calling an ordered pair of morphisms $(h:A \to D,s: B \to D)$ with $s \in \Tcal$ a \textbf{$\Tcal$-cospan} from $A$ to $B$, suppose the stability axiom defines a mapping from $\Tcal$-cospans to $\Tcal$-spans satisfying the following conditions:
\begin{enumerate}
	\item For any $\Tcal$-morphism $t:B \to A$, the $\Tcal$-span coming from $(\id_A,t)$ is isomorphic in $\Span_{\Tcal}(A,B)$ to $(t,\id_B)$.
	\item If $f:A \to D$, $g:B \to D$ and $k:A' \to A$ such that $g$ is a $\Tcal$-morphism, the $\Tcal$-span obtained by applying the stability mapping along $f$ and then $k$ is isomorphic in $\Span_{\Tcal}(A',B)$ to that obtained by applying it along $f \circ k$.
	\item If $f:A \to D$, $g:B \to D$ and $e:B' \to B$ such that $e$ and $g$ are $\Tcal$-morphisms, the $\Tcal$-span obtained by applying the stability mapping to $g$ and then $e$ is isomorphic in $\Span_{\Tcal}(A,B')$ to that obtained by applying it along $g \circ e$.
\end{enumerate}
Then there is a weak composition on $\Tcal$-arches, in the sense that there are mappings
\[ \circ : \Arch_{\Tcal}(B,C) \times \Arch_{\Tcal}(A,B) \to
\Arch_{\Tcal}(A,C), \]
which are associative and unital up to isomorphism of $\Tcal$-arches. Moreover, this composition is natural in the second component up to isomorphism, in the sense that for each fixed $\Tcal$-arch $(u,h)$ in $\Arch_{\Tcal}(B,C)$, a morphism $x: (t,g) \to (t',g')$ in $\Arch_{\Tcal}(A,B)$ induces a morphism $(u,h) \circ x: (u,h) \circ (t,g) \to (u,h) \circ (t',g')$ in $\Arch_{\Tcal}(A,C)$, and the resulting mapping $(u,h) \circ - : \Arch_{\Tcal}(A,B) \to \Arch_{\Tcal}(A,C)$ is functorial up to unit and associativity isomorphisms.
\end{proposition}
\begin{proof}
Even without the listed conditions, stability provides a putative definition of the composition operation: given a consecutive pair of $\Tcal$-arches, we simply apply the stability axiom to the pair of morphisms with common codomain,
\[\begin{tikzcd}
& & P \ar[dl, "\Tcal \ni t''"'] \ar[dr, "g''"] & & \\
& E \ar[dl, "t"'] \ar[dr, "g"] & & E' \ar[dl, "t'"'] \ar[dr, "g'"] & \\
A & & B & & B';
\end{tikzcd}\]
that the resulting $\Tcal$-span $(t \circ t'', g' \circ g'')$ is a $\Tcal$-arch is easily checked. The extra conditions are needed to make this operation weakly unital and associative. The naturality in the second component is a direct consequence of the second condition.
\end{proof}

For brevity, we leave the analogous statement and proof of Proposition \ref{prop:composition} for finitely generated sites to the reader, noting that the analogue of $\Tcal$-cospans will not be duals of $\Tcal'$-multispans, but the more restrictive shape of diagram relevant to the stability axiom 3'.

In the best case scenario where it \textit{is} possible to construct a weak composition on arches, we \textit{may} obtain a bicategory (see \cite[Definition 1.1]{bicat} for a definition of bicategory) from the principal site $(\Ccal,J_{\Tcal})$, whose $0$-cells are the objects of $\Ccal$, whose $1$-cells are $\Tcal$-arches and whose $2$-cells are morphisms between these.

\begin{remark}
\label{rmk:sievecat}
The fact that $\Tcal$-arches do not assemble into a bicategory in full generality is not merely an artifact of us having restricted ourselves to the data of the stable classes of morphisms (resp. finite families), rather than the principal (resp. finitely generated) Grothendieck topologies they generate. If we expand our collections of morphisms to multiarches indexed by arbitrary $J_{\Tcal}$-covering families, the construction in \cite[Proposition 2.5(iv)]{Dense} does give a canonical family representing the composite, but it typically fails to be unital, since composing with an identity $\Tcal$-span gives a strictly larger family. We can restrict to $J$-covering sieves to avoid this problem, but even then, without pullbacks the composition may not be weakly associative, since multi-composition of $J$-covering sieves is not necessarily associative in the required sense.
\end{remark}

One situation where the hypotheses of Proposition \ref{prop:composition} are satisfied is when $\Ccal$ has pullbacks, by Lemma \ref{lem:pbstable}.

\begin{corollary}
Let $(\Ccal,J_{\Tcal})$ be a principal site where $\Ccal$ has pullbacks, such as the canonical sites on locally regular categories we shall see in Definition \ref{dfn:regcoh}. Then the objects of $\Ccal$, the $\Tcal$-spans on $\Ccal$ and the morphisms between these assemble into a bicategory. In particular, the composition operations of Proposition \ref{prop:composition} are natural in the first component as well as the second in this case.
\end{corollary}
\begin{proof}
The verification of the conditions in Proposition \ref{prop:composition} is straightforward; note that we actually require a specified choice of pullbacks, but the mediating isomorphisms are provided by universal properties. For the final claim, we observe that in the third condition of Proposition \ref{prop:composition}, we no longer need to restrict ourselves to the case where $e:B' \to B$ is a $\Tcal$-morphism, since we can complete the defining rectangle with a pullback square ($\Tcal$-morphisms are indicated with double-headed arrows):
\[\begin{tikzcd}
A' \ar[dr, phantom, "\lrcorner", very near start]
\ar[rr, bend left, "(g\circ e)'", two heads]
\ar[d, "f''"'] \ar[r, "e'"'] &
A \ar[dr, phantom, "\lrcorner", very near start]
\ar[r,"g'"', two heads] \ar[d, "f'"'] & B \ar[d, "f"] \\
B' \ar[r, "e"] \ar[rr, bend right, "g\circ e"', two heads] &
B \ar[r,"g", two heads] & D.
\end{tikzcd}\]
The morphism $e'$ provides the morphism of $\Tcal$-spans corresponding to $e$ to make the composition natural (again, up to relevant isomorphisms) in the first component, as claimed.

The commutativity of the associativity and identity coherence diagrams which are required to formally make this a bicategory are guaranteed by the uniqueness in the universal property of the pullbacks involved.
\end{proof}

As usual, the analogous result for finitely generated sites holds, but we omit the proof.

Whether it satisfies all of the requirements of a bicategory or not, however, the relationship between the $\Tcal$-arch structure and the subcategory of $\Sh(\Ccal,J_{\Tcal})$ on the representables is simply that of collapsing the $2$-cells, in the following sense:

\begin{lemma}
\label{lem:Tarch}
Two $\Tcal$-arches (resp. $\Tcal'$-multiarches) from $A$ to $B$ on a principal (resp. finitely generated) site are identified by $\ell$ if and only if they are in the same component in the category $\Arch_\Tcal(A,B)$ (resp. $\mArch_{\Tcal'}(A,B)$) described in Lemma \ref{lem:archcat}.
\end{lemma}
\begin{proof}
In one direction, if $x: (t:E \to A, g:E \to B) \to (t':E' \to A, g':E' \to B)$, then by definition the unique morphism $\ell(t,g): \ell(A) \to \ell(B)$ with $\ell(g) = \ell(t,g) \circ \ell(t)$ also satisfies $\ell(g') = \ell(t,g) \circ \ell(t')$, whence $\ell(t',g') = \ell(t,g)$. Thus $\ell$ identifies $\Tcal$-spans in the same component.

Conversely, applying \cite[Proposition 2.5(iii)]{Dense}, two $\Tcal$-arches $(t:E \to A, g:E \to B)$ and $(t':E' \to A, g':E' \to B)$ induce the same morphism in $\Sh(\Ccal,J_{\Tcal})$ if and only if there is a $\Tcal$-morphism $s:D \to A$ and morphisms $h:D \to E$ and $h':D' \to E'$ with $t \circ h = s = t' \circ h'$ and $g \circ h \equiv_{J_{\Tcal}} g' \circ h'$. Expanding on the latter condition, this implies the existence of some $\Tcal$-morphism $u:E \to D$ equalizing $g \circ h$ and $g' \circ h'$. But then $(t \circ h \circ u, g \circ h \circ u) = (t' \circ h' \circ u, g' \circ h' \circ u)$ is easily shown to be a $\Tcal$-arch, and it admits morphisms $h \circ u$ and $h' \circ u'$ to $(t,g)$ and $(t',g')$ respectively, whence these are in the same connected component, as required.
\end{proof}

We can strengthen Lemma \ref{lem:Tarch} to reconstruct the full subcategory of $\Sh(\Ccal,J_{\Tcal})$ on the representable sheaves.

\begin{scholium}
Let $(\Ccal,J_{\Tcal})$ be a principal site. Then the full subcategory of $\Sh(\Ccal,J_{\Tcal})$ on the representable sheaves is equivalent to the category whose objects are the objects of $\Ccal$ and whose morphisms $A \to B$ are indexed by the connected components of the category $\Arch_{\Tcal}(A,B)$.

Similarly, if $(\Ccal,J_{\Tcal'})$ is a finitely generated site, then the full subcategory of $\Sh(\Ccal,J_{\Tcal'})$ on the representable sheaves is equivalent to the category whose objects are the objects of $\Ccal$ and whose morphisms $A \to B$ are indexed by the connected components of the category $\mArch_{\Tcal'}(A,B)$.
\end{scholium}
\begin{proof}
Observe that in the definition of composition given in the proof of Proposition \ref{prop:composition}, \textit{any} choice of stability square will produce a $\Tcal$-arch (resp. $\Tcal'$-multiarch) lying in the same component of $\Arch_\Tcal(A,C)$ (resp. $\mArch_{\Tcal'}(A,C)$), since this $\Tcal$-arch (resp. $\Tcal'$-multiarch) will necessarily be mapped by $\ell$ to the composite of the morphisms corresponding to the pair of arches (resp. multiarches) being composed. Thus, even without weak functoriality, the composition is well-defined on connected components, as required.
\end{proof}

In the subcanonical case, where all $\Tcal$-morphisms are strict epimorphisms (resp. all $\Tcal'$-families are jointly strictly epimorphic), the computations from this section simplify greatly. Indeed, $\ell$ is full and faithful in this case, which means that every component of each category $\Arch_{\Tcal}(A,B)$ (resp. $\mArch_{\Tcal'}(A,C)$) contains a unique (multi)arch of the form $(\id_A,f)$. In a $\Tcal$-arch $(t,g)$, $g$ coequalizes every pair of morphisms which $t$ coequalizes by definition, whence the morphism $\ell(t,g)$ corresponds to the unique morphism $A \to B$ factorizing $g$ through $t$; the morphisms representing multiarches are recovered analogously from the universal properties of jointly strictly epic families.

\subsection{Quotients of principal sites}
\label{ssec:quotient}

Rather than directly computing the category of representable sheaves in $\Sh(\Ccal,J_{\Tcal})$ via $\Tcal$-arches, we might hope to simplify things by first modifying the principal site.

In Kondo and Yasuda's definition of `$B$-site', they assume that the underlying category is an `$E$-category', which is to say that every morphism is an epimorphism, \cite[Definitions 4.1.1, 4.2.1]{SPM}, which seems very restrictive. However, by taking the quotient of $\Ccal$ by the canonical congruence, we show here that we may assume that $\Tcal$ is contained in the class of epimorphisms of $\Ccal$ without loss of generality, since the corresponding topos of sheaves is equivalent to that on the original site.

\begin{proposition}
\label{prop:congruence}
Let $(\Ccal,J_{\Tcal})$ be a principal site. Then there is a canonical congruence $\sim$ on $\Ccal$ such that $(\Ccal/{\sim}, J_{\Tcal/{\sim}})$ is a principal site with $\Tcal/{\sim}$ a subclass of the epimorphisms of $\Ccal/{\sim}$, and with $\Sh(\Ccal,J_{\Tcal}) \simeq \Sh(\Ccal/{\sim},J_{\Tcal/{\sim}})$.

Similarly, if $(\Ccal,J_{\Tcal'})$ is a finitely generated site, there is a congruence $\sim$ on $\Ccal$ such that $(\Ccal/{\sim}, J_{\Tcal'/{\sim}})$ is a finitely generated site with $\Tcal'/{\sim}$ a subclass of the epimorphisms of $\Ccal/{\sim}$, and with $\Sh(\Ccal,J_{\Tcal'}) \simeq \Sh(\Ccal/{\sim},J_{\Tcal'/{\sim}})$.
\end{proposition}
\begin{proof}
The congruence $\sim$ is that induced by the functor $\ell:\Ccal \to \Sh(\Ccal,J_{\Tcal})$. Explicitly, let $f\sim f':C \to D$ whenever there is a morphism $h:C' \rightrightarrows C$ in $\Tcal$ with $fh=f'h$. To verify that this is a congruence, suppose we are also given morphisms $g \sim g':D \rightrightarrows E$ which are equalized by $k:D'\to D$. Stability of $k$ along $fh = f'h$ gives $k' \in \Tcal$:
\[\begin{tikzcd}
C'' \ar[d, "k'"] \ar[rr] & & D' \ar[d, "k"] & \\
C' \ar[r, "h"] & C \ar[r,shift left, "f"] \ar[r,shift right, "f'"'] &
D  \ar[r,shift left, "g"] \ar[r,shift right, "g'"'] & E;
\end{tikzcd}\]
$\Tcal$ is closed under composition and $gfhk' = g'f'hk'$, so $gf \sim g'f'$ as required.

By the definition of the congruence it is immediate that the morphisms in $\Tcal/{\sim}$ are epimorphisms. Moreover, the canonical functor $(\Ccal,J_{\Tcal}) \to (\Ccal/{\sim},J_{\Tcal/{\sim}})$ is a morphism and comorphism of sites\footnote{See Definitions \ref{dfn:morsite} and \ref{dfn:comorphism} below.}, since it is cover-preserving, cover-lifting and flat. Either by checking denseness conditions of \cite{Dense} or since the congruence is induced by the functor $\ell$, we recover the equivalence $\Sh(\Ccal,J_{\Tcal}) \simeq \Sh(\Ccal/{\sim},J_{\Tcal/{\sim}})$.

The congruence for a finitely generated site has $f \sim f'$ whenever $fh_i = f'h_i$ for each $h_i$ in a $\Tcal'$-family. The remainder of the proof is analogous.
\end{proof}

The reduced site $(\Ccal/{\sim},J_{\Tcal/{\sim}})$ in Proposition \ref{prop:congruence} can be obtained in various alternative ways. By construction, the functor $\Ccal \to \Ccal/{\sim}$ is the universal functor with domain $\Ccal$ sending $\Tcal$-morphisms to epimorphisms, so it is not surprising thanks to the proof of Proposition \ref{prop:representable} that $\ell$ canonically factors through it. This site also coincides with the one obtained by lifting the hyperconnected--localic factorization of the geometric equivalence of toposes $\Sh(\Ccal,J_{\Tcal}) \simeq \Sh(\Sh(\Ccal,J_{\Tcal}), J_{\mathrm{can}})$ to the level of sites, also described in \cite[\S 6.5]{Dense}. See Section \ref{ssec:morsites} below for more on morphisms of sites.

After making this simplification, $J_{\Tcal}$-local equality (resp. $J_{\Tcal'}$-local equality) reduces to ordinary equality, so that for example a $\Tcal$-arch from $A$ to $B$ is a $\Tcal$-span as in \eqref{eq:span} such that for any $h,k:D\rightrightarrows C$ with $f \circ h = f \circ k$ we have $g \circ h = g \circ k$. In particular, by considering the arches in which the $\Tcal$-morphism is an identity, we see that the functor $\ell:(\Ccal/{\sim},J_{\Tcal/{\sim}}) \to \Sh(\Ccal,J_{\Tcal})$ is \textit{faithful}, and further that $\ell:(\Ccal,J_{\Tcal}) \to \Sh(\Ccal,J_{\Tcal})$ is faithful if and only if every $\Tcal$-morphism is an epimorphism, which is to say that the congruence ${\sim}$ is trivial.

\begin{corollary}
\label{crly:full}
Let $(\Ccal,J_{\Tcal})$ be a principal site and let ${\sim}$ be the congruence on $\Ccal$ from Proposition \ref{prop:congruence}. Then the functor $\ell:(\Ccal,J_{\Tcal}) \to \Sh(\Ccal,J_{\Tcal})$ is full if and only if every $\Tcal/{\sim}$-morphism in $\Ccal/{\sim}$ is a \textbf{strict} epimorphism.

Similarly, if $(\Ccal,J_{\Tcal'})$ is a finitely generated site, then $\ell:(\Ccal,J_{\Tcal'}) \to \Sh(\Ccal,J_{\Tcal'})$ is full if and only if every $\Tcal'/{\sim}$-family in $\Ccal/{\sim}$ is a strictly epimorphic family.
\end{corollary}
\begin{proof}
If the hypothesis holds, then the site $(\Ccal/{\sim},
J_{\Tcal/{\sim}})$ is subcanonical, so the induced functor to the topos is full and faithful, whence $\ell$ with domain $(\Ccal,J_{\Tcal})$ is full.

Conversely, given that $\ell$ is full, suppose that $t:E \to A$ is in $\Tcal$. Suppose that we have $g:E \to B$ such that whenever $h,k:E'\to E$ with $t \circ h = t \circ k$, we have $g \circ h = g \circ k$. Then $(t,g)$ is a $\Tcal$-arch, and there is a morphism $\ell(t,g)$ in $\Sh(\Ccal,J_\Tcal)$ completing the triangle. By fullness of $\ell$, this is the image of a morphism $A \to B$ in $\Ccal/{\sim}$, and by the definition of $\sim$ there is at most one such morphism, whence $t/{\sim}$ is a strict epimorphism, as claimed.

The argument for finitely generated sites is analogous.
\end{proof}

It follows from Corollary \ref{crly:full} that the restriction of the codomain of $\ell$ to the full subcategory of $\Sh(\Ccal,J_{\Tcal})$ on the objects of the form $\ell(C)$ is the universal functor sending $\Tcal$-morphisms to strict epimorphisms. This site can alternatively be obtained by considering the lifting of the surjection--inclusion factorization of the geometric equivalence of toposes to the level of sites, as seen in \cite[\S 6.1]{Dense}.

\begin{example}
\label{xmpl:atomic}
Recall that a category $\Ccal$ satisfies the \textbf{right Ore condition} if any cospan can be completed to a commutative square. This is exactly the condition needed to make the class of \textit{all} morphisms of $\Ccal$ stable, and the corresponding principal topology is more commonly called the \textbf{atomic topology}, $J_{at}$, while the site $(\Ccal, J_{at})$ is called an \textbf{atomic site}. The above results show that we may reduce any atomic site to one in which every morphism is epic (hence a `$B$-site' in the terminology of Kondo and Yasuda).
\end{example}

\subsection{Reductive and coalescent categories}
\label{ssec:redcat}

Returning to our study of the subcategories of supercompact and compact objects, we observe that the epimorphisms they inherit from $\Ecal$ always meet most of the requirements for stability.

\begin{lemma}
\label{lem:stable}
For any Grothendieck topos $\Ecal$, let $\Ccal_s$, $\Ccal_c$ the usual subcategories. Then the class $\Tcal$ of epimorphisms in $\Ccal_s$ which are inherited from $\Ecal$ satisfies axioms 1,2 and 4 for stable classes, while the class $\Tcal'$ of finite jointly epimorphic families on $\Ccal_c$ inherited from $\Ecal$ satisfies axioms 1',2',4' and 5'.
\end{lemma}
\begin{proof}
Clearly $\Ccal_s$ and $\Ccal_c$ inherit identities from $\Ecal$, so $\Tcal$ contains these and $\Tcal'$ contains the singleton families of the identities. Since epimorphisms are stable under composition in $\Ecal$, $\Tcal$ is closed under composition. Multicomposition of finite jointly covering families in $\Ecal$ is similarly direct, giving the second axiom for $\Tcal'$. Axioms 4' and 5' are straightforward.
\end{proof}

By Theorem \ref{thm:canon}, if $\Ecal$ is supercompactly generated (resp. compactly generated) then axiom 3 (resp. axiom 3') must also be satisfied in each case. Note that the converse fails: stability of $\Ecal$-epimorphisms in $\Ccal_s$ or stability of $\Ecal$-epimorphic finite families in $\Ccal_c$ are not sufficient to guarantee that $\Ecal$ is supercompactly or compactly generated. Indeed, if $\Ecal$ a non-degenerate topos where the only compact object is the initial object, such as that exhibited in Example \ref{xmpl:R}, the stability axioms are trivially satisfied but $\Ecal$ is neither supercompactly nor compactly generated.

In the remainder of this section, we refine the concepts of principal and finitely generated sites in order to obtain a characterization of the categories $\Ccal_s$ (resp. $\Ccal_c$) of supercompact (resp. compact) objects in supercompactly (resp. compactly) generated toposes.

\begin{definition}
\label{dfn:reductive}
We say a small category $\Ccal$ is \textbf{reductive} if it has funneling colimits and its class of strict epimorphisms is stable. The \textbf{reductive topology} $J_r$ on a reductive category is the principal topology generated by its class of strict epimorphisms.

We say $\Ccal$ is \textbf{coalescent} if it has multifunneling colimits and its class of (strictly) epic finite families is stable. The \textbf{coalescent topology} $J_c$ on a coalescent category is the finitely generated topology on its class of strictly epic finite families.
\end{definition}

\begin{remark}
Some justification for this naming and notation is warranted. 

The names of the categories are intended to evoke the presence of funneling (resp. multifunneling) colimits, since any diagram in them of the respective shapes `is reduced' (resp. `coalesces') by composing with a suitable epimorphism (resp. jointly epimorphic family). If we consider the reductive category
\[\Ccal : = A \rightrightarrows B \too C,\]
featuring in Example \ref{xmpl:nonreg} below, for example, we can identify a functor $F: \Ccal \to \Set$ with a directed graph, equipped with a mapping on its set of vertices which identifies nodes in the same connected component; if $F$ preserves funneling colimits, then the image of the epimorphism $B \too C$ must exactly be the map reducing the graph to its set of connected components.

The names were also chosen to have their first few letters in common with \textit{regular} and \textit{coherent} respectively, since the regular and coherent objects in a topos are respectively subclasses of the supercompact and compact objects. Thus, while the $r$ and $c$ in $J_r$ and $J_c$ stand for \textit{reductive} and \textit{coalescent} respectively, we shall see that when a category is both regular and reductive, the regular topology (see Definition \ref{dfn:regtop} below) coincides with the reductive topology, so the $r$ on $J_r$ could also mean `regular'; similarly for coalescent and coherent.
\end{remark}

While not every stable class of finite families need contain a stable class of singleton morphisms, we record the fact that this does happen when the families involved are strictly epimorphic families.

\begin{lemma}
\label{lem:collred}
Any coalescent category is a reductive category with finite colimits and a strict initial object.
\end{lemma}
\begin{proof}
If $\Ccal$ is a coalescent category, it certainly has the required colimits, so it suffices to show stability of strict epimorphisms. Indeed, if $t:D \too C$ is strict and $g: B \to C$ is any morphism in $\Ccal$, then since $\{t\}$ is a strictly epic family, there is some strictly epic family $\{h_j: A_j \to B \mid i = 1, \dotsc, m\}$ over $B$ such that each $g \circ h_j$ factors through $t$. Factoring this family through the coproduct $A_1 + \cdots + A_m$ gives a strict epimorphism completing the required stability square (even in the case $m=0$).

To see that the initial object is strict, observe that the empty family is a strict jointly epic family on the initial object, so given any morphism $A \to 0$, stability forces the empty family to be jointly epic over $A$, whence $A$ is also an initial object.
\end{proof}

It would be easy to mistakenly conclude based on the results presented thus far that the subcategory of supercompact objects in the category of sheaves on a reductive category $\Ccal$ should be equivalent to $\Ccal$. The flaw in this reasoning lies in the fact that, while the functor $\ell: \Ccal \to \Sh(\Ccal,J_r)$ is full and faithful (since $(\Ccal,J_r)$ is a subcanonical site) and this functor preserves strict epimorphisms, it \textit{does not} preserve all funneling colimits; a similar argument applies for coalescent categories.

\begin{example}
\label{xmpl:tworel}
Consider the following categories, $\Ccal$ and $\Ccal'$ on the left and right respectively. It is easily checked that they are reductive, with strict epimorphisms identified with two heads.
\[\begin{tikzcd}[row sep = small]
R_1 \ar[dr, shift left] \ar[dr, shift right] & &
R_2 \ar[dl, shift left] \ar[dl, shift right] \\
& A \ar[dd, two heads] \ar[dl, phantom] & \\
\cdot & & \\
& B &
\end{tikzcd}
\hspace{1em}
\begin{tikzcd}[row sep = small]
R_1 \ar[dr, shift left] \ar[dr, shift right]
\ar[dd, dashed, shift left] \ar[dd, dashed, shift right] & &
R_2 \ar[dl, shift left] \ar[dl, shift right]
\ar[dd, dashed, shift left] \ar[dd, dashed, shift right] \\
& A \ar[dd, two heads] \ar[dr, two heads] \ar[dl, two heads] & \\
C \ar[dr, two heads] & & D \ar[dl, two heads] \\
& B &
\end{tikzcd}\]
The coequalization is that suggested by the positioning, so that in the first diagram, the coequalizer of the pair coming from $R_1$ is the terminal object $B$, but in the second diagram, $C$ is the coequalizer of the pair coming from $R_2$.

One can calculate directly that the category of supercompact objects in $\Sh(\Ccal,J_r) \simeq \Sh(\Ccal',J_r)$ is equivalent to $\Ccal'$. Indeed, the functor $\ell: \Ccal \to \Sh(\Ccal,J_r)$ does not preserve the coequalizer diagram $R_1 \rightrightarrows A \too B$.
\end{example}

We shall see a further example of a failure of $\ell$ to preserve coequalizers (and hence funneling colimits) in Example \ref{xmpl:TF}. In order to understand which colimits are preserved by $\ell$, we apply criteria derived in \cite[Corollary 2.25]{Dense}, which we recall here; we refer the reader to that monograph once again for the proof.
\begin{lemma}
\label{lem:colimitpres}
Let $(\Ccal,J)$ be a site, $F: \Dcal \to \Ccal$ a diagram and $\vec{\lambda} = \{\lambda_D: F(D) \to C_0 \mid D \in \Dcal\}$ a cocone under $F$ with vertex $C_0$. Then $\vec{\lambda}$ is sent by the canonical functor $\ell:\Ccal \to \Sh(\Ccal,J)$ to a colimit cone if and only if:
\begin{enumerate}[label = ({\roman*})]
	\item For any object $C$ and morphism $g:C \to C_0$ in $\Ccal$, there is a $J$-covering family $\{f_i: C_i \to C \mid i \in I\}$ and for each $i \in I$, an object $D_i$ of $\Dcal$ and an arrow $h_i: C_i \to F(D_i)$ such that $\lambda_{D_i} \circ h_i = g \circ f_i$.
	\item For any object $C$ in $\Ccal$ and morphisms $g_1:C \to F(X)$, $g_2:C \to F(Y)$ such that $\lambda_{X} \circ g_1 = \lambda_{Y} \circ g_2$, there is a $J$-covering family $\{f_i: C_i \to C \mid i \in I\}$ such that for each $i \in I$, $g_1 \circ f_i$ and $g_2 \circ f_i$ lie in the same connected component of $(C_i \downarrow F)$.
\end{enumerate}
\end{lemma}

Observe that the first condition can be simplified.

\begin{lemma}
\label{lem:simplify}
Condition (i) of Lemma \ref{lem:colimitpres} is equivalent to the requirement that $\{\lambda_D: F(D) \to C_0 \mid D \in \Dcal\}$ is a $J$-covering family. 
\end{lemma}
\begin{proof}
Consider the case where $g$ is the identity on $C_0$. There must be some $J$-covering family $\{f_i: C_i \to C \mid i \in I\}$ and for each $i \in I$, an object $D_i$ of $\Dcal$ and an arrow $h_i: C_i \to F(D_i)$ such that $\lambda_{D_i} \circ h_i = g \circ f_i$. Since every morphism in this covering family factors through a leg of the colimit cone, the legs of the cone must form a $J$-covering family. Conversely, given any morphism $g:C \to C_0$, since $J$-covering families are required to be stable, pulling back $\vec{\lambda}$ gives the required $J$-covering family over $C$ to fulfill condition (i).
\end{proof}

Applying this in the particular case of funneling or multifunneling colimits in a principal or finitely generated site $(\Ccal,J)$, we get the following:
\begin{proposition}
\label{prop:colims}
Let $(\Ccal,J_{\Tcal})$ be a principal site and $F: \Dcal \to \Ccal$ a funneling diagram with weakly terminal object $F(D_0)$ and $\vec{\lambda} = \{\lambda_D: F(D) \to C_0 \mid D \in \Dcal\}$ a cocone under $F$ with vertex $C_0$. Then $\vec{\lambda}$ is sent by the canonical functor $\ell:\Ccal \to \Sh(\Ccal,J_{\Tcal})$ to a colimit cone if and only if:
\begin{enumerate}[label = ({\roman*})]
	\item $\lambda_{D_0} \in \Tcal$.
	\item For any object $C$ in $\Ccal$ and morphisms $g_1, g_2:C \rightrightarrows F(D_0)$, such that $\lambda_{D_0} \circ g_1 = \lambda_{D_0} \circ g_2$, there is a $\Tcal$-morphism $t: C' \to C$ such that $g_1 \circ t$ and $g_2 \circ t$ lie in the same connected component of $(C' \downarrow F)$.
\end{enumerate}

Similarly, if $(\Ccal,J_{\Tcal'})$ is a finitely generated site and $F: \Dcal \to \Ccal$ a multifunneling diagram with weakly terminal objects $F(D_1),\dotsc,F(D_n)$ and $\vec{\lambda} := \{\lambda_D: F(D) \to C_0 \mid D \in \Dcal\}$ a cocone under $F$ with vertex $C_0$, then $\vec{\lambda}$ is sent by $\ell:\Ccal \to \Sh(\Ccal,J_{\Tcal'})$ to a colimit cone if and only if:
\begin{enumerate}[label = ({\roman*})]
	\item $\{\lambda_{D_1},\dotsc,\lambda_{D_n}\} \in \Tcal'$.
	\item For any object $C$ in $\Ccal$ and morphisms $g_1:C \to F(D_k)$ and $g_2: C \to F(D_l)$ with $1 \leq k,l \leq n$, such that $\lambda_{D_k} \circ g_1 = \lambda_{D_l} \circ g_2$, there is a $\Tcal'$-family $\{t_i: C_i \to C \mid 1 \leq i \leq N\}$ such that $g_1 \circ t_i$ and $g_2 \circ t_i$ lie in the same connected component of $(C_i \downarrow F)$.
\end{enumerate}
\end{proposition}
\begin{proof}
For (i) in each case, we apply Lemma \ref{lem:simplify} and then the fact that every morphism in the cone factors through $\lambda_{D_0}$ (resp. one of the $\lambda_{D_k}$) to deduce, thanks to stability axiom 4 (resp. axioms 4' and 5') that condition (i) of Lemma \ref{lem:colimitpres} is equivalent to the given statement.

Condition (ii) in each case is a consequence of condition (ii) in Lemma \ref{lem:colimitpres}, having simply taken the special case $X = Y = D_0$ (resp. $X = D_k$ and $Y = D_l$). Conversely, for arbitrary $g_1:C \to F(X)$ and $g_2:C \to F(Y)$ such that $\lambda_{X} \circ g_1 = \lambda_{Y} \circ g_2$, we may extend this via any of the morphisms $p_1: X \to D_0$ and $p_2: Y \to D_0$ in the diagram (resp. $p_1: X \to D_k$ and $p_2: Y \to D_l$) so that $\lambda_{D_0} \circ Fp_1 \circ g_1 = \lambda_{D_0} \circ Fp_2 \circ g_2$ (resp. $\lambda_{D_k} \circ Fp_1 \circ g_1 = \lambda_{D_k} \circ Fp_2 \circ g_2$). Given a $\Tcal$-morphism $t:C' \to C$ such that $Fp_1 \circ g_1 \circ t$ and $Fp_2 \circ g_2 \circ t$ are in the same connected component of $(C' \downarrow F)$ (resp. a $\Tcal'$-family $\{t_i: C_i \to C\}$ such that $Fp_1 \circ g_1 \circ t_i$ and $Fp_2 \circ g_2 \circ t_i$ are in the same connected component of $(C_i \downarrow F)$), it is clear that $g_1 \circ t$ and $g_2 \circ t$ (resp. $g_1 \circ t_i$ and $g_2 \circ t_i$ for each $i$) also lie in this same component, as required.
\end{proof}

As an illustration of the second half of Proposition \ref{prop:colims}, we deduce the special case of binary coproducts (cf. Lemma \ref{lem:multi}).
\begin{corollary}
\label{crly:coproduct}
For $(\Ccal,J_{\Tcal'})$ a finitely generated site, a cospan $\lambda_1: X_1 \rightarrow Y \leftarrow X_2: \lambda_2$ is mapped by $\ell: \Ccal \to \Sh(\Ccal,J_{\Tcal'})$ to a coproduct cocone if and only if:
\begin{enumerate}[label = ({\roman*})]
	\item $\{\lambda_1,\lambda_2\} \in \Tcal'$,
	\item Whenever $f_1: C \to X_1$ and $f_2: C \to X_2$ have $\lambda_1 \circ f_1 = \lambda_2 \circ f_2$, the empty family is $\Tcal'$ covering on $C$, and
	\item Whenever $f,f':C \rightrightarrows X_1$ are coequalized by $\lambda_1$, there is a $\Tcal'$-covering family on $C$ consisting of morphisms equalizing $f$ and $f'$ (and similarly for pairs of morphisms into $X_2$).
\end{enumerate}
If $\Tcal'$-covering families are jointly epic, we can replace the last condition by the condition that $\lambda_1$ and $\lambda_2$ must be monic.
\end{corollary}
\begin{proof}
This is a direct application of Proposition \ref{prop:colims} for multifunneling colimits in the case of a finite discrete diagram, where two morphisms are in the same connected component of $(C \downarrow F)$ if and only if they are equal (which is impossible if they have distinct codomains).
\end{proof}

Imposing these conditions on funneling and multifunneling colimit cones in reductive and coalescent categories, we arrive at the following definitions. 

\begin{definition}
\label{dfn:redeff}
We say a reductive or coalescent category $\Ccal$ is \textbf{effectual} if the respective conditions of Proposition \ref{prop:colims} hold for every funneling (resp. multifunneling) colimit diagram. 

Explicitly, a reductive category is effectual if for every funneling diagram $F: \Dcal \to \Ccal$ with colimit expressed by $\lambda: F(D_0) \too C_0$, given morphisms $g_1, g_2:C \rightrightarrows F(D_0)$ in $\Ccal$ such that $\lambda \circ g_1 = \lambda \circ g_2$, there is a strict epimorphism $t: C' \too C$ such that $g_1 \circ t$ and $g_2 \circ t$ lie in the same connected component of $(C' \downarrow F)$.

Similarly, a coalescent category is effectual if for every multifunneling diagram $F: \Dcal \to \Ccal$ with colimit expressed by $\{\lambda_i: F(D_i) \to C_0\}$, given morphisms $g_1:C \to F(D_k)$ and $g_2:C \to F(D_l)$ such that $\lambda_k \circ g_1 = \lambda_l \circ g_2$, there is a jointly strictly epimorphism family $\{t_j: C'_j \too C : 1 \leq j \leq m\}$ such that $g_1 \circ t_j$ and $g_2 \circ t_j$ lie in the same connected component of $(C'_j \downarrow F)$ for each index $j$.
\end{definition}

Recall that a category $\Ccal$ with finite coproducts is said to be \textbf{positive} if and only if coproduct inclusions are monomorphisms and finite coproducts are disjoint, in the sense that for all objects $A$ and $B$ of $\Ccal$, the following square is a pullback:
\[\begin{tikzcd}
0 \ar[r] \ar[d] & B \ar[d] \\
A \ar[r] & A + B.
\end{tikzcd}\]
Applying Corollary \ref{crly:coproduct}, we deduce:
\begin{corollary}
\label{crly:poscoal}
An effectual coalescent category is positive.
\end{corollary}

With these definitions to hand, we can finally express a definitive correspondence result between supercompactly or compactly generated toposes and their canonical sites from Theorem \ref{thm:canon}.

\begin{theorem}
\label{thm:correspondence}
Up to equivalence, there is a one-to-one correspondence between supercompactly generated Grothendieck toposes and essentially small effectual, reductive categories. The correspondence sends a topos to its essentially small category of supercompact objects and a reductive category to the topos of sheaves for the reductive topology on that category.

Similarly, there is an up-to-equivalence correspondence between compactly generated Grothendieck toposes and effectual coalescent categories.
\end{theorem}
\begin{proof}
Passing from a topos to its subcategory of supercompact (resp. compact) objects and back again gives an equivalent topos by Theorem \ref{thm:canon}; the small category must be an effectual reductive (resp. effectual coalescent) category, since the subcategory of supercompact (resp. compact) objects is closed under funneling (resp. multifunneling) colimits, so the inclusion must preserve them.

In the other direction, since the strict (resp. strict finite family) Grothendieck topology is subcanonical, a reductive (resp. coalescent) category is included faithfully as a full subcategory of the corresponding sheaf topos, and all of the representable sheaves are supercompact (resp. compact). The supercompact (resp. compact) objects are funneling (resp. multifunneling) colimits of the representable sheaves in this topos, but these colimits are preserved by $\ell$ by construction when the category is effectual, so the category of representable sheaves coincides with the category of supercompact (resp. compact) objects, as required.
\end{proof}

We shall extend this correspondence to an equivalence of $2$-categories in Section \ref{ssec:morsites}. Since geometric morphisms shall come into play at that point, we add here the following extra definition:
\begin{definition}
\label{dfn:augmented}
A reductive category $\Ccal$ is \textbf{augmented} if it has an initial object. The \textbf{augmented reductive topology} $J_{r+}$ on such a category has covering sieves generated by singleton or empty strictly jointly epic families. The resulting \textbf{augmented reductive site} $(\Ccal,J_{r+})$ is quasi-principal.
\end{definition}

Before moving on, a word of warning. In spite of Lemma \ref{lem:collred}, an effectual coalescent category \textit{may fail to be an effectual reductive category}, because the condition for a funneling colimit in a coalescent category to be preserved upon passing to the corresponding topos is weaker than the condition on reductive categories; we shall see a counterexample in Example \ref{xmpl:wedge}. Even when a coalescent category is `reductive-effectual', the corresponding toposes from Theorem \ref{thm:correspondence} are always distinct, since we have seen that the initial object of a topos is never supercompact but always compact.

\subsection{Localic supercompactly generated toposes}
\label{ssec:localic}

Recall that a Grothendieck topos is \textbf{localic} if it is of the form $\Sh(L)$, the category of (set-valued) sheaves on some locale $L$; this is equivalent to the global sections geometric morphism being localic in the sense described in Section \ref{ssec:hype}. A further characterization is the fact that their set of subterminal objects is separating; these subterminals form a frame isomorphic to the frame of opens of the locale $L$.

The conditions for $\Sh(L)$ to be supercompactly or compactly generated reduce to conditions on the frame of subterminal objects. It is worth mentioning that the cases where $\Sh(L)$ has enough coherent objects produces the notion of \textit{coherent frame}, while having enough regular objects gives a \textit{supercoherent frame}. These are respectively discussed by Banaschewski and a coauthor in \cite{cohframe} and \cite{superframe}, where they employ the subobject characterizations in the form of `way below' and `totally below' relations, respectively.

\begin{proposition}
\label{prop:localic}
A localic topos is supercompactly generated if and only if it is equivalent to $\Setswith{\Ccal}$ for some poset $\Ccal$. Moreover, any poset is an instance of an effectual reductive category.
\end{proposition}
\begin{proof}
The `if' direction is a consequence of Proposition \ref{prop:xmpls} and the fact such a presheaf topos is necessarily localic.

Conversely, if $\Ecal$ is supercompactly generated, each subterminal must be covered by its supercompact subobjects, so the supercompact subterminal objects generate the topos. A subterminal object is supercompact if and only if it has no proper cover by strictly smaller subterminals. This forces the canonical topology on the supercompact objects to be trivial, whence $\Ecal \simeq [\Ccal_s^{\mathrm{op}},\Set]$. By considering the expression of $\Ecal$ as a category of sheaves on a locale, we see that the supercompact objects must all be quotients of subterminals, and hence themselves subterminals, so $\Ccal_s$ is a poset.
\end{proof}

To characterize localic compactly generated toposes, we need to present a more obscure definition.

\begin{definition}
\label{dfn:JSL}
A \textbf{join semilattice} is a poset having all finite joins, including the bottom element. We say a join semilattice is \textbf{distributive} if for any triple of objects $(a,b,c)$ with $a \leq b \vee c$, there are elements $b' \leq b$ and $c' \leq c$ such that $a = b' \vee c'$; see \cite[\S II.5]{GLT}. This in particular holds in any distributive lattice.
\end{definition}

Distributivity, which inductively extends to arbitrary finite joins, is precisely the condition ensuring that the collection of finite join covers is a stable class of finite families. Note that finite join covers are coproduct injections, so that \textit{any distributive join semilattice is a coalescent category}. Considering the join of any element with itself, we see that any non-degenerate distributive lattice is an example of a coalescent category which fails to be positive (and hence fails to be effectual).

\begin{proposition}
\label{prop:localic2}
A localic topos is compactly generated if and only if it is the category of sheaves on a distributive join semilattice with respect to the topology that makes finite joins covering. 
\end{proposition}
\begin{proof}
Unlike in the supercompact case, the compact subterminal objects no longer populate all of $\Ccal_c$, but a similar argument applies: if $\Ccal$ is the subcategory of $\Ecal$ on the compact subterminal objects, then $\Ccal$ is a distributive join semilattice since a finite union of compact subterminal objects is compact, and so $\Ecal = \Sh(\Ccal,J)$ where $J$ is the topology on $\Ccal$ whose covering families are finite joins.
\end{proof}

By applying Theorem \ref{thm:closure}, we can immediately conclude that the above characterizations apply to the localic reflections of supercompactly and compactly generated toposes.
\begin{corollary}
The localic reflection of any supercompactly generated topos is a (localic) presheaf topos. The localic reflection of a compactly generated topos is a topos of sheaves on a distributive join semilattice.
\end{corollary}

\begin{rmk}
\label{rmk:Stone}
The preprint version of this paper contained an extensive discussion of the above results in the context of Caramello's work on Stone-type dualities, \cite{Stonetype}, which we merely summarize here.

Besides the posets and distributive lattices which we obtain as sites in Propositions \ref{prop:localic} and \ref{prop:localic2}, we may also consider the corresponding locales, which are respectively characterized as having a base of supercompact or compact opens, respectively. Extracting those locales having enough points, and selecting sufficient sets of points appropriately, we recover classical Stone-type dualities.

Taking either essential points or all points in the supercompactly generated case, we respectively obtain Alexandroff spaces (which are not sober!) and Alexandroff locales, and by comparing these with posets, recover two versions of Alexandroff duality, cf. \cite[Theorems 4.1 and 4.2]{Stonetype}.

Taking all points of a compactly generated locale (assuming the axiom of choice, there are enough of these) we recover the duality presented by Gr\"{a}tzer in \cite[\S II.5]{GLT} and refined by Celani and Calomino in \cite[Theorem 20]{DSLat}. This latter duality is more interesting for two reasons. First, it is a non-trivial extension of the better-known Stone-type duality between distributive lattices and coherent locales recalled in \cite[\S 4.2]{Stonetype} with reference to \cite[II.3.2, II.3.3]{Stone}. Second, this route to the duality entirely circumvents the voluminous proofs required by Gr\"{a}tzer to present the result.

We have not examined the maps between spaces involved here. Caramello provides many other examples of this Stone-type duality machinery corresponding to other `$C$-compactness' properties (the localic analogue of the $P$-compactness properties referenced in Remark \ref{rmk:Pcompact} above) in \cite{Stonetype}. The final results can typically be extended to $(1,2)$-categorical equivalences.
\end{rmk}

\subsection{Locally regular and coherent categories}
\label{ssec:regcoh}

As we observed earlier, supercompactness and compactness have been studied in the context of regular and coherent toposes, which are toposes of sheaves on regular and coherent categories respectively, equipped with suitable Grothendieck topologies. Here we recall the definitions of these classes of categories, as well as some more general classes, for comparison with reductive and coalescent categories.

\begin{definition}
\label{dfn:regcoh}
Recall that an epimorphism $e$ in a category $\Ccal$ is \textbf{extremal} if whenever $e = m \circ g$ with $m$ a monomorphism, then $m$ is an isomorphism.

A category is \textbf{locally regular} if it is closed under connected finite limits, it has an orthogonal (extremal epi, mono) factorization system, and every span factors through a jointly monic pair via an extremal epimorphism. Such a category is \textbf{regular} if it also has finite products (equivalently, a terminal object). Clearly, a slice (also called an `over-category') of a locally regular category is regular.

We say a category is \textbf{locally coherent} if it is locally regular and finite unions of subobjects (including the minimal subobject) exist and are stable under pullback. Such a category is \textbf{coherent} if it also has finite products.
\end{definition}

\begin{lemma}
\label{lem:extremal}
Every extremal epimorphism in a locally regular category is regular.
\end{lemma}
\begin{proof}
We adapt the proof, \cite[Proposition A1.3.4]{Ele}, that in a regular category, covers are regular epimorphisms; we omit the composition symbol for conciseness in this proof.

Let $f:A \too B$ be an extremal epimorphism in a locally regular category and let $a,b: R \rightrightarrows A$ be its kernel pair. We show that $f$ coequalizes $a$ and $b$.

Suppose $c:A \to C$ has $ca = cb$, and factorize the span $(f,c)$ as an extremal epimorphism followed by a jointly monic pair:
\[\begin{tikzcd}[row sep = small]
& & B \\
A \ar[urr, "f", two heads, bend left] \ar[drr, "c"', bend right] \ar[r, "d", two heads] &
D \ar[ur, "g"] \ar[dr, "h"'] & \\
& & C.
\end{tikzcd}\]
If we can show that $g$ is monic, then extremalness of $f$ will force it to be an isomorphism, so that $c = hg^{-1}f$ factors through $f$.

Given $k,l:E \rightrightarrows D$ with $gk = gl$, consider the following diagram composed of pullback squares:
\[\begin{tikzcd}
P \ar[rr, "m", bend left] \ar[dd, "n"', bend right] \ar[d, two heads]
\ar[dr, "p", two heads] \ar[r, two heads] &
\cdot \ar[d, two heads] \ar[r] & A \ar[d, two heads, "d"] \\
\cdot \ar[d] \ar[r, two heads] &
E \ar[d, "k"] \ar[r, "l"] & D \\
A \ar[r, two heads, "d"'] & D &
\end{tikzcd}\]
The morphism labelled $p$ is a composite of extremal epimorphisms (by stability) and hence is itself an extremal epimorphism. From this diagram and the preceding assumptions, we have $fm = gdm = gkp = glp = gdn = fn$, whence $(m,n)$ factors through $(a,b)$ via some morphism $q:P \to R$, and we have:
\[hkp = hdm = cm = caq = cbq = cn = hdn = hlp,\]
whence $hk = hl$ by epicness of $p$, but since $(g,h)$ was jointly monic, we have $k = l$, which completes the proof.
\end{proof}

\begin{definition}
\label{dfn:regtop}
The \textbf{regular topology} on a regular or locally regular category is simply the principal topology generated by the extremal epimorphisms, and similarly the \textbf{coherent topology} on a coherent or locally coherent category is the finitely generated topology generated by the finite jointly extremal epic families.
\end{definition}

By Lemma \ref{lem:extremal}, we may replace `extremal' with `regular' in the descriptions of the stable classes in this definition, whence we see that these sites are subcanonical. Accordingly we obtain a regular (resp. locally regular, coherent, locally coherent) topos of sheaves on such a site, where here the adjective merely indicates that the topos can be generated by such a site; any topos is automatically a coherent (and hence regular, locally regular and locally coherent) category\footnote{The reason for the somewhat unfortunate naming convention which we are extending here is explained in \cite[D3.3]{Ele}.}.

By Proposition \ref{prop:representable}, any locally regular topos is supercompactly generated, and any locally coherent topos is compactly generated. We are therefore led to wonder when the classes of categories coincide.

\begin{theorem}
\label{thm:intersect}
A small category is locally regular with funneling colimits if and only if it is a reductive category with pullbacks.

A small category is locally coherent with multifunneling colimits if and only if it is a coalescent category with pullbacks.

In each case, we can remove the ``locally'' adjective in exchange for adding a terminal object.
\end{theorem}
\begin{proof}
In one direction, by Lemma \ref{lem:extremal} we have that in a locally regular category, the classes of extremal, strict and regular epimorphisms all coincide, since any strict epimorphism is extremal, and they form a stable class by assumption, whence a locally regular category with funneling colimits is a reductive category with pullbacks.

Conversely, given a reductive category $\Ccal$ with pullbacks, we must show that $\Ccal$ has equalizers, since a category has connected finite limits if and only if it has both pullbacks and equalizers; the remaining conditions follow from Corollary \ref{crly:orthog} and Lemma \ref{lem:jmonic}, thanks to Lemma \ref{lem:pbstable}. Given a pair of morphisms $h,k: A \rightrightarrows B$ in $\Ccal$, consider their coequalizer $c:B \too C$. Then $\Ccal/C$ is regular, since it has pullbacks and a terminal object (so all finite limits), and it inherits the required factorization system, including that for spans, from $\Ccal$. Therefore there exists an equalizer of $h$ and $k$ as morphisms over $C$, and it is clear that this will also be their equalizer in $\Ccal$.

For the locally coherent case, by considering the strictly epic finite families of subobjects, we see that the fact that unions are stable under pullback ensures that strictly epic finite families form a stable class, as required.

Conversely, given a coalescent category $\Ccal$ with pullbacks, the argument above, Corollary \ref{crly:orthog} and Lemma \ref{lem:jmonic} give that $\Ccal$ is locally regular. For finite unions of subobjects, observe that it suffices to consider nullary and binary unions. The former are guaranteed by the strict initial object of a coalescent category, seen in Lemma \ref{lem:collred}. For the latter, observe that the union can be expressed as the pushout (a multifunneling colimit) along the intersection of the two subobjects (the pullback of the monomorphisms defining the subobjects), since any subobject containing the given pair of subobjects forms a cone under this diagram. The fact that strictly epic families are stable under pullback ensures that these unions are too, again thanks to Lemma \ref{lem:pbstable}.
\end{proof}

\begin{example}
A distributive join semilattice has pullbacks and a top element if and only if it is a distributive lattice, so the (bounded) distributive lattices are precisely the coherent distributive join semilattices.
\end{example}

The reader may have noticed that we did not include the properties of effectiveness for regular and coherent categories in Definition \ref{dfn:regcoh}: 
\begin{definition}
\label{dfn:regeff}
A (locally) regular category is \textbf{effective}\footnote{Referred to as \textbf{Barr-exactness} in older texts; we follow Johnstone in our terminology.} if all equivalence relations are kernel pairs.
\end{definition}

We chose a similar name, `effectual', for the concepts appearing in Definition \ref{dfn:redeff} because both effectuality and effectiveness are conditions equivalent to the relevant categories being recoverable from the associated topos. Indeed, a locally regular, effective category $\Ccal$ can be recovered from the topos of sheaves on $\Ccal$ for the regular topology, $\Sh(\Ccal,J_r)$, as the category of \textit{regular objects}, which were defined in Example \ref{xmpl:regular}. Similarly, if $\Ccal$ is locally coherent, positive and effective, it can be recovered from $\Sh(\Ccal,J_c)$ as the category of \textit{coherent objects}, which are defined analogously. As a special case, we recover the familiar correspondences between effective regular categories and regular toposes, or between effective, positive coherent categories (also known as \textit{pretoposes}) and coherent toposes. These results are comparable to Theorem \ref{thm:correspondence}. The concepts of effectuality and effectiveness are directly related:
\begin{proposition}
\label{prop:effective}
Let $\Ccal$ be a reductive category with pullbacks. Then if $\Ccal$ is effectual as a reductive category, it is also effective as a regular category.
\end{proposition}
\begin{proof}
First suppose that $\Ccal$ is effectual, let $a,b:R \rightrightarrows A$ be an equivalence relation on $A$, and let $\lambda: A \too B$ be its coequalizer. We must show that $(a,b)$ is the kernel pair of $c$. Given $g_1,g_2: C \rightrightarrows A$ with $\lambda \circ g_1 = \lambda \circ g_2$, by effectuality of $\Ccal$ there is a strict epimorphism $t: C' \too C$ such that $g_1 \circ t$ and $g_2 \circ t$ lie in the same connected component of $(C' \downarrow F)$, where $F: \Dcal \to \Ccal$ is the diagram picking out the parallel pair $(a,b)$.

The only form a connecting zigzag can have in $(C' \downarrow F)$ is (omitting the morphisms from $C'$ and any identity morphisms):
\[\begin{tikzcd}
& R \ar[dl, "x_1"'] \ar[dr, "y_1"] & & R \ar[dl, "x_2"'] \ar[dr, "y_2"]
& & \cdots \ar[dl, "x_3"'] \ar[dr, "y_{n-1}"] & & R \ar[dl, "x_n"'] \ar[dr, "y_n"] & \\
A && A && A && A && A,
\end{tikzcd}\]
with each $x_i$ and $y_i$ equal to $a$ or $b$. By reflexivity of $R$ we may construct a zigzag consisting of $n > 0$ spans. We show by downward induction that there must always be a zigzag of exactly $1$ span, which corresponds to a factorization of the span $(g_1 \circ t,g_2 \circ t)$ through the relation $(a,b)$.

Clearly if $x_1 = y_1$ we may omit the first zigzag, and similarly for all of the others, so we may assume that $x_i \neq y_i$. By symmetry of $R$, if $(x_i,y_i) = (b,a)$, this factorizes through $(a,b)$, which gives an alternative zigzag of the same length in $(C' \downarrow F)$, so we may assume $x_i = a$ and $y_i = b$. If $n \geq 2$, we may factor through the pullback of $x_2$ along $y_1$, and transitivity of $R$ means that we get a strictly shorter zigzag; iterating this, we reach a zigzag with $n = 1$, as required.

Finally, taking the (regular epimorphism, relation) factorization of $(g_1 \circ t,g_2 \circ t)$, we conclude that the resulting relation (and hence $(g_1,g_2)$) must factor uniquely through $R$, as required.
\end{proof}

As we shall see in Example \ref{xmpl:countable}, the converse of Proposition \ref{prop:effective} fails, which is why we did not employ the same name for these concepts.

\begin{remark}
While we provided a direct proof of Proposition \ref{prop:effective} for completeness, we could more succinctly have reasoned as follows. When a category is both regular and reductive, the strict and regular topologies on the category coincide. If such a category is effectual, therefore, all of the supercompact objects in its topos of sheaves are regular, and hence it must also be effective, since it is equivalent to the category of regular objects in its associated topos.
\end{remark}

More generally, one might take an interest in the regular objects in a supercompactly generated topos, or the coherent objects in a compactly generated topos. However, this class of objects need not be stable under pullback in general, and hence may not assemble into a locally regular (resp. locally coherent) category. Nonetheless, by considering the induced Grothendieck topology on this subcategory, we obtain a supercompactly generated subtopos of the original topos. Iterating this process recursively, in the countable limit we obtain a maximal pullback-stable class of regular objects, although the resulting subcategory still may not be a locally regular category in the sense of Definition \ref{dfn:regcoh}, since that definition also required the presence of equalizers. Since it is unclear to us whether this class of objects or the corresponding subtopos have an interesting universal property, and since we lack interesting specific examples of this construction, we terminate our analysis here.

\subsection{Morphisms of sites}
\label{ssec:morsites}

Morphisms of sites are most easily defined on sites whose underlying category has finite limits. However, there is no reason for this property to hold in a general principal or finitely generated site. We must therefore use the more general definition of morphism of sites, which we quote from \cite[Definition 3.2]{Dense}.

\begin{definition}
\label{dfn:morsite}
Let $(\Ccal,J)$ and $(\Dcal,K)$ be sites. Then a functor $F: \Ccal \to \Dcal$ is a \textbf{morphism of sites} if it satisfies the following conditions:
\begin{enumerate}
	\item $F$ sends every $J$-covering family in $\Ccal$ to a $K$-covering family in $\Dcal$.
	\item Every object $D$ of $\Dcal$ admits a $K$-covering family $\{g_i: D_i \to D \mid i \in I\}$ by objects $D_i$ admitting morphisms $h_i: D_i \to F(C'_i)$ to objects in the image of $F$.
	\item For any objects $C_1,C_2$ of $\Ccal$ and any span $(\lambda'_1:D \to F(C_1), \lambda'_2:D \to F(C_2))$ in $\Dcal$, there exists a $K$-covering family $\{g_i: D_i \to D \mid i \in I\}$, a family of spans in $\Ccal$, $\{(\lambda_1^i: C'_i \to C_1, \lambda_2^i: C'_i \to C_2) \mid i \in I \},$ and a family of morphisms in $\Dcal$, $\{h_i: D_i \to F(C_i)\}$, such that the following diagram commutes:
	\[\begin{tikzcd}[column sep = small]
		& D_i \ar[dl, "h_i"'] \ar[dr, "g_i"] & \\
		F(C'_i) \ar[d, "F(\lambda_1^i)"']
		\ar[drr, "F(\lambda_2^i)", very near start] & &
		D \ar[dll, "\lambda'_1"', very near start]
		\ar[d, "\lambda'_2"] \\
		F(C_1) & & F(C_2)
	\end{tikzcd}\]
	\item For any pair of arrows $f_1,f_2:C_1 \rightrightarrows C_2$ in $\Ccal$ and any arrow $\lambda':D \to F(C_1)$ of $\Dcal$ satisfying $F(f_1) \circ \lambda' = F(f_2) \circ \lambda'$, there exist a $K$-covering family in $\Dcal$ $\{g_i:D_i \to D \mid i \in I\},$ and a family of morphisms of $\Ccal$ $\{\lambda^i:C'_i \to C_1 \mid i \in I\},$ satisfying $f_1 \circ \lambda^i = f_2 \circ \lambda^i$ for all $i \in I$ and of morphisms of $\Dcal$ $\{h_i:D_i \to F(C'_i) \mid i \in I\},$ making the following squares commutative:
	\[\begin{tikzcd}
		D_i \ar[d, "h_i"'] \ar[r, "g_i"] & D \ar[d, "g"] \\
		F(C'_i) \ar[r, "F(\lambda^i)"'] & F(C_1)
	\end{tikzcd}\]
\end{enumerate}
\end{definition}

\begin{remark}
\label{rmk:findiagcover}
It is not difficult to show by induction on finite diagrams that the last three conditions are equivalent to the following more condensed condition:

Given a finite diagram $A:\Ibb \to \Ccal$ and a cone $L'$ over $F \circ A$ in $\Dcal$ with apex $D$, there is a $K$-covering family of morphisms $\{g_i: D_i \to D \mid i \in I\}$ and cones $L_i$ over $A$ in $\Ccal$ with apex $C'_i$ such that $L \circ g_i$ factors through $F(L_i)$ for each $i \in I$, in the sense that there exist morphisms $h_i:D_i \to F(C_i)$ with
\[\begin{tikzcd}
	D_i \ar[d, "h_i"'] \ar[r, "g_i"] & D \ar[d, "\lambda'_j"] \\
	F(C'_i) \ar[r, "F(\lambda^i_j)"'] & F(A(X_j))
\end{tikzcd}\]
for each $i$, where $\lambda'_j$ and $\lambda^i_j$ are the $j$th legs of cones $L'$ and $L_i$ respectively.

Moreover, when the domain site \textit{does} have finite limits, these three conditions reduce to the requirement that $F$ preserves finite limits.
\end{remark}

A functor is a morphism of sites precisely if $\ell_{\Dcal} \circ F:\Ccal \to \Sh(\Dcal,K)$ is a $J$-continuous flat functor, so that this composite extends along $\ell_{\Ccal}$ to provide the inverse image functor of a geometric morphism $\Sh(\Dcal,K) \to \Sh(\Ccal,J)$.

\begin{corollary}
\label{crly:sitepristine}
Suppose $(\Ccal,J)$ and $(\Dcal,K)$ are principal (resp. quasi-principal, finitely generated) sites. Then any morphism of sites $F: (\Ccal,J) \to (\Dcal,K)$ induces a relatively pristine (resp. relatively polished, relatively proper) geometric morphism $f:\Sh(\Dcal,K) \to \Sh(\Ccal,J)$.
\end{corollary}
\begin{proof}
Since the conditions on the sites ensure that the representables are supercompact (resp. supercompact or initial, compact), and the restriction of the inverse image functor to these is precisely $\ell_{\Dcal} \circ F$, we conclude that $f^*$ preserves these objects, whence $f$ is relatively pristine (resp. relatively polished, relatively proper) by Proposition \ref{prop:relpres}.
\end{proof}

\begin{remark}
\label{rmk:morsite}
More generally, suppose $(\Ccal,J)$ is any small-generated site such that $\Sh(\Ccal,J)$ is supercompactly (resp. compactly) generated and $(\Dcal,K)$ is a principal (resp. finitely generated) site. The geometric morphism induced by a morphism of sites $F: (\Ccal,J) \to (\Dcal,K)$ has inverse image functor sending any funneling (resp. multifunneling) colimit of representables to a colimit of supercompact (resp. compact) objects of the same shape, whence by Lemma \ref{lem:closed2} it must in particular preserve supercompact (resp. compact) objects. As such, we can replace the hypotheses of Corollary \ref{crly:sitepristine} with these weaker conditions if we so choose.
\end{remark}

Beyond individual morphisms, it is natural to make the extra step of forming a $2$-category of sites. Indeed, any natural transformation between functors underlying morphisms of sites induces a natural transformation between the inverse images of the corresponding geometric morphisms; for subcanonical sites, this mapping is full and faithful. Thus, for example, any equivalence of sites (an equivalence of categories which respects the Grothendieck topologies) lifts to an equivalence between the corresponding toposes.

We say that a principal site $(\Ccal,J_{\Tcal})$ is \textit{epimorphic} (resp. \textit{strictly epimorphic}) if $\Tcal$ is contained in the class of epimorphisms (resp. strict epimorphisms). An (effectual) reductive site is an (effectual) reductive category equipped with its reductive topology. We employ the \textit{ad hoc} notation of $\mathbf{EffRedSite}$, $\mathbf{RedSite}$, $\mathbf{StrEpPSite}$, $\mathbf{EpPSite}$ and $\mathbf{PSite}$ for the $2$-categories of effectual reductive sites, reductive sites, strictly epimorphic principal sites, epimorphic principal sites and all principal sites respectively, each endowed with morphisms of sites as $1$-cells and natural transformations as $2$-cells. Clearly, we have forgetful $2$-functors:
\begin{equation}
\label{eq:redsite}
\mathbf{EffRedSite} \to \mathbf{RedSite} \to \mathbf{StrEpPSite} \to \mathbf{EpPSite} \to \mathbf{PSite}.
\end{equation}

We apply analogous terminology and notation for the comparable kinds of finitely generated sites and coalescent sites. For example, we write $\mathbf{EffCoalSite}$, $\mathbf{PosCoalSite}$ and $\mathbf{FGSite}$ for the $2$-categories of effectual coalescent sites, positive coalescent sites and finitely generated sites, respectively. There results an analogous chain of $2$-functors:
\begin{equation}
\label{eq:coalsite}
\begin{tikzcd}[column sep = 5pt]
\mathbf{EffCoalSite} \ar[r] & \mathbf{PosCoalSite} \ar[r] &
\mathbf{CoalSite} \ar[r] & \mathbf{StrEpFGSite} \ar[r] & 
\mathbf{EpFGSite} \ar[r] & \mathbf{FGSite}.
\end{tikzcd}
\end{equation}

Consolidating the results of Section \ref{ssec:representable}, we find that several of these forgetful functors have adjoints. 

\begin{corollary}
\label{crly:reflect}
Let $(\Ccal,J_{\Tcal})$ be a principal site, ${\sim}$ the canonical congruence of Proposition \ref{prop:congruence}, $\ell(\Ccal)$ the full subcategory of $\Sh(\Ccal,J_{\Tcal})$ on the representable sheaves and $\Ccal_s$ the (essentially small) category of supercompact objects in that topos. Then the canonical functors underlie morphisms of sites:
\[\begin{tikzcd}
(\Ccal_s{,}J_r)  &
(\ell(\Ccal){,}J_{can}|_{\ell(\Ccal)}) \ar[l] &
(\Ccal/{\sim}{,}J_{\Tcal/{\sim}}) \ar[l] &
(\Ccal{,}J_{\Tcal}) \ar[l]
\end{tikzcd}\]
which are the units of reflections to the forgetful functors
\[\mathbf{EffRedSite} \to \mathbf{StrEpPSite} \to \mathbf{EpPSite} \to \mathbf{PSite}\]
found in Diagram \eqref{eq:redsite}.

Similarly, if $(\Ccal,J_{\Tcal'})$ is a finitely generated site, then with analogous notation, we have morphisms of sites:
\[\begin{tikzcd}
(\Ccal_c{,}J_c) &
(\ell(\Ccal){,}J_{can}|_{\ell(\Ccal)}) \ar[l] &
(\Ccal/{\sim}{,}J_{\Tcal'/{\sim}}) \ar[l] &
(\Ccal{,}J_{\Tcal'}) \ar[l]
\end{tikzcd}\]
which are units for the reflections of the forgetful functors
\[\mathbf{EffCoalSite} \to \mathbf{StrEpFGSite} \to \mathbf{EpFGSite} \to \mathbf{FGSite}\]
appearing in Diagram \eqref{eq:coalsite}.

All of these units induce equivalences at the level of the associated toposes.
\end{corollary}
\begin{proof}
We omit the straightforward checks that these are indeed morphisms of sites. The universality of the middle and right hand units has been discussed in and beneath Proposition \ref{prop:congruence} and Corollary \ref{crly:full}; it remains only to show that the final morphism $(\ell(\Ccal){,}J_{can}|_{\ell(\Ccal)}) \to (\Ccal_s{,}J_r)$ is universal.

Let $\Ccal'$ be an effectual, reductive category. A morphism of sites $F: (\Ccal,J_{\Tcal}) \to (\Ccal',J_r)$ corresponds to a geometric morphism $\Sh(\Ccal',J_r) \to \Sh(\Ccal,J_{\Tcal})$ whose inverse image functor restricts to $F$ on the representable sheaves, and so sends these to supercompact objects in $\Sh(\Ccal',J_r)$. Since a quotient of a supercompact object is supercompact and inverse image functors preserve quotients, $F$ extends uniquely (up to isomorphism) to a morphism of sites $(\Ccal_s,J_r) \to (\Ccal',J_r)$ inducing the same geometric morphism, as required.

As usual, the proof for finitely generated sites is analogous.
\end{proof}

The morphisms appearing in Corollary \ref{crly:reflect} allow us to give another characterization of reductive and coalescent categories.

\begin{lemma}
\label{lem:redsect}
Let $(\Ccal,J_{\Tcal})$ be a strictly epimorphic principal site in which $\Tcal$ is the class of all strict epimorphisms of $\Ccal$. Then, assuming the axiom of choice, $\Ccal$ is reductive if and only if the (underlying functor of the) composed unit morphism $(\Ccal,J_{\Tcal}) \to (\Ccal_s,J_r)$ has a left adjoint.

Similarly, a strictly epimorphic finitely generated site $(\Ccal,J_{\Tcal'})$ where $\Tcal'$ consists of the strict jointly epic families has $\Ccal$ coalescent if and only if the morphism of sites $(\Ccal,J_{\Tcal}) \to (\Ccal_c,J_c)$ has a left adjoint.
\end{lemma}
\begin{proof}
If $\Ccal$ is reductive and $\Tcal$ is its class of strict
epimorphisms, consider a supercompact object $C$ in $\Sh(\Ccal,J_{\Tcal}) = \Sh(\Ccal,J_r)$. $C$ is a quotient of some representable $\ell(C_0)$, and so is the colimit of some funneling diagram in $\ell(\Ccal)$ with weakly terminal object $\ell(C_0)$. Lifting this to a funneling diagram in $\Ccal$, call its colimit $L(C)$. There is a universal morphism $\eta: C \to \ell(L(C))$, since the image of the strict epimorphism $\ell(C_0 \too L(C))$ forms a cone under the original funnel in $\Sh(\Ccal,J_{\Tcal})$. This $\eta$ is the universal morphism from $C$ to a representable object, since given $C \to \ell(D)$, we have that the composite $\ell(C_0) \to C \to \ell(D)$ is a morphism in the image of $\ell$ (since $\ell$ is full and faithful on a strict principal site) forming a cone under the same funnel, so there is a factoring morphism $\ell(L(C)) \to \ell(D)$, as required. This universality means that $\ell(L(C))$ is well-defined up to isomorphism, and we can use choice to select a representative for each $C$; the universality then ensures that $L$ is functorial, and is a left adjoint to the inclusion $\Ccal \to \Ccal_s$, as required.

Conversely, suppose we have a left adjoint functor $L: \Ccal_s \to \Ccal$. Given a funnel in $\Ccal$, consider its colimit in $\Ccal_s$; this is preserved by $L$, so the colimit exists in $\Ccal$, which is enough to make $\Ccal$ a reductive category.

The argument for coalescent categories is analogous, passing via a finite coproduct to define $L(C)$ in the first part.
\end{proof}

In other words, any reductive category is a coreflective subcategory of an effectual reductive category, and similarly any coalescent category is a coreflective subcategory of an effectual coalescent category.

\begin{theorem}
\label{thm:2equiv}
Let $\mathbf{SG}\TOP_{\mathrm{rel prec}}$ be the $2$-category of supercompactly generated Grothendieck toposes, relatively pristine geometric morphisms and all geometric transformations. Then the object mapping:
\begin{align*}
\mathbf{PSite} &\to \mathbf{SC}\TOP_{\mathrm{rel prec}}\op \\
(\Ccal,J_{\Tcal}) &\mapsto \Sh(\Ccal,J_{\Tcal}).
\end{align*}
extends to a $2$-functor between these $2$-categories. This functor is faithful on $2$-cells if we restrict the domain to $\mathbf{EpPSite}$. It is full and faithful on $2$-cells and faithful (up to isomorphism) on $1$-cells if we restrict the domain to $\mathbf{StrEpPSite}$. Finally, it is a $2$-categorical equivalence if we restrict the domain to $\mathbf{EffRedSite}$.

Analogously, letting $\mathbf{CG}\TOP_{\mathrm{rel prop}}$ be the $2$-category of compactly generated Grothendieck toposes and relatively proper geometric morphisms, there is a $2$-functor whose effect on objects is:
\begin{align*}
\mathbf{FGSite} &\to \mathbf{CG}\TOP_{\mathrm{rel prop}}\op \\
(\Ccal,J_{\Tcal'}) &\mapsto \Sh(\Ccal,J_{\Tcal'}).
\end{align*}
This restricts to an equivalence between the $2$-category $\mathbf{EffCoalSite}$ and $\mathbf{CG}\TOP_{\mathrm{rel prop}}\op$.

Finally, there is an equivalence between the $2$-category $\mathbf{SG}\TOP_{\mathrm{rel pol}}\op$ of supercompactly generated Grothendieck toposes, relatively polished geometric morphisms and all geometric transformations (with $1$-cells reversed), and the $2$-category $\mathbf{EffRed}^{+} \mathbf{Site}$ of effective augmented reductive sites of Definition \ref{dfn:augmented}.
\end{theorem}
\begin{proof}
The claim of $2$-functoriality is fulfilled thanks to Corollary \ref{crly:sitepristine}.

Faithfulness when we restrict to $\mathbf{EpPSite}$ is by virtue of the observations after Proposition \ref{prop:congruence} that $\ell$ is faithful (on $\Tcal$-arches and hence) on morphisms coming from the site, and this applies in particular to the components of natural transformations.

Fullness on $2$-cells and faithfulness on $1$-cells when we restrict further to $\mathbf{StrEpPSite}$ is more directly derived from full faithfulness of $\ell$ for such subcanonical sites, and the fact that natural transformations between inverse image functors are determined by their components at the representables. Conversely, any such natural transformation (including a natural isomorphism) restricts along $\ell$ to a natural transformation between the underlying morphisms of sites.

The equivalence when we restrict to $\mathbf{EffRedSite}$ is thanks to Theorem \ref{thm:correspondence}. Indeed, since an effectual reductive category is equivalent to the category of supercompact objects in the corresponding topos, any relatively pristine morphism between such toposes restricts to a morphism of sites between the underlying effective reductive sites.

As ever, the argument for finitely generated sites and quasi-principal sites is analogous.
\end{proof}

\subsection{Comorphisms of sites and points}
\label{ssec:comorsite}

We would be remiss not to also discuss comorphisms of sites.

\begin{definition}
\label{dfn:comorphism}
Given sites $(\Ccal,J)$ and $(\Dcal,K)$, a functor $F:\Ccal \to \Dcal$ is a \textbf{comorphism of sites} $F: (\Ccal,J) \to (\Dcal,K)$ if it has the \textbf{cover-lifting property}, so that for any object $C$ of $\Ccal$ and $K$-covering sieve $S$ on $F(C)$, there exists a $J$-covering sieve $R$ on $C$ with $F(R) \subseteq S$. Such a comorphism of sites induces a geometric morphism $f:\Sh(\Ccal,J) \to \Sh(\Dcal,K)$, whose inverse image functor maps a sheaf $X$ to $\mathbf{a}_{J}(\Hom_{\Dcal}(F(-),X))$.
\end{definition}

While comorphisms of sites could in principle provide a way to extend the correspondence of Theorem \ref{thm:correspondence} to a covariant duality, we introduce them for another purpose. First, we make the following observation regarding toposes as canonical sites:
\begin{lemma}
\label{lem:adjoint}
Let $f:\Fcal \to \Ecal$ be a geometric morphism. The inverse image functor $f^*$ is a comorphism of sites $(\Ecal,J_{can}) \to (\Fcal,J_{can})$ if and only if the direct image $f_*$ has a right adjoint.
\end{lemma}
\begin{proof}
Suppose $f^*$ is a comorphism of sites. The induced geometric morphism has inverse image functor defined by $X \mapsto \mathbf{a}_{J_{can}}(\Hom_{\Fcal}(f^*(-),X))$, and by inspection, this coincides with $f_*$. As such, $f_*$ is the inverse image functor of a geometric morphism, and so has a right adjoint. Conversely, if $f_*$ has a right adjoint, then it is a morphism of sites since it is left exact and preserves jointly epimorphic families, whence its left adjoint $f^*$ is a comorphism of sites by \cite[Lemma VII.10.3]{MLM}.
\end{proof}

Indeed, comparing this definition with Lemma \ref{lem:relprop}, we obtain the following:

\begin{proposition}
\label{prop:comorph}
Suppose $\Fcal$ is a Grothendieck topos in which every small jointly epimorphic family $\{x_i: X_i \to X\}$ can be refined to a jointly epimorphic family of monomorphisms $\{x'_j: X'_j \hookrightarrow X\}$, in the sense that each $x'_j$ factors through some $x_i$. Then $f: \Fcal \to \Ecal$ is relatively pristine if and only if its inverse image functor is a comorphism of sites $(\Ecal,J_{can}) \to (\Fcal,J_{can})$, if and only if $f_*$ has a right adjoint.
\end{proposition}
\begin{proof}
Recall that the covering sieves for the canonical topology on a topos are those containing small jointly epimorphic families. Given a jointly epimorphic family on an object of the form $f^*(Y)$, we may by assumption refine it to a jointly epimorphic family of monomorphisms, whence from the characterization of relative pristine morphisms in Lemma \ref{lem:relprop}, $f$ is relatively pristine if and only if $f^*$ is a comorphism of sites. The final line follows from Lemma \ref{lem:adjoint}.
\end{proof}

Any localic topos satisfies the conditions of Proposition
\ref{prop:comorph}, but so does the topos of actions of the
two-element monoid, for example. We apply it in the case that $\Fcal$ is $\Set$.

Recall that a geometric morphism is \textbf{totally connected} if it is locally connected and the extra left adjoint preserves finite limits; see \cite[\S C3.6]{Ele} for basic results about these morphisms. A topos is said to be totally connected if its global sections geometric morphism is.

\begin{corollary}
\label{crly:totconn}
A Grothendieck topos has a relatively pristine point if and only if it is totally connected, and such a point is unique up to isomorphism.
\end{corollary}
\begin{proof}
By Proposition \ref{prop:comorph}, a geometric morphism from $\Set$ is relatively pristine if and only if its direct image functor has a right adjoint, which must then coincide with the direct image of the unique global sections morphism. As such, this happens if and only if the inverse image functor of the global sections morphism has a left adjoint preserving finite limits, making the topos totally connected.
\end{proof}

In light of Corollary \ref{crly:totconn}, considering the points of a localic topos demonstrates how large the (ostensibly subtle) difference between relatively polished and relatively pristine morphisms can be.

\begin{lemma}
\label{lem:point}
Every point of a localic topos is relatively polished.
\end{lemma}
\begin{proof}
We use Lemma \ref{lem:relprop}. Since $\Set$ is atomic, given a point $p:\Set \to \Ecal$ and an object $X$ of $\Ecal$, it suffices to consider the minimal inhabited covering of $p^*(X)$ by singletons $\{ x_i : 1 \hookrightarrow p^*(X) \}$ or, if $p^*(X)$ is empty, by the identity $\{ 0 \hookrightarrow p^*(X) \}$. The characterization of localic toposes as being generated by subterminal objects means we have a covering of $X$ by subterminals $\{ y_j: X_j \to X \}$, and since $f^*(X_j)$ is necessarily subterminal for every $j$, these factorize through this minimal covering, as required.
\end{proof}

Returning to relatively pristine points in the case of supercompactly generated toposes, we find a convenient characterization of when such a point exists.

\begin{proposition}
\label{prop:precpoint}
Given a principal site $(\Ccal,J_{\Tcal})$, the topos $\Sh(\Ccal,J_{\Tcal})$ has a relatively pristine point (and so is totally connected) if and only if $\Ccal\op$ is filtered.
\end{proposition}
\begin{proof}
By Theorem \ref{thm:2equiv}, any relatively pristine point of $\Sh(\Ccal,J_{\Tcal})$ comes from a morphism of sites $(\Ccal_s,J_s) \to (1,J_{can})$, where $\Ccal_s$ is the usual subcategory in $\Sh(\Ccal,J_{\Tcal})$. Composing with the canonical morphism of sites $(\Ccal,J_{\Tcal}) \to (\Ccal_s,J_s)$ from Corollary \ref{crly:reflect}, we conclude that such a point exists if and only if there exists a morphism of sites $(\Ccal,J_{\Tcal}) \to (1,J_{can})$. Since there is exactly one functor $\Ccal \to 1$, which automatically preserves covers (since all $J_{\Tcal}$ are inhabited), the result follows after observing that the remaining conditions of Definition \ref{dfn:morsite} correspond to filteredness of $\Ccal\op$.
\end{proof}
Note that since the condition on $\Ccal$ is independent of the choice of principal topology it is equipped with, we may deduce that any relatively pristine subtopos of $\Sh(\Ccal,J_{\Tcal})$ is totally connected when the hypotheses hold.

We immediately recover the following result, which although deducible from \cite[Example C3.6.17(c)]{Ele}, does not seem to have been recorded explicitly anywhere that we know of.
\begin{corollary}
Any regular topos is totally connected. 
\end{corollary}
\begin{proof}
It suffices to observe that if $\Ccal$ has finite limits, $\Ccal\op$ is filtered, so Proposition \ref{prop:precpoint} applies.
\end{proof}
In particular, by \cite[Theorem C3.6.16(iv)]{Ele}, the category of models of a regular theory in any Grothendieck topos has a terminal object (easily described as the model in which every sort is interpreted as the terminal object). The logic of supercompactly generated toposes shall be the subject of the next chapter.

\subsection{Examples of reductive categories}
\label{ssec:xmpl}

In this final subsection, we present examples of reductive and coalescent categories, as well as principal and finitely generated sites and their toposes of sheaves, in order to address some hypotheses about relationships between the concepts presented in this chapter. We begin with localic examples, then examples relating to categories of finite sets, then abelian categories.

\begin{example}
In the proof of Corollary \ref{crly:hypepres}, we observe that hyperconnected morphisms are pristine, so that any supercompactly generated two-valued topos is supercompact. Taking $\Ccal$ to be any non-trivial poset with a maximal element and taking $\Ecal$ to be the category of presheaves on this poset provides an example of a supercompactly generated, supercompact topos which is not two-valued.
\end{example}

\begin{example}
\label{xmpl:nonsc}
The objects of a reductive category need not be supercompact within this category, in spite of Corollary \ref{crly:strict}. For example, consider the four-element lattice:
\[\begin{tikzcd}
	& 1 & \\
	a \ar[ur] & & b \ar[ul] \\
	& 0 \ar[ur] \ar[ul] &
\end{tikzcd}\]
By Proposition \ref{prop:localic}, it forms a reductive category. However, the colimit of the span $a \leftarrow 0 \rightarrow b$ is $1$, which is to say that the arrows from $a$ and $b$ to $1$ form a strictly epic family containing no strict epimorphism. Even if we relax to mere epimorphisms here, the empty family is strictly epic over $0$ yet has no inhabited subfamilies.

By considering the lattice of subsets of $\Nbb$ as a distributive join semilattice, we similarly find that objects of coalescent categories need not be compact in those categories.
\end{example}

\begin{example}
\label{xmpl:R}
A familiar example of a localic, locally connected topos which is not supercompactly (or even compactly) generated is the topos of sheaves on the real numbers: no non-trivial open sets in the reals are compact.
\end{example}

\begin{example}
\label{xmpl:Nlocalic}
As a more original non-example, here is a localic, totally connected topos which is not compactly generated. Consider the poset $P$ whose objects are the natural numbers (excluding $0$), and with order given by $n \leq m$ if and only if $n$ is divisible by $m$, so that $1$ is terminal. Endow this poset with the Grothendieck topology $J$ whose covering sieves on a natural number $n$ are those containing cofinitely many prime multiples of $n$.

All of these sieves are connected and effective epimorphic; that is, $(P,J)$ is a localic, locally connected, subcanonical site. For every $n$, $\ell(n)$ is therefore an indecomposable subterminal object of $\Sh(P,J)$. Since $P$ has finite limits, the topos $\Sh(P,J)$ is moreover totally connected. To show that $\Sh(P,J)$ fails to be compactly generated it suffices to show that none of the $\ell(n)$ are. But by construction each $\ell(n)$ has a nontrivial infinite covering family by other representables which contains no finite subcovers. Thus this topos has no supercompact objects, and the only compact object is the initial object.
\end{example}

\begin{example}
\label{xmpl:surjprec}
As promised earlier, we demonstrate that it is not possible to extend Theorem \ref{thm:hype2}(i) or (ii) to relatively pristine or relatively proper surjections.

Consider the poset $P$ constructed as a fractal tree with countably many roots and branches. Explicitly, it has elements \textit{non-empty} finite sequences of natural numbers, $\coprod_{n=1}^\infty \Nbb^n$, with $\vec{x} \leq \vec{y}$ if $\vec{y}$ is an initial segment of $\vec{x}$. The Alexandroff locale $L$ corresponding to $P$ has opens which are downward-closed subsets in this ordering, so that for any sequence $\vec{x}$ in an open set, all extensions of $\vec{x}$ also lie in that open.

Consider the collection of opens $U$ such that whenever $(x_1,\dotsc,x_{k-1},x_k) \in U$, we have $(x_1,\dotsc,x_{k-1},y) \in U$ for cofinitely many values of $y$. This collection is closed under finite intersections and arbitrary unions (we needed the sequences in $P$ to be non-empty to ensure that the empty intersection of opens was included here), which makes it a subframe of $\Ocal(L)$. There is a corresponding quotient locale $L'$ of $L$, and hence a geometric surjection $s: \Sh(L) \to \Sh(L')$. Moreover, this surjection is relatively pristine. Indeed, if $X$ is a sheaf on $L'$ and we are given a covering of $s^*(X)$ in $\Sh(L)$, we may without loss of generality assume that $s^*(X)$ is covered by supercompact opens of $L$, and each supercompact open contains the open of $L'$ consisting of the strict extensions of sequences it contains; $X$ is necessarily the union of these in $\Sh(L')$.

However, $\Sh(L')$ is not supercompactly or even compactly generated, since the opens of the form $s^*(U)$ are not compact, with the exception of the initial open, despite $s^*$ preserving supercompact objects.
\end{example}

\begin{example}
\label{xmpl:nonreg}
There exist reductive categories without equalizers, products or pullbacks (or even pullbacks of monomorphisms), so which in particular are not locally regular. Indeed, simplifying Example \ref{xmpl:tworel}, consider the category of presheaves on $f,g: A \rightrightarrows B$. The subcategory of supercompact objects in this topos is simply the coequalizer diagram $A \rightrightarrows B \too C$, so that in particular the pair of monomorphisms $f,g$ has neither an equalizer nor a pullback, and the product $B \times B$ does not exist. We obtain a similar example from any finite category containing a parallel pair of morphisms lacking an equalizer. 
\end{example}

\begin{example}
\label{xmpl:simple}
Any discrete category (a category with no non-identity morphisms) with more than one object is a reductive and locally regular category which is not regular.

The free finite cocompletion of such a discrete category $\Ccal$ is a coalescent category, since it is equivalent to $[\Ccal\op,\FinSet]$, where $\FinSet$ is the category of finite sets\footnote{This expression for the free finite cocompletion applies if and only if $\Ccal$ has finite hom-sets, since this is necessary and sufficient for the representable presheaves to lie in $[\Ccal\op,\FinSet]$. This is trivially the case when $\Ccal$ is discrete.}, and this is a coalescent category, whence $[\Ccal\op,\FinSet]$ has funneling colimits computed pointwise. $\FinSet$ is in some sense the archetypal coalescent category, since by inspection, $\Sh(\FinSet,J_c) \simeq \Set$. Since $\FinSet$ has pullbacks, these free finite cocompletions are coherent categories.
\end{example}

\begin{example}
We may also consider the category $\FinSet_+$ of \textit{inhabited} finite sets; since a funneling colimit of inhabited finite sets is inhabited and finite, and taking the pullback in $\FinSet$ of an epimorphism in $\FinSet_+$ along a morphism with inhabited domain gives another epimorphism with inhabited domain, we have that $\FinSet_+$ is another example of a reductive category without pullbacks (since the two inclusions $1 \rightrightarrows 2$ `should' have empty intersection).

By reintroducing the empty set and declaring that the empty sieve is covering over it, we obtain an augmented reductive site with underlying category $\FinSet$ having an equivalent topos of sheaves; since the coalescent topology on $\FinSet$ is a refinement of the augmented reductive one, we see that we have a (relatively proper) inclusion of toposes:
\[\Set \simeq \Sh(\FinSet,J_c) \hookrightarrow \Sh(\FinSet_+,J_r), \]
which is not an equivalence since the sheaves represented by the finite sets of cardinality at least $2$ are supercompact in the latter topos but merely compact in the former.
\end{example}

\begin{example}
As a dual construction, we observe that $\FinSet\op$ is a coalescent category, and $\Sh(\FinSet\op,J_c)$ embeds into the \textit{classifying topos for the theory of objects}, $[\FinSet,\Set]$. Incidentally, the latter topos provides a counterexample to the hypothesis that every point of every Grothendieck topos is relatively polished, which we might suppose as an extension of Lemma \ref{lem:point}: the points of $[\FinSet,\Set]$ correspond to sets (objects of $\Set$), and the correspondence sends a geometric morphism to the set which is the image of the representable functor $\yon(1)$. But $\yon(1)$ is supercompact, so a point corresponding to any set with more than one element fails to be relatively polished.
\end{example}

\begin{example}
\label{xmpl:simplex}
For yet another related example, consider the simplex category $\Delta$, whose objects are \textit{inhabited} finite ordinals,
\[ [n] = \{0, \dotsc, n-1\}, \, \text{ } \, n \geq 1 \]
and whose morphisms are the order-preserving maps between these. Clearly this does not have pullbacks (since the intersection of the two inclusions $[1] \rightrightarrows [2]$ would be the empty ordinal which is not an object of $\Delta$).

$\Delta$ has funneling colimits: given any collection of morphisms into the object $[n]$ of $\Delta$, their colimit is the quotient of $[n]$ identifying $f(x),f(x)+1,\cdots,g(x)$ (or $g(x),g(x)+1,\cdots,f(x)$) for each parallel pair $f,g:[n']\rightrightarrows [n]$ in the diagram and each $x \in [n']$. Moreover, each epimorphism $g:[n] \too [m]$ is split (so in particular is strict) by the monomorphism $\mathrm{min}(g^{-1}):[m]\rightrightarrows [n]$, say.

In particular, by Remark \ref{rmk:split} the collection of strict epimorphisms is stable, which makes $\Delta$ a reductive category with $\Sh(\Delta,J_r) = [\Delta\op,\Set]$, the topos of \textit{simplicial sets}. This recovers a non-trivial fact about the topos of simplicial sets: every quotient of a representable simplicial set is also representable, so that every simplicial set is a union of its representable subsets. We shall see some related examples in Section \ref{ssec:xmpl1}.
\end{example}

\begin{example}
\label{xmpl:nonpresh}
To contrast Examples \ref{xmpl:simple} and \ref{xmpl:simplex}, we recall an example of a supercompactly generated topos which is not equivalent to a presheaf topos; compare also Proposition \ref{prop:localic} above.

Consider the Schanuel topos, $\Sh(\FinSet\op_{\mathrm{mono}},J_{at})$. We see that $(\FinSet\op_{\mathrm{mono}},J_{at})$ is an atomic site with pullbacks, but moreover it is a reductive and regular site, since all of the morphisms in the category are regular epimorphisms which are stable under pullback.

We know of two ways to show that this topos is not a presheaf topos. The first is to show that the site is effectual, which explicitly involves identifying an algorithm which, given a cofunnel $F$ in $\FinSet_{\mathrm{mono}}$ with weakly initial object $F(D_0)$ and $x,y: F(D_0) \rightrightarrows C$ equalized by its limit, constructs an inclusion $i: C \hookrightarrow C'$ and a connecting zigzag between $i \circ x$ and $i \circ y$ in $(F \downarrow C')$. It then follows that the supercompact objects (equivalently, atoms) in the Schanuel topos are precisely the representable sheaves coming from $\FinSet\op_{\mathrm{mono}}$. However, if $I$ is a finite set, $A$ is any set of cardinality larger than that of $I$, and we have inclusions from a further finite set $B$ into both $A$ and $I$, there can be no monomorphism completing the triangle,
\begin{equation*}
\label{eq:injective}
\begin{tikzcd}
& A \ar[dl, dashed] \\
I & B \ar[u, hook] \ar[l, hook],
\end{tikzcd}
\end{equation*}
whence $I$ is not injective in $\FinSet_{\mathrm{mono}}$ and hence is not projective in $\FinSet\op_{\mathrm{mono}}$. It follows that no object of $\FinSet\op_{\mathrm{mono}}$ is projective, whence the Schanuel topos has no indecomposable projective objects, but is non-degenerate, and so is not a presheaf topos.

An alternative proof, which we thank Olivia Caramello for describing to us, is to observe that the category of representables in a presheaf topos can be identified, up to idempotent-completion, with the full subcategory of finitely presentable objects in its category of points; see the next chapter for a more precise explanation of this. We also have that a presheaf topos is atomic if and only if the representing category is a groupoid. The Schanuel topos classifies infinite decidable objects, so its category of points corresponds to the category of infinite sets (and monomorphisms); since any inhabited full subcategory of this category is not a groupoid (every infinite set has an injective endomorphism which is not invertible) and the topos is non-degenerate, it again follows that this cannot be a presheaf topos.
\end{example}

\begin{example}
\label{xmpl:abelian}
Since any abelian category is effective regular, any small abelian category with funneling colimits is reductive. This is the case for the finitely presented (right) modules of a (right) Noetherian ring, say, since these coincide with finitely generated modules and so the category is closed under quotients in the large category of modules. For example, the category of finitely generated abelian groups is a reductive category with finite limits and colimits.

In order to construct a small abelian category which does \textit{not} have funneling colimits, we look for a coherent ring $R$ whose collection of finitely generated ideals is not closed under infinite intersections (note that if the ring is an integral domain, coherence ensures that it will be closed under finite intersections). Let $I$ be an infinitely generated ideal which is obtained as such an intersection. In the large category of modules, we may identify $R/I$ as the colimit of the funneling diagram consisting of the inclusions of finitely generated sub-ideals of $I$, along with the parallel zero maps, into the ring, viewed as a (right) module over itself. If the colimit of this diagram existed in the category of finitely presented $R$-modules, it would have to be a quotient of $R$ by some finitely generated ideal, but by construction there is no initial finitely presented quotient under this diagram, whence the colimit does not exist.

Consider the ring $R$ of eventually constant sequences valued in the field on two elements (with point-wise operations). Observe that all finitely generated ideals in this ring are principal. An ideal generated by a sequence $g$ is the cokernel of the module homomorphism $R \to R$ sending $x$ to $x \cdot (1-g)$, so this is indeed a coherent ring. For each index $i$ we have a `basis element' $e_i$ which is $1$ at $i$ and $0$ elsewhere. The ideal $I_i$ generated by $(1-e_i)$ consists of those sequences which are $0$ at $i$. Consider $\bigcap_{j = 1}^{\infty} I_{2j}$: this consist of sequences which are non-zero only at odd indices, but by the eventually constant criterion, no single sequence can generate this ideal, whence the ideal fails to be finitely generated.

We would like to thank Jens Hemelaer for helping us to identify the sufficient structure needed to find this counterexample and Ryan C. Schwiebert for identifying a ring realizing that structure (via math.stackexchange.com).

Note that a non-trivial abelian category cannot be coherent or coalescent since the initial object is not strict in such a category (cf Lemma \ref{lem:collred}).
\end{example}

\begin{example}
\label{xmpl:TF}
As an example of a regular and reductive category that fails to be effective, let alone effectual, we adapt \cite[before Example A1.3.7]{Ele}.

Consider the category $\mathbf{TF}_{fg}$ of finitely-generated torsion-free abelian groups. By considering it as a reflective subcategory of the category of finitely generated abelian groups, we find that it is regular and has all coequalizers. We can moreover check that it has funneling colimits, since the full category of abelian groups is cocomplete and any quotient of a finitely generated group is finitely generated (to obtain the quotient in $\mathbf{TF}_{fg}$, the torsion parts of such a quotient are annihilated). Thus it is a reductive category with finite limits. However, as Johnstone points out, the equivalence relation
\[R = \{(a,b) \in \Zbb \times \Zbb \mid a \equiv b \, \mathrm{mod} \,
2 \} \cong \Zbb \times \Zbb \]
is not a kernel pair of any morphism in $\mathbf{TF}_{fg}$, so this category is not effective regular.
\end{example}

\begin{example}
\label{xmpl:countable}
For an example of a regular and reductive category that is effective but not effectual, we modify \cite[Example D3.3.9]{Ele}. Let $\Set_\omega$ be the full subcategory of $\Set$ on the finite and countable sets. This has pullbacks and funneling colimits which are stable under pullback, inherited from $\Set$. Moreover, all epimorphisms are regular, so this is a regular and reductive category (indeed, it is a coherent and coalescent category too!). It also inherits the property of being effective from $\Set$.

However, it is not effectual. Johnstone exhibits the following coequalizer diagram:
\begin{equation}
\label{eq:Ncoeq}
\begin{tikzcd}
\Nbb \ar[r, shift left, "\id"] \ar[r, shift right, "s"'] &
\Nbb \ar[r, two heads] & 1,
\end{tikzcd}
\end{equation}
where $s$ is the successor function. He concludes by considering the natural number object in $\Sh(\Set_\omega, J_c)$ that this coequalizer is not preserved by the canonical functor $\ell: \Set_\omega \to \Sh(\Set_\omega,J_c)$; we could deduce the failure of effectuality from that. Instead, we prove it directly by considering the morphisms
\[\begin{tikzcd}
\Nbb \ar[rr, shift left, "\id"]
\ar[rr, shift right, "g := 2 \times -"'] & & \Nbb,
\end{tikzcd}\]
which are clearly coequalized by the epimorphism in \eqref{eq:Ncoeq}. If $\Set_\omega$ were effectual, there would exist some epimorphism $t: X \too \Nbb$ such that $t$ and $g \circ t$ lie in the same connected component of $(X \downarrow F)$, where $F$ is the parallel arrow diagram $(\id,s)$ whose coequalizer is shown in \eqref{eq:Ncoeq}. However, given any finite zigzag,
\[\begin{tikzcd}
X \ar[dd, "t"'] \ar[r, equal] & X \ar[d] \ar[r, equal]& X \ar[dd] \ar[rr, equal] & & X \ar[dd] \ar[r, equal] & X \ar[d] \ar[r, equal] & X \ar[dd, "g \circ t"] \\
& \Nbb \ar[dl, "x_1"'] \ar[dr, "y_1"] & & \cdots \ar[dl, "x_2"'] \ar[dr, "y_{n-1}"] & & \Nbb \ar[dl, "x_n"'] \ar[dr, "y_n"] & \\
\Nbb && \Nbb && \Nbb && \Nbb,
\end{tikzcd}\]
where $x_i$ and $y_i$ are $\id$ or $s$, composing with any morphism $m: 1 \to X$ such that $t \circ m > n$ as elements of $\Nbb$, we conclude that since the difference between the image of $m$ in the first copy of $\Nbb$ and the last copy of $\Nbb$ is at most $n$, we have a contradiction. That is, no zigzag of finite length is sufficient to connect $t$ and $g \circ t$ in $(X \downarrow F)$.
\end{example}

\begin{example}
\label{xmpl:wedge}
Finally, we provide an example of an effectual coalescent category which is not effectual as a reductive category. Consider the following category, $\Ccal$:
\[\begin{tikzcd}
R_1 \ar[dr, shift left] \ar[dr, shift right] \ar[r, "c_1"] &
Z \ar[d, shift left = 3] \ar[d, shift left]
\ar[d, shift right] \ar[d, shift right = 3] &
R_2 \ar[dl, shift left] \ar[dl, shift right] \ar[l, "c_2"'] \\
& A, &
\end{tikzcd}\]
where the composition is such as to make $Z$ the coproduct of $R_1$ and $R_2$. We make this into a strictly epimorphic finitely generated site with the stable class $\Tcal'$ consisting of the singleton families on the identity morphisms and the pair $\{c_1,c_2\}$. This embeds into the effectual coalescent category $\Ccal_c$ of compact objects in $\Sh(\Ccal,J_{\Tcal'})$. Consider the funneling diagram consisting of the pairs $R_1 \rightrightarrows A$ and $R_2 \rightrightarrows A$. By construction, the colimit of this diagram will also coequalize all of the morphisms $Z \to A$, since $c_1$ and $c_2$ are jointly epimorphic in $\Ccal_c$. However, there is no strict epimorphism $C' \too Z$ verifying the definition of effectuality for a reductive category. Indeed, in the completion of $\Ccal_c$ to an effectual reductive category, $c_1$ and $c_2$ are no longer jointly epimorphic, so that the funneling colimit described fails to coequalize the morphisms $Z \to A$.
\end{example}

%% file: The_SL.tex
\chapter{Logic of Supercompactly Generated Toposes}
\label{chap:logic}

Having introduced the classes of supercompactly generated toposes\footnote{From this point forward we put aside compactly generated toposes, since supercompactly generated toposes will be more relevant in the subsequent chapters and are simpler to study.}, a natural question to ask about these toposes is `which theories do they classify?' Recall that, given a fragment of logic $\Lcal$ interpretable in any Grothendieck topos and a theory $\Tbb$ in that fragment, a Grothendieck topos $\Ecal$ is said to \textit{classify $\Tbb$} if there are equivalences of categories
\begin{equation}
\label{eq:equivalence}
\Geom(\Fcal,\Ecal) \simeq \Tbb\text{-mod}(\Fcal),
\end{equation}
which are natural in the argument $\Fcal$ as it varies over Grothendieck toposes. In the categorical logic literature, the fragment $\Lcal$ is typically taken to be geometric logic, or a sub-fragment thereof such as coherent or regular logic. Geometric logic is important because of the classical result that every geometric theory has a classifying topos, and conversely that every Grothendieck topos classifies a geometric theory (in fact, many geometric theories; see \cite[Proposition D3.1.12 and Remark D3.1.13]{Ele}, originating in \cite[Theorem 9.1.1 and Proposition 9.1.5]{MakkaiReyes}). Meanwhile, regular and coherent logics form important fragments of geometric logic because they are expressive enough to include the host of theories studied in mathematical practice which involve only unary (resp. finitary) logical connectives and existential quantification. All three fragments have well-developed semantics in the context of toposes, with formulae-in-context interpreted as subobjects of products, and entailments between formulae in the same context realized as containments of those subobjects.

Recall from Example \ref{xmpl:regular} in the last chapter that a topos classifying a regular theory is called a regular topos. We saw that any regular topos is automatically supercompactly generated, as a consequence of Proposition \ref{prop:representable}. As such, regular theories form a strict sub-class of the geometric theories classified by supercompactly generated toposes.

The purpose of the present short chapter is to recall enough basic categorical logic results and constructions, mostly sourced from \cite[Part D]{Ele} and \cite{TST}, in order to characterize the theories which are classified by supercompactly generated toposes and present some examples. This will be refined in Chapter \ref{chap:TSGT} to the special case of toposes of topological monoid actions.

\subsection*{Overview}

In Section \ref{sec:bg'} we recount the basics of geometric logic up to the point of constructing the (geometric) syntactic category. In Section \ref{sec:geometric} we present and apply the relevant details of Caramello's framework from \cite{TST} to obtain two characterizations of theories classified by supercompactly generated toposes.

In Section \ref{sec:apply}, we leverage these results to present some examples. Since all presheaf toposes are supercompactly generated, the class of \textit{supercompactly generated theories} trivially includes all examples of theories of presheaf type, and many of the quotients which we consider end up being of presheaf type too, but we have managed to identify some non-trivial supercompactly generated quotients of these theories.

Finally, in Section \ref{sec:observe} we reflect on our progress.

Except insofar as any of the examples have original aspects, there is very little original material to report in the contents of this chapter, but it may be enlightening for the reader who has not seen the details of the syntactic category framework worked through in full to the point of explicitly computing quotients from first principles.

\section{Background}
\label{sec:bg'}

\subsection{Geometric categories}
\label{ssec:geomcat}

For the purposes of discussing categorical logic in a topos theory context, we shall need a class of categories which we have not encountered yet.

\begin{dfn}
A \textbf{geometric category} is a well-powered regular category whose subobject lattices have all small unions, and such that these unions are stable under pullback. In particular, this means that covers by (small) unions of subobjects form a Grothendieck pretopology on any geometric category. The Grothendieck topology generated from it coincides with the \textit{canonical topology} $J_{can}$, which we encountered in Definition \ref{dfn:effective} of the last chapter; we call a geometric category endowed with this topology a \textbf{geometric site}.
\end{dfn}

Every Grothendieck topos is a geometric category, and conveniently the category of sheaves for the canonical topology on a Grothendieck topos is equivalent to that topos. The main purpose of this section is to concisely recount how the syntax for a geometric theory $\Tbb$ gives rise to a geometric site whose topos of sheaves classifies $\Tbb$. However, in the subsequent developments we will need some tools which we can derive categorically in advance.

Recall that a sieve $S$ on an object $C$ in a category $\Ccal$ is said to be \textbf{$J$-closed} for a Grothendieck topology $J$ on $\Ccal$ if and only if, for any arrow $f: D \to C$ in $\Ccal$, $f^*(S) \in J$ implies that $f \in S$. Such sieves are of interest because the lattice of subobjects of a sheaf represented by $C$ in $\Sh(\Ccal,J)$ is isomorphic to the lattice of $J$-closed sieves on $C$.

Let $(\Ccal,J_{can})$ be a geometric site. We shall write $\ell:\Ccal \to \Sh(\Ccal,J_{can})$ for the composite of the Yoneda embedding with the $J_{can}$-sheafification functor, which is automatically full and faithful. We paraphrase \cite[Lemma 3.1.6(ii)]{TST} in order to characterize $J_{can}$-closed sieves in a geometric category.
\begin{lemma}
\label{lem:Jclosed}
Let $S$ be a sieve on an object $C$ of a geometric category $\Ccal$. Then $S$ is $J_{can}$-closed if and only if $S$ is a principal sieve generated by a monomorphism.
\end{lemma}
\begin{proof}
First, suppose $S$ is a $J_{can}$-closed sieve, and let $s:B \hookrightarrow C$ be the union of the images of all morphisms in $S$. Since $\Ccal$ is well-powered, we may reduce to a presieve $\{r_i:B_i \to C \mid i \in I\}$ generating $S$ such that $s = \bigcup_{i \in I}\im(r_i)$. By construction, $S$ is a subsieve of the sieve generated by $s$. Conversely, each of the $r_i$ factors through $s$, say $r_i = s \circ h_i$, whence $s^*(S)$ is generated by the $h_i$, but these are jointly epic since images and unions are stable under pullback, so $s \in S$ and the sieve generated by $s$ is contained in $S$, as required.

Now suppose $S$ is principal, generated by a monomorphism $s:B \hookrightarrow C$. Given $f:A \to C$ such that $f^*(S)$ contains a small covering family, these all factor through the pullback $f^*(s)$ by definition (which is a monomorphism since $s$ is), and hence so does the union of their images, which is the identity. As such, $f^*(s)$ is an isomorphism, so $f$ factors through $s$ and hence lies in $S$.
\end{proof}

An immediate consequence is that subobjects of representable sheaves over a geometric site are representable.
\begin{prop}
\label{prop:subobj}
Let $(\Ccal,J_{can})$ be a geometric site and $\ell:\Ccal \to \Sh(\Ccal,J_{can})$ the functor described above. Then the image of $\ell$ induces isomorphisms on subobject lattices.
\end{prop}
\begin{proof}
As discussed above, subobjects of a representable $\ell(C)$ correspond to $J_{can}$-closed sieves on $C$ in $\Ccal$. A subobject $s:B \hookrightarrow C$ is sent by $\ell$ to the subobject corresponding to the $J_{can}$-closure of the principal sieve generated by $s$, but by Lemma \ref{lem:Jclosed} all such sieves are already $J_{can}$-closed and all $J_{can}$-closed sieves are principal, so $\ell$ induces an isomorphism, as claimed.
\end{proof}

\subsection{Geometric languages}
\label{ssec:geomlang}

We summarize the account given in \cite[\S D1.1]{Ele}, although this is by no means the only place it can be found. Recall that a \textbf{signature} $\Sigma$ contains the basic, generating data for the formal languages considered in categorical logic, namely the \textit{sorts} $A,B,\dotsc$, \textit{function symbols} $f,g,\dotsc$, and \textit{relation symbols} $R,S,\dotsc$. The function and relation symbols are each equipped with a type consisting of a list of sorts, denoted for example as,
\[f: A_1,\dotsc,A_n \to B \hspace{9pt}\text{ and }\hspace{9pt} R \hookrightarrow A_1,\dotsc,A_n, \]
respectively; this notation evokes the intended semantics which we shall recap in the next subsection. From here on, we employ the notation $\vec{A}$ as shorthand for the generic list of sorts $A_1,\dotsc,A_n$.

Given a signature, we may consider \textit{variables}, which are formal symbols $x,y,\dotsc$ equipped with a \textit{type} consisting of a sort in $\Sigma$. From these we inductively build \textit{terms}, also each having a type consisting of a single sort: individual variables are terms of their respective types, while for terms $t_i$ of type $A_i$ and a function symbol $f:\vec{A} \to B$, $f(t_1,\dotsc,t_m)$ is a term of type $B$.

Everything up to this point can be applied to every fragment of infinitary first-order logic. In the next step of inductively constructing \textit{formulae} from terms, on the other hand, the construction operations which are permitted determine the fragment of logic needed to interpret, manipulate or reason about the resulting language. For geometric logic, the following expressions are well-formed (which is to say permissible) geometric formulae:
\begin{itemize}
	\item \textit{Relations}: $R(t_1,\dotsc,t_n)$, where $R$ is a relation symbol of type $R \hookrightarrow \vec{A}$ and $t_i$ is a term of type $A_i$;
	\item \textit{Equality}: $t_1 = t_2$ where $t_1,t_2$ are terms of the same type;
	\item \textit{Finitary conjunctions}: If $\phi_1,\dotsc,\phi_n$ are geometric formulae, then so is $\bigwedge_{i=1}^n \phi_i$ (which may also be denoted with infix wedges); this includes the empty conjunction, called \textit{truth} and denoted $\top$, as a special case;
	\item \textit{Infinitary disjunctions}: If $\{\phi_i : i \in I\}$ is a set of geometric formulae whose free variables form a finite set (see below), then $\bigvee_{i \in I} \phi_i$ is a geometric formula; this includes the empty disjunction, called \textit{false} and denoted $\bot$, as a special case;
	\item \textit{Existential quantification}: for any geometric formula $\phi$ containing a free variable $x$ (see below), $(\exists x)\phi$ is a geometric formula where $x$ is now a bound variable.
\end{itemize}
Informally, a variable appearing in a formula $\phi$ is \textit{bound} if the formula includes quantification over it, and \textit{free} otherwise; Johnstone explicitly specifies the free variables of formulae in tandem with their inductive construction in \cite[Definition D1.1.3]{Ele}. A geometric formula with no free variables is called a \textit{propositional formula}.

We will occasionally need to be able to substitute variables in formulae. If $\phi$ features (pairwise distinct) free variables $\vec{x} = x_1,\dotsc,x_n$, and $\vec{x}' = x'_1,\dotsc,x'_n$ is a list of variables of the same respective types, we write $\phi[\vec{x}'/\vec{x}]$ for the formula obtained by uniformly substituting $x'_i$ for $x_i$ in $\phi$.

A \textit{context} for a term $t$ or a geometric formula $\phi$ is a finite list $\vec{x}$ of variables including all free variables appearing in $\phi$; finiteness of contexts is the reason we required the $\phi_i$ to have only finitely many free variables between them in the rule for infinite disjunctions. We write $[]$ for the empty context.

Accordingly, a geometric \textbf{formula-in-context} is a pair, denoted $\{ \vec{x} : \phi\}$, consisting of a geometric formula $\phi$ and a valid context $\vec{x}$ for that formula. Formally, a context $\vec{x}$ comes equipped with its type $\vec{A}$; since we will mostly deal with single-sorted theory, there will usually only be one possible type, but when we need to make the type of a context explicit in formulae-in-context or in the sequents defined below, we expand it to $\vec{x}:\vec{A}$.

A \textbf{geometric sequent} over $\Sigma$ is a formal expression of the form,
\[\phi \vdash_{\vec{x}} \psi,\]
where $\phi$ and $\psi$ are geometric formulae over $\Sigma$ and $\vec{x}$ is a valid context for both $\phi$ and $\psi$; when we need to indicate the context, we call this a sequent \textit{over $\vec{x}$}. The informal intended interpretation of such a sequent is,
\[`(\forall \vec{x}) \, \phi(\vec{x}) \Rightarrow \psi(\vec{x})';\]
since geometric formulae do not include universal quantification or implication operations, this interpretation will be made formal only in the semantics used to interpret geometric logic. A sequent over the empty context is called a \textit{sentence}.

A \textbf{geometric theory} $\Tbb$ over a signature $\Sigma$ simply consists of a collection of geometric sequents, the \textit{axioms} of the theory.

\subsection{Semantics}
\label{ssec:geomsem}

We will primarily be interested in models of theories in Grothendieck toposes, but any geometric category has the structure needed to interpret geometric theories, as follows. This is a summary of \cite[\S D1.2]{Ele}.

A \textbf{$\Sigma$-structure} in a geometric category $\Ecal$ consists of an interpretation of each sort $A$ in $\Sigma$ as an object $\Mbb(A)$ of $\Ecal$, each function symbol $f:\vec{A} \to B$ as a morphism
\[\Mbb f : \Mbb(A_1) \times \cdots \times \Mbb(A_n) \to \Mbb(B)\]
in $\Ecal$, and each relation symbol $R \hookrightarrow \vec{A}$ in $\Sigma$ as a relation
\[\Mbb(R) \hookrightarrow \Mbb(A_1) \times \cdots \times \Mbb(A_n)\]
in $\Ecal$. We shall abbreviate $\Mbb(A_1) \times \cdots \times \Mbb(A_n)$ to $\Mbb(\vec{A})$ from now on.\footnote{We have chosen the notation `$\Mbb$' to help distinguish models of theories from the monoids discussed elsewhere in the thesis, which are typically denoted by `$M$'.}

Given a $\Sigma$-structure $\Mbb$, we may interpret the derived structures from the previous section inductively, following their constructions. First, a term-in-context $\vec{x} \cdot t$ is either a single variable $x_i$ in $\vec{x}:\vec{A}$, which is interpreted as the $i$th projection
\[ [[\vec{x} \cdot t]]_\Mbb = \pi_i: \Mbb(\vec{A}) \to \Mbb(A_i), \]
or it is of the form $f(t_1,\dotsc,t_m)$ for some function symbol $f:\vec{C} \to B$ and terms $t_i:C_i$, in which case $[[\vec{x} \cdot t]]_\Mbb$ is defined as the composite:
\[\Mbb(\vec{A}) \xrightarrow{({[[\vec{x}\cdot t_1]]}_\Mbb{,} \dotsc {,} {[[\vec{x}\cdot t_m]]}_\Mbb)} \Mbb(\vec{C}) \xrightarrow{\Mbb f} \Mbb(B).\]

Having interpreted terms, a geometric formula-in-context $\{\vec{x} \cdot \phi\}$ over $\Sigma$ is interpreted as a subobject of $\Mbb(\vec{A})$ determined by the structure of $\phi$. In brief, a relation $R(t_1,\dotsc,t_m)$ is interpreted by pulling back the interpretation of $R$ along the product of the morphisms interpreting the terms $t_i$; equalities of terms are interpreted as equalizers of the terms' interpretations; conjunctions and disjunctions are interpreted as intersections and unions of subobjects, respectively; and existential quantification is interpreted as the image of the projection of a given subobject along the relevant product projection morphism.

Given $\Sigma$-structures $\Mbb,\Mbb'$ in $\Ecal$, a \textbf{$\Sigma$-structure homomorphism} from $\Mbb$ to $\Mbb'$ consists of morphisms $h_A: \Mbb(A) \to \Mbb'(A)$ in $\Ecal$ for each sort $A$ in $\Sigma$ which commute with the respective interpretations of function symbols and relation symbols. Thanks to the universal properties used in the interpretations of the derived structures, these homomorphisms extend to the full structures, in the sense that (products of) the $h_A$ commute with the interpretations of terms, and there exist canonical morphisms derived from the $h_A$ between the interpretations of geometric formulae-in-context.

Let $\Tbb$ be a geometric theory over $\Sigma$. A $\Sigma$-structure $\Mbb$ in $\Ecal$ is a \textbf{model of $\Tbb$} if for each sequent $\phi \vdash_{\vec{x}} \psi$ in $\Tbb$ is \textit{satisfied} in $\Mbb$, in the sense that there is a containment of subobjects $[[\vec{x} \cdot \phi]]_\Mbb \subseteq [[\vec{x} \cdot \psi]]_\Mbb$. Morphisms between models are just $\Sigma$-structure homomorphisms as above, so the category of $\Tbb$-models is a full subcategory of the category of $\Sigma$-structures. We write $\Tbb\text{-mod}(\Ecal)$ for this category.

There are numerous details of this account that we have omitted, but we have all of the terminology we shall need, so we shall move on.

\subsection{Syntactic Categories}
\label{ssec:syncat}

In many ways, the above account is an entirely intuitive (albeit inevitably long-winded, even in summary) account of how formal language structure can be translated into categorical structure or vice versa. The challenge of categorical logic lies in determining an exhaustive collection of deduction rules on the formal language side which entirely characterize the behaviour on the categorical side. This is the content of \cite[\S D1.3]{Ele}, and we unapologetically omit most of it, with the exception of the following `normal form'-type result:
\begin{fact}[{\cite[Lemma D1.3.8(ii)]{Ele}}]
\label{fact:normalform}
Any geometric formula-in-context is provably equivalent (for any $\Tbb$) to one of the form $\{\vec{x} \cdot \bigvee_{i \in I}(\exists \vec{y})\phi_i\}$, where each $\phi_i$ is a finite conjunction of relations or equalities; we call such a $\phi_i$ a \textbf{regular formula}.
\end{fact}

For our purposes, the important takeaway is that, for a suitable collection of deduction rules (rules for transforming sequents into related sequents), most of which are intuitive formalizations of rules derived in classical first-order logic, there is a sound and complete proof theory for the categorical semantics of geometric languages; as such, we shall appeal to the reader's `logical intuition' when manipulation according to these deduction rules is required. The completeness result for geometric categories, which is the subject of \cite[\S D1.4]{Ele}, involves building the \textbf{syntactic category} $\Ccal_{\Tbb}$ for the theory $\Tbb$, as follows:

Objects of $\Ccal_{\Tbb}$ consist of \textit{formulae-in-context} $\{\vec{x} \cdot \phi\}$ up to `$\alpha$-equivalence', which is to say up to substitution of the variables in the context and formula for other variables of the same types (such that none of the substituted variables coincide with any variables already present); this allows us to assume without loss of generality that the contexts of any given finite set of formulae-in-context are disjoint when defining morphisms.

A morphism $\{\vec{x} \cdot \phi\} \to \{\vec{y} \cdot \psi\}$ in the syntactic category is a \textit{$\Tbb$-provable equivalence class of geometric formulae} $[\theta]$, where $\theta(\vec{x},\vec{y})$ is a formula in the union of the contexts such that the following three sequents are provable in $\Tbb$:
\begin{align*}
\phi &\vdash_{\vec{x}} (\exists \vec{y}) \theta, \\
\theta &\vdash_{\vec{x},\vec{y}} \phi \wedge \psi, \\
\theta \wedge \theta[\vec{y}'/\vec{y}] &\vdash_{\vec{x},\vec{y},\vec{y}'} \vec{y} = \vec{y}',
\end{align*}
where $\vec{y}'$ is any context of the same type as $\vec{y}$.

\begin{prop}
\label{prop:geomstruct}
The structure $\Ccal_{\Tbb}$ defined above is a geometric category. 
\end{prop}
\begin{proof}
We summarize Lemmas D1.4.1, D1.4.2 and D1.4.10 of \cite{Ele} and refer the reader to those results for some of the missing details. First, to show that $\Ccal_{\Tbb}$ is a category, we define the composite
\[\{\vec{x} \cdot \phi\} \xrightarrow{[\theta]} \{\vec{y} \cdot \psi\} \xrightarrow{[\gamma]} \{\vec{z} \cdot \chi\}\]
to be $[(\exists \vec{y})(\gamma \wedge \theta)]$, while the identity $\{\vec{x} \cdot \phi\} \to \{\vec{x}' \cdot \phi\}$ is $[\phi \wedge (\vec{x} = \vec{x}')]$.

Note that formulae in the same context which are provably equivalent in the sense alluded to in Fact \ref{fact:normalform} are automatically isomorphic in the syntactic category via the $\Tbb$-provable equivalence class of the formula which is a conjunction of them both (up to substitution) with a suitable equality formula.

To demonstrate the presence of finite limits, we observe that $\{[] \cdot \top\}$ is a terminal object, that $\{\vec{x},\vec{y} \cdot (\phi \wedge \psi)\}$ is the binary product of $\{\vec{x} \cdot \phi\}$ and $\{\vec{y} \cdot \psi\}$, and that the equalizer of a pair,
\[\begin{tikzcd}
\{\vec{x} \cdot \phi\} \ar[r, shift left, "{[\theta]}"] \ar[r, shift right, "{[\gamma]}"'] & \{\vec{y} \cdot \psi\}
\end{tikzcd}\]
is $\{\vec{x}' \cdot (\exists \vec{y})(\theta[\vec{x}'/\vec{x}] \wedge \gamma[\vec{x}'/\vec{x}])\}$ (the respective projection maps are given by suitable conjunctions of the formulae involved with equalities of variables).

The image of $[\theta]$ as above is $\{\vec{y} \cdot (\exists \vec{x}) \theta\}$, and a morphism is a cover (extremal epimorphism) if and only if $\psi \vdash_{\vec{y}} (\exists \vec{x})\theta$ is provable in $\Tbb$. The union of a collection of subobjects $\{\vec{x} \cdot \phi_i\} \hookrightarrow \{\vec{x} \cdot \psi\}$ is simply $\{\vec{x} \cdot \bigvee_{i \in I} \phi_i\}$; in particular, the minimal subobject is $\{\vec{x} \cdot \bot\}$. The stability of unions under pullback amounts to a consequence of distributivity of finite conjunctions over arbitrary disjunction (this is one of the intuitive deduction rules which we omitted earlier).

Finally, well-poweredness is a further consequence of Fact \ref{fact:normalform}, since any subobject can be expressed as a disjunction of regular formulae, and there are essentially only a set of such formulae over any signature (up to $\alpha$-equivalence).
\end{proof}

Completeness of the theory $\Tbb$ with respect to models in geometric categories follows from the very existence of the syntactic category, since this category comes with a canonical model of $\Tbb$ given by interpreting a formula-in-context $\{\vec{x}\cdot \phi\}$ as the corresponding object of the category; all of the axioms of $\Tbb$ are satisfied by construction. Moreover, this model is universal: the category of models of $\Tbb$ in any geometric category $\Dcal$ is equivalent to the category of geometric functors $\Ccal_{\Tbb} \to \Dcal$, which is to say those functors preserving the defining finite limit and subobject structure of geometric category.

Since $\Ccal_{\Tbb}$ is a geometric category, we can equip it with the canonical Grothendieck topology, which we shall now denote $J_{\Tbb}$. Explicitly, considering the descriptions of unions and covers in the proof of Proposition \ref{prop:geomstruct}, $J_{\Tbb}$-covering families are sets of $\Tbb$-provably functional formulae,
\[\left\{ \{\vec{x}_i \cdot \phi_i\} \xrightarrow{[\theta_i]} \{\vec{x} \cdot \phi\} \bigm\vert i \in I \right\},\]
such that $\phi \vdash_{\vec{x}} \bigvee_{i \in I} (\exists \vec{x}_i) \theta_i$ is provable in $\Tbb$.

The resulting topos $\Sh(\Ccal_{\Tbb},J_{\Tbb})$ is the classifying topos of $\Tbb$, because any geometric functor from $\Ccal_{\Tbb}$ to a Grothendieck topos $\Fcal$ has a unique factorization as the functor $\ell_{\Tbb}:\Ccal_{\Tbb} \to \Sh(\Ccal_{\Tbb},J_{\Tbb})$ followed by the inverse image of a geometric morphism.

\begin{rmk}
\label{rmk:othersyntax}
Note that while we describe $\Ccal_{\Tbb}$ as `the' syntactic category here, for a theory $\Tbb$ lying in a smaller fragment of geometric logic such as regular logic (in the sense that the geometric formulae appearing in the axioms of $\Tbb$ can be built from the strict subset of the available operations which are permitted in regular logic), we may construct a smaller syntactic category consisting of just the regular formulae-in-context. This will be a regular category in the sense of Definition \ref{dfn:regcoh}, and is the canonical regular category containing a model of $\Tbb$. It comes equipped with the corresponding Grothendieck topology. An important `conservativity' result in categorical logic is that the topos of sheaves on this regular site is equivalent to the topos of sheaves on the geometric syntactic site described above.
\end{rmk}

In using the geometric syntactic site, we will need the following result, whose proof is extracted from that of \cite[Theorem 6.1.3(i)]{TST} (which previously appeared in \cite[Theorem 2.2(i)]{universal}).
\begin{corollary}
The functor $\ell_{\Tbb}$ is closed under subobjects. That is, given any geometric formula-in-context $\{\vec{x} \cdot \phi\}$ and a subobject $S \hookrightarrow \ell_{\Tbb}(\{\vec{x} \cdot \phi\})$, there exists a geometric formula $\psi$ in the same context such that $S \cong \ell_{\Tbb}(\{\vec{x} \cdot \psi\})$.
\end{corollary}
\begin{proof}
Just as in any sheaf topos, the subobjects of a representable sheaf $\ell_{\Tbb}(\{\vec{x} \cdot \phi\})$ in $\Sh(\Ccal_{\Tbb},J_{\Tbb})$ correspond to $J_{\Tbb}$-closed sieves on $\{\vec{x} \cdot \phi\}$ in $\Ccal_{\Tbb}$. But since $\Ccal_{\Tbb}$ is a geometric category, these correspond to subobjects in $\Ccal_{\Tbb}$ by Proposition \ref{prop:subobj}, which to quote \cite[Lemma D1.4.4(iv)]{Ele} are represented by formulae-in-context of the form,
\[\{\vec{x} \cdot \psi[\vec{x}'/ \vec{x}]\} \xrightarrow{[\psi \wedge (\vec{x}' = \vec{x})]} \{\vec{x} \cdot \phi\},\]
where $\psi$ is a geometric formula such that $\psi \vdash_{\vec{x}} \phi$ is $\Tbb$-provable, and $\{\vec{x} \cdot \psi\} \leq \{\vec{x} \cdot \chi\}$ as subobjects if and only if $\psi \vdash_{\vec{x}} \chi$ is $\Tbb$-provable.
\end{proof}

\section{Geometric theories}
\label{sec:geometric}

Having established the basics of classifying toposes and syntactic sites, we examine in this part the classes of geometric theory which will be relevant to our investigation.

\subsection{Theories of presheaf type}
\label{ssec:preshtype}

A (geometric) theory $\Tbb$ is said to be a \textbf{theory of presheaf type} if its classifying topos $\Set[\Tbb]$ is equivalent to $\Setswith{\Dcal}$ for some small category $\Dcal$.\footnote{We write $\Dcal$ for various small categories in this section, including those underlying principal sites, to avoid confusion with the geometric syntactic categories $\Ccal_{\Tbb}$.}

We saw in Corollary \ref{crly:idem} of Chapter \ref{chap:TDMA} a special case of the fact that the category of points of $\Setswith{\Dcal}$ can be identified with the category $\Flat(\Dcal,\Set)$ of flat functors from $\Dcal$ to $\Set$. This is moreover equivalent to the \textbf{inductive completion} $\Ind(\Dcal\op)$ of $\Dcal\op$, which is to say the free cocompletion of $\Dcal\op$ with respect to filtered colimits, which is discussed and constructed explicitly in \cite[\S C4.2]{Ele}. Since it shall be relevant in Chapter \ref{chap:TSGT}, we observe that the categories which are equivalent to ones of the form $\Ind\Dcal$ for $\Dcal$ a small category are precisely the \textit{finitely accessible categories}, which is to say those having filtered colimits and a separating set of finitely presentable objects; \cite{Accessible} is a standard reference for accessible categories more generally.

We have also seen that $\Dcal$ can be recovered up to idempotent-completion as the opposite of the category of essential points of $\Setswith{\Dcal}$. However, if we are given only the category of points $\Ind(\Dcal\op)$ then we must identify the essential points as the \textbf{finitely presentable objects}, which is to say those objects $P$ in $\Ind(\Dcal\op)$ such that the representable functor $\Hom_{\Ind(\Dcal\op)}(P,-) : \Ind(\Dcal\op) \to \Set$ preserves filtered colimits.

Since points correspond to models of $\Tbb$ in $\Set$, we can make the further identification of $\Dcal\op$ with the \textbf{finitely presentable models} of $\Tbb$ in $\Set$, which have the corresponding defining property.

There is another way of identifying $\Dcal$ in the data associated to the classifying topos of $\Tbb$. Consider the presentation of this topos via the syntactic site $(\Ccal_{\Tbb},J_{\Tbb})$. By Proposition \ref{prop:subobj}, representable sheaves in $\Sh(\Ccal_{\Tbb},J_{\Tbb})$ are closed under subobjects and hence under retracts (up to isomorphism) and any projective indecomposable object is necessarily a retract of a representable. As such, we can identify the objects of $\Dcal$ with sheaves represented by geometric formulae-in-context $\ell_{\Tbb}(\{\vec{x} \cdot \phi\})$. These can be characterized intrinsically as follows:
\begin{dfn}[{\cite[Definition 6.1.9]{TST}}]
A geometric formula-in-context $\{\vec{x} \cdot \phi\}$ is said to be \textbf{$\Tbb$-irreducible} if and only if, for any $J_{\Tbb}$-covering family $\{\theta_i(\vec{x},\vec{x}_i) \mid i \in I\}$ of $\Tbb$-provably functional formulae, there exists some $i \in I$ and a $\Tbb$-provably functional formula $\theta'(\vec{x},\vec{x}_i)$ from $\{\vec{x} \cdot \phi\}$ to $\{\vec{x}_i \cdot \phi_i\}$ such that $\phi \vdash_{\vec{x}} (\exists \vec{x}_i) (\theta' \wedge \theta_i)$. In other words, one of the morphisms is a split epimorphism.
\end{dfn}

In particular, using Fact \ref{fact:normalform} to put $\phi$ into disjunctive normal form and considering the covering family of morphisms of the form,
\[\{\vec{x}_i \cdot (\exists \vec{y})\phi_i\} \xrightarrow{[\phi_i(\vec{x}) \wedge (\vec{x}_i=\vec{x})]} \{\vec{x} \cdot \bigvee_{i \in I}(\exists \vec{y})\phi_i\},\]
this means that any $\Tbb$-irreducible formula is provably equivalent to a formula involving only finite conjunctions and finitely many existential quantifications.

The $\Tbb$-model corresponding to a $\Tbb$-irreducible formula $\{\vec{x} \cdot \phi\}$ is the corresponding representable functor on the syntactic category. Taking a provably equivalent finite formula in the above sense therefore gives a `finite presentation' of the model, in a literal sense. This model $\Mbb$ will have the universal property that for any other $\Set$-model $\Mbb'$, we have:
\[ \Hom_{\Tbb\text{-mod}(\Set)}(\Mbb,\Mbb') \cong [[\vec{x} \cdot \phi]]_{\Mbb'}.\]
Expanded fully, this yields the definition of \textit{finitely presented model} in \cite[Definition 6.1.11(b)]{TST}. Passing via $\Dcal$, we have arrived at the equivalence between the category of \textit{finitely presentable} models of $\Tbb$ and the category of \textit{finitely presented} models of $\Tbb$ in $\Set$. 

It is worth noting that in the setting where $\Tbb$ has free models, we can treat the formula-in-context as a finite presentation in an algebraic sense, as noted in \cite[\S 9.4]{MakkaiReyes}: $\Mbb$ can be identified with the quotient of the free model on the finite set of typed variables $\vec{x}$ by the congruence generated by the formula $\phi$. This does not apply to a generic theory of presheaf type. The theories of ordered structures we shall consider in Section \ref{ssec:xmpl1} do not admit free structures, for instance, but the algebraic special case lends us enough intuition to concretely reconstruct an irreducible formula presenting a given finitely presentable model; as we shall see, this is the primary challenge when attempting to apply this theory. 

For the remainder of the theoretical section, we assume that we are provided with a functor (a choice of representing finitary formulae),
\[i: \Dcal \to \Ccal_{\Tbb}\]
which induces an equivalence $\Setswith{\Dcal} \simeq \Sh(\Ccal_{\Tbb},J_{\Tbb})$.

A consequence of the above investigation, which we make explicit because we will emulate it later, is the syntactic characterization of theories of presheaf type.
\begin{lemma}[{\cite[Corollary 6.1.10]{TST}}]
\label{lem:irreducible}
A geometric theory $\Tbb$ over a signature $\Sigma$ is of presheaf type if and only if every geometric formula over $\Sigma$ is $\Tbb$-provably equivalent to a disjunction of $\Tbb$-irreducible formulae.
\end{lemma}

\subsection{Supercompactly generated theories}

We shall characterize the theories classified by supercompactly generated toposes both as quotients of theories of presheaf types and in syntactic terms.

\subsubsection*{Quotient theory characterization}

Let $\Tbb$ be a theory of presheaf type, classified by a topos $\Setswith{\Dcal}$ as in the last subsection. In \cite[Chapter 8]{TST}, Caramello gives a detailed presentation of the relationship between Grothendieck topologies on the category $\Dcal$ and quotients of the theory $\Tbb$. In particular, Theorem 8.1.3 of \cite{TST} tells us that $\Sh(\Dcal,J)$ classifies the \textit{$J$-homogeneous $\Tbb$-models} across Grothendieck toposes $\Fcal$, while Theorem 8.1.8 guarantees that there is a unique (up to syntactic equivalence) geometric quotient theory $\Tbb'$ of $\Tbb$ whose models coincide with the $J$-homogeneous ones, which Caramello goes on to abstractly derive. As a particular example, she observes in Remark 8.1.9 that the topos of sheaves for the atomic topology on $\Dcal$, when $\Dcal$ satisfies the right Ore condition, corresponds to the Booleanization of any theory classified by $\Setswith{\Dcal}$. We do not reproduce her developments in full; instead, we directly apply her Theorem 8.1.12 to the special case of the principal sites introduced in the last chapter.

\begin{prop}
\label{prop:quotienttheory}
Let $\Tbb$ be a theory of presheaf type and let $\Tcal$ be a stable collection of morphisms in the category $\Dcal$ which is dual to the category of finitely presentable $\Tbb$-models in $\Set$. Then the quotient of $\Tbb$ which is classified by $\Sh(\Dcal,J_{\Tcal})$ is obtained by adding all axioms of the form
\[ (\psi \vdash_{\vec{y}} (\exists \vec{x})\theta), \]
where $[\theta] : \{\vec{x} \cdot \phi\} \to \{\vec{y} \cdot \psi\}$ is a morphism in $\Ccal_{\Tbb}$ corresponding via the functor $i:\Dcal \to \Ccal_{\Tbb}$ to a morphism in $\Tcal$ (so $\phi(\vec{x})$ and $\psi(\vec{y})$ are formulae over the signature of $\Tbb$ presenting $\Tbb$-models in $\Dcal$).
\end{prop}

Note that this is a natural extension of the atomic case mentioned by Caramello \cite[Corollary 8.1.14]{TST}.

Proposition \ref{prop:quotienttheory} is a relative characterization of theories with supercompactly generated classifying toposes, since we can present any such topos with a principal site to which the Proposition may be applied, thanks to the existence of generic theories, cf.\ Section \ref{ssec:generic} below.

\subsubsection*{Syntactic characterization}

By Lemma \ref{lem:closed} in Chapter \ref{chap:sgt}, a Grothendieck topos is supercompactly generated if and only if every object is covered by its supercompact subobjects. As we have seen, the geometric syntactic category of a theory $\Tbb$ forms a full, generating subcategory of the classifying topos of $\Tbb$ which, by Proposition \ref{prop:subobj}, is closed under subobjects. As such, a classifying topos of $\Tbb$ which is supercompactly generated necessarily has enough \textit{representable} supercompact objects. We thus arrive at the following definition:
\begin{dfn}[{\cite[Definition 3.1]{SCCT}}]
\label{dfn:Tscompact}
Let $\Tbb$ be a theory over a signature $\Sigma$. A geometric formula-in-context $\phi(\vec{x})$ over $\Sigma$ is \textbf{$\Tbb$-supercompact} if whenever a sequent of the form
\[ \phi(\vec{x}) \vdash_{\vec{x}} \bigvee_{i \in I} \psi_i(\vec{x}),\]
is provable in $\Tbb$, the sequent
\[ \phi(\vec{x}) \vdash_{\vec{x}} \psi_i(\vec{x})\]
is provable in $\Tbb$ for some $i \in I$.
\end{dfn}

\begin{thm}
\label{thm:scompform}
A geometric theory $\Tbb$ over a signature $\Sigma$ has a supercompactly generated classifying topos if and only if every geometric formula over $\Sigma$ is $\Tbb$-provably equivalent to a disjunction of $\Tbb$-supercompact formulae. We naturally call such a theory \textbf{supercompactly generated}.
\end{thm}
\begin{proof}
By the subobject characterization of supercompact objects of Lemma \ref{lem:scompact}, a representable object is supercompact in $\Sh(\Ccal_{\Tbb},J_{\Tbb})$ if and only its preimage along $\ell_{\Tbb}$ is supercompact, which exactly corresponds to the formula-in-context being $\Tbb$-supercompact. Thus there are enough of these if and only if every geometric formula-in-context is $\Tbb$-provably equivalent to a disjunction of $\Tbb$-supercompact formulae.
\end{proof}

\section{Applications}
\label{sec:apply}

Now we shall present some concrete applications of the theory above at various levels of generality.

It will be useful to explicitly recall from Definition \ref{dfn:stable} and Lemma \ref{lem:pbstable} of the last chapter that a class $\Tcal$ of morphisms in $\Ccal$ is called \textbf{stable} if it satisfies the following four conditions:
\begin{enumerate}
	\item $\Tcal$ contains all identities.
	\item $\Tcal$ is closed under composition.
	\item For any $f:C\to D$ in $\Tcal$ and any morphism $g$ in $\Ccal$ with codomain $D$, there exists a commutative square,
	\begin{equation}
	\begin{tikzcd}
	A \ar[r, "f'"] \ar[d, "g'"'] & B \ar[d, "g"]\\
	C \ar[r, "f"'] & D
	\end{tikzcd}
	\end{equation}
	in $\Ccal$ with $f'\in \Tcal$.
	\item Given any morphism $f$ of $\Ccal$ such that $f \circ g \in \Tcal$ for some morphism $g$ of $\Ccal$, we have $f \in \Tcal$.
\end{enumerate}

\subsection{The generic cases}
\label{ssec:generic}

First, let $\Dcal$ be any idempotent complete small category. Recall from \cite[Theorem 2.1.11]{TST} that we can view $\Dcal$ as a syntactic category for the theory $\Tbb_{\Dcal}$ of \textbf{flat functors on $\Dcal$}. This theory is defined over the signature $\Sigma_{\Dcal}$ having:
\begin{itemize}
	\item a sort $\name{C}$ for each object $C$ in $\Dcal$,
	\item a function symbol $\name{f}:\name{C} \to \name{D}$ for each morphism $f:C \to D$ in $\Dcal$, and
	\item no relation symbols.
\end{itemize}
The axioms of $\Tbb_{\Dcal}$ are of the form,
\begin{gather}
  \label{seq:id}
  \top \vdash_{x:\name{C}} \name{\id_C}(x) = x\\
  \label{seq:triangle}
  \top \vdash_{x} \name{h}(\name{g}(x)) = \name{f}(x)\\
  \label{seq:inhabited}
  \top \vdash_{[]} \bigvee_{C \in \ob(\Dcal)} (\exists x:\name{C}) \top\\
  \label{seq:span}
  \top \vdash_{y_1:\name{C_1}, \, y_2:\name{C_2}} \bigvee_{C_1 \xleftarrow{p_1} C \xrightarrow{p_2} C_2} (\exists x:\name{C})\left( (\name{p_1}(x) = y_1) \wedge (\name{p_2}(x) = y_2) \right)\\
  \label{seq:equalize}
  \name{f}(x) = \name{g}(x) \vdash_{x:\name{C}} \bigvee_{B \xrightarrow{h} C, f \circ h = g \circ h} (\exists w:\name{B})\left( \name{h}(w) = x \right),
\end{gather}
where we have an axiom of the form \eqref{seq:id} for each object $C$ of $\Dcal$, an axiom of the form \eqref{seq:triangle} for each commuting triangle $f = g \circ h$ in $\Dcal$, one of the form \eqref{seq:span} for each pair of objects $C_1,C_2$ of $\Dcal$, and an axiom of the form \eqref{seq:equalize} for each pair of parallel morphisms $f,g:C \rightrightarrows D$ in $\Dcal$.

\begin{rmk}
Some of these axioms will typically be redundant or reducible depending on the structure of the category. We can omit cases of axiom \eqref{seq:triangle} where $g$ or $h$ is an identity morphism, for example.

If $\Dcal$ has a terminal object, then the disjunction in \eqref{seq:inhabited} can be reduced to the singleton consisting of that object. Similarly, if the product of $C_1$ and $C_2$ exists, then the disjunction in \eqref{seq:span} can be reduced to the singleton consisting of the product span, and if the equalizer of $f,g$ exists, the disjunction in \eqref{seq:equalize} can be reduced to a singleton. More generally, given a final subcategory of either $\Dcal$, the category of spans over $(C_1,C_2)$, or the category of equalizers of $(f,g)$, the objects of this subcategory index a smaller, $\Tbb$-provably equivalent disjunction for the right-hand side of \eqref{seq:inhabited}, \eqref{seq:span} and \eqref{seq:equalize}, respectively.
\end{rmk}

Syntactically, we identify each object $C$ of $\Dcal$ with the formula-in-context $\{x:\name{C} \cdot \top\}$, and each morphism $f:C \to C'$ with the $\Tbb_{\Dcal}$-provable equivalence class of formulae represented by $\name{f}(x) = y$, where of course $x:\name{C}$ and $y:\name{C'}$.

As such, a principal site $(\Dcal,J_{\Tcal})$ classifies the theory obtained by adding to $\Tbb_{\Dcal}$ the axioms
\[\top \vdash_{y:\name{C'}} (\exists x:\name{C})(\name{f}(x) = y),\]
ranging over $f \in \Tcal$.

We can in principle achieve much simpler, more elegant theories by taking advantage of the existence of simpler (presentations of) theories of presheaf type, which is what we shall attempt to do in the remainder of the chapter.

\subsection{Essentially Algebraic theories}
\label{ssec:essalg}

Recall that a single-sorted theory $\Tbb$ over a signature $\Sigma$ is called \textit{essentially algebraic} if it can be expressed in terms of partial operations and equations between these. When all of the partial operations are total, we obtain \textit{algebraic theories} as a special case. When the operations of an essentially algebraic theory are finitary, its models are classified by a category with finite limits, in the sense that we may construct a \textit{cartesian syntactic category} for such a theory $\Tbb$ (cf.\ Remark \ref{rmk:othersyntax}), which we shall denote $\Bcal_{\Tbb}$ to avoid confusing it with the geometric syntactic category introduced earlier, with the universal property that models of the theory in any category with finite limits correspond to limit-preserving functions from the $\Bcal_{\Tbb}$. 

Any essentially algebraic theory $\Tbb$ is of presheaf type; indeed, in the finitary case the category $\Dcal$ of finitely presentable models is exactly $\Bcal_{\Tbb}\op$. For example, letting $\Grp_{fp}$ be the category of finitely presentable groups, the topos $[\Grp_{fp},\Set]$ classifies the theory of groups. The usual signature for this theory, consisting of only one sort, one constant symbol and a pair of function symbols, is evidently much smaller than the canonical signature which carries a sort for each finitely presented group!

Let $\Tbb$ be any theory. We write $\Tbb\text{-mod}_{fp}(\Set)$ for its category of finitely presentable models in $\Set$ and define $\Dcal_{\Tbb} := \Tbb\text{-mod}_{fp}(\Set)\op$. Since any finite colimit of finitely presentable objects is finitely presentable, (cf. \cite[Proposition 1.3]{Accessible}) when $\Tbb$ is an essentially algebraic theory, $\Dcal_{\Tbb}$ always has all finite limits, as expected. Thus a stable class in $\Dcal_{\Tbb}$ consists of a class of morphisms in $\Tbb\text{-mod}_{fp}(\Set)$ which contains identities and is stable under composition, pushouts and left factors.

\begin{rmk}
These requirements bear a striking (and perhaps misleading) resemblance to the \textit{saturated classes} found in \cite[Chapter 1]{Osmond} (alternatively, in \cite{spectra1}), which generate the left class in an orthogonal factorization system on the larger category $\Tbb\text{-mod}(\Set)$.

A class $\Vcal$ of morphisms in $\Tbb\text{-mod}_{fp}(\Set)$ is said to be a \textbf{saturated class} if:
\begin{itemize}
	\item $\Vcal$ contains isomorphisms;
	\item $\Vcal$ is stable under composition;
	\item $\Vcal$ is right-cancellative, in the sense that given a commuting triangle
	\[ \begin{tikzcd}
		X \ar[dr, "g"'] \ar[rr, "f"] & & Y, \\
		& Z \ar[ur, "h"'] &
	\end{tikzcd}\]
	where $f$ and $g$ are in $\Vcal$, $h$ also lies in $\Vcal$;
	\item $\Vcal$ is closed under finite colimits in $\Tbb\text{-mod}_{fp}(\Set)^2$;
	\item $\Vcal$ is closed under pushouts along arbitrary morphisms in $\Tbb\text{-mod}_{fp}(\Set)$.
\end{itemize}

Unfortunately, the right-cancellation property, besides being slightly too strong, is on the wrong side; a saturated class in the above sense `behaves like a class of epimorphisms', whereas we require a class that `behaves like a pushout-stable class of monomorphisms'.
\end{rmk}

The fact that $\Dcal_{\Tbb}$ has finite limits means that any principal topology on this category will yield a regular topos, since all of the representables will become regular objects. We include some amongst our examples anyway.

\subsection{Examples}
\label{ssec:xmpl1}

In each of the following examples, we apply Proposition \ref{prop:quotienttheory} to principal sites $(\Ccal,J_{\Tcal})$ for which the topos of presheaves on the underlying category $\Ccal$ classifies a known theory (simpler than the generic one described above).

\subsubsection{Theories of objects}
\label{sssec:objects}

Recall that the theory of objects $\Obb$ is the empty theory on a signature with a single sort and no relation or function symbols. It is classified by $[\FinSet,\Set]$, so $\Dcal_{\Obb} = \FinSet\op$. Let $\Tcal$ be a stable class in $\Dcal_{\Obb}$.

All of the monomorphisms in $\FinSet$, with the exception of those of the form $0 \hookrightarrow n$, are split monomorphisms, so $\Tcal\op$ must contain these; these constitute the stable class corresponding to the trivial topology. Moreover, $\Tcal\op$ contains $0 \hookrightarrow n$ for some $n>0$ if and only if it contains $0 \hookrightarrow 1$.

Axiom 4 of stable classes means that $\Tcal\op$ contains a morphism $m \to n$ of $\FinSet$ with $m>0$ if and only if $\Tcal\op$ contains the epic part of this morphism. If this epimorphism is non-trivial, in the sense that it identifies two elements $x \neq y$, we can take a pushout of this morphism along $m \to 2$ sending $x$ to $0$ and $y$ to $1$, so $2 \too 1$ is in $\Tcal\op$. From there, we may push out along a monomorphism $2 \hookrightarrow k + 1$ to produce an epimorphism $k+1 \too k$ identifying any given choice of elements. Since any non-trivial epimorphism in $\FinSet$ is a composite of such morphisms, it follows that $\Tcal\op$ contains all epimorphisms!

In summary, there are four possibilities for a stable class in $\Dcal_{\Obb}$:
\begin{enumerate}[label = {\alph*})]
	\item $\Tcal\op$ consists of just the split monomorphisms;
	\item $\Tcal\op$ contains all monomorphisms;
	\item $\Tcal\op$ contains the split monomorphisms and all epimorphisms;
	\item $\Tcal\op$ contains all morphisms.
\end{enumerate}
To identify the theories classified by the respective classes, we need to identify the formulae corresponding to the objects and morphisms. The set with $n$ elements has a finite presentation as the formula-in-context $\{x_0, \dotsc, x_{n-1} \cdot \top\}$, and the formula corresponding to a function $\sigma:m \to n$ is (recalling that the morphism in $\Dcal_{\Obb}$ goes in the opposite direction) is simply
\[x_0,\dotsc,x_{m-1},x'_0,\dotsc,x'_{n-1} \cdot \bigwedge_{i=0}^{n-1} (x'_{\sigma(i)} = x_i).\]
Our observations above demonstrate that it suffices to consider the morphisms $0 \hookrightarrow 1$ and $2 \too 1$, so that $\Sh(\Dcal_{\Obb},J_{\Tcal})$ respectively classifies:
\begin{enumerate}[label = {\alph*})]
	\item the theory of objects;
	\item the theory of inhabited objects, adding the axiom $\top \vdash_{[]} (\exists x) \top$;
	\item the theory of objects with at most one element, adding the axiom $\top \vdash_{x_0,x_1} (\exists x)(x = x_0 \wedge x = x_1)$;
	\item the theory of objects with exactly one element.
\end{enumerate}
Incidentally, observe that all of these are theories of presheaf type: (a) is the presheaf topos we started with, while (b), (c) and (d) are respectively classified by $[\FinSet_+,\Set]$, $[\bullet \to \bullet,\Set]$ and $\Set$, where $\FinSet_+$ is the category of inhabited finite sets.

Next, let $\Obb_*$ be the theory of \textbf{pointed objects}, which has a single sort and a single constant symbol, $0$. Letting $\FinSet_*$ be the category of finite pointed sets, we have that $\Dcal_{\Obb_*} = \FinSet_*\op$; note that this is distinct from $\FinSet_+$, since the distinguished point must be preserved by all morphisms. Without loss of generality we can take $0$ to be the distinguished point in an $n$-element set $\{0,\dotsc,n-1\}$; this covers all isomorphism classes of finite pointed sets. Let $\Tcal$ be a stable class in $\Dcal_{\Obb_*}$.

All monomorphisms in $\Dcal_{\Obb_*}$ split, so $\Tcal\op$ must contain them all. Thus, as above, we may reduce to considering which epimorphisms lie in $\Tcal\op$. Given any non-trivial epimorphism $g:n \to m$ in $\Tcal\op$, suppose $x \neq y$ are identified by $g$ and without loss of generality that $y \neq 0$. Then we can consider the mapping $n \to 2$ sending $x$ to $0$ and all other elements to $1$; the pushout of $g$ along this gives the morphism $2 \too 1$. Pushing out further and composing, we conclude that $\Tcal\op$ must contain all functions in which elements are successively identified with $0$, notably including $3 \too 1$. Observe that the mapping $3 \too 2$ mapping $1$ and $2$ to $1$ is a left factor of $3 \too 1$, and so must also be in $\Tcal\op$, despite not identifying elements with $0$. Pushing out this function and composing again, we conclude that $\Tcal\op$ must contain arbitrary epimorphisms.

In summary, we have just two principal topologies:
\begin{enumerate}[label = {\alph*})]
	\item $\Tcal\op$ consists of just the monomorphisms;
	\item $\Tcal\op$ contains all morphisms.
\end{enumerate}
To identify the corresponding theories, we must identify formulae presenting the finite pointed sets. A pointed set with $n + 1$ elements is the free pointed set on $n$ generators, since the free construction freely adds a distinguished point. Thus, $n$ is presented by the formula $\{x_1, \dotsc,x_{n-1} \cdot \top \}$. Accordingly, the $\Obb_*$-provably functional formula corresponding to a function $\sigma: n \to m$ is,
\[x_1,\dotsc,x_{n-1},x'_1,\dotsc,x'_{m-1} \cdot \bigwedge_{i = 1}^{n-1} (x'_{\sigma(i)} = x_i),\]
where any instances of $x_0$ on the right-hand side are substituted for the constant symbol $0$.

To determine the quotients, we need only identify the $\Obb_*$-provably functional formulae corresponding to the map $2 \to 1$, which is simply $x_1 \cdot x_1 = 0$. As such, we find that the corresponding toposes $\Sh(\Dcal_{\Obb_*},J_{\Tcal})$ respectively classify: 
\begin{enumerate}[label = {\alph*})]
	\item the theory of pointed objects;
	\item the theory of objects with exactly one element (obtained by adding $\top \dashv_x x = 0$).
\end{enumerate}
The latter topos is by inspection equivalent to $\Set$.

\begin{rmk}
While it has been excised from the account above, the translation from morphisms into formulae was very helpful in identifying the minimal classes involved here; for example, the formula corresponding to the function $3 \too 2$ described mapping $1$ and $2$ to $1$ is $x_1,x'_1,x'_2 : x_1 = x'_1 \wedge x_1 = x'_2$, which means that the theory classified by the topos of sheaves for the Grothendieck topology generated by that morphism is obtained by adding the axiom $\top \dashv_{x'_1,x'_2} (\exists x) x = x'_1 \wedge x = x'_2$. By inspection, this gives another presentation of the theory of objects with one element, so we deduced that $\Tcal\op$ must contain $2 \too 1$ if and only if it contained $3 \too 2$.
\end{rmk}

As a third example in the same vein, consider the theory $\Dbb$ of decidable objects, obtained from the theory of objects by adding a binary `apartness' relation ${\#}$ and the axioms,
\[(x \# x) \vdash_x \bot \hspace{1cm} \text{ and } \hspace{1cm} \top \vdash_{x,y} (x \# y)\vee (x = y).\]
By \cite[Proposition D3.2.7]{Ele}, this theory is classified by the topos of presheaves on the category $\Dcal_{\Dbb} := \FinSet_{\mathrm{inj}}\op$, where $\FinSet_{\mathrm{inj}}$ is the category of finite sets and injections.

This time we jump straight to the presentations of objects by formulae. The simplest option is to identify the $n$-element set with the formula-in-context
\[\left\{x_1,\dotsc,x_n \cdot \bigwedge_{1 \leq i < j \leq n}(x_i \# x_j) \right\},\]
while an injective function $\sigma:n \hookrightarrow m$ corresponds to the usual $\Dbb$ formula,
\[x_1,\dotsc,x_{n},x'_1,\dotsc,x'_{m} \cdot \left(\bigwedge_{k=1}^{n} (x'_{\sigma(k)} = x_k) \right) \wedge \left(\bigwedge_{1 \leq i < j \leq m}(x'_i \# x'_j) \right).\]
Thus, given a stable class $\Tcal$ of morphisms in $\Dcal_{\Dbb}$, the axiom which must be added to $\Dbb$ if $\sigma \in \Tcal\op$ can be reduced to,
\[ \bigwedge_{1 \leq i < j \leq n}(x_i \# x_j) \vdash_{\vec{x}} (\exists \vec{x}') \left(\bigwedge_{1 \leq i < j \leq m}(x'_i \# x'_j) \right), \]
or less formally, `any set of $n$ distinct elements can be extended to a set of $m$ distinct elements'.

Knowing this, it will come as no surprise that if $\Tcal\op$ contains a morphism $\sigma: n \hookrightarrow m$ then we can use stability axioms 3 and 4 to show that $\Tcal\op$ must contain all of the morphisms $n' \hookrightarrow m'$ with $n \leq n' \leq m' \leq m$. Moreover, no other morphisms are necessary besides identity morphisms, since the resulting collection of morphisms is clearly closed under axioms 1, 2 and 4, and given any morphism $n' \to m''$ with $n' \geq n$ and $m'' \geq m$, we can complete a stability square with an identity morphism.

As such, the stable classes in $\Dcal_{\Dbb}$ can be indexed by finite or countable subsets $S \subseteq \Nbb$, where $S$ is the collection of cardinalities $N$ such that $\Tcal\op$ contains no non-trivial covers of the set of size $N$. The class $\Tcal$ corresponding to $S$ produces a subtopos of $\Setswith{\Dcal_{\Dbb}}$ classifying a theory which we shall denote $\Dbb(S)$: the theory of `decidable objects such that any collection of $n$ distinct elements can be extended to a collection of $n^+$ distinct elements, where $n^+ := \inf \{N \in S \cup \{\infty\} \mid N \geq n\}$', or more succinctly, the theory of `decidable objects with $N$ elements, where $N \in S \cup \{\infty\}$'. With this notation, $\Dbb = \Dbb(\Nbb)$, while the other extreme case, $\Dbb(\emptyset)$ is the theory of infinite decidable sets discussed in \cite[Example D3.4.10]{Ele}.

\begin{lemma}
The classifying topos of $\Dbb(S)$ is a presheaf topos if and only if $S$ is infinite.
\end{lemma}
\begin{proof}
If $S$ is infinite, then for every finite set $n$, there exists a minimal $N \in S$ with $n \leq N$, which is the largest number such that $n \hookrightarrow N$ is in $\Tcal\op$; the $N \in S$ therefore correspond to the $J_{\Tcal}$-irreducible objects in the sense of \cite[\S 8.2.2]{TST}, and there are enough of them to cover all of the object of $\Dcal_{\Dbb}$.

On the other hand, we know that none of the infinite decidable sets are finitely presentable in the category of $\Dbb(S)$-models in $\Set$ (for any $S$), since we proved this for $\Dbb(\emptyset)$ in Example \ref{xmpl:nonpresh} from the last chapter, and the category of infinite decidable sets is clearly closed under filtered colimits in the category of $\Dbb$-models. Thus, when $S$ is finite there are exactly $|S|$ finitely presentable models of $\Dbb(S)$, but at least $|S| + 1$ non-isomorphic points, which is more than the category of presheaves on any category with $|S|$ objects can have (in fact there is a proper class of points, but one too many is enough to accomplish the argument).
\end{proof}

\subsubsection{Boolean algebras}
\label{sssec:Boole}

Recall that classical Stone duality identifies the category of Boolean algebras with the dual of the category of profinite sets. In fact, the latter category is more commonly identified as the category of compact Hausdorff totally disconnected topological spaces, but the profinite set description is more useful for us because it identifies the category of Boolean algebras as the inductive completion of the category of finite Boolean algebras. As such, the theory of Boolean algebras is classified by the topos $\Setswith{\FinSet}$. For completeness, we include a presentation of the theory $\Bb$ of Boolean algebras which has one sort, two binary infix function symbols ${\cup}$ and ${\cap}$, a unary operation $\neg$ and two constant symbols $0$ and $1$ subject to:
\begin{align*}
\top & \vdash_{x,y,z} (x \cup y) \cup z = x \cup (y \cup z) &
\top & \vdash_{x,y,z} (x \cap y) \cap z = x \cap (y \cap z) \\
\top & \vdash_{x,y} x \cup y = y \cup x &
\top & \vdash_{x,y} x \cap y = y \cap x \\
\top & \vdash_{x} x \cup 1 = 1 &
\top & \vdash_{x} x \cap 0 = x \\
\top & \vdash_{x,y,z} x \cup (y \cap z) = (x \cup y) \cap (x \cup z) &
\top & \vdash_{x,y,z} x \cap (y \cup z) = (x \cap y) \cup (x \cap z) \\
\top & \vdash_{x} x \cup \neg x = 1 &
\top & \vdash_{x} x \cap \neg x = 0
\end{align*}
Obviously, we use the rounded cup and cap for the binary operations to distinguish them from the logical operations in our logical formalism.

The identification of the category of finite Boolean algebras with the dual of finite sets is straightforward, via powersets. In particular, writing $n$ for the set $\{0,\dotsc,n-1\}$ as before, the Boolean algebra $\Pcal(X)$ corresponds to the formula-in-context
\[\left\{x_0,\dotsc,x_{n-1} \cdot \left(\bigwedge_{0 \leq i < j \leq n-1} (x_i \cap x_j = 0)\right) \wedge \left( \bigcup_{i=0}^{n-1} x_i = 1\right) \right\},\]
where the large cup expression has the evident shorthand meaning, becoming $0$ when $n = 0$, while a function $\sigma:n \to m$ corresponds to the $\Bb$-provably functional formula,
\[x_0,\dotsc,x_{n-1},x'_0,\dotsc,x'_{m-1} \cdot \left(\bigwedge_{k=0}^{n-1} (x'_{\sigma(k)} = x_k)\right) \wedge \left(\bigwedge_{0 \leq i < j \leq n-1} (x_i \cap x_j = 0)\right) \wedge \left( \bigcup_{i=0}^{n-1} x_i = 1\right); \]
note that the expression is covariant with $\sigma$ here.

Let $\Dcal_{\Bb} := \FinSet$. Proceeding more or less dually to the `theory of objects' situation, we observe first that all epimorphisms in $\Dcal_{\Bb}$ split, so any stable class $\Tcal$ of morphisms in $\Dcal_{\Bb}$ contains them; by axiom 4, this implies that a morphism lies in $\Tcal$ if and only if its image does.

If $\Tcal$ contains any non-trivial monomorphism $n \hookrightarrow m$ with $0 < n < m$, then pulling back along a suitable morphism $2 \to m$, we conclude that the subobject classifier map $1 \hookrightarrow 2$ of $\FinSet$ (recalling that this category is an elementary topos) lies in $\Tcal$, whence every monomorphism does by further pullback. Finally, if $0 \hookrightarrow n$ is a member of $\Tcal$ for any $n$, then it is so for every $n$, but moreover $1 \hookrightarrow 2$ is a right factor of $0 \hookrightarrow 2$. Thus we somewhat disappointingly have only two possibilities:
\begin{enumerate}[label = {\alph*})]
	\item $\Tcal$ consists of just the split epimorphisms;
	\item $\Tcal$ contains all morphisms;
\end{enumerate}
these respectively classify:
\begin{enumerate}[label = {\alph*})]
	\item the theory of Boolean algebras;
	\item the theory of the trivial Boolean algebra.
\end{enumerate}
The latter topos is once again equivalent to $\Set$, and is the double-negation subtopos of $\Setswith{\FinSet}$.

Things get more interesting when we restrict the morphisms in the category by considering \textit{decidable} Boolean algebras, but we shall not present that example here. See \cite[Example D3.4.12]{Ele} for this example and the extremal case of the atomic topology on the resulting category.

\subsubsection{Ordered structures}
\label{sssec:order}

Consider the augmented\footnote{The ordinary simplex category $\Delta$ considered in Example \ref{xmpl:simplex} of the last chapter contains only the \textit{non-empty} finite ordinals.} simplex category $\Delta_0$, whose objects are finite ordinals,
\[ [n] = \{0, \dotsc, n-1\} \]
and whose morphisms are the order-preserving maps between these. This can equivalently be expressed as the category of finite total orders, which are models of the theory $\Tbb$ on a signature with a single sort and a single binary relation symbol subject to the axioms:
\begin{align*}
\top & \vdash_{x} x \leq x \\
\left(\exists y \right) \left((x \leq y) \wedge (y \leq z)\right) & \vdash_{x,z} x \leq z \\
(x \leq y) \wedge (y \leq x) & \vdash_{x,y} x=y \\
\top & \vdash_{x,y} (x \leq y) \vee (y \leq x).
\end{align*}

Note that this category doesn't have pullbacks, since although it has a terminal object, the product $[2] \times [2]$ does not exist: there are two morphisms $[1] \to [2]$, so such a product would necessarily have four elements, but $[4]$ is not this product, since there are ten morphisms $[2] \to [4]$ but, since there are three morphisms $[2] \to [2]$, there would need to be exactly nine morphisms $[2] \to [2] \times [2]$. On the other hand, $\Delta_0$ does have pullbacks of monomorphisms.

The category $\Setswith{\Delta_0}$ is known to classify the theory $\Ibb$ of \textit{bounded} total orders (also known as \textit{intervals}), which is obtained from $\Tbb$ by adding two constant symbols $0$ and $1$, as well as the extra axiom:
\begin{equation*}
\top \vdash_{x}(0 \leq x) \wedge (x \leq 1).
\end{equation*}

It will serve us to partly verify this claim\footnote{We take for granted that the stated theories have these categories of finitely presentable models.} by demonstrating that the category of finite bounded total orders is the dual of the augmented simplex category.

\begin{lemma}
\label{lem:reversal}
Let $m,n \in \Nbb$ with $n > 0$. Then there is a one-to-one, order-reversing correspondence between maps $[m] \to [n]$ and maps $[n-1] \to [m+1]$.
\end{lemma}
\begin{proof}
Given an increasing function $\sigma:[m] \to [n]$, consider its graph (in the sense of a bar chart); in Figure \ref{fig:graphs} we give an example with $m = 8$ and $n = 7$.
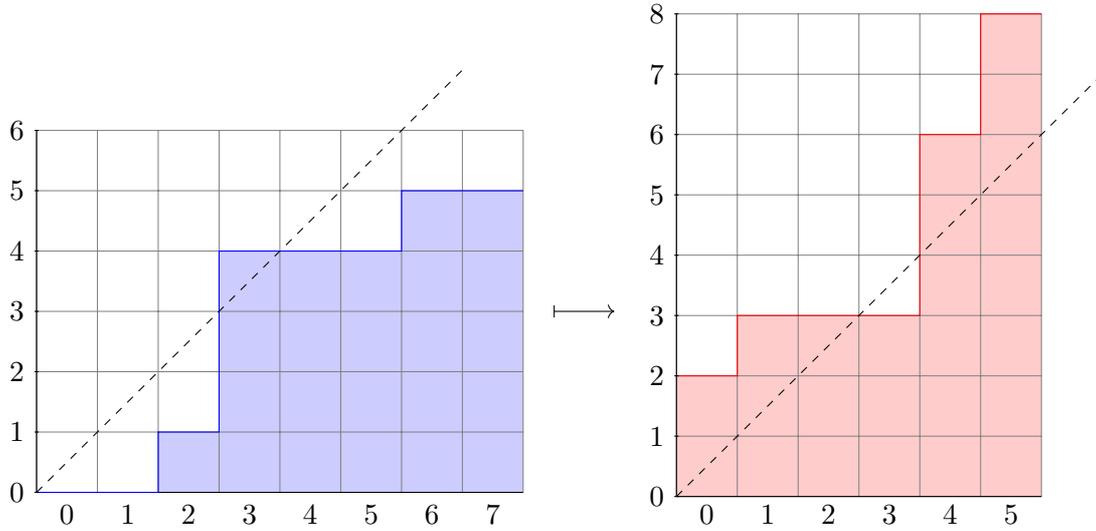
\begin{figure}
\[\begin{tikzpicture}[scale = 0.8]
\fill[blue!20!white] (0,0) -- (2,0) -- (2,1) -- (3,1) -- (3,4) -- (6,4) -- (6,5) -- (8,5) -- (8,6) -- (8,0) -- (0,0);
\draw[step=1cm,gray,very thin] (0,0) grid (8,6);
\draw (0,0) -- (8,0);
\draw (0,0) -- (0,6);
\draw[blue] (0,0) -- (2,0) -- (2,1) -- (3,1) -- (3,4) -- (6,4) -- (6,5) -- (8,5);
\draw[dashed] (0,0) -- (7,7);
\foreach \x in {0,...,7}
\draw (\x cm,-1pt) + (0.5 cm,0pt) node[anchor=north] {$\x$};
\foreach \y in {0,...,6}
\draw (1pt,\y cm) -- (-1pt,\y cm) node[anchor=east] {$\y$};
\draw [|->] (8.5,3) -- (9.5,3);
\end{tikzpicture}
\begin{tikzpicture}[scale = 0.8]
\draw[white] (-1,2) -- (-1,4);
\fill[red!20!white] (0,0) -- (0,2) -- (1,2) -- (1,3) -- (4,3) -- (4,6) -- (5,6) -- (5,8) -- (6,8) -- (6,0) -- (0,0);
\draw[step=1cm,gray,very thin] (0,0) grid (6,8);
\draw (0,0) -- (6,0);
\draw (0,0) -- (0,8);
\draw[red] (0,2) -- (1,2) -- (1,3) -- (4,3) -- (4,6) -- (5,6) -- (5,8) -- (6,8);
\draw[dashed] (0,0) -- (7,7);
\foreach \x in {0,...,5}
\draw (\x cm,1pt) + (0.5 cm,0pt) node[anchor=north] {$\x$};
\foreach \y in {0,...,8}
\draw (1pt,\y cm) -- (-1pt,\y cm) node[anchor=east] {$\y$};
\end{tikzpicture}\]
\caption{The graph of a morphism $\sigma:[8] \to [7]$ and its transpose $\sigma^{\circ}:[6] \to [9]$.}
\label{fig:graphs}
\end{figure}
Such a graph has $m$ bars of height at most $n-1$. If we transpose this graph along the diagonal and invert the shading, we obtain the graph of an increasing function $[n-1] \to [m+1]$, with $n-1$ bars of height at most $m$. Specifically, this is the graph of the function $\sigma^{\circ}:[n-1] \to [m+1]$ defined by
\begin{equation}
\label{eq:sigmaop}
	\sigma\op(k) := \# \{x \mid \sigma(x) \leq k \}.
\end{equation}
Since this transformation is clearly invertible, this defines the desired bijection; that it reverses the order is similarly straightforward.
\end{proof}

\begin{crly}
$\Delta_0\op$ is equivalent to the category of finite bounded total orders.
\end{crly}
\begin{proof}
The correspondence identifies the ordinal $[n]$ with the bounded total order having $n+1$ elements. Indeed, a morphism from the $n+1$-element total order to that with $m+1$ elements must send the top and bottom elements to their counterparts, but its restriction to the intermediate elements can be identified with an order-preserving function $[n-1] \to [m+1]$, which correspond to morphisms $[m] \to [n]$ by Lemma \ref{lem:reversal}. It is straightforward to check that this identification respects composition.
\end{proof}

Let $\Dcal_{\Tbb}:= \Delta_0$, and let $\Tcal$ be a stable class in $\Dcal_{\Tbb}$. As in several of the `theories of objects' examples, all epimorphisms are split, so $\Tcal$ must contain all of these, and the (epi,mono)-factorization system means that $\Tcal$ is once again determined by the monomorphisms it contains.

If $\Tcal$ contains any non-trivial monomorphism $[n] \hookrightarrow [m]$, then picking out an element not lying in the image and pulling back, the morphism $[0] \hookrightarrow [1]$ lies in $\Tcal$, and hence every morphism does by axioms 3 and 4 as usual. As such, the only supercompactly generated subtopos of $\Setswith{\Delta_0}$ is the double negation subtopos, equivalent to $\Set$ and classifying the empty total order.

Dually, let $\Dcal_{\Ibb} := \Delta_0\op$. Given a stable class $\Tcal$ in $\Dcal_{\Ibb}$, $\Tcal\op$ contains all of the monomorphisms with non-empty domain, since these are all split. Given any non-trivial epimorphism $[n] \too [m]$ in $\Tcal\op$, pushing out along a suitable morphism $[m] \too [2]$ demonstrates that $[2] \too [1]$ lies in $\Tcal\op$, and hence (by inductive application of axioms 2 and 3) all epimorphisms lie in $\Tcal\op$. Finally, if any $[0] \hookrightarrow [n]$ is in $\Tcal\op$ then they all are, by axioms 2 and 4, so we have four possibilities in total, whose corresponding theories have descriptions analogous to the quotients of the theory of objects described earlier.

\section{Observations}
\label{sec:observe}

We focused our efforts in the last section on applying the method of Section \ref{ssec:preshtype} to theories classified by toposes of presheaves on categories $\Dcal$ simple enough that we could hope to characterize all of the principal topologies on them. Besides the case of the theory of decidable objects $\Dbb$, the results were not especially exciting, in the sense that most of the theories we ended up with were of presheaf type.

There are other situations where this is bound to happen. Recall from Proposition \ref{prop:localic} of Chapter \ref{chap:sgt} that a localic topos is supercompactly generated if and only if it is equivalent to $\Setswith{\Dcal}$ for some poset $\Dcal$. Recall also that a theory has a localic classifying topos if and only if it is Morita-equivalent to a propositional theory. Since any subtopos of a localic topos is localic, we conclude:
\begin{crly}
Every supercompactly generated propositional theory is of presheaf type.
\end{crly}

Considering the work from the previous chapter, we must stress that the toposes we have arrived at in the above analysis are merely the \textit{relatively pristine} subtoposes of the presheaf toposes under consideration, in the sense of Definition \ref{dfn:relpolished}. Indeed, any full subcategory of $\Dcal_{\Tbb}$ (for $\Tbb$ any of the theories considered) provides a non-trivial presheaf subtopos corresponding to an essential \textit{relatively polished} inclusion, and these have further relatively pristine subtoposes. Actually, this procedure produces all of the relatively polished subtoposes:
\begin{schl}
Let $f: \Ecal \hookrightarrow \Setswith{\Dcal}$ be a geometric inclusion (over $\Set$). Then $f$ is relatively polished if and only if it factors as a relatively pristine inclusion followed by an essential inclusion.
\end{schl}
\begin{proof}
In the proof of Corollary \ref{crly:incl} in the last chapter, we observed that an inclusion into $\Setswith{\Dcal}$ is relatively polished if and only if the corresponding Grothendieck topology $J$ on $\Dcal$ is \textit{quasi-principal}; letting $\Dcal'$ be the full subcategory of $\Dcal$ on the objects over which the empty sieve is not $J$-covering, the induced topology $J|_{\Dcal'}$ is principal, and $f$ factors as
\[\Sh(\Dcal',J|_{\Dcal'}) \hookrightarrow \Setswith{\Dcal'} \hookrightarrow \Setswith{\Dcal}, \]
where the latter is an essential inclusion, as required. This proof is manifestly non-constructive: it can fail in a topos where not every object is decidable.
\end{proof}

There may be still further supercompactly generated (but not relatively polished) subtoposes of the presheaf toposes considered in this chapter. In spite of our syntactic characterization of supercompactly generated toposes, the logical framework is ill-equipped to detect them: verifying that a general classifying topos is supercompactly generated typically requires a very robust meta-logical understanding of the proof theory of $\Tbb$. Theorem \ref{thm:scompform} is more useful in the opposite direction: if a classifying topos can be shown to be supercompact by other means, this imposes a powerful constraint on the proof theory.

It is also worth noting that all of our specific examples began from known theories of presheaf type. This is symptomatic of some of the obstacles one is confronted with when trying to apply this method in a more general setting. Foremost is the fact that, while we do get a theory for any principal site $(\Dcal,J_{\Tcal})$, as we saw in Section \ref{ssec:generic}, this standard theory is automatically just as complicated as $\Dcal$. To find simpler (and therefore more informative) theories, we must rely on having access ahead of time to a simpler theory which is classified by the presheaf topos $\Setswith{\Dcal}$, which is typically only possible in special cases.

Even given a theory $\Tbb$ which is classified by $\Setswith{\Dcal}$, identifying the quotient theory classified by a subtopos of the form $\Sh(\Dcal,J_{\Tcal})$ will typically be difficult because it requires us to first find explicit finitary presentations of the objects of $\Dcal\op$ as $\Tbb$-models and to identify $\Tbb$-provably functional formulae presenting the morphisms of $\Tcal$ or $\Tcal'$. As such, we think it would be useful to have a way to present supercompactly generated theories syntactically, much as regular, coherent and geometric theories are presented by axioms over a signature, as we described earlier. We shall discuss this idea briefly in Section \ref{ssec:redlogic} of the Conclusion.

%% file: The_TTMA.tex
\chapter{Toposes of Topological Monoid Actions}
\label{chap:TTMA}

In Chapter \ref{chap:TDMA}, we investigated properties of presheaf toposes of the form $\Setswith{M}$ for a monoid $M$, whose objects are sets equipped with a right action of $M$. A natural direction to generalize this study is to view sets as discrete spaces and to consider the actions of a topological monoid on them. In order to analyze this case, we parallel the analogous categories for topological groups, which are well-studied.

For a topological group $(G,\tau)$, the category $\Cont(G,\tau)$ of continuous $G$-actions on discrete topological spaces is a Grothendieck topos. One way to prove this is to observe that there is a canonical geometric morphism $\Setswith{G} \to \Cont(G,\tau)$ which is a surjection, see \cite[A4.2.4(a)]{Ele}. The inverse image functor of this morphism is the forgetful functor which sends a continuous $(G,\tau)$-set to its underlying $G$-set. The direct image functor is constructed explicitly by Mac Lane and Moerdijk in \cite[\S VII.3]{MLM}: it sends a $G$-set $X$ to the subset consisting of those elements whose `isotropy subgroup' is open; it follows that the counit of this morphism is monic and so (by \cite[A4.6.6]{Ele}, say) that the geometric morphism is moreover hyperconnected.

In this article we begin by extending these observations to categories of continuous actions of monoids. We take a rather classical approach at first: rather than considering genuine topological monoids (that is, monoids in the category of topological spaces), we consider endowing the underlying set of a monoid $M$ with an arbitrary topology $\tau \subseteq \Pcal(M)$. This approach is motivated by the fact that no part of the description of a continuous $(M,\tau)$-set relies on the fact that the topology $\tau$ makes $M$ a topological monoid, and we shall indeed see that the argument is valid even when this fails. \textit{A reader critical of this decision should be reassured by the fact that, as we shall eventually see in Theorem \ref{thm:tau} and Proposition \ref{prop:ctsx}, any `monoid with a topology' is in any case Morita-equilavent to a genuine topological monoid.}

\subsection*{Overview}

In Section \ref{ssec:necessary}, we exhibit the necessary data to establish that the forgetful functor from the category $\Cont(M,\tau)$ of continuous actions of a monoid with respect to an arbitrary topology $\tau$ to the topos $\Setswith{M}$ is left exact and comonadic (Proposition \ref{prop:hyper}). The adjoint can be expressed using either clopen subsets of $(M,\tau)$ or open relations. From the existence of this adjunction we conclude that $\Cont(M,\tau)$ is an elementary topos, so that the forgetful functor is the inverse image of a hyperconnected geometric morphism, just as in the group case. In Section \ref{ssec:apply}, we apply the theoretical results from Chapter \ref{chap:sgt} to deduce that any topos of the form $\Cont(M,\tau)$ is moreover a supercompactly generated Grothendieck topos, which brings us to an intuitive Morita-equivalence result in terms of the (essentially small) category of continuous principal $M$-sets, Corollary \ref{crly:Morita'}. Finally, in \ref{ssec:jcp} we show another property of the categories of principal continuous $M$-sets which has not yet been covered, indicating that our characterization of toposes of the form $\Cont(M,\tau)$ is not yet complete.

In Section \ref{sec:montop}, we examine the question of how much is recoverable about a topology $\tau$ on a monoid $M$ from the hyperconnected morphism $\Setswith{M} \to \Cont(M,\tau)$. To do this, we construct the classical powerset $\Pcal(M)$ of $M$ as a right $M$-set in Section \ref{ssec:inverse}, and from this object recover in Section \ref{ssec:action} a canonical topology $\tilde{\tau}$, contained in $\tau$, making $(M,\tilde{\tau})$ a genuine topological monoid with an equivalent category of actions; given that Morita-equivalence for topological monoids is non-trivial, this is as good a result as we could have hoped for. In Section \ref{ssec:powder}, we show that we can further reduce this topological monoid to obtain a Hausdorff monoid $(\tilde{M},\tilde{\tau})$, still retaining the same topos of actions. The resulting class of representative topological monoids for toposes of the form $\Cont(M,\tau)$, which we call \textit{powder monoids}, have many special properties. In Section \ref{ssec:prodiscrete}, we show that this class includes, but is not limited to, the classes of prodiscrete monoids and nearly discrete groups; indeed, the reduction of a monoid to a powder monoid is analogous of the reduction of a group to a nearly discrete group in \cite[Example A2.1.6]{Ele}.

In Section \ref{sec:surjpt}, we consider the canonical surjective point of $\Cont(M,\tau)$, which is the composite of the canonical essential surjective point of $\Setswith{M}$ and the hyperconnected morphism obtained in Section \ref{sec:properties}. Our aim is to characterize toposes of this form in terms of the existence of a point of this form, just as Caramello does for topological groups in \cite{TGT}. First, in Section \ref{ssec:EqRel}, we obtain a canonical small site for $\Setswith{M}$ whose objects are right congruences, equivalent to the site of principal $M$-sets, and show that hyperconnected morphisms out of $\Setswith{M}$ correspond to suitable subsites of this one. In particular, this provides a small site for $\Cont(M,\tau)$ (Scholium \ref{schl:Morita2}). By taking a limit indexed over such a site, we show in Section \ref{ssec:complete} that, analogously to the case of groups in \cite{TGT}, we can recover a presentation for the codomain of a hyperconnected geometric morphism out of $\Setswith{M}$ as a topos of topological monoid actions. This presentation is obtained by topologizing the monoid of endomorphisms of the canonical point. Thus, the existence of a point factorizing as an essential surjection followed by a hyperconnected morphism characterizes this class of toposes (Theorem \ref{thm:characterization}). Moreover, the resulting \textit{complete monoids} are powder monoids (Proposition \ref{prop:Lpowder}), and any powder monoid presenting the same topos (equipped with the same canonical point) admits a dense injective monoid homomorphism to the canonical representative (Corollary \ref{crly:extend}). Paralleling the introduction of (algebraic) bases for topological groups, we show in Section \ref{ssec:base} that we can re-index the limit defining a complete monoid over a \textit{base of open congruences} in order to obtain a simpler expression for it and in certain cases deduce further properties. We briefly consider the topologies on the original monoid $M$ induced by hyperconnected morphisms out of $\Setswith{M}$ in Section \ref{ssec:factor}.

Finally, since in Chapter \ref{chap:TDMA} we saw that semigroup homomorphisms correspond to essential geometric morphisms between toposes of discrete monoid actions, in Section \ref{sec:homomorphism} we show that continuous semigroup homomorphisms between topological monoids induce geometric morphisms between the corresponding toposes of continuous actions (Lemma \ref{lem:cts}). As such, we show that $\Cont(-)$ defines a $2$-functor extending the presheaf construction for discrete monoids in Chapter \ref{chap:TDMA}, which we may restrict to the class of complete monoids. In Section \ref{ssec:intrinsic} we record some intrinsic properties of the hyperconnected geometric morphism $\Setswith{M} \to \Cont(M,\tau)$ when $M$ is a powder monoid or complete monoid, enabling us in Sections \ref{ssec:id} and \ref{ssec:monhom} to examine how, when a geometric morphism $g$ is induced by a continuous semigroup homomorphism $\phi$ between complete monoids, the properties of $g$ are reflected as properties of $\phi$. We show that the surjection--inclusion factorization of $g$ is canonically represented by the factorization of $\phi$ into a monoid homomorphism followed by an inclusion of a subsemigroup (Theorem \ref{thm:inccomplete}). Moreover, the hyperconnected--localic factorization of $g$ can be identified with the dense--closed factorization of $\phi$ (Theorem \ref{thm:locextend}). In both cases, the intermediate monoid is complete. Finally, in \ref{ssec:monads} we show that the classes of monoids we have been working with throughout assemble into reflective sub-($2$-)categories of the ($2$-)category of topological monoids.

\section{Properties of Categories of Continuous Monoid Actions}
\label{sec:properties}

\subsection{Necessary Clopens}
\label{ssec:necessary}

Throughout, we will refer to pairs $(M,\tau)$ where $M$ is a monoid (in $\Set$) and $\tau \subseteq \Pcal(M)$ is a topology on $M$. The multiplication on $M$ will be largely left implicit, but when we need to make it explicit we shall denote it by $\mu$. There is no assumption here that $\tau$ makes $\mu$ continuous; when it makes makes $\mu$ continuous in its first (resp. second) argument, we say that \textit{$\tau$ makes the multiplication of $M$ left (resp. right) continuous}.

\begin{rmk}
As mentioned above, the motivation for using these `monoids with topologies' rather than genuine topological monoids is that the definition of continuous $M$-set to follow applies without modification to this larger class of objects, and because it shall turn out to be useful to think of the topology as equipped (rather than intrinsic) structure. We reassure the reader that we shall eventually be able to reduce any `monoid with a topology' to a Morita-equivalent topological monoid.
\end{rmk}

Consider a (right) $M$-set, expressed in the form of a set $X$ equipped with a right action $\alpha: X \times M \to X$ subject to the usual conditions. We say this is an \textbf{$(M,\tau)$-set} if the action $\alpha$ is continuous when $X \times M$ is endowed with the product topology of the discrete topology on $X$ and the topology $\tau$ on $M$. An $(M,\tau)$-set will be referred to simply as a \textbf{continuous $M$-set} when the topology $\tau$ is understood.

We begin by exhibiting necessary and sufficient conditions for an $M$-set to be continuous.
\begin{lemma}
\label{lem:Inx}
Let $M$ be a monoid equipped with a topology $\tau$ and $X$ an $M$-set. Then $X$ is an $(M,\tau)$-set if and only if for each $x \in X$ and $p \in M$, the set
\[\Ical_x^p := \{m \in M \mid xm = xp\}  \]
is open in $\tau$. We call the collection of all such $\Ical_x^p$ the \textbf{necessary clopens} for $X$.
\end{lemma}
\begin{proof}
For continuity we precisely require that for each open subset $U \subseteq X$, its preimage under the action is open. Since $X$ is discrete, without loss of generality we may assume $U = \{x'\}$ for some $x' \in X$. A subset of $X \times M$ is open if and only if its intersection with each open subspace of the form $\{x\} \times M$ with $x \in X$ is open.

Thus we require $\{m \in M \mid xm = x'\}$ to be open for each pair $x, x' \in X$. However, if $x' \neq xp$ for every $p \in M$, the corresponding set is empty and so automatically open. Otherwise, $x' = xp$ for some $p$, which gives the result.
\end{proof}

To justify the name `necessary clopens' rather than merely `necessary opens', note that for each fixed $x$, the sets $\Ical_x^p$ partition $M$, so $\Ical_x^p$ being open for every $p$ forces each such set to also be closed. 

An $M$-set $X$ being continuous requires the `stabilizer submonoids' $\Ical_x^1$ to be both open and closed for every $x$. When $M$ is a topological group we know that this condition is actually sufficient, since the other subsets in the partition are simply the right cosets of $\Ical_x^1$ (which are open because a topological group acts on itself by homeomorphisms) but this is not the case for monoids in general.

While necessary clopens are the most direct generalization of (the right cosets of) the stabilizer subgroups for the action of a group on a set, we can avoid the additional need to index over these by working with equivalence relations:
\begin{crly}
\label{crly:Rnx}
Let $M$ be a monoid equipped with a topology $\tau$ and $X$ an $M$-set. Then $X$ is an $(M,\tau)$-set if and only if for each $x \in X$, the equivalence relation
\[\rfrak_x := \{(p,q) \in M \times M \mid xp = xq\}  \]
is open in the product topology $\tau \times \tau$ on $M \times M$. When $M$ is a group, $\rfrak_x$ is the relation that partitions $M$ into the right cosets of the stabilizer subgroup of $x$.
\end{crly}
\begin{proof}
If $X$ is an $(M,\tau)$-set, by Lemma \ref{lem:Inx} we have $\Ical_x^n \in \tau$ for each $n \in M$, and $\rfrak_x$ is precisely $\bigcup_{n \in M} \Ical_x^n \times \Ical_x^n$, so is open in $\tau \times \tau$. On the other hand, if $\rfrak_x$ is open, for fixed $p\in M$ each $(p,q) \in \rfrak_x$ is contained in an open rectangle $U_q \times V_q$ with $U_q,V_q$ open in $\tau$. Thus $\Ical_x^p = \bigcup_{(p,q)\in \rfrak_x}U_q$ is open, as required.
\end{proof}

Working concretely with an equivalence relation from Corollary \ref{crly:Rnx} is equivalent to working with all of the clopens in a partition at once. For each result to follow we can therefore give an expression in terms of either the necessary clopens or the open relations.

\begin{prop}
\label{prop:hyper}
Suppose a monoid $M$ is equipped with a topology $\tau$. Then the forgetful functor $V : \Cont(M,\tau) \to \Setswith{M}$ is left exact and comonadic; its right adjoint $R$ sends an $M$-set $X$ to:
\begin{align*}
R(X) & := \{x \in X \mid \forall p,q \in M, \, \Ical^p_{xq} \in \tau \} \\
& = \{x \in X \mid \forall q \in M, \, \rfrak_{xq} \in \tau \times \tau \}.
\end{align*}
Moreover, if $\tau$ makes the multiplication of $M$ left continuous then the expression for $R(X)$ simplifies to
\begin{align*}
R(X) & := \{x \in X \mid \forall p \in M, \, \Ical^p_x \in \tau \}\\
& = \{x \in X \mid \rfrak_{x} \in \tau \times \tau\}.
\end{align*}
\end{prop}
\begin{proof}
The definition ensures that $R(X)$ is closed under the action of $M$, since for any $x \in R(X)$ and $q \in M$, $\Ical_{xq}^p$ is open for every $p \in M$ by assumption, ensuring $xq \in R(X)$. Taking $q = 1$ for each $x \in R(X)$ demonstrates (by Lemma \ref{lem:Inx}) that $R(X)$ is a continuous $M$-set.

The inclusion $R(X) \hookrightarrow X$ is the universal morphism from a continuous $M$-set into $X$. Indeed, suppose $f : Y \to X$ is an $M$-set homomorphism with $Y$ a continuous $M$-set. Given $m \in \Ical^p_{f(y)q}$, there is an inclusion of subsets $\Ical_{yq}^m \subseteq \Ical_{f(y)q}^m$ since each $m' \in \Ical_{yq}^m$ has $f(y)qm = f(yqm) = f(yqm') = f(y)qm'$. So every $\Ical_{f(y)q}^p$ is open and the image of $f$ is contained in $R(X)$. It follows that $X \mapsto R(X)$ is a right adjoint for the forgetful functor, as required.

Since $V$ is full and faithful, it is conservative. A finite limit of discrete spaces is discrete, so a finite limit of continuous $(M,\tau)$-sets is precisely the limit of the corresponding $M$-sets. Thus $V$ is left exact, in particular preserving all equalizers. By any version of the (co)monadicity theorem, it follows that $V$ is comonadic.

Finally, observe that for $x \in R(X)$, $p,q \in M$, we have:
\[\Ical_{xq}^p = \{m \in M \mid xqm = xqp\} = \{m \in M \mid qm \in \Ical_{x}^{qp}\} = q^*(\Ical_{x}^{qp}),\]
where $q^*$ is the inverse image of multiplication on the left by $q$ (which shall be described in more detail in Section \ref{ssec:inverse}). Thus if $\tau$ makes multiplication by $q$ continuous then $\Ical_{xq}^p$ is open whenever $\Ical_{x}^{qp}$ is, whence we obtain the simplified expressions.
\end{proof}

We call $R(X)$ the subset of \textbf{continuous elements} of $X$ with respect to $\tau$ (even when multiplication is not left continuous with respect to $\tau$).

\begin{crly}
\label{crly:topos}
$\Cont(M,\tau)$ is an elementary topos.
\end{crly}
\begin{proof}
We have shown that $\Cont(M,\tau)$ is equivalent to the category of algebras for a cartesian comonad on $\Setswith{M}$, which by \cite[Theorem A4.2.1]{Ele} makes $\Cont(M,\tau)$ an elementary topos.
\end{proof}

\begin{rmk}
\label{rmk:sgrp}
One might wonder what can be said of the continuous actions of a \textit{semigroup} endowed with a topology. In Remark \ref{rmk:sgrp0} of Chapter \ref{chap:TDMA}, we observed that an action of a semigroup $S$ extends canonically to an action of the monoid $S_1$ obtained by adjoining a unit element (which must act as the identity). Given a topology on $S$, we may extend it to a topology on $S_1$ with an equivalent category of actions by making the singleton consisting of the adjoined unit an open subset, and extending this to a topology by taking unions with the existing opens. Thus once again, no generality is lost by considering only monoids equipped with topologies rather than arbitrary semigroups. 
\end{rmk}

\subsection{Toposes of actions are supercompactly generated}
\label{ssec:apply}

In light of Corollary \ref{crly:topos}, Proposition \ref{prop:hyper} demonstrates that the adjunction $(V \dashv R)$ is a hyperconnected geometric morphism $\Setswith{M} \to \Cont(M,\tau)$. Recalling the properties of $\Setswith{M}$ from Chapter \ref{chap:TDMA}, we may apply Theorem \ref{thm:hype2} of Chapter \ref{chap:sgt} to conclude that:

\begin{crly}
\label{crly:conttop}
Any topos of the form $\Cont(M,\tau)$ is a supercompactly generated, two-valued Grothendieck topos with enough points.
\end{crly}

As such, we can employ all of the theory developed in Chapter \ref{chap:sgt} to toposes of topological monoid actions. The fact that $\Setswith{M}$ is supercompactly generated is implicitly important in Hemelaer's work in \cite{TGRM}: when identifying those toposes of $G$-equivariant sheaves on a space $X$ which are equivalent to one of the form $\Setswith{M}$, they arrive at the definition of a \textit{minimal basis}, which corresponds to a base of supercompact open sets.

As we already discussed in Chapter \ref{chap:sgt}, the supercompact objects in $\Setswith{M}$ and hence in $\Cont(M,\tau)$ have a straightforward algebraic description.

\begin{dfn}
We shall call an object $N$ in $\Setswith{M}$ a \textbf{principal}\footnote{Some readers might prefer the term \textit{cyclic}.} \textbf{right $M$-set} if it is a quotient of $M$, in that there exists an epimorphism $M \too N$. Such an $M$-set is generated by a single element, the image of $1 \in M$ under the given epimorphism. Similarly, given a topology $\tau$ on $M$, we say an $(M,\tau)$-set $N$ is \textbf{principal} if $V(N)$ is a principal right $M$-set.

Similarly, we call an object $Q$ \textbf{finitely generated} if there is a finite jointly epic family of morphisms $M \to Q$, or equivalently if it has a finite (possibly empty) set of generators.
\end{dfn}

\begin{prop}
\label{prop:prince}
The supercompact objects of $\Ecal:= \Cont(M,\tau)$ are precisely the principal $M$-sets. As such, these form an effectual, reductive category, whose properties we record here:
\begin{enumerate}
	\item All monomorphisms in $\Ccal_s$ are regular and coincide with those in $\Ecal$;
	\item All epimorphisms in $\Ccal_s$ are strict and coincide with those in $\Ecal$;
	\item The classes of epimorphisms and monomorphisms in $\Ccal_s$ form an orthogonal factorization system;
	\item $\Ccal_s$ has a terminal object $1$, and every object is well-supported;
	\item $\Ccal_s$ has cokernels coinciding with those in $\Ecal$.
\end{enumerate}
\end{prop}
\begin{proof}
Clearly this is true in $\Setswith{M}$, since by definition the principal $M$-sets are exactly the quotients of the representable $M$-set $M$. It follows that the supercompact objects of $\Cont(M,\tau)$ are the continuous principal $M$-sets, since the inverse image of a hyperconnected morphism preserves and reflects supercompact objects. Properties 1 to 3 are derived in Lemma \ref{lem:monocoincide} and Scholium \ref{schl:regularmono}; Corollary \ref{crly:strict} and Lemma \ref{lem:presic}, and Corollary \ref{crly:orthog}, respectively. Properties 4 and 5 are consequences of Proposition \ref{prop:hype2} and Lemma \ref{lem:coker}.
\end{proof}

Thus the statement that $\Cont(M,\tau)$ is supercompactly generated is a formalization of the intuitive fact that every $(M,\tau)$-set is the union of its principal sub-$M$-sets. This enables us to extract a site presentation for $\Cont(M,\tau)$. From Theorem \ref{thm:canon} and Corollary \ref{crly:MoritaCs}, we have:
\begin{crly}
\label{crly:Morita'}
Let $\Ccal_s$ be the category of continuous principal $(M,\tau)$-sets. Then we have $\Cont(M,\tau) \simeq \Sh(\Ccal_s,J_r)$. In particular, topological monoids $(M,\tau)$ and $(M',\tau')$ are \textbf{Morita equivalent}, which is to say that $\Cont(M,\tau) \simeq \Cont(M',\tau')$, if and only if they have equivalent categories of continuous principal $M$-sets.
\end{crly}

\begin{xmpl}
\label{xmpl:infeq}
This result can be practically applied. For example, it shows that any monoid endowed with a topology for which there are infinitely many distinct isomorphism classes of continuous principal actions cannot be Morita-equivalent to any finite monoid. Of course, when the monoids involved are large enough, even the categories of principal actions can be hard to work with, so some alternative ways of generating Morita equivalences are desirable; we shall see some in subsequent sections.
\end{xmpl}

\begin{xmpl}
\label{xmpl:zero}
To present a more categorical example, recall from Definition \ref{dfn:absorb} of Chapter \ref{chap:mpatti} that a \textbf{zero element} of a monoid $M$ is an element $z \in M$ such that $mz = z = zm$ for all $m \in M$.

Let $(M,\tau)$ be a topological monoid with a zero element and $(M',\tau')$ another topological monoid. Then if $\Cont(M,\tau) \simeq \Cont(M',\tau')$, it must be that every principal $(M',\tau')$-set has a unique fixed point, since this is true in $\Setswith{M}$ and the category of principal $(M,\tau)$-sets is a full subcategory containing $1$. In particular, if $M'$ is a group and $M$ is as above, then $\Cont(M,\tau) \simeq \Cont(M',\tau')$ if and only if both $\tau$ and $\tau'$ are indiscrete topologies.
\end{xmpl}

In Section \ref{ssec:EqRel}, we shall provide an alternative presentation of the site $\Ccal_s$ of continuous principal $M$-sets in terms of right congruences.

\begin{rmk}
\label{rmk:compact}
In Chapter \ref{chap:sgt}, we also treat the broader class of \textbf{compactly generated toposes}. Without going into extraneous detail, the compact objects of $\Cont(M,\tau)$ are the \textit{finitely generated continuous $M$-sets}, and the category of these provides a larger site expression for $\Cont(M,\tau)$ and another Morita equivalence condition. We felt that there was not sufficient added theoretical value to cover this perspective in detail in this chapter, although they may eventually provide further sites for toposes of $M$-sets which can be identified with sites arising elsewhere.
\end{rmk}

A feature of hyperconnected morphisms which was not covered in Chapter \ref{chap:sgt} is that they provide a way to compute exponential objects in the codomain topos using those in the domain topos.

\begin{lemma}
\label{lem:expo}
Let $h:\Fcal \to \Ecal$ be a (hyper)connected geometric morphism and let $X$, $Y$ be objects of $\Ecal$. Then the exponential object $Y^X$ in $\Ecal$ can be computed as $h_*\left(h^*(Y)^{h^*(X)}\right)$.
\end{lemma}
\begin{proof}
We check the universal property:
\begin{align*}
& \Hom_{\Ecal}\left(Z,h_*\left(h^*(Y)^{h^*(X)}\right)\right) \\ \cong \, &
\Hom_{\Fcal}(h^*(Z),h^*(Y)^{h^*(X)}) \\ \cong \, &
\Hom_{\Ecal}(Z \times X,Y),
\end{align*}
where the latter isomorphism is obtained from full faithfulness of $h^*$.
\end{proof}

\begin{crly}
\label{crly:expo}
Let $X$, $Y$ be $(M,\tau)$-sets. Then the exponential object $Y^X$ in $\Cont(M,\tau)$ is $R\left(\Hom_{\Setswith{M}}(M \times V(X),V(Y))\right)$, which consists of the continuous elements of the exponential object $V(Y)^{V(X)}$ in $\Setswith{M}$.
\end{crly}
\begin{proof}
Applying Lemma \ref{lem:expo}, it suffices to compute $V(Y)^{V(X)}$ in $\Setswith{M}$. The underlying set is given by $\Hom_{\Setswith{M}}(M,V(Y)^{V(X)}) \cong \Hom_{\Setswith{M}}(M \times V(X),V(Y))$, by the universal property of exponentials. $M$ acts by multiplication in the first component, so that given $h:M \times V(X) \to V(Y)$, $h \cdot m$ is the mapping $(n,p) \mapsto h(mn,p)$.
\end{proof}

\subsection{The joint covering property}
\label{ssec:jcp}

One might wonder if the properties of $\Cont(M,\tau)$ identified in Corollary \ref{crly:conttop} are enough to characterize toposes of this form. For comparison, in the work of Caramello in \cite{TGT}, it is shown amongst many other results that a topos is equivalent to the topos of actions of a topological group if and only if it is an atomic, two-valued topos admitting a special surjective point\footnote{The inverse image of this point is an extension of the $J_{at}$-flat functor represented by a $\Ccal$-universal and $\Ccal$-ultrahomogeneous object $u$ in $\Ind\Ccal$; see \cite[Theorem 3.5]{TGT}.}. These conditions look a lot like the properties in Corollary \ref{crly:conttop}, except we have replaced `atomic' by `supercompactly generated' and have weakened the existence of a special point to the mere existence of \textit{enough} points.

Of course, we also know that toposes of the form $\Cont(M,\tau)$ have a canonical surjective point, obtained as the composite of the canonical point of $\Setswith{M}$ and the hyperconnected morphism $\Setswith{M} \to \Cont(M,\tau)$. Here we observe an additional property of categories of principal $M$-sets and an example of a topos having all of the properties of Corollary \ref{crly:conttop} but whose category of supercompact objects fails to have this additional property.

\begin{dfn}
\label{dfn:jcp}
We say a small category $\Ccal$ has the \textbf{joint covering property} if for any pair of objects $A,B$ of $\Ccal$ there exists an object $N$ of $\Ccal$ admitting epimorphisms to $A$ and $B$.
\end{dfn}

If $\Ccal$ is a poset, the joint covering property is equivalent to $\Ccal$ having a lower bound for any pair of elements. If $\Ccal$ has binary products, it corresponds to the property that the projection maps from any binary product should be epimorphisms. The category of non-empty sets has this property; more generally, the category of well-supported objects of a topos always has this property. In contrast, any non-trivial category with a strict initial object must fail to have the joint covering property.

\begin{lemma}
\label{lem:Mjcp}
Consider the topos $\Setswith{M}$; let $\Ccal_s$ be its subcategory of supercompact objects. Then $\Ccal_s$ has the joint covering property.
\end{lemma}
\begin{proof}
Given principal $M$-sets $N_1,N_2$ with generators $n_1,n_2$, consider the product $N_1 \times N_2$. The principal sub-$M$-set $N$ of this product generated by $(n_1,n_2)$ clearly admits the desired epimorphisms to $N_1$ and $N_2$.
\end{proof}

By applying a topological argument, we could directly extend the proof of Lemma \ref{lem:Mjcp} to the corresponding result for $\Cont(M,\tau)$. However, in the spirit of our study of hyperconnected morphisms in Chapter \ref{chap:sgt}, we once again give a more general argument for hyperconnected morphisms.

\begin{prop}
\label{prop:hypejcp}
Let $\Fcal$ be a topos and $\Ccal'_s$ its subcategory of supercompact objects. Suppose $\Ccal'_s$ has the joint covering property and $f:\Fcal \to \Ecal$ is a hyperconnected geometric morphism. Then the corresponding subcategory $\Ccal_s$ of $\Ecal$ also has the joint covering property.
\end{prop}
\begin{proof}
Since $f$ is hyperconnected, $\Ecal$ is closed in $\Fcal$ under products and subobjects, so that any joint cover in $\Ccal'_s$ of a pair of objects in $\Ccal_s$ also lies in $\Ccal_s$. Note that since the functor $\Ccal_s \to \Ccal'_s$ is full and faithful, we do not need to worry whether epimorphisms in $\Ccal'_s$ coincide with those in $\Fcal$: any epimorphism in $\Ccal'_s$ will also be one in $\Ccal_s$.
\end{proof}

\begin{crly}
\label{crly:contjcp}
The category of principal $(M,\tau)$-sets in $\Cont(M,\tau)$ has the joint covering property. 
\end{crly}

\begin{xmpl}
\label{xmpl:nonjcp}
At this point we can present an example of a two-valued, supercompactly generated topos with enough points which is not equivalent to $\Cont(M,\tau)$ for any topological monoid $(M,\tau)$. Consider the following category, $\Ccal$:
\[\begin{tikzcd}
	X \ar[r, shift left, two heads] \ar[loop left] &
	1 \ar[l, shift left, hook] \ar[r, shift left, hook] &
	Y \ar[l, shift left, two heads] \ar[loop right, "{,}"]
\end{tikzcd}\]
where identity morphisms are omitted and the outside loops are the idempotent endomorphisms whose splitting gives the terminal object. We can check directly that this is a reductive category: there are relatively few colimits that need to be checked, and since all of the (strict) epimorphisms split, they are stable and the reductive topology coincides with the trivial topology, as noted in Remark \ref{rmk:split}.

Therefore, let $\Ecal$ be the presheaf topos $\Sh(\Ccal,J_r) \simeq \Setswith{\Ccal}$, which is supercompactly generated and, being a presheaf topos, has enough points. We can compute directly that the category $\Ccal_s$ of supercompact objects of $\Ecal$ is equivalent to $\Ccal$ (meaning $\Ccal$ is an effective reductive category); since $X$ and $Y$ are well-supported, this also verifies two-valuedness. But $\Ccal$ does not have the joint covering property (there is no joint cover of $X$ and $Y$) and hence $\Ecal$ is not equivalent to a topos of the form $\Cont(M,\tau)$.

For a family of related examples, we can let $M$ and $M'$ be non-trivial monoids each having a zero element (see Example \ref{xmpl:zero} above). Then their idempotent-completions each have a terminal object; we may construct a category $\Ccal$ by gluing these idempotent completions along their respective terminal objects. The category of presheaves on this category will have the properties of Corollary \ref{crly:conttop}, but $\Ccal_s$ will not have the joint covering property (because there can be no joint covering of $M$ and $M'$). The above is the case where $M = M'$ is the two-element monoid with both elements idempotent.
\end{xmpl}

\begin{rmk}
\label{rmk:strict}
The category of supercompact objects in $\Cont(M,\tau)$ has the even more restrictive property that the covering morphisms in Definition \ref{dfn:jcp} may be chosen to be \textit{strict} epimorphisms, although as we saw in Corollary \ref{crly:epic} of Chapter \ref{chap:sgt}, this is equivalent to them being mere epimorphisms in any two-valued topos. Conversely, this `strict joint covering property' for supercompact objects actually forces two-valuedness of a supercompactly generated topos. The ordinary joint covering property does not have this implication, since the category of supercompact objects in the topos of presheaves on any meet semi-lattice has the joint covering property, and any non-trivial such topos is not two-valued.
\end{rmk} 

Even including the joint covering property to the list of properties derived previously, it is not clear at this point whether we obtain a complete characterization of toposes of the form $\Cont(M,\tau)$, since there is no canonical way of reconstructing a topological monoid given only the reductive category of principal $(M,\tau)$-sets and no additional data (such as their underlying sets). In particular, we have not yet arrived at a complete answer to the question of when a supercompactly generated, two-valued Grothendieck topos $\Ecal$ is equivalent to one of the form $\Cont(M,\tau)$. We shall return to this question in Section \ref{sec:surjpt}., but we shall see in Chapter \ref{chap:TSGT} that adding the joint covering property to the list of properties can be sufficient when $\Ccal_s$ is small enough.

\section{Monoids with topologies}
\label{sec:montop}

In this section we examine the extent to which the topology on the monoid $(M,\tau)$ can be recovered from the hyperconnected geometric morphism $\Setswith{M} \to \Cont(M,\tau)$.

\subsection{Powersets and inverse image actions}
\label{ssec:inverse}

If $M$ acts on a set $X$ on the \textit{left}, then $M$ has a corresponding \textit{right} action on its powerset $\Pcal(X)$ via the `inverse image' action, $A \mapsto g^*(A) = \{x \in X \mid gx \in A\}$; it is easily checked that $(gh)^* = h^*g^*$. Note that if $M$ is a group then $g^*$ is simply (element-wise) left multiplication by $g^{-1}$.

If $t : X \to Y$ is a homomorphism of left $M$-sets, so $t(g \cdot x) = g\cdot t(x)$ for every $x \in X$, then we can define $t^{-1} : \Pcal(Y) \to \Pcal(X)$ sending $B$ to $t^{-1}(B)$, since
\begin{align*}
g^*(t^{-1}(B)) & = \{x \in X \mid g \cdot x \in t^{-1}(B)\}\\
&= \{x \in X\mid t(g \cdot x) \in B\}\\
&= \{x \in X\mid g \cdot t(x) \in B\}\\
&= \{x \in X\mid t(x) \in g^*(B)\} = t^{-1}(g^*(B)).
\end{align*}
Thus we obtain a functor $\Pcal : [M,\Set]\op \to \Setswith{M}$, which is self-adjoint: the dual functor $\Pcal\op: \Setswith{M} \to [M,\Set]\op$ is left adjoint to $\Pcal$. This adjunction is, by construction, a lifting of the powerset adjunction on $\Set$ along the forgetful functor from $\Setswith{M}$, in the sense that the following diagram commutes:
\[\begin{tikzcd}
\Set \ar[d,bend right,"\Pcal\op"'] \ar[d, phantom, "\dashv"] & {\Setswith{M}} \ar[l,"U"'] \ar[d,bend right,"\Pcal\op"'] \ar[d, phantom, "\dashv"]\\
\Set\op \ar[u,bend right,"\Pcal"'] & {[M,\Set]}\op \ar[l,"U"'] \ar[u,bend right,"\Pcal"'].
\end{tikzcd}\]
The purpose of introducing this adjunction is to identify some special $M$-sets. First and foremost, the action of $M$ on itself by left multiplication gives a canonical right $M$-action on $\Pcal(M)$ which (even \textit{a priori}) seems a good starting point from which to recover a topology.

In Chapter \ref{chap:TDMA}, we were able to identify a representing monoid $M$ for $\Setswith{M}$ as the representing object for the forgetful functor $U$ in the diagram above. We can do something very similar here:
\begin{lemma}
\label{lem:represent}
$\Pcal(M)$ represents the composite functor $\Pcal\op \circ U : \Setswith{M} \to \Set\op$. In particular, it is uniquely determined as an object of $\Setswith{M}$ by the choice of representing monoid $M$.
\end{lemma}
\begin{proof}
Passing around the square and applying Yoneda, we obtain natural isomorphisms:
\[\Hom_{\Setswith{M}}(X,\Pcal(M)) \cong \Hom_{[M,\Set]}(M,\Pcal\op(X)) \cong \Pcal\op(U(X)),\]
as required. 
\end{proof}
We can in fact deduce that this composite functor is comonadic, so that $\Setswith{M}$ is comonadic over $\Set\op$, but since the existing tools for comparing toposes with cotoposes (beyond those used to show the existence of colimits in toposes) are not well-developed to the author's knowledge, we shall take a different route to derive further properties of $\Pcal(M)$.

Note that the two-element set $2$ represents $\Pcal: \Set \to \Set\op$. By passing through the available adjunctions, we find that for every right $M$-set $X$,
\begin{align*}
\Hom_{\Setswith{M}}(X,\Pcal(M)) & \cong \Pcal\op(U(X))\\
&\cong \Hom_{\Set}(U(X),2)\\
&\cong \Hom_{\Setswith{M}}(X,\Hom_{\Set}(M,2)).
\end{align*}
That is, $\Pcal(M) \cong \Hom_{\Set}(M,2)$ as right $M$-sets, which is clear at the level of underlying sets, but the fact that the actions coincide was not apparent \textit{a priori}.

\begin{rmk}
Localic geometric morphisms over a topos $\Ecal$ correspond to internal locales in $\Ecal$, by \cite[Lemma 1.2]{factorizationI}, say. The correspondence sends a morphism $f: \Fcal \to \Ecal$ to the internal locale $f_*(\Omega_{\Fcal})$, where $\Omega_{\Fcal}$ is the subobject classifier of $\Fcal$.

Recalling that $2$ is the subobject classifier for $\Set$, we have just shown that $\Pcal(M)$ is (the underlying object of) the internal locale corresponding to the canonical point of $\Setswith{M}$; this provides another way to deduce the second statement in Lemma \ref{lem:represent}, and endows $\Pcal(M)$ with a canonical order relation (which coincides with the usual inclusion ordering).
\end{rmk}

\begin{lemma}
\label{lem:Pcal}
$\Pcal(M)$ is an internal Boolean algebra in $\Setswith{M}$. In particular, it has a distinguished non-trivial automorphism, complementation. There are exactly two morphisms $1 \to \Pcal(M)$. Also, $\Pcal(M)$ has the subobject classifier $\Omega$ of $\Setswith{M}$ as a subobject. Finally, $\Pcal(M)$ is a coseparator. 
\end{lemma}
\begin{proof}
The structure of a Boolean algebra involves only finite limits, so Boolean algebras are preserved by both direct and inverse image functors; thus $\Pcal(M)$ inherits the Boolean algebra structure from $2$. The two morphisms $1 \to \Pcal(M)$ correspond to the empty set and the whole of $M$; these are the only two since composing the canonical point with the global sections morphism must give the identity geometric morphism on $\Set$, which means $\Gamma(\Pcal(M)) = \Hom_{\Setswith{M}}(1,\Pcal(M)) \cong 2$.

The usual argument showing that the category of coalgebras for a left exact comonad is a topos (see \cite[\S V.8]{MLM}) exhibits the subobject classifier as an equalizer of two endomorphisms of the free coalgebra on the subobject classifier; this free algebra is precisely $\Pcal(M)$. More specifically, the endomorphisms are the identity and the morphism sending a subset $A$ to those $m \in M$ for which $m^*(A) = M$. From these expressions we recover the fact that $\Omega \hookrightarrow \Pcal(M)$ is the collection of right ideals of $M$. Since the subobject classifier of a topos is always injective, we in fact can conclude that $\Omega$ is a retract of $\Pcal(M)$.

Finally, the functor $\Pcal \circ U$ is a composite of faithful functors so it is faithful, meaning its representing object must be a coseparator.
\end{proof}

Having established some key properties of $\Pcal(M)$ as an object of $\Setswith{M}$, we examine how the necessary clopens from Lemma \ref{lem:Inx} behave as elements of $\Pcal(M)$.

\begin{lemma}
\label{lem:InA}
Given $A \in \Pcal(M)$, $p \in M$, we have $\Ical_{A}^{p} = \Ical_{M \backslash A}^{p}$; moreover,
\[\Ical_{A}^{p} \subseteq
\begin{cases}
A & \text{if } p \in A\\
M \backslash A & \text{if } p \notin A.
\end{cases}\]
\end{lemma}
\begin{proof}
By definition, $\Ical_{A}^{p} = \{m \in M \mid m^*(A) = p^*(A)\}$. Since inverse images respect complementation, we have $m^*(A) = p^*(A)$ if and only if $m^*(M \backslash A) = p^*(M \backslash A)$, and hence $\Ical_{A}^{p} = \Ical_{M \backslash A}^{p}$, as claimed.

Now, without loss of generality, suppose $p \in A$, else we may exchange $A$ and $M \backslash A$. Then $1 \in p^*(A)$. Given $m \in \Ical_{A}^{p}$, it follows that $1 \in m^*(A)$ which forces $m \in A$. Thus $\Ical_{A}^{p} \subseteq A$.
\end{proof}

\begin{lemma}
\label{lem:InA2}
Suppose $X$ is any $M$-set, $x \in X$ and $p \in M$. Let $A = \Ical_{x}^{p} \in \Pcal(M)$. Then for any $p' \in A$, the inclusion in Lemma \ref{lem:InA} holds with equality: $\Ical_{A}^{p'} = A$.
\end{lemma}
\begin{proof}
Suppose $m \in A$ so that $xm = xp = xp'$. Then $m^*(A) = \{m' \in M \mid xmm' = xp\} = \{m' \in M \mid xp'm' = xp\} = p'{}^*(A)$, so $m \in \Ical_{A}^{p'}$. This proves the reverse inclusion to that in Lemma \ref{lem:InA}.
\end{proof}

Note that the complement of $A$ in Lemma \ref{lem:InA2} may split into multiple sets of the form $\Ical_{A}^{p}$ for $p \notin A$, but we at least retain that $\Ical_{x}^{p} \subseteq \Ical_{A}^{p}$ for each $p$.

\subsection{Action topologies}
\label{ssec:action}

We have by now developed sufficient tools to reconstruct a topology from the hyperconnected morphism $\Setswith{M} \to \Cont(M,\tau)$.

\begin{thm}
\label{thm:tau}
Suppose $M$ is a monoid equipped with a topology $\tau$, and $V,R$ are as in Proposition \ref{prop:hyper}. Consider $\Pcal(M)$ equipped with the inverse image action. Then the underlying set of
\[T := VR(\Pcal(M)) = \{A \subseteq M \mid \forall p, q \in M, \, \Ical_{q^*(A)}^{p} \in \tau \}\]
is a base of clopen sets for a topology $\tilde{\tau} \subseteq \tau$ such that $\Cont(M,\tilde{\tau}) = \Cont(M,\tau)$ as sub-categories of $\Setswith{M}$. Moreover, $\tilde{\tau}$ is the coarsest topology on $M$ with this property.
\end{thm}
\begin{proof}
We extracted the expression for $T$ from the construction of $R$ in Proposition \ref{prop:hyper}. By Lemma \ref{lem:InA}, every $A \subseteq M$ is a union over its elements $t$ of the sets $\Ical_{A}^{t}$, so if $A \in T$ then $A$ is necessarily open. Similarly, Lemma \ref{lem:InA} guarantees that $M \backslash A \in T$ whenever $A \in T$, since $\Ical_{M \backslash A}^{p} = \Ical_{A}^{p}$ and $q^*(M \backslash A) = M \backslash q^*(A)$. It follows that each $A \in T$ is clopen with respect to $\tau$.

To show that $T$ is a base for a topology it suffices to show that $A \cap B$ is in $T$ whenever $A$ and $B$ are. Directly,
\begin{equation}
\label{eq:intersection}
\Ical_{q^*(A \cap B)}^{p} = \{m \in M \mid (qm)^*(A) \cap (qm)^*(B) = (qp)^*(A) \cap (qp)^*(B)\};
\end{equation}
if $p'$ is any element of this set, then by inspection $\Ical_{q^*(A)}^{p'} \cap \Ical_{q^*(B)}^{p'} \subseteq \Ical_{q^*(A \cap B)}^{p'} = \Ical_{q^*(A \cap B)}^{p}$ is an open neighbourhood of $p'$ contained in it, ensuring that the latter is open. We conclude $A \cap B \in T$, as required.

If $X$ is an $M$-set which is continuous with respect to the generated topology $\tilde{\tau}$, then $\Ical_{x}^{p} \in \tilde{\tau} \subseteq \tau$ for every $x \in X$, $p \in M$ so $X$ is continuous with respect to $\tau$.

Conversely, if $X$ is continuous with respect to $\tau$, so $\Ical_{x}^{p} \in \tau$ for all $x \in X$ and $p \in M$, we want to show that each $\Ical_{x}^{p} \in \tilde{\tau}$. Writing $A = \Ical_{x}^{p}$, this is equivalent to showing that $\Ical_{q^*(A)}^{p} \in \tau$ for each $p, q \in M$. Given $p_1 \in \Ical_{q^*(A)}^{p}$, consider the open set $\Ical_{xq}^{p_1}$. We have $p_2 \in \Ical_{xq}^{p_1}$ if and only if $xqp_2 = xqp_1$. Consequently,
\[p_2^*q^*(A) = \{m \in M \mid xqp_2m = xp\} = \{m \in M \mid xqp_1m = xp\} = p_1^*q^*(A).\]
But $\Ical_{q^*(A)}^{p_1}$ is precisely $\{m \in M \mid m^*q^*(A) = p_1^*q^*(A)\}$ so we conclude $\Ical_{xq}^{p_1} \subseteq \Ical_{q^*(A)}^{p_1} = \Ical_{q^*(A)}^{p}$, and hence the latter is open as required.

Finally, to show that $\tilde{\tau}$ is the coarsest such topology, suppose $\tau'$ is some topology on $M$ such that any $M$-set $X$ is continuous with respect to $\tau'$ if and only if it is continuous with respect to $\tau$. Then the respective inclusions of $\Cont(M,\tau)$ and $\Cont(M,\tau')$ into $\Setswith{M}$ are isomorphic. Thus $T$ is computed in the same way with respect to either topology, and by repeating the above argument, we have $\tilde{\tau} \subseteq \tau'$, as claimed.
\end{proof}

\begin{rmk}
Note that the caveat `as subcategories of $\Setswith{M}$' in Theorem \ref{thm:tau} likely cannot be removed in full generality, since a sufficiently large monoid could admit two topologies with distinct categories of continuous $M$-sets which happen to be equivalent. We have not constructed such an example, since in this chapter we are primarily interested in examining $\Cont(M,\tau)$ as a topos under $\Setswith{M}$.
\end{rmk}

\begin{dfn}
\label{dfn:action}
The topology $\tilde{\tau}$ derived in Theorem \ref{thm:tau} will be called the (right) \textbf{action topology induced by $\tau$}. By the final statement of Theorem \ref{thm:tau}, the construction of $\tilde{\tau}$ is idempotent (see Lemma \ref{lem:G3} below for a deeper exploration of this). As such, we say $\tau$ is an \textbf{action topology} if $\tilde{\tau} = \tau$.
\end{dfn}

We will continue to employ the notation $T:=VR(\Pcal(M))$ for the Boolean algebra of necessary clopens when the topology $\tau$ is understood. Rather than considering the full action topology $\tilde{\tau}$, it will sometimes be more convenient to work directly with $T$, since this is an object residing in the toposes we are studying.

\begin{schl}
\label{schl:Tseparator}
Considering the Boolean algebra $T$ as an object of $\Cont(M,\tau)$, it inherits all of the properties we observed in $\Pcal(M)$ in Lemmas \ref{lem:represent} and \ref{lem:Pcal}: it represents $\Pcal\op \circ U \circ V: \Cont(M,\tau) \to \Set\op$, is a complete internal Boolean algebra with exactly two global sections, and is a coseparator which contains the subobject classifier of $\Cont(M,\tau)$ as an (order-inheriting) subobject. Explicitly, the subobject classifier of $\Cont(M,\tau)$ consists of the left ideals of $M$ lying in $T$.
\end{schl}
\begin{proof}
For the first part, we extend the proof of Lemma \ref{lem:represent} with the observation that
\[\Hom_{\Cont(M,\tau)}(X,R(\Pcal(M))) \cong \Hom_{\Setswith{M}}(V(X),\Pcal(M)),\]
where $R(\Pcal(M))$ is $T$ viewed as an object of $\Cont(M,\tau)$.

For the second part, all of the arguments in the proof of Lemma \ref{lem:Pcal} carry over with $T$ in place of $\Pcal(M)$.
\end{proof}

\begin{schl}
\label{schl:Tsuff}
Let $X$ be a right $M$-set continuous with respect to $\tau$, and let $T$ be as in Theorem \ref{thm:tau}. Then for every $x \in X$, $p \in M$, we have $\Ical_{x}^{p} \in T$. In particular, we do not need to generate $\tilde{\tau}$ in order to verify continuity.
\end{schl}
\begin{proof}
Consider the construction in the proof of Theorem \ref{thm:tau}; in it, we showed that $\Ical_{q^*(\Ical_{x}^{p})}^{p'} \in \tau$ for each $x \in X$ and $p, p', q \in M$. But this is exactly the condition needed for $\Ical_{x}^{p}$ to be in $T$, since it ensures that the action of $M$ on $\Pcal(M)$ is continuous on the sub-$M$-set generated by $\Ical_{x}^{p}$.
\end{proof}

\begin{schl}
\label{schl:base}
The clopen sets of the form $\Ical_A^p$ for $A \in T$ (or more generally, the necessary clopens of all $(M,\tau)$-sets) also form a base for $\tilde{\tau}$.
\end{schl}
\begin{proof}
Given $a \in A$, we have $a \in \Ical_A^a \subseteq A$, so each member of $T$ is a union of members of the given form, as required.
\end{proof}

\subsection{Powder monoids}
\label{ssec:powder}

Action topologies have much more convenient properties than arbitrary topologies. Most notably:
\begin{prop}
\label{prop:ctsx}
The multiplication on $M$ is continuous with respect to $\tilde{\tau}$ for any starting topology $\tau$.
\end{prop}
\begin{proof}
Given $A \in \tilde{\tau}$ and a pair $(a,b) \in \mu^{-1}(A)$, we have $a \in \Ical_{A}^{a}$ and $b \in a^*(A)$ by inspection. Since the inverse image action commutes with arbitrary unions and the generating set $T$ of $\tilde{\tau}$ is closed under the action of $M$ on $\Pcal(M)$, we deduce that $a^*(A) \in \tilde{\tau}$. Given any $m \in \Ical_{A}^{a}$, $n \in a^*(A)$, we have $m^*(A) = a^*(A)$ and hence $n \in m^*(A)$ and $mn \in A$. Thus $\Ical_{A}^{a} \times a^*(A) \subseteq \mu^{-1}(A)$. It follows that $\mu^{-1}(A) \in \tilde{\tau} \times \tilde{\tau}$, as required.
\end{proof}

Thus, almost miraculously, $(M,\tilde{\tau})$ is a topological monoid. That is, from the perspective of continuous actions on discrete sets, there is no loss in generality in assuming that the topology on the monoid makes its multiplication continuous, which shall come as a relief to the modern algebraist.

Passing to the action topology also sheds the extraneous local richness of the original topology which the discrete sets being acted on are oblivious of.
\begin{lemma}
\label{lem:connected}
Let $(M,\tau)$ be a locally connected topological monoid. Then $\tilde{\tau}$ is generated by the connected components with respect to $\tau$. In particular, if $(M,\tau)$ is connected, $(M,\tilde{\tau})$ is indiscrete.
\end{lemma}
\begin{proof}
If $\tau$ makes $M$ locally connected, the connected components of $M$ are clopen. For each $x \in M$ let $C_x$ be the connected component containing $x$. We claim each $C_x$ is a member of $\tilde{\tau}$. Indeed, given $p \in M$, $p^*(C_x)$ is clopen (since $(M,\tau)$ \textit{is} a topological monoid here), and hence is some union of connected components. Observe that $\Ical_{C_x}^{p} = \{m \in M \mid m^*(C_x) = p^*(C_x)\}$ contains $C_p$: multiplication on the right is continuous, so whenever $py \in C_x$, we have $my \in C_x$ for every $m$ in $C_p$. It follows that $\Ical_{C_x}^{p}$ is a union over $q \in \Ical_{C_x}^{p}$ of components $C_q$, and so is open. Thus $C_x \in \tilde{\tau}$ and since these are the minimal clopen sets we are done.

If $(M,\tau)$ is connected, the only clopen subsets of $M$ are $\emptyset$ and $M$, so $\tilde{\tau}$ contains only these.
\end{proof}

Lemma \ref{lem:connected} means that, for example, $\Rbb$ with its usual topology goes from being Hausdorff (or even stronger, normal) to being indiscrete upon passing to $\tilde{\tau}$. On the other hand, other properties of a topology $\tau$ are preserved by passing to $\tilde{\tau}$. For example:
\begin{lemma}
\label{lem:compactau}
Suppose $\tau$ is a compact topology on a monoid $M$. Then $\tilde{\tau}$, being a coarser topology than $\tau$, is compact too.
\end{lemma}

By definition, action topologies are \textbf{zero-dimensional}, since they have a base of clopen sets. See \cite[Section I.4, Figure 9]{Counter} for a helpful diagram of how this property interacts with some basic separation properties. We list some of them here:
\begin{lemma}
\label{lem:T0T2}
Suppose $\tau = \tilde{\tau}$ is an action topology and $(M,\tau)$ is Kolmogorov (satisfies the $T_0$ separation axiom). Then $(M,\tau)$ has the properties (listed in order of decreasing strength) of being totally separated and regular, totally disconnected and Urysohn, and Hausdorff ($T_2$).
\end{lemma}

\begin{dfn}
\label{dfn:powder}
A topological monoid $(M,\tau)$ which is $T_0$ and such that $\tau$ is a right action topology (that is, such that $\tau$ has a basis of clopen sets $U$ such that $\Ical_U^p = \{q \mid q^*(U) = p^*(U)\} \in \tau$ for every $p$) shall be called a (right) \textbf{powder monoid}; the name is motivated the separation properties exhibited in Lemma \ref{lem:T0T2}. We have avoided the name `action monoid' since it conflicts with the terminology `monoid actions'.
\end{dfn}

Note that there is implicit asymmetry in Definition \ref{dfn:powder}, and indeed we may define a \textit{left} powder monoid to be a topological monoid $(M,\tau)$ which is Hausdorff and such that $\tau$ is a right action topology on $M\op$. We shall discuss these in more detail in Section \ref{ssec:monads}. 
For the time being, we write `powder monoid' to mean `right powder monoid' unless otherwise stated.

\begin{rmk}
The properties of Lemma \ref{lem:T0T2} do not characterize powder monoids. For example, $\Qbb$ with its usual topology is a $T_0$ and zero-dimensional topological group, but we find that, just like $\Rbb$, its corresponding action topology is trivial.
\end{rmk}

\begin{thm}
\label{thm:Hausd}
Given a monoid with an arbitrary topology $(M,\tau)$, there is a canonical (right) powder monoid, which we shall by an abuse of notation denote by $(\tilde{M},\tilde{\tau})$, such that $\Cont(M,\tau) \simeq \Cont(\tilde{M},\tilde{\tau})$ and the canonical points of these toposes coincide. 
\end{thm}
\begin{proof}
We first construct the action topology $\tilde{\tau}$ corresponding to $\tau$ from Theorem \ref{thm:tau}. We shall show that the equivalence relation on $M$ relating points which are topologically indistinguishable with respect to $\tilde{\tau}$ is a two-sided congruence on $M$.

Suppose $m_1,m'_1$ and $m_2,m'_2$ are two pairs of topologically indistinguishable points in the sense that every open set of $\tilde{\tau}$ containing $m_i$ also contains $m'_i$ and vice versa. Then given an open set $U \in \tilde{\tau}$ containing $m_1m_2$, we have $m_2 \in m_1^*(U)$ and hence $m'_2 \in m_1^*(U)$. Moreover, $m_1 \in \Ical_U^{m_1}$ so $m'_1 \in \Ical_U^{m_1}$, which is to say that $m_1^*(U) = {m'_1}^*(U)$ and so $m'_1m'_2 \in U$.

Moreover, the actions of any pair of topologically indistinguishable points $m,m'$ of $(M,\tilde{\tau})$ on any $(M,\tilde{\tau})$-set are forced to be equal: if we had $xm \neq xm'$ we would have $\Ical_x^m$ containing $m$ but not $m'$ and therefore not open. Thus the continuous actions restrict to the quotient $\tilde{M}$ of $M$ by this relation.

The monoid $\tilde{M}$ inherits its topology from $(M,\tilde{\tau})$; we abuse notation and call the inherited topology $\tilde{\tau}$ too, since the frames of open sets of the two topologies are isomorphic. It is easily checked that $\tilde{\tau}$ is still an action topology on $\tilde{M}$, and $(\tilde{M},\tilde{\tau})$ is Kolmogorov by construction, as required.

Since we have not modified the forgetful functor (the underlying sets of the actions remain the same), this construction depends only on $M$ and the canonical point of $\Cont(M,\tau)$.
\end{proof}

While convenient, it is unavoidable that the construction of Theorem \ref{thm:Hausd} relies on our original representing monoid $(M,\tau)$. In Section \ref{ssec:complete}, we shall construct a representing powder monoid for a topos of the form $\Cont(M,\tau)$, in general different from the one constructed above, which depends only on the canonical point.

\subsection{Prodiscrete Monoids, Nearly Discrete Groups}
\label{ssec:prodiscrete}

There are plenty of nontrivial examples of powder monoids.

\begin{dfn}
Recall that a \textbf{prodiscrete monoid} is a topological monoid $(M,\tau)$ obtained as a (projective) limit of discrete monoids,
\[M = \varprojlim_{i\in I} M_i,\]
with $\tau$ the coarsest topology making each projection map continuous, which has a base of opens of the form $\pi_i^{-1}(\{m_i\})$ with $m_i \in M_i$. Often the limit is taken to be filtered or such that all of the monoid homomorphisms involved are surjections, but we do not require these restrictions.
\end{dfn}

\begin{xmpl}
\label{xmpl:idemclosed}
For those readers unfamiliar with prodiscrete monoids, we construct an example now which will be useful later. Consider the `truncated addition' monoids $N_{a,1}$, indexed by integers, $a \geq 0$ consisting of the integers $\{0,\dotsc,a\}$ equipped with the operation $\mu(p,q) = \min\{p+q,a\}$. For each $a \leq a'$ we have a surjective monoid homomorphism $N_{a',1} \too N_{a,1}$. The projective limit of the resulting sequence of monoids can be identified with $\Nbb \cup \{\infty\}$, equipped with addition extended in the obvious way, and the topology on it coincides with the one-point (Alexandrov) compactification of $\Nbb$ as a discrete topological space.
\end{xmpl}

\begin{prop}
\label{prop:prodisc}
Any prodiscrete monoid is a (right) powder monoid.
\end{prop}
\begin{proof}
It suffices to show that open sets of the form $U = \pi_i^{-1}(A)$ with $A \subseteq M_i$ are continuous elements of the topology. Indeed, given $\alpha = (a_i)_{i \in I} \in M$,
\begin{align*}
\Ical_U^{\alpha} &=
\{\beta = (b_i)_{i \in I} \in M \mid \alpha^*(U) = \beta^*(U)\}\\
&= \{\beta \in M \mid \forall c \in M_i, \, a_ic \in A \Leftrightarrow b_ic \in A \}\\
&\supseteq \{\beta \in M \mid b_i = a_i\} = \pi_i^{-1}(\{a_i\})
\end{align*}
contains an open neighbourhood of $\alpha$ and hence is open in the prodiscrete topology $\tau$, and by a similar argument, for any $\alpha' = (a'_i)_{i \in I}$,
\[\alpha'{}^*(U) = \pi_i^{-1}({a'_i}^*(A))\]
is of the same form, so is open in $\tau$ as required.
\end{proof}

Given the motivation of the present work, it is natural to wonder what happens when we apply the construction of Theorem \ref{thm:Hausd} to groups.

\begin{dfn}
A topological group is said to be \textbf{nearly discrete} if the intersection of all open subgroups contains only the identity element; see Johnstone \cite[Example A2.1.6]{Ele} or Caramello \cite[comments following Proposition 2.4]{TGT}.
\end{dfn}

\begin{lemma}
\label{lem:powdergrp}
A topological group is a (right) powder monoid if and only if it is nearly discrete and has a neighbourhood base of open subgroups at the identity; we accordingly call such groups \textbf{powder groups}.
\end{lemma}
\begin{proof}
Given such a topological group $(G,\tau)$, $g \in G$ and a neighbourhood $U \in \tau$ of $g$, since multiplication by any element of $G$ on either side preserves opens, we may without loss of generality suppose that $g = 1 \in U$, and so that $U$ is an open subgroup.

Then $\Ical_U^1$, being the set of $h \in G$ such that $h^{-1}U = U$, contains (and so is equal to) $U$ and in particular is open in $G$, ensuring that $U$ is in the action topology corresponding to $\tau$, whence $\tau$ is an action topology and hence $(G,\tau)$, being Hausdorff (since given any element distinct from the identity we can find an open subgroup which does not contain it), is a powder monoid.

Conversely, a powder group has a basis of the identity consisting of the isotropy subgroups $\Ical_U^1$ with $U$ varying over the open neighbourhoods of the identity, and being Hausdorff forces such a group to be nearly discrete.
\end{proof}

\begin{xmpl}
\label{xmpl:notpro}
Lemma \ref{lem:powdergrp} allows us to present an example of a powder monoid which is not a prodiscrete monoid. Consider the group of automorphisms of $\Nbb$ with the stabilizers of finite subsets defined to be open subgroups (as suggested in \cite[Example A2.1.6]{Ele}). Any prodiscrete group is the limit of its quotients by normal open subgroups, but this group has no proper open normal subgroups. Incidentally, this nearly discrete group is one of the many powder monoids representing the Schanuel topos which we encountered in Example \ref{xmpl:nonpresh} of Chapter \ref{chap:sgt}, and which we shall discuss further in Example \ref{xmpl:notpro2} below.
\end{xmpl}

Of course, the class of prodiscrete groups forms a subclass of the class of powder groups. An even more refined class is the following:
\begin{dfn}
A group is \textbf{profinite} if it is (expressible as) a directed projective limit of finite groups. They can alternatively be characterized as compact, totally disconnected groups; see \cite[Proposition 1.1.7]{Profinite}.
\end{dfn}

\begin{schl}
\label{schl:Hausgrp}
If $(G,\tau)$ is a group with an arbitrary topology, then the corresponding powder monoid $(\tilde{G},\tilde{\tau})$ is a powder group. If $(G,\tau)$ is compact, then $(\tilde{G},\tilde{\tau})$ is profinite.
\end{schl}
\begin{proof}
The proof that the equivalence relation identifying $\tilde{\tau}$-indistinguishable points respects multiplication also demonstrates that the equivalence class containing the inverse $g^{-1}$ of an element $g$ is an inverse for the equivalence class of $g$ in the quotient $\tilde{M}$. Thus $(\tilde{G},\tilde{\tau})$ is a group, and hence a powder group, as claimed. Adding Lemma \ref{lem:compactau} gives the profinite result.
\end{proof}

Thus applying Theorem \ref{thm:Hausd} to any topological group brings us naturally to the class of groups described by Johnstone in \cite[Example A2.1.6]{Ele} as canonical representatives for toposes of topological group actions, which are precisely the powder groups. This is not the end of the story, however: see Remark \ref{rmk:completegrp}.

\section{The Canonical Surjective Point}
\label{sec:surjpt}

In Section \ref{sec:properties}, we made extensive use of the hyperconnected morphism $f: \Setswith{M} \to \Cont(M,\tau)$ which we constructed to demonstrate that $\Cont(M,\tau)$ is a topos. In this section, we show that the existence of a hyperconnected morphism with domain $\Setswith{M}$ entirely characterizes toposes of this form.

From Chapter \ref{chap:TDMA}, we know that the (discrete) monoids presenting $\Setswith{M}$ correspond to the surjective essential points of that topos. There may be multiple possible presentations, but there is a unique one for each point. Without loss of generality, we fix a presentation of $\Setswith{M}$ and its corresponding canonical point (whose inverse image functor is the forgetful functor), and examine equivalent presentations in the next section.

Suppose we are given a topos $\Ecal$ having a point $p$ which factorizes through the canonical point of $\Setswith{M}$ via a hyperconnected morphism $h$:
\begin{equation}
\label{eqn:p}
\begin{tikzcd}
\Set \ar[r, bend left=50, "- \times M"] \ar[r, bend right=50, "\Hom_{\Set}(M{,}-)"']
\ar[r, phantom, shift left=6, "\bot", near end] \ar[r, phantom, shift right=4, "\bot", near end]
& {\Setswith{M}} \ar[l, "U"', near start] \ar[r, shift right = 2, "h_*"'] \ar[r, phantom, "\bot"]
& \Ecal, 
\ar[l, shift right = 2, "h^*"']
\end{tikzcd}
\end{equation}
where $U$ is the usual forgetful functor. Applying Theorem \ref{thm:hype2} of Chapter \ref{chap:sgt} once again, it follows that $\Ecal$ is a supercompactly generated, two-valued Grothendieck topos whose category of supercompact objects has the joint covering property. The given point of $\Ecal$ is surjective, being a composite of surjective morphisms, but also localic, since we can construct arbitrary coproducts of the terminal object $Uh^*(1)$ to cover any set. 

In order to have a complete picture of what $\Ecal$ can look like, our first task is to classify the hyperconnected morphisms under $\Setswith{M}$. For this task, we re-express the canonical site of principal $M$-sets in terms of relations on the monoid.

\subsection{Equivariant relations and congruences}
\label{ssec:EqRel}

We can use the powerset adjunction of Section \ref{ssec:inverse} to construct another canonical object of $\Setswith{M}$. Since $M \times M$ is naturally equipped with a left $M$-action, we obtain an inverse image action on $\Pcal(M \times M)$, the set of all relations on $M$ \textit{viewed as a set}. However, this contains more relations than we need!

When we consider $M \times M$ as an object of $\Setswith{M}$, its subobjects (which correspond to the relations on $M$ \textit{as an $M$-set}) are its sub-right-$M$-sets. We call such subobjects (right) \textbf{equivariant relations} on $M$, because they are the relations $r$ with the property that $(p,q) \in r$ implies $(pm,qm) \in r$ for every $m \in M$.\footnote{We have chosen to consistently use a lower case $r$ for relations to avoid a clash of notation with the right adjoint functor $R$ constructed earlier.}

\begin{lemma}
\label{lem:relations}
The sets of equivariant relations, reflexive relations, symmetric relations and transitive relations are sub-$M$-sets of $\Pcal(M\times M)$ with the inverse image action, each inheriting the ordering from $\Pcal(M \times M)$. Thus their intersection, the collection $\Rfrak$ of \textbf{right congruences}, is an ordered sub-$M$-set of $\Pcal(M \times M)$.
\end{lemma}
\begin{proof}
Equivariance of a relation is unaffected by composition on the left; that is, if $r$ is an equivariant relation, then so is $k^*(r)$, so these form a sub-$M$-set.

The diagonal relation $\Delta : M \hookrightarrow M \times M$ clearly satisfies $k^*(\Delta) \supseteq \Delta$ for every $k \in M$, and the inverse image action preserves containment, so reflexive relations form a sub-$M$-set. 

Given a symmetric relation $r \hookrightarrow M \times M$ and $(m,m') \in k^*(r)$, we have $(km, km') \in r$ and hence $(km', km) \in r$ and $(m',m) \in k^*(r)$, so symmetric relations form a sub-$M$-set.

By a similar argument, if $(m,m'), (m',m'') \in k^*(r)$ then $(m,m'') \in k^*(r)$, so transitive relations form a sub-$M$-set.
\end{proof}

\begin{rmk}
Let $\Omega$ be the subobject classifier of $\Setswith{M}$. The \textbf{internal power-object} of $M$, the exponential object $\Omega^M$, has an underlying set which coincides with the collection of equivariant relations. Indeed, the elements of $\Omega^M$ correspond to elements of
\[\Hom_{\Setswith{M}}(M,\Omega^M) \cong \Hom_{\Setswith{M}}(M \times M,\Omega) \cong \Sub_{\Setswith{M}}(M\times M).\]
However, the action of $M$ on $\Omega^M$ is by inverse images \textit{only in the first component}, so $\Omega^M$ does not coincide with the first sub-$M$-set of $\Pcal(M \times M)$ described in Lemma \ref{lem:relations}. In fact, for $M$ a non-trivial monoid, the subsets of $\Omega^M$ on the reflexive, symmetric or transitive relations are typically not even sub-$M$-sets.
\end{rmk}

Now observe that the relations $\rfrak_x$ of Corollary \ref{crly:Rnx} are always right congruences, so we can examine their behaviour as elements of $\Rfrak$, just as we considered the behaviour of necessary clopens as elements of $\Pcal(M)$ in the last section.

\begin{lemma}
\label{lem:rinrfrak}
Let $r \in \Rfrak$ and $p \in M$. Then, as elements of $\Rfrak$ with the restriction of the action and ordering from $\Pcal(M \times M)$, we have $r \subseteq \rfrak_r$ and $p^*(\rfrak_r) = \rfrak_{p^*(r)}$. That is, $\rfrak_{(-)}$ is an order-increasing $M$-set endomorphism of $\Rfrak$.
\end{lemma}
\begin{proof}
We must show that $(x,y) \in r$ implies that $x^*(r) = y^*(r)$. Indeed, since $r$ is equivariant, for any $(p,q) \in M \times M$, $(xp,yp)$ and $(xq,yq)$ are in $r$, whence $(xp,xq) \in r$ if and only if $(yp,yq) \in r$, which gives the desired equality. Both $p^*(\rfrak_r)$ and $\rfrak_{p^*(r)}$ are equal to the set $\{(x,y) \in M \times M \mid x^*(p^*(r)) = y^*(p^*(r))\}$.
\end{proof}

\begin{rmk}
\label{rmk:order}
While $\rfrak$ may be order-increasing, there is no reason for it to be order-preserving: $r \subseteq s$ does not imply $\rfrak_r \subseteq \rfrak_s$ in general, since given $(m,n)$ such that $(mu,mv) \in r$ if and only if $(nu,nv) \in r$, we could still have $(mu,mv) \in s$ with $(nu,nv) \notin s$ if the former pair of elements are related in $s$ but not $r$. However, by expanding the definitions we at least find that $\rfrak_{r \cap s} \supseteq \rfrak_r \cap \rfrak_s$, in analogy with \eqref{eq:intersection}.
\end{rmk}

\begin{rmk}
\label{rmk:tauhat}
Suppose $M$ is a monoid equipped with a topology $\tau$, and $V,R$ are as in Proposition \ref{prop:hyper}. Consider $\Rfrak$ equipped with the inverse image action. In light of Theorem \ref{thm:tau}, we might consider the continuous right congruences, which is to say those lying in $\Tcal:= VR(\Rfrak)$, and construct a topology $\hat{\tau}$ on $M$ by taking the equivalence classes with respect to congruences lying in $\Tcal$ as a base of clopen sets. We might then wonder if $\Cont(M,\tau) \simeq \Cont(M,\hat{\tau})$; this turns out not to be the case.

It follows from Remark \ref{rmk:order} that the base of $\hat{\tau}$ is closed under intersections. We can also show that $\tilde{\tau} \subseteq \hat{\tau}$. Let $A \in T$ be a clopen set in $\tilde{\tau}$ and consider the congruence $\rfrak_A$, which we know to be open in $\tau \times \tau$. Then the inclusion $\rfrak_A \subseteq \rfrak_{\rfrak_A}$ from Lemma \ref{lem:rinrfrak} ensures that the latter is also in $\tau \times \tau$, and similarly stability under the inverse image action is guaranteed, so the equivalence classes $\Ical_A^p$ (and by extension $A$) are open in $\hat{\tau}$.

However, there is no reason that the opposite inclusion should hold, or even that $\hat{\tau}$ should be contained in $\tau$, since the congruence classes of $r$ need not be in $\tau$ when those of $\rfrak_r$ are. Indeed, consider a monoid $M$ with two distinct right-absorbing elements $x,y$, so that $xm = x$ and $ym = y$ for all $m \in M$. Let $\tau$ be the topology on $M$ generated by asserting that every singleton except $\{x\}$ and $\{y\}$ is open, and also $\{x,y\}$ is open. Then the diagonal relation $\Delta: M \to M \times M$ has $\rfrak_{\Delta} = \{(p,q) \mid p^*(\Delta) = q^*(\Delta)\}$, which is open in $\tau \times \tau$ since $x^*(\Delta) = M \times M = y^*(\Delta)$, so $(x,y)$ and $(y,x) \in \rfrak_{\Delta}$. As such, $\hat{\tau}$ is the discrete topology in this case. In particular, while continuity with respect to $\tau$ requires $x$ and $y$ to act identically, continuity with respect to $\hat{\tau}$ does not, so $\Cont(M,\tau) \not\simeq \Cont(M,\hat{\tau})$.

Moreover, the construction of $\hat{\tau}$ is not obviously idempotent, and we have not even been able to prove that $(M,\hat{\tau})$ is a topological monoid (recall Proposition \ref{prop:ctsx}), let alone that $\hat{\tau}$ is an action topology.
\end{rmk}

Instead, the reason we introduced the object of right congruences was to provide a canonical indexing of principal $M$-sets.

\begin{lemma}
\label{lem:prel}
The quotients of $M$ in $\Setswith{M}$ correspond precisely to the right congruences on $M$ (which are internal equivalence relations on $M$ in this topos).
\end{lemma}
\begin{proof}
Any right congruence $r$ on $M$ gives a quotient $M \too M/r$. Conversely, given a quotient $q:M \too N$, let $r := \rfrak_{q(1)}$; then we clearly have $N \cong M/r$, and this operation is an inverse to the preceding one by inspection.
\end{proof}

Note that we specifically refer to quotients of $M$ in Lemma \ref{lem:prel}, rather than principal $M$-sets, since distinct congruences can give isomorphic principal $M$-sets: each generator of a principal $M$-set presents it as a quotient in a distinct way. Nonetheless, we do get all principal $M$-sets at least once in this way.

Our next task is to recover the categorical structure on these objects. For a start, the natural ordering on the collection of congruences $\Rfrak$ is reflected in the subcategory $\Ccal_s$ of supercompact objects of $\Ecal$. Indeed, if $r \subseteq r'$ there is a corresponding quotient map $M/r \too M/r'$.

\begin{lemma}
\label{lem:factor}
Let $\Ccal$ be the full subcategory of $\Setswith{M}$ on objects of the form $M/r$. Then any morphism $g: M/r_1 \to M/r_2$ in $\Ccal$ factors uniquely as a quotient map of the form described above followed by an inclusion of the form $M/m^*(r_2) \hookrightarrow M/r_2$.
\end{lemma}
\begin{proof}
Consider the canonical generator $[1]$ of $M/r_1$. Let $m$ be any representative of $g([1])$. Then the image part of $g$ is precisely the inclusion described in the statement of the Lemma, and the factoring quotient map is of the desired form.

For uniqueness, note that epimorphisms are orthogonal to monomorphisms in $\Ccal_s$ by Corollary \ref{crly:orthog}, so any other such factorization has an isomorphic intermediate object. But the isomorphism must commute with the quotient maps, in particular preserving the canonical generator, and hence the corresponding relations are equal.
\end{proof}

As such, we extend the partial order $\Rfrak$ to a category $\Rfrakk$ as follows.

\begin{dfn}
\label{dfn:Rfrak}
The objects of $\Rfrakk$ are the right congruences in $\Rfrak$. A morphism $r_1 \to r_2$ is an equivalence class $[m]$ of $r_2$ such that $r_1 \subseteq m^*(r_2)$, and composition is given by multiplication in $M$, which is easily seen to be compatible with the containment condition.
\end{dfn}

That this category is well-defined follows from Lemma \ref{lem:rinrfrak}, since $(m,m') \in r$ implies $m^*(r) = {m'}^*(r)$. By inspection of Lemma \ref{lem:factor}, the resulting category is isomorphic to the full subcategory $\Ccal$ appearing there, and \textit{we identify $\Rfrakk$ with that category}. In particular, since that category is equivalent to the category $\Ccal_s$ of all principal $M$-sets, we have the following.

\begin{crly}
\label{crly:Rsite}
The topos $\Setswith{M}$ is equivalent to $\Sh(\Rfrakk,J_r)$; the strict epimorphisms in $\Rfrakk$ are precisely those indexed by an equivalence class containing $1 \in M$.
\end{crly}

It is worth noting that this expression for $\Setswith{M}$ is independent on the representing monoid $M$, since for any other choice of $M$, we obtain the same category $\Rfrakk$ up to equivalence (being equivalent to the category $\Ccal_s$ of supercompact objects of $\Setswith{M}$). We adapt the presentation of Corollary \ref{crly:Rsite} in order to obtain a canonical site for a topos $\Ecal$ admitting a hyperconnected morphism from $\Setswith{M}$.

\begin{prop}
\label{prop:prince2}
Let $h : \Setswith{M} \to \Ecal$ be a hyperconnected geometric morphism. Let $\Rfrakk_h$ be the full subcategory of $\Rfrakk$ on the right congruences $r$ such that the principal $M$-set $M/r$ lies in $\Ecal$ (up to isomorphism). Then $\Rfrakk_h$ is non-empty and closed under subobjects and quotients, has the joint covering property, and $\Ecal \simeq \Sh(\Rfrakk_h,J_r)$.

Conversely, non-empty subcategories of $\Rfrakk$ which are closed under subobjects, quotients and joint covers correspond bijectively with the hyperconnected morphisms under $\Setswith{M}$ (up to isomorphism).
\end{prop}
\begin{proof}
Clearly $\Rfrakk_h$ always contains the maximal equivalence relation, since $\Ecal$ contains the terminal object, while being closed under quotients and subobjects is a consequence of the fact that this is true of $\Ecal$, and that monomorphisms and epimorphisms in $\Rfrakk$ and $\Rfrakk_h$ coincide with those in $\Setswith{M}$ and $\Ecal$ by the proof of Proposition \ref{prop:prince} (using the fact that $\Rfrakk$ and $\Rfrakk_h$ can be identified up to equivalence with the respective categories of supercompact objects).

Similarly, using the proof of Proposition \ref{prop:hypejcp} we conclude that $\Rfrakk_h$ must inherit joint covers from $\Rfrakk$.

By Theorem \ref{thm:canon}, we recover $\Ecal \simeq \Sh(\Rfrakk_h,J_r)$, as claimed.

Conversely, given a subcategory $\Rfrakk'$ of $\Rfrakk$ satisfying the given conditions, consider the full subcategory $\Ecal$ of $\Setswith{M}$ on those objects which are colimits of objects in $\Rfrakk'$.

$\Rfrakk'$ contains the terminal object, so $\Ecal$ also does. A product of $M$-sets in $\Ecal$ has elements generating $M$-sets corresponding to (quotients of) joint covers in $\Rfrakk'$, and an equalizer in $\Setswith{M}$ of $M$-sets in $\Ecal$ is in particular a sub-$M$-set of one in $\Ecal$, so is covered by principal $M$-sets in $\Rfrakk'$. As such, $\Ecal$ is closed under finite limits, the embedding of $\Ecal$ into $\Setswith{M}$ is a left exact coreflection, and this embedding is the inverse image functor of a hyperconnected geometric morphism $\Setswith{M} \to \Ecal$, as required.
\end{proof}

We may use the site resulting from Proposition \ref{prop:prince2} in the special case when $\Ecal = \Cont(M,\tau)$ to recover a small (rather than merely essentially small) site for $\Cont(M,\tau)$. Given a topological monoid $(M,\tau)$, we write $\Rfrakk_{\tau}$ for $\Rfrakk_h$, where $h:\Setswith{M} \to \Cont(M,\tau)$ is the canonical hyperconnected morphism. 
\begin{schl}
\label{schl:Morita2}
Given a topological monoid $(M,\tau)$, the category $\Rfrakk_{\tau}$ is equivalent to the category $\Ccal_s$ of supercompact objects in $\Cont(M,\tau)$. In particular, another topological monoid $(M',\tau')$ is Morita equivalent to $(M,\tau)$ if and only if $\Rfrakk_{\tau} \simeq \Rfrakk_{\tau'}$.
\end{schl}
\begin{proof}
We see from the proof of Proposition \ref{prop:prince2} that the observations leading up to Corollary \ref{crly:Rsite} also apply in $\Cont(M,\tau)$. Thus we have the desired result, and the Morita equivalence statement follows by Corollary \ref{crly:Morita'}.
\end{proof}

Note that we can restrict the ordering on $\Rfrak$ to the objects of $\Rfrakk_h$ or $\Rfrakk'$, or equivalently recover that ordering by considering only morphisms indexed by $1 \in M$. We denote the resulting posets by $\Rfrak_h$ and $\Rfrak'$ respectively, extending this convention where needed.

\begin{rmk}
\label{rmk:Mequivariant}
In Proposition \ref{prop:prince2}, if we instead consider only the ordered sets, we can characterize the sub-poset $\Rfrak_h$ as an \textit{$M$-equivariant filter} in $\Rfrak$: a subset which is non-empty, upward closed, downward directed and closed under the inverse image action. These conditions are easier to verify in practice, so we use them occasionally in examples to follow.
\end{rmk}

\begin{lemma}
\label{lem:collect}
Let $M$ be a topological monoid, $\Rfrakk$ the category of right congruences on $M$, and $\Rfrakk'$ a subcategory of $\Rfrakk$ satisfying the conditions of Proposition \ref{prop:prince2}. Given a topology $\tau$ on $M$, $M/r$ is continuous with respect to $\tau$ for every $r \in \Rfrakk'$ if and only if every $r$ is open in $\tau \times \tau$.
\end{lemma}
\begin{proof}
For $M/r$ to be continuous with respect to $\tau$, it is certainly necessary that $r$ be open, since $r = \rfrak_{[1]}$. Conversely, given $[m] \in M/r$, $\rfrak_{[m]} = m^*(r) \in \Rfrakk'$, so if all of the right congruences in $\Rfrakk'$ are open then all elements of $M/r$ are continuous.
\end{proof}

In light of Lemma \ref{lem:collect}, we call the objects of $\Rfrakk_h$ the \textbf{open congruences of $h$}. This name will take on further significance in Section \ref{ssec:complete}.

\begin{xmpl}
\label{xmpl:grprln}
Suppose $M$ is a group. Then the right congruences on $M$ are precisely the partitions of $M$ into right cosets of a subgroup of $M$. As such, we can identify the objects of $\Rfrakk$ with these subgroups, and form a category of \textit{open subgroups} corresponding to a given hyperconnected geometric morphism instead.
\end{xmpl}

We are now in a position to show that not every topos admitting a hyperconnected morphism from $\Setswith{M}$ is of the form $\Cont(M,\tau)$ for some topology $\tau$ on $M$.

\begin{xmpl}
\label{xmpl:N+}
Consider the monoid $\Nbb$ of non-negative integers with addition. It is easily shown that the proper principal $\Nbb$-sets are indexed by pairs $(a,b)$ of non-negative integers with $b \geq 1$, consisting of the set $\{0, 1, \dotsc, a + b - 1\}$ acted on by addition such that sums greater than $a + b - 1$ are reduced modulo $b$ into the interval $[a, a + b - 1]$; we write $N_{a,b}$ for this $M$-set; the reader should recognize the $M$-sets $N_{a,1}$ of Example \ref{xmpl:idemclosed} amongst these.

We have an epimorphism $N_{a,b} \too N_{a',b'}$ if and only if $a' \leq a$ and $b'\mid b$, and a monomorphism $N_{a,b} \hookrightarrow N_{a',b'}$ if and only if $a \leq a'$ and $b = b'$ (so all monos split). The joint cover of $N_{a,b}$ and $N_{a',b'}$ is $N_{\max \{a,a'\},\mathrm{lcm}(b,b')}$. Note that $N_{a,b}$ is always finite, so that $\Nbb$ itself is the only infinite principal $\Nbb$-set.

By the above, the collection of right congruences corresponding to finite $\Nbb$-sets is therefore an upward-closed, downward-directed set closed under the inverse image action. By Remark \ref{rmk:Mequivariant}, there is a corresponding hyperconnected morphism $h : \Setswith{\Nbb} \to \Ecal$, where the subcategory $\Ecal$ of $\Setswith{\Nbb}$ consists of the $\Nbb$-sets all of whose elements generate finite subsets under the action. This could be compared to the adjunction between the category of abelian groups and the category of torsion abelian groups.

For $a \geq 0$, the necessary clopens of $N_{a,1}$ are $\{\{0\},\{1\}, \dotsc, \{a - 1\}, [a,1)\}$. Thus the only topology on $M$ making all of the $N_{a,1}$ continuous is the discrete topology, but $\Ecal \not\simeq \Setswith{\Nbb}$: not only is the stated hyperconnected morphism not an equivalence, but it is clear that the subcategories of supercompact objects of these toposes cannot be equivalent, since no supercompact object of $\Ecal$ admits epimorphisms to all of the others. By computing the action topologies corresponding to coarser topologies on $\Nbb$, we can verify by a similar argument that $\Ecal \not\simeq \Cont(\Nbb,\tau)$ for any topology $\tau$, as claimed.

On the other hand, since $\Nbb$ is a commutative monoid, each $N_{a,b}$ is canonically a monoid, and it is easy to see that $\Ecal$ is the topos of continuous actions of the \textbf{profinite completion of $\Nbb$}, which is the profinite monoid obtained as the inverse limit of the $N_{a,b}$ along the epimorphic maps between them.
\end{xmpl}

\begin{rmk}
\label{rmk:sgtxmpl}
The topos $\Ecal$ of Example \ref{xmpl:N+} provides the example, promised earlier, of a supercompactly generated topos which is neither regular, nor atomic, nor a presheaf topos. Indeed, the actions $N_{a,b}$ are not closed under products, so $\Ecal$ is not a regular topos. $\Ecal$ is not a presheaf topos, since the topos contains no indecomposable projectives: any such would necessarily be supercompact, but for any $(a,b)$, there is an epimorphism $N_{a,2b} \too N_{a,b}$ which is not split. Finally, it is not atomic, since the atoms are those of the form $N_{0,b}$, and $N_{a,b}$ is not covered by such an $\Nbb$-set for $a>0$.
\end{rmk}

The argument in Example \ref{xmpl:N+} is somewhat overzealous; even if we had found that $\Ecal \simeq \Cont(\Nbb,\tau)$ for some $\tau$, the important conclusion is that this hyperconnected morphism fails to \textit{express} $\Ecal$ as $\Cont(\Nbb,\tau)$. By extending the argument of Theorem \ref{thm:tau}, we shall show in Theorem \ref{thm:factor} that for any hyperconnected morphism out of $\Setswith{M}$ there is a canonical coarsest topology $\tau$ such that this morphism factorizes through the hyperconnected morphism $\Setswith{M} \to \Cont(M,\tau)$; in Example \ref{xmpl:N+}, this topology happens to be the discrete topology.

However, the fact that we were able to find a representing topological monoid for $\Ecal$ in the end turns out to be a general fact, as we might have hoped, and the latter part of Example \ref{xmpl:N+} suggests how it can be constructed. We perform this construction in general in Proposition \ref{prop:lim} and conclude the proof that it gives a representing monoid with Theorem \ref{thm:characterization}.

\subsection{Endomorphisms of the Canonical Point}
\label{ssec:complete}

Let $p$ be the point of the topos $\Ecal$ factorized in \eqref{eqn:p}. Consider the endomorphisms of $p$; that is, the monoid of natural transformations $\alpha: p^* \Rightarrow p^*$. Since we know $\Ecal$ is supercompactly generated, any such endomorphism is determined by its components on the subcategory $\Rfrakk_h$ of $\Ecal$. That is, each $\alpha$ consists of a collection of endomorphisms $\alpha_r: p^*(M/r) \to p^*(M/r)$ for $r \in \Rfrakk_h$ satisfying naturality conditions relative to morphisms in $\Rfrakk_h$.

The following argument replicates the proof of \cite[Theorem 5.7]{ATGT}, but without the need for a subsequent restriction to automorphisms. For consistency, we consider the opposite of the endomorphism monoid, acting on the right, so that composition is from left to right (this precludes the need to dualize subsequently).

\begin{prop}
\label{prop:lim}
Let $h:\Setswith{M} \to \Ecal$ be a hyperconnected geometric morphism, let $\Rfrakk_{h}$ be the corresponding category of right congruences on $M$ described in Proposition \ref{prop:prince2}, and let $\Rfrak_h$ be the underlying order. Then $\End(p^*)\op$ can be identified with the limit
\begin{equation}
\label{eq:limit}
L := \varprojlim_{r \in \Rfrak_{h}}U(M/r)
\end{equation}
in $\Set$. Explicitly, the elements are tuples $\alpha = ([a_r])_{r \in \Rfrak_h}$ with each $[a_r] \in M/r$, represented by $a_r \in M$, such that whenever $r \subseteq r'$, $[a_r] = [a_{r'}]$ in $M/r'$. The composite of $\alpha$ and $\beta = ([b_r])_{r \in \Rfrak_h}$ is $\alpha\beta = ([a_r b_{a_r^*(r)}])_{r \in \Rfrak_h}$. 
\end{prop}
\begin{proof}
Consider the factorization of morphisms in $\Rfrakk_h$ from  Lemma \ref{lem:factor}, which corresponds to the factorization of a morphism $m:r_1 \to r_2$ as,
\[ \begin{tikzcd} 
r_1 \ar[r, "{[1]}", two heads] & m^*(r_2) \ar[r, "{[m]}", hook] & r_2.
\end{tikzcd} \]
The naturality conditions for a natural transformation $p^*|_{\Rfrakk_h} \Rightarrow p^*|_{\Rfrakk_h}$ can be reduced to naturality along factors of the form $[1]: r \too r'$ and $[m]: m^*(r) \hookrightarrow r$.

First, consider $[m]: m^*(r) \hookrightarrow r$, corresponding to the inclusion $M/m^*(r) \hookrightarrow M/r$. Naturality forces $([m])\alpha_{r}$ to be the image of $([1])\alpha_{m^*(r)}$ under this inclusion. Thus, any endomorphism $\alpha$ of $p^*$ is determined by its values on the equivalence classes represented by the identity of $M$. Write $[a_r] \in U(M/r)$ for $([1])\alpha_r$; then $([m])\alpha_r = [m a_{m^*(r)}]$.

Now consider $[1]: r \too r'$, corresponding to the quotient map $M/r \too M/r'$. This forces $[a_r] \mapsto [a_{r'}]$. Thus we may identify each endomorphism $\alpha$ with an element of the stated limit; these observations also determine the composition.

Conversely, any element of the limit defines an endomorphism, since we have shown that collections satisfying these conditions are guaranteed to be natural.
\end{proof}

Let $L$ be the monoid defined in Proposition \ref{prop:lim}. For each $r \in \Rfrakk_h$, we write $\pi_{r}:L \to U(M/r)$ for the universal projection map. Being expressed as a limit of discrete sets, $L$ is canonically equipped with a prodiscrete topology $\rho$ making it a topological monoid. The basic opens for this topology form a genuine base of open sets:

\begin{lemma}
\label{lem:basics}
The collection of open sets of the form $\pi_r^{-1}(\{[m]\})$ generating $\rho$ are closed under (non-empty) finite intersection.
\end{lemma}
\begin{proof}
The empty intersection is all of $L = \pi^{-1}_{L \times L}(\{[1]\})$. Thus it suffices to consider binary intersections.

Given opens $\pi_{r_1}^{-1}(\{[m_1]\})$ and $\pi_{r_2}^{-1}(\{[m_2]\})$ having non-empty intersection, let $r = r_1 \cap r_2$. Since the intersection of the open sets in non-empty, there is some element $\alpha \in L$ whose component $[a_r]$ at $M/r$ maps under the canonical quotient maps to $[m_1], [m_2]$ respectively. Taking $[m]=[a_r]$ provides the desired expression for the intersection.
\end{proof}

Thus we need only concern ourselves with basic opens when checking continuity.

\begin{lemma}
\label{lem:Mdense}
Let $M$ and $L$ be as above. There is a canonical monoid homomorphism $u: M \to L$ sending $m$ to the endomorphism $\alpha_m$ represented at every principal $M$-set $M/r$ by $[m]$. Its image is dense in $(L,\rho)$.
\end{lemma}
\begin{proof}
The stated definition does give a well-defined endomorphism for each $m$, since the canonical quotient morphisms preserve representatives by definition. To see that this is a monoid homomorphism, simply observe from the expression for composition in Proposition \ref{prop:lim} that the component of $\alpha_m \cdot \alpha_{m'}$ at $M/r$ is always represented by $mm'$.

Now observe that for each $m \in M$, the basic opens of the form $\pi_r^{-1}(\{[m]\})$ contain $\alpha_m$, so the image of $M$ intersects every open set by Lemma \ref{lem:basics} and has dense image, as claimed.
\end{proof}

The monoid homomorphism $u$ is the key to demonstrating that we can in fact present $\Ecal$ as $\Cont(L,\rho)$. Indeed, note that it induces a geometric morphism $q: \Setswith{M} \to \Cont(L,\rho)$, whose inverse image is the restriction of $L$-actions along the homomorphism $\alpha_{(-)}$. We will explore the geometric morphisms induced by general continuous (monoid and) semigroup homomorphisms in Section \ref{sec:homomorphism}.

\begin{prop}
\label{prop:representation}
Let $h:\Setswith{M} \to \Ecal$ be hyperconnected, and let $(L,\rho)$ be the endomorphism monoid constructed above. The geometric morphism $q:\Setswith{M} \to \Cont(L,\rho)$ induced by the continuous dense homomorphism $u: M \to L$ of Lemma \ref{lem:Mdense} is hyperconnected, and the principal $M$-sets lying in $\Cont(L,\rho)$ are precisely those lying in the topos $\Ecal$ from which $L$ was defined. That is, $\Ecal \simeq \Cont(L,\rho)$.
\end{prop}
\begin{proof}
Since $\Cont(L,\rho)$ is supercompactly generated, it is necessary and sufficient to show that $q^*$ is full and faithful and that any principal $(L,\rho)$-set is a quotient of (and hence equal to) an object which is mapped by $q^*$ to a principal $M$-set: indeed, arbitrary colimits are computed in both $\Setswith{M}$ and $\Cont(L,\rho)$ at the level of underlying sets, so considering the principal objects is sufficient.

It is immediate that $q^*$ is faithful since the underlying function of an $L$-set homomorphism is unaffected by applying $q^*$. To show that any $M$-set homomorphism $g:P \to Q$ between $(L,\rho)$-sets $P$ and $Q$ is an $(L,\rho)$-set homomorphism, consider an element $p \in P$ and $\alpha \in L$. We have that $\alpha_m \in \Ical_p^\alpha$ for some $m \in M$ by density of $M$. Thus $g(p \cdot \alpha) = g(p \cdot \alpha_m) = g(p) \cdot \alpha_m = g(p) \cdot \alpha$, where the final equality is by continuity of the action of $L$ on $Q$. Thus $q^*$ is full, as expected.

Let $\Rfrakk_h$ be the category of right congruences defined in Section \ref{ssec:EqRel}. There is a canonical right action of $L$ on $M/r$ for each $r \in \Rfrak_h$: if $\alpha = ([a_r])_{r \in \Rfrak_h}$, then $\alpha$ acts on $[m] \in M/r'$ by `projected multiplication', sending it to $[ma_{m^*(r')}]$.

We should verify that this action is well-defined and continuous with respect to $\rho$. If $\beta = ([b_r])_{r \in \Rfrak_h}$ is another element of $L$, then acting by $\alpha$ and then $\beta$ gives $[m] \mapsto [ma_{m^*(r')}] \mapsto [ma_{m^*(r')}b_{(ma_{m^*(r')})^*(r')}]$, which indeed is equal to the action of $\alpha\beta$, so $M/r'$ is a right $L$-set.

For continuity of the action on $M/r'$, consider
\begin{align*}
\Ical_{[1]}^\alpha &= \{\beta = ([b_r])_{r \in \Rfrak_h} \mid [a_{r'}] = [b_{r'}]\}\\
&= \pi_{r'}^{-1}(\{[a_{r'}]\}).
\end{align*}
By definition of the induced prodiscrete topology, this is open in $L$, as required. By Lemma \ref{lem:collect}, since $r'$ was arbitrary, this is sufficient to conclude that the actions of $L$ on the principal $M$-sets lying in $\Ecal$ are continuous with respect to $\rho$. Moreover, it is clear that these are principal $L$-sets, since we can obtain all of $M/r'$ by applying the elements $\alpha_m$ to the generator of $M/r'$.

Now $q^*$ returns the $L$-sets defined above to the principal $M$-sets they extended. Given a quotient $L \too K$ in $\Setswith{L}$, observe that to be continuous with respect to $\rho$ it must be that for each $\alpha \in L$, $\Ical_{[1]}^\alpha$ is open in $\rho$, so must contain an open set (and hence a basic open set) around $\alpha$, say $\pi_{r'}^{-1}(\{[a_{r'}]\})$. But by the equation above, this is precisely $\Ical_{[1]}^\alpha$, where here $[1]$ is the canonical generator of $M/r'$. Thus $K$ must be a quotient of $M/r'$ as an $L$-set, and hence as an $M$-set, as required.
\end{proof}

We can summarize the results obtained so far in this section with the following theorem.
\begin{thm}
\label{thm:characterization}
A topos is equivalent to one of the form $\Cont(M,\tau)$ if and only if it has a surjective point which factors as an essential surjection followed by a hyperconnected geometric morphism. Moreover, every topological monoid is canonically Morita equivalent to a monoid endowed with a prodiscrete topology.
\end{thm}

\begin{rmk}
We have had to be careful with the language used in Theorem \ref{thm:characterization}: $(L,\rho)$ is \textit{not} in general a prodiscrete monoid, since the principal $M$-sets are not in general equipped with the structure of monoids. Even so, $\rho$ is always a `good' topology in the following sense.
\end{rmk}

\begin{prop}
\label{prop:Lpowder}
Given a monoid $M$ and a hyperconnected geometric morphism $h:\Setswith{M} \to \Ecal$, the corresponding topological monoid $(L,\rho)$ is a powder monoid.
\end{prop}
\begin{proof}
Clearly the basic opens at fixed $r \in \Rfrak_h$ partition $L$, so they are clopen; it therefore suffices to check that the basic open sets are continuous elements of $\Pcal(L)$ in $\Setswith{L}$.

Using the established notation for the components of $\alpha$ and $\beta$,
\begin{align*}
\alpha^*(\pi_r^{-1}(\{[m]\}))
&= \{\beta \in L\op \mid [a_r b_{a_r^*(r)}] = [m] \in M/r \}\\
&= \{\beta \in L\op \mid b_{a_r^*(r)} \in a_r^*(\Ical_n^m) \}\\
&= \begin{cases}
\pi_{a_r^*(r)}^{-1}(\{[m']\}) & \text{if } \exists m' \in a_r^*(\Ical_n^m)\\
\emptyset & \text{otherwise.}
\end{cases}
\end{align*}
This shows that the inverse image action preserves basic opens, and consequently we need only consider:
\begin{align*}
\Ical_{\pi_r^{-1}(\{[m]\})}^{\alpha}
&= \{\beta \in L\op \mid \beta^*(\pi_r^{-1}(\{[m]\})) = \alpha^*(\pi_r^{-1}(\{[m]\})) \}\\
& \supseteq \pi_r^{-1}(\{[a_r]\}),
\end{align*}
by inspection of the fact that $\alpha^*(\pi_r^{-1}(\{[m]\}))$ depends only on $a_r$. Thus the basic opens are continuous elements; since an action topology has at most as many opens as the topology it is derived from, it follows that $\tilde{\rho}=\rho$ as claimed.

To see that $(L,\rho)$ is Hausdorff we simply note that two points are equal if and only if they are equal in every component (by the definition of $L$ as a limit) and points which differ in any component are separated by basic opens from the corresponding projection maps.
\end{proof}

\begin{schl}
\label{schl:completion}
The topological monoid $(L,\rho)$ constructed in Proposition \ref{prop:lim} from a hyperconnected geometric morphism $h: \Setswith{M} \to \Ecal$ depends only on the point $\Set \to \Setswith{M} \to \Ecal$ of $\Ecal$ which, after composing with the equivalence $\Ecal \simeq \Cont(L,\rho)$, is naturally isomorphic to the canonical point of the latter topos. In particular, \textbf{any} factorization of the point of $\Ecal$ into an essential surjection followed by a hyperconnected geometric morphism produces the same representing monoid.
\end{schl}
\begin{proof}
While we can identify $L$ with the limit described in Proposition \ref{prop:lim}, it remains the opposite of the monoid of endomorphisms of the stated point, so is independent of $M$. Since $\rho$ is an action topology by Proposition \ref{prop:Lpowder}, it is uniquely determined by the topos $\Cont(L,\rho)$ after establishing $L$ by Theorem \ref{thm:tau}, which completes the proof.
\end{proof}

\begin{crly}
\label{crly:extend}
Suppose that $(M,\tau)$ is a powder monoid and let $(L,\rho)$ be the opposite of the topological monoid of endomorphisms of the corresponding canonical point of $\Cont(M,\tau)$. Then the monoid homomorphism $u: M \to L$ of Lemma \ref{lem:Mdense} is injective and continuous with respect to $\tau$ and $\rho$.
\end{crly}
\begin{proof}
Let $(M,\tau)$ be a powder monoid and $m, m' \in M$. Suppose $\alpha_m = \alpha_{m'}$, and let $U \in T$ with $m \in U$, whence the same is true for $\Ical_U^m$. Consider the principal sub-$M$-set of $T$ generated by $U$, say $M/r$. Consider $\pi_r^{-1}(\{[m]\})$ and $\pi_r^{-1}(\{[m']\})$; since $\alpha_m = \alpha_{m'}$, these open sets must be equal in $L$, which is to say that $[m] = [m']$ in $M/r$, and hence $m' \in U$. As such, $m,m'$ are topologically indistinguishable and hence equal in $(M,\tau)$, and so $u$ is injective as claimed.

To demonstrate continuity, consider a basic open $U' := \pi_r^{-1}(\{[a]\})$ in $L$. Then
\[u^{-1}(U') = \{m \mid \alpha_m \in U' \} = \{m \mid [m] = [a] \text{ in } M/r\} \supseteq \Ical_{[1]}^a,\]
where $[1]$ is the canonical generator of $M/r$. Since $M/r$ is a continuous $M$-set by assumption, $\Ical_{[1]}^a$ is open in $(M,\tau)$, as required.
\end{proof}

\begin{xmpl}
\label{xmpl:Z+}
Consider the following example, to be contrasted with Example \ref{xmpl:N+}. Let $\Zbb$ be the group of integers under addition. Taking the topology $\tau$ on $\Zbb$ in which the subgroups $n\Zbb \subseteq \Zbb$ and their cosets are open for each $n\neq 0$, we find that the continuous principal $\Zbb$-sets are precisely the finite cyclic groups $\Zbb/n\Zbb$ with $n > 0$; we thus obtain the topos $\Cont(\Zbb,\tau)$ of torsion $\Zbb$-sets. The topology $\tau$ is not discrete since every neighbourhood of $0$ is infinite, but it is nearly discrete and hence $(\Zbb,\tau)$ is a powder group. However, the monoid obtained from Proposition \ref{prop:lim} is the \textit{profinite completion} of the integers. Thus even for a powder monoid, the comparison map may fail to be an isomorphism.
\end{xmpl}

\begin{dfn}
\label{dfn:cpltmon}
In light of Corollary \ref{crly:extend} and Example \ref{xmpl:Z+}, we say that a monoid is (right) \textbf{complete} if the comparison morphism $u:(M,\tau) \to (L,\rho)$ is an isomorphism of topological monoids.
\end{dfn}

We shall see in Section \ref{ssec:monads} that complete monoids form a reflective subcategory of the ($2$-)category of monoids and also of the ($2$-)category of powder monoids; the comparison homomorphism $u$ is the unit of these reflections.

\begin{crly}
\label{crly:Mlim}
Let $(M,\tau)$ be a powder monoid and $h:\Setswith{M} \to \Cont(M,\tau)$ the canonical hyperconnected morphism. Consider the poset $\Rfrak_h$ as a subcategory of $\Setswith{M}$. Then $M$ is complete if and only if it is the limit of $\Rfrak_h$ in $\Setswith{M}$.
\end{crly}
\begin{proof}
Since limits in $\Setswith{M}$ are computed from their underlying sets, this follows from Proposition \ref{prop:lim} and Corollary \ref{crly:extend}, after observing that the action of $M$ expressed as the limit is respected by these morphisms by definition.
\end{proof}

\begin{rmk}
\label{rmk:completegrp}
It must be stressed that our notion of completeness for powder monoids does not quite coincide with that of completeness for groups described in Caramello's paper \cite[\S 2.3]{TGT}. This is clear from a comparison between our Proposition \ref{prop:lim} and Caramello and Lafforgue's construction in \cite[Proposition 5.7]{ATGT}, since they begin by constructing the topological monoid of endomorphisms of the canonical point as we do (which is complete in our sense), but then restrict to the subgroup of automorphisms in order to obtain the representing topological group (which is complete in their sense), and it is an important fact that this gives a genuinely different representation in general; see Example \ref{xmpl:notpro2} below.

For consistency, we say a powder group is \textbf{complete} if it is isomorphic to the topological subgroup of units of the corresponding complete monoid; this is true to Caramello and Lafforgue's terminology, and coincides with ours whenever the complete monoid happens to be a group.
\end{rmk}

\begin{xmpl}
\label{xmpl:notpro2}
We encountered the Schanuel topos in Example \ref{xmpl:nonpresh}. Thanks to the work of Caramello, \cite[\S 6.3]{TGT}, we know that any infinite set $X$ provides, via its corresponding point, an equivalence $\Ecal \simeq \Cont(\Aut(X),\tau_{\mathrm{fin}})$, where $\tau_{\mathrm{fin}}$ is the topology generated from the base of subgroups stabilizing finite subsets. The topological group of automorphisms of $\Nbb$ considered in Example \ref{xmpl:notpro} is one such representation. In particular, \textit{all points of the Schanuel topos are of the form required by Theorem \ref{thm:characterization}}, and the corresponding complete monoids are simply $(\End_{\mathrm{mono}}(X),\tau_{\mathrm{fin}})$, where this time $\tau_{\mathrm{fin}}$ has basic open sets consisting of subsets of the form
\[\{f \in \End_{\mathrm{mono}}(X) \mid f(x_1) = y_1, \dotsc, f(x_k) = y_k\}\]
for each finite set of pairs of elements $(x_i,y_i) \in X \times X$. 

In computing these complete monoids, we have the advantage of being able to inspect the endomorphisms of objects in the category of models for the theory classified by the topos, since these represent endomorphisms of the corresponding points. Computing one of these complete monoids with \eqref{eq:limit} and identifying the result with a monoid of injective endomorphisms is a more demanding, albeit instructive, exercise. This is one motivation for characterizing the theories classified by toposes of topological monoid actions in Chapter \ref{chap:TSGT}.
\end{xmpl}

\subsection{Bases of Congruences}
\label{ssec:base}

While we made abstract use of the limit expression \eqref{eq:limit} in proofs in the last subsection, in practice it will be more convenient to compute $L$ after re-indexing the limit over a smaller collection of right congruences. 

\begin{dfn}
\label{dfn:basecong}
Given a hyperconnected morphism $h:\Setswith{M} \to \Ecal$, a collection of right congruences $\Rfrak' \subseteq \Rfrak_h$ is called a \textbf{base of open congruences for $h$} if for every $r \in \Rfrak_h$ there is some $r' \in \Rfrak'$ with $r' \subseteq r$. A base of open congruences is precisely an initial subcategory of the poset $\Rfrak_h$, and as such we can replace \eqref{eq:limit} with
\begin{equation}
\label{eq:limit2}
L := \varprojlim_{r \in \Rfrak'} U(M/r),
\end{equation}
where the morphisms are simply inclusions of congruences. Moreover, the expression for the prodiscrete topology on $L$ restricts to this re-indexing, thanks to the fact that the basic opens for this topology coming from the projection maps along the omitted indices are necessarily unions of opens coming from any initial collection of indices.
\end{dfn}

\begin{xmpl}
\label{xmpl:algbase}
Suppose $(M,\tau)$ is a topological group. Recall from \cite[following Lemma 2.1]{TGT} that an \textit{algebraic base} for $(M,\tau)$ is a neighbourhood base of the identity $\Bcal$, consisting of open subgroups, such that for any $H,K \in \Bcal$ there exists $P \in \Bcal$ with $P \subseteq H \cap K$ and for any $g \in M$, there exists $Q \in \Bcal$ with $Q \subseteq g^{-1}Hg$.

Suppose we are given an algebraic base $\Bcal$ of open subgroups of $(M,\tau)$. In accordance with Example \ref{xmpl:grprln}, we can identify the open congruences for the canonical hyperconnected morphism $h:\Setswith{M} \to \Cont(M,\tau)$ with open subgroups. Every such open subgroup of $M$ must contain one belonging to $\Bcal$, whence the congruences corresponding to groups in $\Bcal$ form a base of open congruences on $(M,\tau)$. Conversely, any base of open congruences gives an algebraic base for $(M,\tau)$. 

The limit \eqref{eq:limit2} corresponds to the $\Bcal$-indexed limit expression for the monoid of endomorphisms presented in \cite[Proposition 5.7(i)]{ATGT}.
\end{xmpl}

\begin{prop}
\label{prop:algbase}
Let $h:\Setswith{M} \to \Ecal$ be a hyperconnected morphism. Suppose $\Rfrakk' \subseteq \Rfrakk_h$ is a base of open congruences. Suppose further that we extend $\Rfrakk'$ to a subcategory of $\Rfrakk_h$ such that given any $r,r' \in \Rfrakk'$ with $r \subseteq m^*(r')$, $\Rfrakk'$ contains a span of morphisms,
\begin{equation}
\label{eq:rspan}
\begin{tikzcd}
r & \ar[l, "{[1]}"', two heads] r'' \ar[r, "{[m]}"] & r'.
\end{tikzcd} 	
\end{equation}
For example, it suffices that $\Rfrakk'$ be a full subcategory. Then the morphisms indexed by $[1]$ in $\Rfrakk'$ form a stable class $\Tcal$ (in the sense of Definition \ref{dfn:stable}), and $(\Rfrakk',J_{\Tcal})$ is a dense subsite of $(\Rfrakk_h,J_r)$, which is to say that
\[ \Ecal \simeq \Sh(\Rfrakk',J_{\Tcal}).\]
\end{prop}
\begin{proof}
Given $[1]:r_1 \too r$ and $[m]:r_2 \to r$ in $\Rfrakk'$, stability of strict epimorphisms in $\Rfrakk_h$ provides a square there,
\[ \begin{tikzcd}
r' \ar[d, "{[m']}"'] \ar[r, "{[1]}", two heads] &
r_2 \ar[d, "{[m]}"] \\
r_1 \ar[r, "{[1]}"', two heads] & r;
\end{tikzcd}\]
without loss of generality we may assume $r' \in \Rfrakk'$ since any $r'$ is covered by a member of $\Rfrakk'$. Then we may construct spans on the upper and left-hand sides using \eqref{eq:rspan} in order to produce a similar square all of whose morphisms lie in $\Rfrakk'$. Moreover, the morphisms indexed by $1$ are precisely the morphisms inherited from $\Rfrakk_h$ which generate covering families, whence we see that $\Rfrakk'$ meets the definition of dense subsite required to apply the Comparison Lemma. This allows us to deduce the stated presentation of $\Ecal$.
\end{proof}

We can use open bases of congruences to address the question of when the completion $(L,\rho)$ of a powder monoid is (isomorphic to) a prodiscrete monoid.

\begin{crly}
\label{crly:prodisc}
Suppose that $h:\Setswith{M} \to \Ecal$ is hyperconnected. The corresponding complete monoid $(L,\rho)$ is discrete if and only if there exists a base of open congruences $\Rfrak' \subseteq \Rfrak_h$ with $\Rfrak'$ finite. More generally, $(L,\rho)$ is prodiscrete if and only if there exists a base of open congruences $\Rfrak' \subseteq \Rfrak_h$ where each $r \in \Rfrak'$ is a two-sided congruence. In the latter case, if $M$ is also a group, or $M/r$ is a group for each $r \in \Rfrak'$, then so is the resulting prodiscrete monoid.

In particular, if $M$ is finite or commutative, so that any right congruence $r$ on $M$ is also a left congruence, then the codomain of any hyperconnected morphism out of $\Setswith{M}$ is equivalent to the topos of continuous actions of a finite discrete monoid or commutative prodiscrete monoid respectively.
\end{crly}
\begin{proof}
As observed in Definition \ref{dfn:basecong}, given any base of open congruences $\Rfrak'$, we may express $L$ as the limit \eqref{eq:limit2} over congruences in $\Rfrak'$. Note that this limit is directed. When $\Rfrak'$ is finite, the induced topology is a finite product of discrete topologies, so it is discrete, and there must be an initial congruence in this case by directedness. Conversely, if $L$ is discrete, then $\Cont(L,\rho) = \Setswith{L}$. By construction, $\Rfrakk_h$ is equivalent to the category of quotients of $L$ in this topos, which contains a generating object, namely $L$ itself. As such, there is some relation $r^*$ in $\Rfrakk_h$ such that the equivalence $\Ecal \simeq \Setswith{L}$ identifies $M/r^*$ with $L$. Then $\{r^*\}$ is a finite base of open congruences for $h$, as required.  

Now suppose instead that each $r \in \Rfrak'$ is also a left congruence. Then the quotients $M/r$ are naturally equipped with a multiplication operation compatible with the multiplication from $M$, and the topological monoid $(L,\rho)$ constructed in Proposition \ref{prop:lim} is their limit as discrete monoids in the category of topological monoids, hence is a prodiscrete monoid. Conversely, if $(L,\rho)$ is prodiscrete, it can be defined as a limit of its discrete quotients, which are quotients of $L$ by a two-sided congruences. The restriction of such a congruence along $u$ is also a two-sided congruence on $M$; the collection of such congruences gives the desired base of open congruences.

If $M$ is a group and $\Rfrak'$ consists of two-sided congruences, then the quotients $M/r$ for $r \in \Rfrak'$ are also groups, whence $(L,\rho)$ is a prodiscrete group, fulfilling the claim regarding groups.
\end{proof}

\begin{xmpl}
As an example application of Corollary \ref{crly:prodisc} on a monoid which is neither commutative nor finite, consider the (non-commutative) monoid $M$ obtained from the non-negative integers $\Nbb$ with addition by freely adjoining a left absorbing element $l$. The elements of $M$ are of the form $(1,n)$ or $(l,n)$ with $n \in \Nbb$, and multiplication is defined by $(1,m)(l,n) = (l,m)(l,n) = (l,n)$, $(l,m)(1,n) = (l,m+n)$ and $(1,m)(1,n) = (1,m+n)$.

We can define a right congruence $r$ on $M$ which identifies all elements of the form $(l,n)$ and has all other equivalence classes being singletons. Then $(1,m)^*(r) = r$ for every $m$ and $(l,m)^*(r) = M \times M$, whence the collection of all equivalence relations containing $r$ is an $M$-equivariant filter in $\Rfrak$, and we have a corresponding hyperconnected morphism $\Setswith{M} \to \Cont(L,\rho)$.

Since $\{r\}$ is initial in $\Rfrak_h$, Corollary \ref{crly:prodisc} informs us that $\rho$ is the discrete topology. Indeed, we find that $L \cong \Nbb \cup \{\infty\}$ with extended addition, and $\Cont(L,\rho) = \Setswith{L}$. It is interesting to note that, $L$ being commutative, any further hyperconnected geometric morphism lands in the topos of actions of a prodiscrete monoid, such as the topologization of $\Nbb \cup \{\infty\}$ seen in Example \ref{xmpl:idemclosed}.
\end{xmpl}

\subsection{Factorizing Topologies}
\label{ssec:factor}

Having made it this far, we would be remiss not to initiate an investigation of when a hyperconnected geometric morphism $\Setswith{M} \to \Ecal$ actually \textit{does} express $\Ecal$ as $\Cont(M,\tau)$ for some topology $\tau$ on $M$. Our first result in this direction is a strengthening of Theorem \ref{thm:tau}.
\begin{thm}
\label{thm:factor}
Let $h: \Setswith{M}\to \Ecal$ be a hyperconnected geometric morphism. Consider $\Pcal(M) \in \Setswith{M}$ with the inverse image action corresponding to left multiplication. Write $T := h^*h_*(\Pcal(M)) \hookrightarrow \Pcal(M)$. Then (the underlying set of) $T$ is a base of clopen sets for the coarsest topology $\tau_h$ on $M$ such that $h$ factors through the canonical morphism $\Setswith{M} \to \Cont(M,\tau_h)$. That is, toposes constructed from topologies on $M$ are universal amongst toposes admitting a hyperconnected morphism from $\Setswith{M}$.
\end{thm}
\begin{proof}
By assumption, the counit at $\Pcal(M)$ is monic, so $T$ is indeed a subobject of $\Pcal(M)$. Further, $T$ must be a sub-Boolean-algebra of $\Pcal(M)$, so it is closed under complementation and finite intersections. Let $\tau_h$ be its closure in $\Pcal(M)$ under arbitrary unions. By Proposition \ref{prop:prince2}, to show that all $M$-sets lying in $\Ecal$ also lie in $\Cont(M,\tau_h)$, it suffices to show that the principal $M$-sets belonging to $\Ecal$ are continuous with respect to $\tau_h$.

Suppose $r \in \Rfrakk_h$, and let $p \in M$. We must show that, for $[1] \in M/r$, $\Ical_{[1]}^p \in U(T)$, or equivalently that the corresponding morphism $\ulcorner \Ical_{[1]}^{p} \urcorner: M \to \Pcal(M)$ factors through the inclusion $T \hookrightarrow \Pcal(M)$. We define a morphism $i^p : M/r \to \Pcal(M)$ by $[q] \mapsto q^*(\Ical_{[1]}^p)$. To see that this is well-defined, note that if $(q, q') \in r$, then
\[q^*(\Ical_{[1]}^p)= \{m \in M \mid [qm] = [p]\} = \{m \in M \mid [q'm] = [p]\} = {q'}^*(\Ical_{[1]}^p).\]
Hence the following diagram commutes:
\begin{equation}
\begin{tikzcd}
M \ar[r, "\ulcorner \Ical_{[1]}^{p} \urcorner"] \ar[d, two heads] & \Pcal(M) \\
M/r \ar[ur,"i^p"] \ar[r,dotted] & T \ar[u,hook]
\end{tikzcd}
\label{eq:factorize}
\end{equation}
Since $\Ecal$ is coreflective and $M/r \in \Ecal$, $i^p$ must further factor through the inclusion $T \hookrightarrow \Pcal(M)$. Thus we are done: $\Ical_{[1]}^{p} \in T$, as required.

It follows that $h$ factors through the morphism $\Setswith{M} \to \Cont(M, \tau_h)$ as claimed, and that $\tau_h$ is an action topology. Moreover, if $h$ factors through $\Cont(M,\tau')$ for any other topology $\tau'$, then $T$ must be continuous with respect to $\tau'$, and hence $\tau_h \subseteq \tau'$, as claimed.
\end{proof}

Note that the fact that $\Pcal(M)$ is an internal Boolean algebra ensures that for any object $N$, the set $\Hom_{\Setswith{M}}(N,\Pcal(M))$ inherits the structure of a Boolean algebra. When $N = M/r$ for some $r \in \Rfrakk_h$, a morphism $a: M/r \to \Pcal(M)$ is determined by the image of the generator $[1]$, and for any $(p,p') \in r$ must satisfy
\[ p^*(a([1])) = a([p]) = a([p']) = {p'}^*(a([1])).\]
In particular, by considering whether $1 \in p^*(a([1]))$, we see that $p \in a([1])$ if and only if $\Ical_{[1]}^p \subseteq a([1])$. Thus the morphisms $i^p$ in the proof above are actually \textit{atoms} in the Boolean algebra $\Hom_{\Setswith{M}}(M/r,\Pcal(M))$, since they have precisely two lower bounds, themselves and the trivial map sending every element of $M/r$ to $\emptyset \in \Pcal(M)$.

\begin{schl}
\label{schl:factors}
Let $h:\Setswith{M} \to \Ecal$ be a hyperconnected morphism. Then $\Ecal$ is equivalent to $\Cont(M,\tau)$ for some topology $\tau$ on $M$ (which necessarily coincides with $\tau_h$) if and only if whenever the image of each atom in $\Hom_{\Setswith{M}}(M/r,\Pcal(M))$ lies in $\Ecal$, we have $M/r$ in $\Ecal$.
\end{schl}
\begin{proof}
Reconstruct the diagram \eqref{eq:factorize} for a right congruence $r$ which is open with respect to $\tau \times \tau$, and let $M/r_p$ be the image of $i^p$:
\begin{equation}
\label{eq:image}
\begin{tikzcd}
M \ar[r, "\ulcorner \Ical_{[1]}^{p} \urcorner"] \ar[d, two heads] & \Pcal(M) \\
M/r \ar[ur,"i^p"] \ar[r, two heads] & M/r_p \ar[u,hook].
\end{tikzcd}
\end{equation}
Since $\Ecal$ is closed under quotients, if $M/r$ is in $\Ecal$ then so are the $M/r_p$ for every $p \in M$.

Conversely, since $\Ecal$ is closed under subobjects, the inclusion $M/r_p \hookrightarrow \Pcal(M)$ factors through $T$ if and only if $M/r_p$ lies in $\Ecal$; that is, $\Ical_{[1]}^p$ is in the topology induced by $h$ if and only if $M/r_p$ lies in $\Ecal$. Thus if $M/r_p$ lies in $\Ecal$ for every $p \in M$, this forces $M/r$ to be continuous.
\end{proof}

Another way of interpreting Scholium \ref{schl:factors} is as a necessary and sufficient condition for $(M,\tau_h)$ to be Morita-equivalent to the complete monoid representing $\Ecal$.

\begin{xmpl}
\label{xmpl:surjtop}
Any surjective monoid homomorphism $\phi: M \to M'$ induces a hyperconnected geometric morphism $f: \Setswith{M} \to \Setswith{M'}$; see Proposition \ref{prop:essgeom} below. The corresponding filter of $M$-equivariant relations is simply the collection of relations containing $r_{\phi} := \{(m,n) \mid \phi(m) = \phi(n)\}$. As such, it suffices to check the conditions of Scholium \ref{schl:factors} for $M/r_{\phi}$.

Suppose $M/r$ is such that for every $p \in M$, the relation $r_p$ from \eqref{eq:image} contains $r_{\phi}$. Then given $(m,n) \in r_{\phi}$, consider $r_m = \{(p,p') \mid p^*(\Ical_{[1]}^m) = p'{}^*(\Ical_{[1]}^m) \}$, where $[1]$ is the generator for $M/r$. By assumption, $(m,n) \in r_m$, whence $m^*(\Ical_{[1]}^m) = n^*(\Ical_{[1]}^m)$ and $n \in \Ical_{[1]}^m$, which is to say that $(m,n) \in r$, as required. So $T = f^*f_*(P(M))$ generates a topology $\tau_f$ on $M$ such that $\Cont(M,\tau_f) \simeq \Setswith{M'}$ via $f$.

Of course, we can calculate $T$ directly as the topology whose open sets are the equivalence classes of $r_{\phi}$, and this coincides with $\tau$: since $f$ is essential, $f^*f_*(\Pcal(M))$ is a complete Boolean algebra in $\Set$ (both $f^*$ and $f_*$ preserve small $\Set$-limits). 
\end{xmpl}

While Example \ref{xmpl:surjtop} illustrates that Scholium \ref{schl:factors} provides a workable necessary and sufficient condition, it does not illuminate precisely which toposes arise in this way. See Conjecture \ref{conj:characterization} in the Conclusion for a more detailed discussion of this issue.

\section{Semigroup Homomorphisms}
\label{sec:homomorphism}

In this section we attempt to functorialize the results obtained so far by examining how homomorphisms lift to geometric morphisms. In light of the results in Section \ref{sec:morphisms0} of Chapter \ref{chap:TDMA}, it is sensible to consider as morphisms between topological monoids not only (continuous) monoid homomorphisms, but also (continuous) semigroup homomorphisms, which correspond to essential geometric morphisms between the corresponding presheaf toposes.

\subsection{Restricting Essential Geometric Morphisms}

Let $\phi : M \to M'$ be a semigroup homomorphism between monoids. Recall that $\phi$ induces a functor $\check{\phi}: \check{M} \to \check{M}'$ between the idempotent-completions of the monoids, and hence induces an essential geometric morphism $f: \Setswith{M} \to \Setswith{M'}$.

Factorizing $f$, we have the following result, which shall be explored further in forthcoming work with Jens Hemelaer \cite{EDMA}:
\begin{prop}
\label{prop:essgeom}
Let $f: \Setswith{M} \to \Setswith{M'}$ be an essential geometric morphism induced by a monoid homomorphism $\phi$, and let $e := \phi(1)$. Then the surjection--inclusion factorization of $f$ is canonically represented by the factorization of $\phi$ as a monoid homomorphism $M \to eM'e$ followed by an inclusion of subsemigroups $eM'e \hookrightarrow M'$. Meanwhile, the hyperconnected--localic factorization of $f$ is canonically represented by the factorization of $\phi$ as surjective monoid homomorphism followed by an injective semigroup homomorphism.
\end{prop}
\begin{proof}
These results are proved by considering the factorization of $\check{\phi}$ corresponding to the surjection--inclusion and hyperconnected--localic factorizations of $f$, which can be found in \cite[Examples 4.2.7(b), 4.2.12(b), 4.6.2(c) and 4.6.9]{Ele}. We find in both cases that the intermediate category is the idempotent completion of the monoid indicated in the statement, whence these factors reduce to the stated semigroup homomorphisms.
\end{proof}

Now consider topologies $\tau$, $\tau'$ on $M$, $M'$ respectively. Then we may consider the square
\[\begin{tikzcd}
{\Setswith{M}} \ar[r, shift right = 4, "f_*"'] \ar[r, shift right = 2, phantom, "\bot"] \ar[r, shift left = 2, phantom, "\bot"] \ar[r, shift left=4, "f_!"] \ar[d, shift left = 2, "R"] &
{\Setswith{M}} \ar[l, "f^*"'{very near start, inner sep = 0pt}] \ar[d, shift left = 2, "R'"] \\
{\Cont(M{,}\tau)} \ar[u, shift left = 2, hook, "V"] \ar[u, phantom, "\dashv"]&
{\Cont(M'{,}\tau')}  \ar[l, "Rf^*V'"] \ar[u, shift left = 2, hook, "V'"] \ar[u, phantom, "\dashv"],
\end{tikzcd}\]
where across the top we have the essential geometric morphism $f$ induced by $\phi$, whose inverse image is induced by tensoring with the left-$M'$-right-$M$-set $M'\phi(1)$ (which coincides with restriction along $\phi$ when $\phi$ is a monoid homomorphism). This situation bears a strong resemblance to that involved in describing morphisms or comorphisms of sites, where the vertical morphisms are inclusions directed upwards rather than hyperconnected morphisms directed downwards. Accordingly, under suitable hypotheses, the lower horizontal map becomes the inverse image functor of a geometric morphism.

\begin{lemma}
\label{lem:cts}
Let $\phi: M \to M'$, $\tau$, $\tau'$ and $f$ be as above. Then the following are equivalent:
\begin{enumerate}
	\item $f^*: \Setswith{M'} \to \Setswith{M}$ maps every $(M',\tau')$-set to an $(M,\tau)$-set.
	\item $\phi$ is continuous with respect to $\tau$ and $\tilde{\tau}'$.
	\item The composite functor $Rf^*V'$ (is left exact and) has a right adjoint $G$ satisfying $GR \cong R'f_*$, which is to say that $f$ restricts along the functors $V,V'$ to a geometric morphism $(G \dashv Rf^*V'): \Cont(M,\tau \to \Cont(M',\tau')$ making the square commute.
\end{enumerate}
\end{lemma}
\begin{proof}
($1 \Leftrightarrow 2$) The precomposition functor $f^*$ maps every $(M',\tau')$-set to an $(M,\tau)$-set if and only if for each $X \in \Cont(M',\tau')$, we have $\Ical_x^{p} \in \tau$ for every $x \in f^*(X)$, $p \in M$. By definition of the action of $M$ on $f^*(X)$, we have $\Ical_x^{p} = \{m \in M \mid x \phi(p) = x \phi(m)\} = \phi^{-1}(\Ical_x^{\phi(p)})$; thus (1) is equivalent to $\phi^{-1}$ preserving the openness of the necessary clopens, which lie in $\tilde{\tau}'$.

Given any $U' \in \tilde{\tau}'$, we may express $U'$ as a union of necessary clopens $\Ical_x^{p'}$, and $\phi^{-1}(U')$ is the corresponding union of $\phi^{-1}(\Ical_x^{p'})$. Each such clopen $\Ical_x^{p'}$ either intersects with the image of $\phi$ and so is of the form $\Ical_x^{\phi(p)}$, or does not and so has empty inverse image. It follows that $\phi$ reflecting openness of the $\Ical_x^{\phi(p)}$ is equivalent to $\phi$ being continuous with respect to $\tilde{\tau}'$ and $\tau$, as required.

($1 \Leftrightarrow 3$) If $f^*$ maps every $(M',\tau')$-set to an $(M,\tau)$-set, then composing $R$ with $f^*V'$ does not affect the underlying set of the image. That is, $Rf^*V'$ preserves finite limits and arbitrary colimits since $f^*V'$ does and these are computed in $\Cont(M,\tau)$ just as in $\Setswith{M}$, making $Rf^*V'$ the inverse image of a geometric morphism by the special adjoint functor theorem. Write $G$ for the direct image. It is immediate that $V(Rf^*V') \cong f^*V'$, which means that the corresponding square of right adjoints commutes up to isomorphism and we have $GR' \cong R'f_*$.

Conversely, given a right adjoint $G$ to $Rf^*V'$ satisfying the given identity we must have $f^*V' \cong VRf^*V'$, which ensures that $f^*$ sends every $(M',\tau')$-set to an $(M,\tau)$-set.
\end{proof}

Thus, since $\phi$ being continuous with respect to $\tau$ and $\tau'$ is in general strictly stronger than condition (2) of Lemma \ref{lem:cts}, we obtain a functorialization of the $\Cont(-)$ construction from the ($1$-)category of topological monoids and continuous semigroup homomorphisms to the ($1$-)category of Grothendieck toposes and geometric morphisms. Let us reintroduce the $2$-morphisms between semigroup homomorphisms.

Recall from Definition \ref{dfn:conjugation} of Chapter \ref{chap:TDMA} that a \textit{conjugation} $\alpha:\phi \Rightarrow \psi$ between semigroup homomorphisms $\phi,\psi:M \to M'$ is an element $\alpha \in M'$ such that $\alpha \phi(1) = \alpha = \psi(1) \alpha$ and for every $m \in M$, $\alpha \phi(m) = \psi(m) \alpha$.

By Theorem \ref{thm:2equiv0}, conjugations correspond bijectively and contravariantly with the natural transformations between the essential geometric morphisms corresponding to $\phi$ and $\psi$. Since $\Cont(M,\tau)$ is a full subcategory of $\Setswith{M}$, any natural transformation $\alpha: f^* \Rightarrow g^*$ restricts along $R$ to give a natural transformation $Rf^*V' \Rightarrow Rg^*V'$; this is shown to be a general property of any connected geometric morphism in Proposition \ref{prop:conncoff}. Thus, we conclude:

\begin{thm}
\label{thm:Cont}
The construction $\Cont(-)$ is a $2$-functor from the $2$-category of topological monoids, continuous semigroup homomorphisms and conjugations to the $2$-category of Grothendieck toposes, geometric morphisms and natural transformations. We may restrict the codomain of this $2$-functor to those Grothendieck toposes satisfying the condition of Theorem \ref{thm:characterization} to make it essentially surjective on objects. Since every such topos has a representative which is a complete monoid, we may also restrict the domain to the class of complete monoids without changing this fact.
\end{thm}

We shall make some further comments about this $2$-functor before Example \ref{xmpl:Schanuel}.

One aim of the remainder of this section is to investigate the surjection--inclusion and hyperconnected--localic factorizations of a geometric morphism corresponding to a continuous semigroup homomorphism, to extend Proposition \ref{prop:essgeom} to topological monoids. We ultimately show in Theorems \ref{thm:surjinc} and \ref{thm:hypeloc} that these factorizations restricts along $\Cont(-)$.

\subsection{Intrinsic Properties of Geometric Morphisms}
\label{ssec:intrinsic}

For reference in the rest of the section, we shall use the following notation for the square of geometric morphisms induced by a continuous semigroup homomorphism $\phi: (M,\tau) \to (M',\tau')$ thanks to Lemma \ref{lem:cts}:
\begin{equation}
\label{eq:square}
\begin{tikzcd}
{\Setswith{M}} \ar[r, "f"] \ar[d, "h"'] &
{\Setswith{M'}} \ar[d, "h'"] \\
{\Cont(M{,}\tau)} \ar[r, "g"'] &
{\Cont(M'{,}\tau')},
\end{tikzcd}	
\end{equation}
where $h$ and $h'$ are hyperconnected and $f$ is essential; we could alternatively have denoted $g$ by $\Cont(\phi)$ in accordance with Theorem \ref{thm:Cont}, but the shorter notation will make some of the results below clearer. In order to understand the relationships between $f$ and $g$, we shall exploit intrinsic properties of the geometric morphisms $h$ and $h'$ as $1$-morphisms in the $2$-category $\TOP$ of Grothendieck toposes\footnote{Some of these results apply more generally, but for the purposes of the present chapter we only concern ourselves with Grothendieck toposes over $\Set$.}, in the special cases that $(M,\tau)$ and/or $(M',\tau')$ are powder monoids or complete monoids. 

Given (Grothendieck toposes) $\Ecal$ and $\Fcal$, we shall write $\Geom(\Ecal,\Fcal)$ for the category of geometric morphisms $\Ecal \to \Fcal$, where a morphism $f \Rightarrow g$ is as usual a natural transformation $f^* \Rightarrow g^*$. We shall also write $\EssGeom(\Ecal,\Fcal)$ for the full subcategory of essential geometric morphisms $\Ecal \to \Fcal$.

First, we can use Corollary \ref{crly:Mlim} to give an intrinsic characterization of the hyperconnected morphism presenting a complete monoid.

\begin{prop}
\label{prop:intrinsic}
A topological monoid $(M,\tau)$ is complete if and only if the geometric morphism $h: \Setswith{M} \to \Cont(M,\tau)$ is internally full and faithful on essential geometric morphisms in $\TOP$, in the sense that for any topos $\Fcal$, the functor
\[h \circ -: \EssGeom(\Fcal,\Setswith{M}) \to \Geom(\Fcal,\Cont(M,\tau))\] 
is full and faithful.

Thus, if $\Ecal$ admits a hyperconnected morphism $h:\Setswith{M} \to \Ecal$ which is internally full and faithful on essential geometric morphisms, the corresponding topological monoid $(L,\rho)$ representing $\Ecal$ has $L \cong M$.
\end{prop}
\begin{proof}
By taking $\Fcal = \Set$ and considering the canonical point of $\Setswith{M}$, we see that the given condition is sufficient, since it forces the monoid of endomorphisms of the canonical point of $\Cont(M,\tau)$ to be isomorphic to that of $\Setswith{M}$, which is precisely $M\op$. The same holds when we are given such a morphism $\Setswith{M} \to \Ecal$.

Conversely, suppose we are given essential geometric morphisms $h,k: \Fcal \to \Setswith{M}$ and a natural transformation $\alpha:h^* \Rightarrow k^*$. Any such natural transformation is determined by its component $\alpha_M:h^*(M) \to k^*(M)$. But since $(M,\tau)$ is complete, by Corollary \ref{crly:Mlim}, $M = \lim_{r \in \Rfrak_{\tau}} M/r$, and $h^*$ and $k^*$ preserve all limits, whence $\alpha_M$ is determined uniquely by the components $\alpha_{M/r}$. The functor induced by $g$ sends $\alpha$ to $\alpha_{g^*}$; considering the components at the principal $(M,\tau)$-sets, we conclude that this functor is full and faithful, as claimed.
\end{proof}

Proposition \ref{prop:intrinsic} should be compared with the following two propositions:
\begin{prop}
\label{prop:inclff}
Inclusions of toposes are internally full and faithful in the $2$-category $\TOP$, in the sense that given an inclusion $g:\Fcal \to \Ecal$ and any topos $\Gcal$, the functor $g \circ -: \Geom(\Gcal,\Fcal) \to \Geom(\Gcal,\Ecal)$ is fully faithful.

In particular, $g$ is full and faithful on essential geometric morphisms in the sense of Proposition \ref{prop:intrinsic}, and when $g$ is an essential inclusion, we may restrict the codomain to deduce that $i \circ -: \EssGeom(\Gcal,\Fcal) \to \EssGeom(\Gcal,\Ecal)$ is full and faithful.
\end{prop}
\begin{proof}
Let $g:\Fcal \to \Ecal$ be a geometric inclusion, and let $h,k: \Gcal \rightrightarrows \Fcal$.

A geometric transformation $h \Rightarrow k$ consists of a natural transformation $h^* \Rightarrow k^*$. Let $\alpha$, $\beta$ be two such transformations. If $g \circ \beta = g \circ \alpha$, then for any object $C$ of $\Gcal$, letting $\epsilon_C$ denote the counit of $(i^* \dashv i_*)$ at $C$, which is an isomorphism, we have:
\begin{align*}
\alpha_C & = k^*\epsilon_{C} \circ \alpha_{g^*g_*(C)} \circ h^*\epsilon_C^{-1} = k^*\epsilon_{C} \circ (g \circ \alpha)_{g_*(C)} \circ h^*\epsilon_{C}^{-1}\\
& = k^*\epsilon_{C} \circ (g \circ \beta)_{g_*(C)} \circ h^*\epsilon_{C}^{-1} = k^*\epsilon_{C} \circ \beta_{g^*g_*(C)} \circ h^*\epsilon_{C}^{-1} = \beta_C,
\end{align*}
so $g \circ -$ is faithful.

Similarly, given $\alpha': h^*g^* \Rightarrow k^*g^*$, define $\alpha: h^* \Rightarrow k^*$ by letting its component at $C$ in $\Gcal$ be $k^*\epsilon_{C}^{-1} \circ \alpha'_{g_*(C)} \circ h^*\epsilon_{C}$. Then for each object $D$ in $\Fcal$ we have
\[(g \circ \alpha)_{D} = \alpha_{g^*(D)} = k^*\epsilon_{g^*(D)}^{-1} \circ \alpha'_{g_*g^*(D)} \circ h^*\epsilon_{g^*(D)} = \alpha'_D,\]
by naturality. So $g \circ -$ is full, as required.
\end{proof}

\begin{prop}
\label{prop:conncoff}
Connected geometric morphisms are internally cofull and cofaithful in the $2$-category $\TOP$, in the sense that given a connected morphism $c:\Fcal \to \Ecal$ and any topos $\Gcal$, the functor $- \circ c: \Geom(\Ecal,\Gcal) \to \Geom(\Fcal,\Gcal)$ is fully faithful.

This applies in particular to hyperconnected geometric morphisms.
\end{prop}
\begin{proof}
Let $c:\Fcal \to \Ecal$ be a connected geometric morphism. Then concretely, $- \circ c$ is simply the application of $c^*$ to the components of any given natural transformation. As such, since $c^*$ is full and faithful, $- \circ c$ is full and faithful.
\end{proof}

\begin{crly}
\label{crly:Cont}
The restriction of $\Cont(-)$ to the category of complete monoids, semigroup homomorphisms and conjugations (with direction reversed) is full and faithful on 2-cells.
\end{crly}
\begin{proof}
By Theorem \ref{thm:2equiv0}, the $2$-equivalence mapping a discrete monoid to its presheaf topos is (contravariantly) full and faithful on $2$-cells, so for essential geometric morphisms $f,f': \Setswith{M} \to \Setswith{M'}$ induced by $\phi$, $\phi'$ respectively, each geometric transformation $\alpha: f \Rightarrow f'$ corresponds to a unique conjugation $\phi \Rightarrow \phi'$. Since $h'$ is full and faithful on essential geometric morphisms, the same is true of $2$-cells $h' \circ f \Rightarrow h' \circ f'$, by Proposition \ref{prop:intrinsic}. Passing across the square \eqref{eq:square}, since $h$ is cofull and cofaithful by Proposition \ref{prop:conncoff}, we obtain a further identification with the geometric transformations $g \Rightarrow g'$ (where $g$, $g'$ are also induced by $\phi$, $\phi'$ respectively), as required.
\end{proof}

\begin{rmk}
Explicitly, the conjugation corresponding to a transformation $\beta: g \Rightarrow g'$ is obtained as follows. First, compose with $h$ and take the limit of the components at the principal $(M',\tau')$-sets to obtain an $M$-set homomorphism $\alpha_{M'}: f^*(M') \to f'{}^*(M')$ (using the limit expression from Corollary \ref{crly:Mlim} and essentialness of $f,f'$ again), and by extension a natural transformation $\alpha: f^* \Rightarrow f'{}^*$. Then we take the mate $\overline{\alpha}:f'_! \Rightarrow f_!$, whose component at $M$ is the desired conjugation.
\end{rmk}

Corollary \ref{crly:Cont} can be understood as a strengthening of Proposition \ref{prop:lim}, since taking the domain monoid $M$ to be the trivial monoid and $\phi = \psi$ to be the unique monoid homomorphism to $(M',\tau')$, the conjugations are precisely the elements of $M'$.

Returning to properties of geometric morphisms, Proposition \ref{prop:inclff} will allow us to constrain the interactions between hyperconnected geometric morphisms and geometric inclusions in Section \ref{ssec:id}, but we also need to consider localic geometric morphisms.
\begin{prop}
\label{prop:locfaith}
Localic geometric morphisms are internally faithful in the $2$-category $\TOP$ of toposes, in the sense that given a localic geometric morphism $f:\Fcal \to \Ecal$ and any topos $\Gcal$, the functor $f \circ -: \Geom(\Gcal,\Fcal) \to \Geom(\Gcal,\Ecal)$ is faithful.
\end{prop}
\begin{proof}
A geometric morphism $g:\Fcal \to \Ecal$ is localic if and only if every object $Y$ of $\Fcal$ is a subquotient of one of the form $g^*(X)$ for $X$ in $\Ecal$, so there exists a diagram
\[\begin{tikzcd}
Y & \ar[l, two heads, "e"'] Z \ar[r, hook, "m"] & g^*(X).
\end{tikzcd}\]
Given geometric morphisms $h,k: \Gcal \rightrightarrows \Fcal$ and geometric transformations $\alpha,\beta: h \rightarrow k$ with $g \circ \alpha = g \circ \beta$, this is equivalent to the condition that $\alpha_{g^*(X)} = \beta_{g^*(X)}$ for all objects $X$ of $\Ecal$. Then considering naturality across the diagram above, since $h^*$ and $k^*$ both preserve epimorphisms and monomorphisms, we have:
\[\begin{tikzcd}
h^*g^*(X) \ar[r,"\alpha_{g^*(X)}", shift left] \ar[r,"\beta_{g^*(X)}"', shift right] &
k^*g^*(X) \\
h^*(Z) \ar[r,"\alpha_{Z}", shift left] \ar[r,"\beta_{Z}"', shift right] \ar[d, "h^*e"', two heads] \ar[u, "h^*m", hook] &
k^*g^*(X) \ar[d, "k^*e", two heads] \ar[u, "k^*m"', hook] \\
h^*(Y) \ar[r,"\alpha_{Y}", shift left] \ar[r,"\beta_{Y}"', shift right] & k^*(Y),
\end{tikzcd}\]
whence we see that $\alpha_{Z} = \beta_{Z}$ and then $\alpha_Y = \beta_Y$. Since $Y$ was a generic object of $\Fcal$, $\alpha = \beta$, as required.
\end{proof}

\begin{crly}
\label{crly:powderfaith}
Let $\tau$ be an action topology on $M$. Then $(M,\tau)$ is a powder monoid (equivalently, $\tau$ is a $T_0$ topology on $M$) if and only if the hyperconnected geometric morphism $\Setswith{M} \to \Cont(M,\tau)$ is internally faithful on essential geometric morphisms.
\end{crly}
\begin{proof}
Let $(L,\rho)$ be the completion of $(M,\tau)$ and consider the continuous, dense monoid homomorphism $u:(M,\tau) \to (L,\rho)$. By Corollary \ref{crly:extend}, when $(M,\tau)$ is a powder monoid, $u$ is injective, so the induced geometric morphism $f:\Setswith{M} \to \Setswith{L}$ is a localic surjection. Considering the square \eqref{eq:square}, since $h'$ is full and faithful on essential geometric morphisms by Proposition \ref{prop:intrinsic} and $f$ is faithful on these, it follows that $gh$ and hence $h$ are both faithful on essential geometric morphisms, as claimed.

Conversely, the morphism $g$ induced by $u$ is an equivalence, so if $h$ is faithful on essential geometric morphisms, then so is $f$, since $h'$ is full and faithful on such. Thus, $u:M \to L$ must be injective, since we can recover it as the restriction of the functor $f \circ -: \EssGeom(\Set,\Setswith{M}) \to \EssGeom(\Set,\Setswith{L})$ to the endomorphisms of the canonical point of $\Setswith{M}$. But $(L,\tau)$ is $T_0$, and any submonoid/subspace of a $T_0$ monoid must also be $T_0$, as required.
\end{proof}

Note that unlike in Proposition \ref{prop:intrinsic}, we cannot deduce that an arbitrary hyperconnected morphism $\Setswith{M} \to \Ecal$ which is faithful on essential geometric morphisms expresses $\Ecal$ as $\Cont(M,\tau)$; the non-topological factor of Theorem \ref{thm:factor} may be non-trivial.

As we noted after Lemma \ref{lem:includes} in Chapter \ref{chap:TDMA}, we can factorize any semigroup homomorphism $\phi:M \to M'$ into a monoid homomorphism followed by the inclusion of a subsemigroup of the form $\phi(1)M'\phi(1)$ into $M'$, and this lifts to (a canonical representation of) the surjection-inclusion factorization of the essential geometric morphism corresponding to $\phi$. Accordingly, we separate the analysis of geometric morphisms coming from continuous semigroup homomorphisms into the analysis of continuous monoid homomorphisms and inclusions of subsemigroups, in Sections \ref{ssec:monhom} and \ref{ssec:id} respectively.

\subsection{Subsemigroups}
\label{ssec:id}

Throughout this section, $(M',\tau')$ is a right powder monoid and $e \in M'$ an idempotent. Out of curiosity, we make some topological observations about the ideals generated by idempotents.

\begin{lemma}
\label{lem:ideal}
The principal left ideal $M'e$ and the principal right ideal $eM'$ of $M'$ are closed in $(M',\tau')$.
\end{lemma}
\begin{proof}
We can characterize $M'e$ and $eM'$ as the subsets of $M'$ on those elements $p$ such that $pe = p$ and $p = ep$ respectively. 

Suppose $x$ is outside $M'e$. Since $(M',\tau')$ is zero-dimensional Hausdorff, we can find a basic clopen set $U$ with $x \in U$ and $xe$ in the complement of $U$. Then $\Ical_U^x$ is an open set containing $x$; if $p \in \Ical_U^x$ then since $e \notin x^*(U)$, we have $e \notin p^*(U)$ so $pe \neq p$. Thus we conclude that $M'e$ is contained in the complement of $\Ical_U^x$, and $M'e$ is closed.

Similarly, if $x$ is outside $eM'$, let $U$ be a basic clopen set containing $x$ but not $ex$. Then $e^*(M \backslash U)$ contains $x$, so we may consider the smaller neighbourhood $U \cap e^*(M \backslash U)$ of $x$. This excludes any element $p$ with $p = ep$, so that in particular $eM'$ is contained in the complement. Thus $eM'$ is closed.
\end{proof}

\begin{prop}
\label{prop:eMe}
Let $(M',\tau')$ be a powder monoid, $e$ an idempotent of $M'$ and $M := eM'e$. Let $\phi: M \hookrightarrow M'$ be the corresponding inclusion of semigroups, and let $\tau$ be the topology on $eM'e$ obtained by restricting $\tau'$. Then $(M,\tau)$ is a powder monoid, and any subsemigroup of this form is closed in $M'$.
\end{prop}
\begin{proof}
Being the coarsest topology on $M$ such that $\phi:M \to M'$ is continuous, $\tau$ is the coarsest topology such that (the hyperconnected part of) the morphism $h' \circ f$ in the square \eqref{eq:square} induced by $\phi$ factors through $\Setswith{M} \to \Cont(M,\tau)$, by Theorem \ref{thm:factor}. Thus $\tau$ is an action topology on $M$. As a subspace of a Hausdorff space, $(M,\tau)$ is Hausdorff (we could alternatively have used Corollary \ref{crly:powderfaith} to deduce this). 

To see that $M$ is closed in $M'$, observe that it is the intersection of the ideals $eM'$ and $M'e$ which we showed to be closed in Lemma \ref{lem:ideal}.
\end{proof}

\begin{xmpl}
\label{xmpl:nopen}
On the other hand, $eM'e$ is not always open in $M'$. Indeed, consider the prodiscrete monoid constructed in Example \ref{xmpl:idemclosed}. The idempotent element $e = \infty$ is a zero element, so that the corresponding subsemigroup is simply $\{\infty\}$, which from the description of the topology on this monoid clearly fails to be open.
\end{xmpl}

Intuitively, we might expect the geometric morphism induced by the subsemigroup inclusion $M \hookrightarrow M'$ in Proposition \ref{prop:eMe} to be a geometric inclusion. To explain why this is the case in Example \ref{xmpl:nopen}, in the sense that the inclusion of $\{\infty\}$ induces a geometric inclusion, we show that this intuition is at least valid for complete monoids\footnote{We have not been able to demonstrate it for powder monoids more generally.}.

\begin{thm}
\label{thm:inccomplete}
Let $(M',\tau')$ be a complete monoid, let $M = eM'e$ for some idempotent $e \in M'$, and let $\phi:M \to M'$ be the subsemigroup inclusion. Then the restricted topology $\tau := \tau'|_M$ makes $(M,\tau)$ a complete topological monoid, and hence the induced geometric morphism $\Cont(M,\tau) \to \Cont(M',\tau')$ is a geometric inclusion.
\end{thm}
\begin{proof}
As usual, let $f:\Setswith{M} \to \Setswith{M'}$ be the essential inclusion induced by $\phi$. Consider the hyperconnected-localic factorization of the composite morphism $\Setswith{M} \to \Cont(M',\tau')$:
\[\begin{tikzcd}
{\Setswith{M}} \ar[r, "f"] \ar[d, "h"'] &
{\Setswith{M'}} \ar[d, "h'"] \\
{\Ecal} \ar[r, "g"] &
\Cont(M'{,}\tau').
\end{tikzcd}\]
Since the upper composite is an inclusion followed by a hyperconnected morphism, by \cite[Proposition A4.6.10]{Ele} the lower geometric morphism is an inclusion: the surjection-inclusion and hyperconnected-localic factorizations of the composite coincide.

Moreover, combining Proposition \ref{prop:intrinsic} with Proposition \ref{prop:inclff}, we have that the composite is full and faithful on essential geometric morphisms, and hence the hyperconnected part $\Setswith{M} \to \Ecal$ also is. Thus the complete monoid representing $\Ecal$ has $M$ as its underlying monoid. That the corresponding topology is the restriction topology follows from Theorem \ref{thm:factor}, just as in the proof of Proposition \ref{prop:eMe}.
\end{proof}

\subsection{Monoid Homomorphisms}
\label{ssec:monhom}

Now suppose $\phi:(M,\tau) \to (M',\tau')$ is a continuous \textit{monoid} homomorphism, so that the essential geometric morphism $f:\Setswith{M} \to \Setswith{M'}$ it induces is a surjection, and hence examining the resulting square \eqref{eq:square}, so is the induced morphism $g: \Cont(M,\tau) \to \Cont(M',\tau')$. Combining this observation with Theorem \ref{thm:inccomplete}, we have that:
\begin{thm}
\label{thm:surjinc}
Let $\phi:(M,\tau) \to (M',\tau')$ be a continuous semigroup homomorphism between complete monoids inducing $g: \Cont(M,\tau)\to \Cont(M',\tau')$, and let $e:= \phi(1)$. Then the surjection--inclusion factorization of $g$ is canonically represented by the factorization of $\phi$ into a monoid homomorphism $M \to eM'e$ followed by an inclusion of subsemigroups $eM'e \hookrightarrow M'$, where $eM'e$ is equipped with the subspace topology.
\end{thm}

We can characterize which morphisms arise from continuous monoid homomorphisms, up to fixing sober representing monoids. Since sobriety has not been a focus of our account of monoids in this thesis, we restrict further to powder monoids.

\begin{prop}
\label{prop:phisom}
Let $(M,\tau)$ and $(M',\tau')$ be powder monoids; let $T$, $T'$ be their respective Boolean algebras of clopen sets. A surjective geometric morphism $g:\Cont(M,\tau)\to \Cont(M',\tau')$ is induced by a continuous monoid homomorphism $\phi:(M,\tau) \to (M',\tau')$ if and only if $T' \cong g_*(T)$ in $\Cont(M',\tau')$.
\end{prop}
\begin{proof}
First suppose that $g$ is induced by some continuous monoid homomorphism $\phi$. Consider the internal Boolean algebras $T$ and $T'$ in the respective toposes. Since $\phi$ is a monoid homomorphism, $g$ commutes with the canonical points of $\Cont(M,\tau)$ and $\Cont(M',\tau')$ (which are induced by the unique monoid homomorphisms $1 \to M$ and $1 \to M'$ respectively). In particular, we have a canonical isomorphism $T' \cong g_*(T)$, as required.

Conversely, given such an isomorphism, both $T'$ and $g_*(T)$ represent $\Pcal\op \circ U \circ V$ by Scholium \ref{schl:Tseparator}, from which it follows that $g$ commutes with the canonical points. As such, $g^*(T')$ has an underlying set which can be identified with that of $T'$, and the counit of $g$ at $T$ provides a homomorphism of Boolean algebras $\phi^{-1}: g^*(T') \to T$ in $\Setswith{M}$ which, since powder monoids are sober as spaces, uniquely defines a map $M \to M'$. The fact that $\phi^{-1}$ is an $M$-set homomorphism ultimately ensures that $\phi$ is a monoid homomorphism. Moreover, when we generate $g$ from a monoid homomorphism $\phi$, we find that this is the morphism we recover from $\phi^{-1}$.
\end{proof}

For an injective monoid homomorphism, we have a partial analogue of Theorem \ref{thm:inccomplete}:
\begin{lemma}
\label{lem:surjective}
Suppose $(M',\tau')$ is a powder monoid and $\phi:M \to M'$ is an injective monoid homomorphism. Then the subspace topology $\tau := \tau'|_M$ on $M$ makes $(M,\tau)$ a powder monoid.
\end{lemma}
\begin{proof}
Consider the square \eqref{eq:square} induced by $\phi$. By Proposition \ref{prop:locfaith}, the localic surjection $f$ is faithful on essential geometric morphisms, and by Corollary \ref{crly:powderfaith}, so is $h'$. It follows that $h$ must also be faithful on essential geometric morphisms, and hence that $(M,\tau)$ also is.
\end{proof}

As in Corollary \ref{crly:powderfaith}, we encounter the problem that we have no guarantee that the morphism $g:\Cont(M,\tau) \to \Cont(M',\tau')$ in Lemma \ref{lem:surjective} induced by an injective monoid homomorphism $\phi$ will have a trivial hyperconnected part. However, we can use whatever hyperconnected part there may be to produce a complete monoid, and when the codomain is also a complete monoid, this gives us a canonical factorization of $\phi$, as follows.

\begin{thm}
\label{thm:locextend}
Suppose $(M',\tau')$ is a complete monoid, and $\phi: M \to M'$ is an injective monoid homomorphism. Let $\tau$ be the restriction of $\tau'$ as in Lemma \ref{lem:surjective}. Then there is a complete monoid $(L,\rho)$ and a dense, continuous, injective monoid homomorphism $(M,\tau) \to (L,\rho)$ such that $\phi$ extends to a continuous injection $\psi: (L,\rho) \to (M',\tau')$ inducing a localic surjection $\Cont(L,\rho) \to \Cont(M',\tau')$.
\end{thm}
\begin{proof}
Of course, we define $(L,\rho)$ to be the complete monoid obtained from the hyperconnected part of the composite $h' \circ f$ in the square \eqref{eq:square} induced by $\phi$. Thus we have a diagram of geometric morphisms:
\begin{equation}
\label{eq:Mfactors}
\begin{tikzcd}
{\Setswith{M}} \ar[rr, "f"] \ar[dr, "u"'] \ar[dd, "h"'] & &
{\Setswith{M'}} \ar[dd, "h'"] \\
& {\Setswith{L}} \ar[dd, "k", near start] \ar[ur, dashed] & \\
{\Cont(M{,}\tau)} \ar[rr, "g", near start] \ar[dr, "v"'] & &
{\Cont(M'{,}\tau')}, \\
& {\Cont(L{,}\rho)} \ar[ur, "w"'] &
\end{tikzcd}	
\end{equation}
in which $h$, $h'$, $k$ and $v$ are hyperconnected, while $f$ and $w$ are localic surjections and $f$ and $u$ are essential. It suffices for us to construct the inclusion $\psi: L \to M'$ to provide the (dashed) essential geometric morphism which restricts to $w$.

The homomorphism $\phi$ induces a morphism $M \to f^*(L') = f^*f_!(M)$, the unit of the adjunction $(f_! \dashv f^*)$, whose element-wise action coincides with $\phi$; we abuse notation and call this unit map $\phi$, too.

Define a mapping $t: \Rfrak_{\tau'} \to \Rfrak_{\tau}$ by pullback: for each $r' \in \Rfrak_{\tau'}$, let $t(r')$ be the pullback of $f^*(r')$ along $\phi$. The resulting relation is such that the intermediate principal $M$-set in the epi-mono factorization of $M \to f^*(M') \too f^*(M'/r')$ is precisely $M/t(r')$, so we have $M \too M/t(r') \hookrightarrow f^*(M'/r')$. By construction, $f^*$ sends $(M',\tau')$-sets to $(M,\tau)$-sets lying in $\Cont(L,\rho)$ and this subcategory is closed under subobjects, so $M/t(r')$ is naturally an $(L,\rho)$-set. As such, we obtain maps $L \too M/t(r') \hookrightarrow f^*(M'/r')$ by factoring each $M \too M/t(r')$ through $L = \varprojlim_{r \in \Rfrak_{v \circ h}} M/r$. These assemble into an $M$-set homomorphism
\[\psi: L \to f^*(M') = \varprojlim_{r \in \Rfrak_{\tau'}} f^*(M'/r'),\]
which explicitly sends $\alpha \in L$ to $\psi(\alpha) = ([\phi(a_{t(r')})])_{r' \in \Rfrak_{\tau'}}$. It remains to show that this map underlies a monoid homomorphism. As such, observe that for $a \in M$ and $r' \in \Rfrak_{\tau'}$:
\begin{align*}
t(\phi(a)^*(r')) & =
t\left( \{(x,y) \in M' \times M' \mid (\phi(a)x,\phi(a)y) \in r' \} \right) \\ & =
\{(m,n) \in M \times M \mid (\phi(a)\phi(m),\phi(a)\phi(n)) \in r' \} \\ & =
a^* \left( \{(m,n) \in M \times M \mid (\phi(m),\phi(n)) \in r' \} \right) \\ & =
a^*(t(r')),
\end{align*}
where we have omitted each instance of $f^*$. Therefore, given $\alpha,\beta \in L$, we have
\begin{align*}
\psi(\alpha\beta) & =
([\phi(a_{t(r')}b_{a_{t(r')}^*(t(r'))})]) \\ & =
([\phi(a_{t(r')})\phi(b_{t(\phi(a_{t(r')})^*(r'))})]) \\ & =
\psi(\alpha)\psi(\beta),
\end{align*}
as required. Preservation of the unit follows from our assumption that $\phi$ was a monoid homomorphism. To demonstrate continuity, observe that if $U = \pi_{r'}^{-1}(\{[m']\})$ is a basic open set in $M'$, then $\psi^{-1}(U)$ is exactly the basic open set $\pi_{t(r')}^{-1}(\{[m]\})$ in $L$ if there is some $m \in M$ with $(\phi(m),m') \in r'$, and is empty otherwise.

Having shown that it exists, we immediately have that $\psi$ is the unique continuous monoid homomorphism making the desired triangle commute, since dense inclusions of Hausdorff spaces are epimorphisms in the category of such spaces.
\end{proof}

The construction of the factoring map $\psi$ in the above relies on the expression of $L$ as a limit. It is useful to have a more topological characterization, for which we need some further preliminary results.

\begin{lemma}
\label{lem:densects}
Let $\phi: (M,\tau) \to (M',\tau')$ be a continuous monoid homomorphism whose image is dense. Then the geometric morphism $g: \Cont(M,\tau) \to \Cont(M',\tau')$ induced by $\phi$ is hyperconnected.
\end{lemma}
\begin{proof}
First, observe that without loss of generality we may assume $\phi$ is a dense inclusion of monoids. Indeed, $\phi$ always factors as a surjective monoid homomorphism followed by a dense inclusion of monoids, and the former factor induces a hyperconnected essential morphism $f$ at the level of the presheaf toposes, whence by consideration of the square \eqref{eq:square}, $g$ must also be hyperconnected. As such, we identify $M$ with its image in $M'$.

Given an $M$-set homomorphism $s: g^*(X) \to g^*(Y)$, $x \in g^*(X)$ and $m' \in M'$, let $m \in M \cap \Ical_x^{m'} \cap \Ical_{s(x)}^{m'}$. Then we have $s(x \cdot m') = s(x \cdot m) = s(x) \cdot m = s(x) \cdot m'$, whence $s$ is an $M'$-set homomorphism, and so $g^*$ is full; it is always faithful when $\phi$ is a monoid homomorphism.

Moreover, the image of $g^*$ is closed under subobjects: given a sub-$M$-set $A \hookrightarrow g^*(Y)$ and $m' \in M'$, for each $y \in A$ we have some $m \in M \cap \Ical_y^{m'}$, whence $y \cdot m = y \cdot m' \in A$, and hence $A$ is a sub-$M'$-set. Altogether, this ensures that $g$ is hyperconnected, as claimed.
\end{proof}

\begin{prop}
\label{prop:denseisom}
Suppose $\phi:(M,\tau) \to (M',\tau')$ is a monoid homomorphism between topological monoids inducing an equivalence $\Cont(M,\tau) \simeq \Cont(M',\tau')$, that $(M,\tau)$ is a complete monoid and that $(M',\tau')$ is a powder monoid. Then $\phi$ is an isomorphism.
\end{prop}
\begin{proof}
Since an equivalence is full and faithful on geometric morphisms, in the square \eqref{eq:square} induced by $\phi$ is follows that $f$ is faithful on essential geometric morphisms and $h$ is full on those in the image of $f$. Since $\phi$ is a monoid homomorphism, it commutes with the canonical points of $\Setswith{M}$ and $\Setswith{M'}$, whence the latter point is in the image of $f$, and hence ($h$ being faithful on essential geometric morphisms by assumption) $M'$ is a complete monoid, and $h$ is full and faithful on essential geometric morphisms. It follows that $f$ is also full and faithful on essential geometric morphisms, and so $\phi$ (the restriction of $f \circ -$ to the canonical point of $\Setswith{M}$) is an isomorphism, as claimed.
\end{proof}

\begin{crly}
\label{crly:closure}
The monoid $(L,\rho)$ constructed in Theorem \ref{thm:locextend} is the closure of $(M,\tau)$ in $(M',\tau')$.
\end{crly}
\begin{proof}
Certainly $(M,\tau)$ is dense in $(L,\rho)$. Consider the dense-closed factorization of $\psi: (L,\rho) \to (M',\tau')$. If the dense part is non-trivial, by Lemma \ref{lem:densects} it produces a hyperconnected factor of the geometric morphism induced by $\psi$, and hence must be an equivalence. The intermediate monoid is a powder monoid by Lemma \ref{lem:surjective}, and thus the dense part is an isomorphism by Proposition \ref{prop:denseisom}.
\end{proof}

\begin{crly}
\label{crly:closed}
Any monoid which is a closed subsemigroup of a complete monoid is complete.
\end{crly}
\begin{proof}
Applying the factorization of Theorem \ref{thm:locextend} to the inclusion of a closed submonoid, the submonoid is dense in the complete intermediate monoid and so, being closed, must coincide with it, whence it is complete by Corollary \ref{crly:closure}. Combining this with Proposition \ref{prop:eMe}, the result follows.
\end{proof}

To summarize, we have that:
\begin{thm}
\label{thm:hypeloc}
Let $\phi:(M,\tau) \to (M',\tau')$ be a continuous semigroup homomorphism between complete monoids inducing $g: \Cont(M,\tau)\to \Cont(M',\tau')$. Then the hyperconnected--localic factorization of $g$ is canonically represented by the dense-closed factorization of $\phi$.
\end{thm}

\subsection{Morita Equivalence}
\label{ssec:Morita}

In Sections \ref{sec:montop} and \ref{sec:surjpt} we saw how an arbitrary monoid can be reduced to a powder monoid and then extended to a complete monoid without changing its topos of actions (up to canonical equivalence). Proposition \ref{prop:denseisom} above demonstrates that a continuous monoid homomorphism $\phi$ between complete monoids induces an equivalence if and only if $\phi$ is an isomorphism, just as in the discrete case. Thus, as far as Morita equivalence for complete monoids \textit{via semigroup homomorphisms} goes, we are reduced to considering subsemigroups of the form $eM'e \hookrightarrow M'$.

First, observe that Morita equivalences of discrete monoids descend to complete topologies on those monoids. We shall discuss this further in Conjecture \ref{conj:inclusion} below.
\begin{schl}
\label{schl:Morita3}
Let $(M',\tau')$ be a complete monoid, and suppose $e \in M'$ is an idempotent such that the inclusion $\iota: M := eM'e \hookrightarrow M'$ induces an equivalence $\Setswith{M} \simeq \Setswith{M'}$. Let $\tau$ be the restriction topology on $M$ from Theorem \ref{thm:inccomplete}. Then $\iota$ induces an equivalence $\Cont(M,\tau) \simeq \Cont(M',\tau')$.
\end{schl}
\begin{proof}
As above, we use the notation of \eqref{eq:square}. We know from the proof of Theorem \ref{thm:inccomplete} that in this situation, $g: \Cont(M',\tau') \to \Cont(M,\tau)$ is the inclusion part of the hyperconnected-inclusion factorization of $h' \circ f$, so when $f$ is an equivalence, the inclusion part must be trivial. That is, $g$ is itself an equivalence.
\end{proof} 

In summary, we have:
\begin{thm}
\label{thm:Moritaphi}
Let $\phi:(M,\tau) \to (M',\tau')$ be a continuous semigroup homomorphism between complete monoids. Then if $\phi$ induces an equivalence $g:\Cont(M,\tau) \to \Cont(M',\tau')$, then $M \cong eM'e$ for some idempotent $e \in M'$, $\phi$ is the canonical inclusion of subsemigroups, and $\tau$ is the restriction of $\tau'$ along this inclusion.
\end{thm}
\begin{proof}
Factoring $\phi$ as a monoid homomorphism followed by an inclusion of semigroups, the induced geometric morphisms between toposes of continuous actions must also be equivalences, whence the former must be an isomorphism by Proposition \ref{prop:phisom}. The topology on the intermediate monoid is the restriction topology by \ref{thm:surjinc}.
\end{proof}

We would have liked to extend the $2$-equivalence between the $2$-category of discrete monoids and the $2$-category of their toposes of actions from \ref{chap:TDMA}
by characterizing the geometric morphisms arising from continuous semigroup homomorphisms. However, this is not possible because, unlike in the discrete case, \textit{not all equivalences of toposes of topological monoid actions lie in the image of $\Cont(-)$}. Since all equivalences have indistinguishable categorical properties, but only some such equivalences are induced by semigroup homomorphisms, we cannot hope for an intrinsic characterization of geometric morphisms lying in the image of $\Cont(-)$.

\begin{xmpl}
\label{xmpl:Schanuel}
Consider once again the Schanuel topos of Examples \ref{xmpl:nonpresh} and \ref{xmpl:notpro2}. Let $X$ be $\mathbb{N}$ or $\mathbb{R}$. In either case, the monoid $(\End_{\mathrm{mono}}(X),\tau_{\mathrm{fin}})$ representing the Schanuel topos has no non-identity idempotents, since $e^2(x) = e(x)$ implies $e(x) = x$ by injectivity, so any semigroup homomorphism in either direction must be a monoid homomorphism. Since the two representing monoids are complete and non-isomorphic, no such homomorphism can induce the equivalence of toposes
\[ (\End_{\mathrm{mono}}(\Nbb),\tau_{\mathrm{fin}}) \simeq
(\End_{\mathrm{mono}}(\Rbb),\tau_{\mathrm{fin}}). \]
\end{xmpl}

\begin{rmk}
Recall from Theorem \ref{thm:2equiv1} of Chapter \ref{chap:TDMA} that general geometric morphisms between toposes of discrete monoid actions $\Setswith{M} \to \Setswith{M'}$ correspond to flat-left-$M'$-right-$M$-sets. The conclusion of the above is that, to fully understand Morita equivalences of toposes of topological monoid actions, a necessary next step is to investigate the analogous class of biactions for topological monoids. We leave this effort to future work.
\end{rmk}

\subsection{Reflective Categories of Topological Monoids}
\label{ssec:monads}

In spite of continuous semigroup homomorphisms not capturing the full richness of geometric morphisms, they do nonetheless produce well-behaved ($2$-)categories of monoids. We show in this section how the classes of monoids we have discussed so far form reflective subcategories of the category of monoids with topologies from which we began.

Let $\mathrm{MonT}_s$, $\mathrm{TMon}_s$, $T_0\mathrm{Mon}_s$, $\mathrm{PMon}_s$, $\mathrm{CMon}_s$ respectively be the $2$-categories of monoids equipped with topologies, topological monoids, $T_0$ topological monoids, \textit{right} powder monoids, and complete monoids, all equipped with continuous semigroup homomorphisms as their $1$-morphisms and conjugations as their $2$-morphisms. We have $2$-functors:
\begin{equation}
\label{eq:monadic}
\begin{tikzcd}
\mathrm{MonT}_s & \ar[l, "G_1"] \mathrm{TMon}_s & \ar[l, "G_2"] \mathrm{T_0Mon}_s & \ar[l, "G_3"] \mathrm{PMon}_s & \ar[l, "G_4"] \mathrm{CMon}_s;
\end{tikzcd}
\end{equation}
all of these subcategories are full on $1$- and $2$-morphisms. In the following results, we demonstrate that all of these functors have adjoints, which makes them reflective ($2$-)subcategories. Recall that for strict $2$-functors $F:\Ccal \to \Dcal$ and $G:\Dcal \to \Ccal$ to form a strict $2$-adjunction $(F \dashv G)$, we require there to be isomorphisms of categories
\[ \Hom_{\Dcal}(FX,Y) \cong \Hom_{\Ccal}(X,GY),\]
natural in $X$ and $Y$. Since the data of a conjugation consists of an element of the codomain monoid, which is not affected by any of the $G_i$, it suffices to prove that the $G_i$ have left adjoints as $1$-functors; preservation of the $2$-morphisms will then be automatic. As such, the functors constructed in the results below are informally referred to as adjoints.

Note that all of the following results also hold when we restrict to the subcategories of monoid homomorphisms, since all of the units of the adjunctions we construct are continuous monoid homomorphisms.

\begin{lemma}
\label{lem:G1}
The functor $G_1: \mathrm{TMon}_s \to \mathrm{MonT}_s$ has a left adjoint.
\end{lemma}
\begin{proof}
Given a monoid with a topology $(M,\tau)$, we can inductively define sub-topologies $\tau_i$ by letting $\tau_0 = \tau$ and $\tau_{i+1}$ consisting of those open subsets $U \in \tau_i$ such that $\mu^{-1}(U)$ is open in $\tau_i \times \tau_i$. Then letting $\tau_{\infty} := \bigcap_{i=0}^\infty \tau_i$, we claim that $(M,\tau_{\infty})$ is a topological monoid. Indeed, given $U \in \tau_{\infty}$, $\mu^{-1}(U) \in \tau_{\infty} \times \tau_{\infty}$ by construction. This is clearly the finest topology contained in $\tau$ with respect to which the multiplication on $M$ is continuous. We could alternatively define $\tau_{\infty}$ as the collection of $U \in \tau$ such that $\mu^{-k}(U)$ (which is a well-defined subset of $M^{k+1}$ by associativity of multiplication) is open in $\tau \times \cdots \times \tau$ ($k+1$ times) for every positive integer $k$.

The identity homomorphism $(M,\tau) \to (M,\tau_{\infty})$ is automatically continuous. Given any topological monoid $(M',\tau')$ and continuous semigroup homomorphism $\phi: (M,\tau) \to (M',\tau')$, since $\phi$ commutes with multiplication we have that for each $U' \in \tau'$,
\[\mu^{-k}(\phi^{-1}(U')) = \phi^{-1}(\mu'{}^{-k}(U'))\]
is open in $\tau$ by continuity of $\phi$, whence $\phi^{-1}(U')$ is a member of $\tau_{\infty}$. Thus $\phi$ factors (uniquely) through $(M,\tau) \to (M,\tau_{\infty})$, as required to make this map the unit of an adjunction.
\end{proof}

\begin{lemma}
\label{lem:G2}
The functor $G_2: T_0\mathrm{Mon}_s \to \mathrm{TMon}_s$ has a left adjoint.
\end{lemma}
\begin{proof}
If $(M,\tau)$ is a topological monoid, then the equivalence relation $\sim$ on $M$ identifying topologically indistinguishable elements is necessarily a two-sided congruence, since given $m_1 \sim m'_1$, $m_2 \sim m'_2$ and a neighbourhood $U$ of $m_1m_2$, we have that $\mu^{-1}(U)$ contains an open rectangle $U_1 \times U_2$ with $m_i \in U_i$, whence $(m'_1,m'_2) \in U_1 \times U_2$ and hence $m'_1m'_2 \in U$. Thus the quotient map $(M,\tau) \to (M/{\sim}, \tau)$ is a continuous semigroup homomorphism.

Given a continuous semigroup homomorphism $\phi$ from $(M,\tau)$ to a $T_0$ topological monoid $(M',\tau')$ and given $m \sim m'$, observe that $\phi(m)$ and $\phi(m')$ must be topologically indistinguishable and hence equal in $(M',\tau')$ by continuity. Thus $\phi$ factors through the quotient map above, as required.
\end{proof}

\begin{lemma}
\label{lem:G3}
The functor $G_3: \mathrm{PMon}_s \to T_0\mathrm{Mon}_s$ has a left adjoint. 
\end{lemma}
\begin{proof}
The construction of the powder monoid $(\tilde{M},\tilde{\tau})$ associated to $(M,\tau)$ in Theorem \ref{thm:Hausd} (via the construction of the action topology in Theorem \ref{thm:tau}) defines the value of the adjoint functor on objects and provides a candidate for the unit in the quotient homomorphism $(M,\tau) \to (\tilde{M},\tilde{\tau})$.

Let $\phi:(M,\tau) \to (M',\tau')$ be a continuous semigroup homomorphism with $(M',\tau')$ a powder monoid. Given $U' \in T'$ and $p \in M$, consider
\[\phi^{-1}(\Ical_{U'}^{\phi(p)}) = \{q \in M \mid \phi(q)^*(U') = \phi(p)^*(U')\},\]
which is clearly contained in
\begin{align*}
\Ical_{\phi^{-1}(U')}^p &= \{q \in M \mid q^*(\phi^{-1}(U')) = p^*(\phi^{-1}(U'))\}\\
&= \{q \in M \mid \phi^{-1}(\phi(q)^*(U')) = \phi^{-1}(\phi(p)^*(U'))\}.
\end{align*}
Since the former is open in $\tau$, so is the latter, and hence $\phi^{-1}(U') \in T$. Since any open in $\tau'$ is a union of basic opens in $T'$, we conclude that $\phi$ factors through $(M,\tilde{\tau})$, and hence through $(\tilde{M},\tilde{\tau})$ by the proof of Lemma \ref{lem:G2}, as required.
\end{proof}

For the fourth result, we recycle the proof of Theorem \ref{thm:locextend}.
\begin{schl}
\label{schl:G4}
The functor $G_4: \mathrm{CMon}_s \to \mathrm{PMon}_s$ has a left adjoint. 
\end{schl}
\begin{proof}
Given a powder monoid $(M,\tau)$, a complete monoid $(L',\rho')$ and a semigroup homomorphism $g:(M,\tau) \to (L',\rho')$, we must show that $g$ factors uniquely through the canonical monoid homomorphism $u: (M,\tau) \to (L,\rho)$, where the latter is the completion of $(M,\tau)$.

As in the proof of Theorem \ref{thm:locextend}, we may assume $\phi: M \to L'$ is a monoid homomorphism, since we may factor $\phi$ as a monoid homomorphism followed by an inclusion of subsemigroups, and the intermediate monoid is canonically a complete monoid with the restriction topology by Theorem \ref{thm:inccomplete}. We then construct the factoring $M$-set homomorphism $\psi: L \to f^*(L')$ just as in the proof of Theorem \ref{thm:locextend}, which did not depend on injectivity of $\phi$.
\end{proof}

Our results from previous sections demonstrate that the units of these four adjunctions all induce equivalences at the level of toposes of continuous actions of monoids.

Recalling the asymmetry in the definition of powder monoids (see the comments after Definition \ref{dfn:powder}), we briefly consider the ($2$-)categories of left powder monoids and left complete monoids.

For this purpose, we employ the dual of the notation introduced in Section \ref{ssec:necessary}, writing ${}_x^p\Ical$ for the necessary clopen associated to an element $x$ in a left $M$-set $X$ and $p \in M$. Then we obtain a complementary result to Lemma \ref{lem:InA2}.
\begin{lemma}
\label{lem:lrclopen}
Let $U$ be a subset of $M$. Let $A = \Ical_U^p$ and $B = {}_U^p\Ical$. Then $\Ical_B^p = {}_A^p\Ical$. In particular, a two-sided powder monoid has a base of clopens expressible in this coincident form.
\end{lemma}
\begin{proof}
After expanding the definitions, we find that both subsets are equal to the set of $q \in M$ such that for all $w,z \in M$, $q \in w^*(U){}^*z$ if and only if $p \in w^*(U){}^*z$ (this is reminiscent of a double-coset construction).
\end{proof}

\begin{schl}
\label{schl:G3'}
The inclusion $G'_3$ of the sub-$2$-category $\mathrm{P'Mon}_s$ of \textbf{left} powder monoids into $T_0\mathrm{Mon}_s$ has a left adjoint. The idempotent monads $P$ and $P': T_0\mathrm{Mon}_s \to T_0\mathrm{Mon}_s$ induced by $G_3$ and $G'_3$ respectively commute, in the sense that $PP' = P'P$, whence the ($2$-)category $\mathrm{P''Mon}_s$ of two-sided powder monoids is also a reflective subcategory of $T_0\mathrm{Mon}_s$.
\end{schl}
\begin{proof}
The first part is clear by inspection of the proof of Lemma \ref{lem:G3}, which can be dualized without difficulty. To see that $PP' = P'P$, fix a ($T_0$) monoid $(M,\tau)$ and consider the left action topology associated to the right action topology associated to $\tau$. This is generated by those $U \in \tilde{\tau}$ such that for every $p \in M$, the subset ${}_U^p\Ical$ is in $\tilde{\tau}$ (and hence, being clopen, in $T$). It is necessary and sufficient to verify that for each fixed $p$ that $B := {}_U^p\Ical$ has $\Ical_B^p \in \tau$. By Lemma \ref{lem:lrclopen}, letting $A := \Ical_U^p$, we have that $\Ical_B^p = {}_A^p\Ical$, whence we have $A$ in the left action topology associated to $\tau$ and hence $U$ is in the right action topology associated to this. In summary, the topologies obtained by applying $P$ and $P'$ in either order are the same.
\end{proof}

\begin{schl}
\label{schl:G4'}
The inclusion $G'_4$ of the sub-$2$-category $\mathrm{C'Mon}_s$ of \textbf{left} complete monoids into $\mathrm{P'Mon}_s$ has a left adjoint.
\end{schl}
\begin{proof}
The proof proceeds exactly as for Scholium \ref{schl:G4}.
\end{proof}

With Scholia \ref{schl:G3'} and \ref{schl:G4'} we may extend the diagram of monadic functors \eqref{eq:monadic} as follows,
\begin{equation}
\label{eq:monadic2}
\begin{tikzcd}
& \ar[dl, "G_3"] \mathrm{PMon}_s & & \ar[ll, "G_4"'] \mathrm{CMon}_s \\
\mathrm{T_0Mon}_s & & \ar[dl, "G_3"'] \ar[ul, "G'_3"] \mathrm{P''Mon}_s \\
& \ar[ul, "G'_3"'] \mathrm{P'Mon}_s & & \ar[ll, "G'_4"] \mathrm{C'Mon}_s.
\end{tikzcd}
\end{equation}
Clearly, there is more of this picture to fill in; see Conjecture \ref{conj:complete} in the Conclusion.

We could have included further reflective subcategories of $\mathrm{TMon}_s$ in this section, such as the category of zero-dimensional monoids, but we hope the examples we have included are sufficiently illustrative. Combined with the fact that the category of topological monoids is `crudely' monadic over the category $\Top$ of topological spaces, we obtain that, upon restricting to monoid homomorphisms and ignoring conjugations, each of the underlying $1$-categories is monadic over $\Top$, which hints at a concrete algebraic way to study complete monoids as algebras for the resulting monad.

%% file: The_TSGT.tex
\chapter{Topological semi-Galois Theory}
\label{chap:TSGT}

Our final task is to combine the results of the last two chapters, to extend results appearing in \cite{TGT}. Our approach is to use the characterization of toposes of topological monoid actions as those which are hyperconnected under a topos of discrete monoid actions. Indeed, we shall see that any point of a topos $\Ecal$ provides a geometric morphism from a candidate topos of discrete monoid actions, and given a site for $\Ecal$ we may thus apply the results of \cite{Dense} to establish when that morphism is hyperconnected.

Our approach is a little different from Caramello's. In \cite{TGT} they use established syntactic results regarding theories classified by Boolean toposes in order to extract a syntactic result first, which is then specialized to produce categorical conditions. While we examine necessary conditions on theories classified by toposes of monoid actions here, we rely more heavily on insights coming from the theory of accessible categories. This demonstrates the versatility of a topos-theoretic approach: we have many tools at our disposal for tackling these problems.

\subsection*{Overview}

In Section \ref{sec:setup} we record how an arbitrary point of Grothendieck toposes has a canonical factorization through a topos of actions of a discrete monoid (Proposition \ref{prop:coreflect}); we use this to extract a general characterization of when a topos of sheaves on a site is equivalent to a topos of topological monoid actions (Theorem \ref{thm:basic}). In Section \ref{sec:cases} we apply this Theorem, first to the case of a geometric syntactic site, then to principal sites. We interrupt this analysis in Section \ref{ssec:wfs} to expose some relevant results about factorization systems, enabling us to re-express the characterization of the point in terms of its properties of as an object of the inductive completion of the dual of the category underlying a principal site with a compatible factorization system (Theorem \ref{thm:qhomogen}). We go on to apply this formulation to the special cases of sites derived from the supercompactly generated theories of Chapter \ref{chap:logic} and then the reductive categories of Chapter \ref{chap:sgt}.

In Section \ref{sec:Fraisse}, we generalize the original \Fraisse construction to demonstrate that, subject to a countability criterion, the necessary conditions on a principal site with a compatible factorization system are also sufficient (Theorem \ref{thm:Fraisse1}), and in fact the resulting point is weakly initial in the category of points with the right properties (Proposition \ref{prop:Fraisse2}).

Finally, in Section \ref{sec:sgal} we explain how these results yield a `semi-Galois theory,' by observing that a principal site with a compatible factorization system admits a canonical anafunctor to the category of open right congruences on the monoid corresponding to a point with the right properties, and we give some examples.

\section{Setup}
\label{sec:setup}

\subsection{Pointed toposes}
\label{ssec:factorpts}

An observation that we used only implicitly in the last chapter, but which we make explicit now, is the following:
\begin{prop}
\label{prop:coreflect}
Consider the $1$-category whose objects are toposes of actions of discrete monoids equipped with their canonical point, and whose morphisms are the (surjective) essential geometric morphisms which commute with these points. This $1$-category is coreflective (up to equivalence) in the $1$-category of pointed Grothendieck toposes and all geometric morphisms.
\end{prop}
\begin{proof}
Given a pointed topos $(\Ecal,p:\Set \to \Ecal)$, let $L = \End(p)\op$. Then there is a canonical morphism $q: \Setswith{L} \to \Ecal$ such that $p$ is the composite of $q$ with the canonical point of $\Setswith{L}$. Indeed, since each element of $L\op$ is a natural transformation $p^* \Rightarrow p^*$, their components at an object $X$ of $\Ecal$ provide a left action of $L\op$, and hence a right action of $L$, on $p^*(X)$, and naturality ensures that morphisms $X \to Y$ are sent to right $L$-set homomorphisms. The fact that the resulting functor $\Ecal \to \Setswith{L}$ preserves finite limits and small colimits follows from the fact that $p^*$ does so and the forgetful functor $\Setswith{L} \to \Set$ creates them, so we obtain a geometric morphism $q_{\Ecal,p}:\Setswith{L} \to \Ecal$.

Now suppose we are given a second pointed topos $(\Fcal,p':\Set \to \Fcal)$ and a geometric morphism $h : \Fcal \to \Ecal$ commuting with the points $p,p'$. Let $L'$ be $\End({p'}^*)\op$. Then for any natural transformation $\alpha:{p'}^* \Rightarrow {p'}^*$, we have (up to isomorphism) an induced endomorphism $\alpha_{h^*}: p^* \Rightarrow p^*$ by restriction along the inverse image $h^*$. This restriction operation is compatible with composition and preserves the identity transformation by inspection, whence it is a monoid homomorphisms $L' \to L$. This induces an essential surjection $Eh:\Setswith{L} \to \Setswith{L'}$ which commutes with the canonical points. Thus, since these points are surjective, we conclude that this essential geometric morphism completes the required commutative square up to isomorphism, which demonstrates naturality of the transformation $q_{-,-}$:
\begin{equation}
\label{eq:qnatural}
\begin{tikzcd}
& \Setswith{L'} \ar[r, "q_{\Fcal,p'}"] \ar[dd,"Eh"'] & \Fcal \ar[dd, "h"] \\
\Set \ar[ur] \ar[dr] & & \\
& \Setswith{L} \ar[r, "q_{\Ecal,p}"'] & \Ecal.
\end{tikzcd}	
\end{equation}
This is compositional by inspection.

For universality, consider any morphism of pointed toposes $k:\Setswith{M} \to \Ecal$, and let $u$ be the canonical point of $\Setswith{M}$, so that $p = k \circ u$. We know from Remark \ref{rmk:endopt} of Chapter \ref{chap:TDMA} that the monoid of endomorphisms of $u$ is $M\op$, whence $q_{\Setswith{M},u}$ is the identity on $\Setswith{M}$. As such, $k = q_{\Ecal,p} \circ Ek$ as a special case of \eqref{eq:qnatural}.
\end{proof}

\begin{rmk}
Following the theme of $2$-categorical results discussed up to now in this thesis, one might like to upgrade this to a sensible $2$-categorical result. We expect this should be possible, but there is an obstacle.

Given geometric morphisms $g,h:\Fcal \rightrightarrows \Ecal$ and a natural transformation $\beta: g^* \Rightarrow h^*$, we can apply the functor $q_{\Fcal,p'}^*$ and translate across the square in \eqref{eq:qnatural} to obtain a transformation $\beta':(Eg)^*q_{\Ecal,p}^* \Rightarrow (Eh)^*q_{\Ecal,p}^*$. Extracting a natural transformation $(Eg)^* \Rightarrow (Eh)^*$ requires appealing to $2$-categorical properties of the morphism $q_{\Ecal,p}$ related to those discussed in Section \ref{ssec:intrinsic} of Chapter \ref{chap:TTMA}. Since we shall only need to use the construction of the morphisms $q_{\Ecal,p}$ in this chapter, we shall not attempt to extract those properties to the point of achieving this more complete result here.
\end{rmk}

As a consequence of Theorem \ref{thm:characterization} from the last chapter, \textit{a topos $\Ecal$ is equivalent to a topos of topological monoid actions if and only if, for some choice of point $p:\Set \to \Ecal$, the geometric morphism $q_{\Ecal,p}$ from Proposition \ref{prop:coreflect} is hyperconnected}. We use this characterization to obtain site-theoretic characterizations of these toposes and by extension syntactic characterizations in terms of theories they classify.

\subsection{Points of sheaf toposes}
\label{ssec:shfpoint}

Suppose $\Ecal \simeq \Sh(\Ccal,J)$ for some site $(\Ccal,J)$. As usual, we write $\ell:\Ccal \to \Sh(\Ccal,J)$ for the composite of the Yoneda embedding with the $J$-sheafification functor. The definition of morphisms of sites (which we recalled in Definition \ref{dfn:morsite} of Chapter \ref{chap:sgt}) is precisely what is required to produce a correspondence,
\[\Geom(\Fcal,\Ecal) \simeq \Site((\Ccal,J),(\Fcal,J_{can})),\]
for any Grothendieck topos $\Fcal$, where the correspondence identifies a geometric morphism $f$ with the morphism of sites $C \mapsto f^*(\ell(C))$. In particular, we can identify a point of $\Ecal$ with a morphism of sites $(\Ccal,J) \to (\Set,J_{can})$, often called a \textit{$J$-continuous flat functor}, which is to say a flat functor $\Ccal \to \Set$ sending $J$-covers to jointly epimorphic families. Note that this correspondence is an equivalence of categories, which is to say that given $p,q:\Fcal \rightrightarrows \Ecal$ the geometric transformations $p^* \Rightarrow q^*$ correspond to natural transformations between the corresponding flat functors.

Given a point $p:\Set \to \Ecal$ with endomorphism monoid $L\op$, the morphism of sites $F$ corresponding to the induced geometric morphism $q_{\Ecal,p}:\Setswith{L} \to \Ecal$ is the one sending an object $C$ to $p^*(\ell(C))$ with the induced right $L$-action. We quote the following result in order to apply it to the present situation.
\begin{prop}[{\cite[Proposition 6.23]{Dense}}]
\label{prop:sitehype}
Let $(\Ccal,J)$ be a small-generated site, $\Fcal$ a Grothendieck topos and $F : \Ccal \to \Fcal$ (the underlying functor of) a morphism of sites. Then the geometric morphism $f : \Fcal \to \Sh(\Ccal,J)$ induced by $F$ is hyperconnected if and only if the following two conditions hold:
\begin{itemize}
	\item $F$ is cover-reflecting with respect to the canonical topology on $\Fcal$, and
	\item for every subobject $A \hookrightarrow F(C)$ in $\Fcal$ there exists a $J$-closed sieve $S$ on $C$ such that $A$ is the union of the images of the arrows $Fh$ for $h \in S$.
\end{itemize}
\end{prop}
By \cite[Theorem 6.1(i)]{Dense}, the first condition is equivalent to the geometric morphism $f$ being a surjection, while the latter condition amounts to the image of $f^*$ being closed under subobjects, which together are necessary and sufficient for $f$ to be hyperconnected, by \cite[Proposition A4.6.6]{Ele}.

We apply this to the $F$ described above, with $\Fcal = \Setswith{L}$. Observe that a family of morphisms $\{A_i \to C \mid i \in I\}$ is sent to a jointly epic family by $F$ if and only if the original point $p^*$ sends it to a jointly epic family, which is to say we can ignore the $L$-action for the first condition. As such, we have the following maximally general result, which we shall refine in the next section.

\begin{thm}
\label{thm:basic}
A topos $\Ecal = \Sh(\Ccal,J)$ is equivalent to a topos of actions of a topological monoid if and only if there exists a $J$-continuous flat functor $F:\Ccal \to \Set$ such that,
\begin{itemize}
	\item $F$ reflects jointly epimorphic families to $J$-covering families, and
	\item letting $L$ be the opposite of the monoid of natural endomorphisms of $F$ and equipping each $F(C)$ with its canonical right $L$-action, given any sub-$L$-set $A$ of $F(C)$ there exists a $J$-closed sieve $S$ on $C$ such that $A$ is the union of the images of the arrows $Fh$ for $h \in S$.
\end{itemize}
Given such a functor, $\Ecal \simeq \Cont(L,\rho)$, where $\rho$ is the coarsest topology making the actions of $L$ on the objects $F(C)$ continuous; in other words the topology is generated by the necessary clopens of the form,
\begin{equation}
\label{eq:endopens}
\Ical_x^m := \{m' \in L \mid x \cdot m = x \cdot m'\},\	
\end{equation}
ranging over $x \in F(C)$, $m \in L$.
\end{thm}
\begin{proof}
The conditions are a direct translation of Proposition \ref{prop:sitehype}. To see that the stated topology coincides with the topology on $L$ from the previous chapter, we observe that every $F(C)$ must be continuous with respect to that topology, and any union or quotient of continuous actions is continuous, so this is the coarsest topology making all objects of $\Sh(\Ccal,J)$ continuous $L$-sets.
\end{proof}

We saw in Chapter \ref{chap:logic} that flat functors on $\Ccal$ can be identified with objects of $\Ind(\Ccal\op)$, the free cocompletion of $\Ccal\op$ under filtered colimits. The point of $\Setswith{\Ccal}$ corresponding to an object $U$ of $\Ind(\Ccal\op)$ is the restriction of the representable functor $C \mapsto \Hom_{\Ind(\Ccal\op)}(C,U)$. As such, a $J$-continuous flat functor on $\Ccal$ corresponds to an object $U$ of $\Ind(\Ccal\op)$ such that $\Hom_{\Ind(\Ccal\op)}(-,U)$ sends $J$-covering families to jointly epimorphic families. In this formulation, $L$ is the opposite of the monoid of endomorphisms of $U$.

In Section \ref{ssec:wfs}, we shall characterize the objects $U$ of $\Ind(\Ccal\op)$ which correspond to points of $\Sh(\Ccal,J)$ that satisfy the conditions of Theorem \ref{thm:basic} in the special case that $J = J_{\Tcal}$ is a principal topology satisfying some extra conditions. This will enable us to arrive at a characterization of these toposes which doesn't explicitly involve the monoid action.

\section{Cases of interest}
\label{sec:cases}

\subsection{Syntactic sites}
\label{ssec:syntactic2}

Throughout, we write $\Ccal_{\Tbb}$ for the (geometric) syntactic category of a geometric theory $\Tbb$ over a signature $\Sigma$, as constructed in Chapter \ref{chap:logic}, and we write $\Set[\Tbb] := \Sh(\Ccal_{\Tbb},J_{\Tbb})$ for the classifying topos of $\Tbb$.

Let $\Mbb$ be a model of $\Tbb$, which corresponds to some point $p: \Set \to \Set[\Tbb]$; the $\Tbb$-model endomorphisms of $\Mbb$ correspond to the endomorphisms of $p$. Letting $L := \End(\Mbb)\op$, the morphism of sites $(\Ccal_{\Phi},J_{\Tcal}) \to \Setswith{L}$ corresponding to $q_{\Set[\Tbb],p}$ sends $\{\vec{x}_i \cdot \phi_i \}$ to its interpretation $[[\vec{x}_i \cdot \phi_i]]_{\Mbb}$, with the action induced by restriction; recall from Chapter \ref{chap:logic} that the interpretations of geometric formulae-in-context are preserved by any $\Tbb$-model homomorphism.

Applying Theorem \ref{thm:basic} to the syntactic category of $\Tbb$, the second condition simplifies rather conveniently.
\begin{crly}
\label{crly:generaltheory}
A theory $\Tbb$ is classified by a topos of topological monoid actions if and only if there exists a model $\Mbb$ of $\Tbb$ in $\Set$ such that the following two conditions are verified:
\begin{itemize}
	\item whenever a family of $\Tbb$-provably functional formulae
	\[\left\{[\theta_i]: \{\vec{x} \cdot \phi_i\} \to \{\vec{x} \cdot \phi\} \mid i \in I \right\}\]
	in $\Ccal_{\Tbb}$ have interpretations $[\theta_i]_{\Mbb}$ which are jointly epimorphic in $\Set$, the sequent $\phi \vdash_{\vec{x}} \bigvee_{i \in I} (\exists \vec{y})\theta_i$ is provable in $\Tbb$, and
	\item letting $L := \End_{\Tbb}(\Mbb)\op$, given a sub-$L$-set $X$ of $[[\vec{x} \cdot \phi]]_{\Mbb}$, there exists a geometric formula $\psi$ in the context $\vec{x}$ such that $\psi \vdash_{\vec{x}} \phi$ is provable in $\Tbb$ and $[[\vec{x} \cdot \psi]]_{\Mbb}$ is isomorphic to $X$ as a subobject of $[[\vec{x} \cdot \phi]]_{\Mbb}$.
\end{itemize}
\end{crly}
\begin{proof}
The first item is a straightforward translation of the first condition in Theorem \ref{thm:basic}, while the latter follows from Lemma \ref{lem:Jclosed} of Chapter \ref{chap:logic}, which states that every $J$-closed sieve on $\{\vec{x} \cdot \phi\}$ in $\Ccal_{\Tbb}$ is principal (so every subobject of a representable is representable).
\end{proof}

Let's break down the consequences of this result; for the most part this amounts to expressing syntactically the categorical requirements we saw in Chapter \ref{chap:TTMA}.

\begin{dfn}
\label{dfn:consmodel}
A model $\Mbb$ of a geometric theory $\Tbb$ in a topos $\Ecal$ (in our case we will have $\Ecal = \Set$, as in classical model theory) is said to be \textbf{$\Tbb$-conservative} if any geometric sequent $\phi \vdash_{\vec{x}} \psi$ which is valid in $\Mbb$ is provable in $\Tbb$.
\end{dfn}

$\Mbb$ is conservative if and only if the corresponding geometric morphism is a surjection (cf. \cite[D3.2.6(ii)]{Ele}). Following the observations after Proposition \ref{prop:sitehype} above, we deduce that \textit{a model $\Mbb$ of a theory $\Tbb$ satisfies the first condition of Corollary \ref{crly:generaltheory} if and only if it is conservative}. This in particular means that for every geometric formula-in-context $\{\vec{x} \cdot \phi\}$, either $\phi \vdash_{\vec{x}} \bot$ is provable in $\Tbb$ or the interpretation $[[\vec{x} \cdot \phi]]_\Mbb$ is inhabited, by consideration of the propositional formula $(\exists \vec{x})\phi$. This leads us to another standard definition.

\begin{dfn}
\label{dfn:completetheory}
A theory $\Tbb$ is said to be (geometrically) \textbf{complete} if for every geometric sentence over $\Sigma$, one of,
\[ \top \vdash_{[]} \phi \hspace{10pt} \text{ or } \hspace{10pt} \phi \vdash_{[]} \bot \]
is provable in $\Tbb$. 
\end{dfn}
Geometric completeness of $\Tbb$ corresponds to the topos $\Set[\Tbb]$ being two-valued (cf. \cite[Remark 2.5]{AtTCC}). Thus \textit{any theory $\Tbb$ satisfying the conditions of Corollary \ref{crly:generaltheory} must be complete}.

Let $\vec{a} \in [[\vec{x} \cdot \top]]_{\Mbb}$. It is true for any model of $\Tbb$ that the image of $\vec{a}$ under a $\Tbb$-model endomorphism of $\Mbb$ must satisfy all geometric formulae in the context $\vec{x}$ which $\vec{a}$ satisfies. The second condition of Corollary \ref{crly:generaltheory} might be described as a \textit{homogeneity condition} which is a partial converse to this observation: applying it to the principal sub-$L$-set generated by $\vec{a}$, it says that there exists a formula-in-context $\{\vec{x} \cdot \psi\}$ such that $\psi(\vec{a})$ holds and given any element $\vec{b}$ such that $\psi(\vec{b})$ holds, there is some $\Tbb$-model endomorphism $m$ of $\Mbb$ such that $m(\vec{a}) = \vec{b}$. This is an analogue of one of the conditions appearing in \cite[Theorem 3.1]{TGT}. The basic opens described in \eqref{eq:endopens} can hence be described in this context as:
\[\Ical_{\vec{a}}^m = \{m' \in L \mid \vec{a} \cdot m = \vec{a} \cdot m'\}. \]

Now recall the notion of $\Tbb$-supercompact formula from Definition \ref{dfn:Tscompact}, and the derived notion of supercompactly generated theory from Theorem \ref{thm:scompform} of Chapter \ref{chap:sgt}, which correspond to $\Set[\Tbb]$ being a supercompactly generated topos. Clearly, \textit{any theory $\Tbb$ satisfying the conditions of Corollary \ref{crly:generaltheory} must be supercompactly generated}.

We shall examine how the conditions of Corollary \ref{crly:generaltheory} can be refined for supercompactly generated theories in Subsection \ref{ssec:sgtheory} below.

\subsection{Principal sites}
\label{ssec:principal2}

Knowing that any topos of topological monoid actions is supercompactly generated, it makes sense to restrict our attention to the principal sites discussed in Section \ref{sec:principal} of Chapter \ref{chap:sgt}. As such, let $\Tcal$ be a stable class of morphisms in a small category $\Ccal$, producing a principal site $(\Ccal,J_{\Tcal})$. We have an immediate refinement of Theorem \ref{thm:basic}.

\begin{crly}
\label{crly:principalcondition}
Let $(\Ccal,J_{\Tcal})$ be a principal site. Then $\Sh(\Ccal,J_{\Tcal})$ is equivalent to a topos of actions of a topological monoid if and only if there exists a flat functor $F:\Ccal \to \Set$ sending $\Tcal$-morphisms to epimorphisms such that,
\begin{itemize}
	\item whenever $\{f_i:C_i \to C\}$ is sent by $F$ to a jointly epimorphic family, one of the $f_i$ lies in $\Tcal$, and
	\item letting $L$ be the opposite of the monoid of natural endomorphisms of $F$ and equipping each $F(C)$ with its canonical right $L$-action, for each sub-$L$-set $A$ of $F(C)$ there exists a $J_{\Tcal}$-closed sieve $S$ on $C$ such that $A$ is the union of the images of the morphisms $Fh$ for $h \in S$.
\end{itemize}
\end{crly}

\begin{crly}
\label{crly:psitescompact}
Letting $(\Ccal,J_{\Tcal})$, $F$ and $L$ be as in the statement of Corollary \ref{crly:principalcondition}, each of the $L$-sets $F(C)$ must be principal, and each supercompact subobject of $F(C)$ is the image of some morphism $Fh:F(D) \to F(C)$.
\end{crly}
\begin{proof}
The first part can be proved directly from the conditions of Corollary \ref{crly:principalcondition}. However, we know from the derivation of Theorem \ref{thm:basic} that the conditions make the geometric morphism $f : \Setswith{L} \to \Sh(\Ccal,J_{\Tcal})$ hyperconnected, from Proposition \ref{prop:representable} that the representable sheaves $\ell(C)$ are supercompact, from Corollary \ref{crly:hypepres} that the inverse image functor of a hyperconnected morphism preserves supercompact objects, and finally that $F(C) \cong f^*(\ell(C))$.

For the second part, if $A \hookrightarrow F(C)$ in $\Setswith{L}$ is supercompact then, since it is the union of the images of morphisms $Fh$ for $h$ ranging over a $J_{\Tcal}$-closed sieve $S$, it is isomorphic to one of these images: there exists $h \in S$ with $A \hookrightarrow F(C)$ the image of $Fh$.
\end{proof}

In this case, the topology on $L$ from Theorem \ref{thm:basic} can be generated by basic clopens of the form,
\[\Ical_a^m := \{m' \in L \mid a \cdot m = a \cdot m'\},\]
for $m \in L$ and $a$ ranging over generators for the principal $L$-sets $F(C)$.

To further illuminate the conditions of Corollary \ref{crly:principalcondition}, we can translate properties of toposes of topological monoid actions into site-theoretic properties, yielding necessary conditions for that Corollary to apply.

\begin{dfn}
Recall the notion of $\Tcal$-span from Definition \ref{dfn:Tspan}. Say a principal site $(\Ccal,J_{\Tcal})$ is \textbf{$\Tcal$-span connected} if, for any objects $A,B$ of $\Ccal$, there exists a $\Tcal$-span from $A$ to $B$.
\end{dfn}

\begin{lemma}
A principal site $(\Ccal,J_{\Tcal})$ is $\Tcal$-span connected if and only if $\Sh(\Ccal,J_{\Tcal})$ is two-valued.
\end{lemma}
\begin{proof}
Recall from Proposition \ref{prop:hype2} that a topos is two-valued if and only if all non-initial objects are well-supported. Given a $J_{\Tcal}$-sheaf $X$ on $\Ccal$, suppose $X(A)$ is inhabited for some object $A$ (else $X$ is initial). Given any other object $B$, we have some $\Tcal$-span
\[\begin{tikzcd}
& E \ar[dl,"f"'] \ar[dr,"g"] &\\
A & & B,
\end{tikzcd}\]
with $f \in \Tcal$. Now $Xf$ must be an isomorphism by the sheaf condition, so $X(E)$ is inhabited and hence $X(B)$ must be inhabited too. Thus $X$ is well-supported, as required.

Conversely, if $\Sh(\Ccal,J_{\Tcal})$ is two-valued, consider the sheaf $\ell(A) \times \ell(B)$. Since $\ell(A),\ell(B)$ are not initial, these objects are well-supported and the projection maps
\[\ell(A) \xleftarrow{\pi_1} \ell(A) \times \ell(B) \xrightarrow{\pi_2} \ell(B)\]
are epimorphisms. Moreover, the product is covered by objects of the form $\ell(C)$, whence by composition with $\pi_1$ we have a jointly epic family $\ell(C_i) \to \ell(A)$ which we can reduce to a singleton by supercompactness of $\ell(A)$. Thus we have a span in $\Sh(\Ccal,J_{\Tcal})$ from $A$ to $B$ whose domain is representable and whose left leg is an epimorphism. We can apply Lemma \ref{lem:Trel} twice to recover representable morphisms,
\begin{equation}
\label{eq:scaffold}
\begin{tikzcd}
\ell(D) \ar[dd, "\ell(q)"'] \ar[dr, two heads, "\ell(p)"] & &
\ell(E) \ar[dd, "\ell(s)"] \ar[ll, "\ell(r)"', two heads] \\
& \ell(C) \ar[dl,"f"', two heads] \ar[dr,"g"] &\\
\ell(A) & & \ell(B).
\end{tikzcd}
\end{equation}
Here $\ell(q)$ is epic by construction, and then $(q \circ r,s)$ is a $\Tcal$-span from $A$ to $B$, since a representable morphism is epic if and only if it lies in $\Tcal$.
\end{proof}

We can derive a stronger property here: we saw in Section \ref{ssec:jcp} that the category of supercompact objects in $\Cont(M,\tau)$ has the \textit{joint covering property}.

\begin{dfn}
\label{dfn:jTcp}
We say a principal site $(\Ccal,J_{\Tcal})$ has the \textbf{joint $\Tcal$-covering property} if for any pair of objects $A$ and $B$, there is a span from $A$ to $B$ with \textit{both} legs lying in $\Tcal$.
\end{dfn}

\begin{schl}
\label{schl:jointT}
A principal site $(\Ccal,J_{\Tcal})$ has the joint $\Tcal$-covering property if and only if the category of supercompact objects in $\Sh(\Ccal,J_{\Tcal})$ has the joint covering property.
\end{schl}
\begin{proof}
For the `if' direction, the joint covering property provides a joint cover of $\ell(A)$ and $\ell(B)$ by some supercompact object, which is a quotient of a representable supercompact object, whence we have a span of epimorphisms
\[\begin{tikzcd}
& \ell(C) \ar[dl,"f"', two heads] \ar[dr,"g", two heads] &\\
\ell(A) & & \ell(B)
\end{tikzcd}\]
in $\Sh(\Ccal,J_{\Tcal})$. Applying the construction of \eqref{eq:scaffold}, we recover a representable such span, and hence a joint $\Tcal$-cover of $A$ and $B$.

Conversely, supercompact objects $X$ and $Y$ are quotients of representables $\ell(A)$ and $\ell(B)$, whence the image under $\ell$ of a joint $\Tcal$-cover of $A$ and $B$ extends to a joint cover of $X$ and $Y$, as required.
\end{proof}

\begin{rmk}
\label{rmk:TGTextend}
Both the $\Tcal$-span connectedness property and the joint $\Tcal$-covering property are natural generalization of (the dual of) the `joint embedding property' of \cite[Definition 3.3]{TGT}, which is one of the necessary conditions on a category $\Ccal$ in order for $\Sh(\Ccal\op,J_{at})$ to be equivalent to a topos of actions of a topological \textit{group}.

The `amalgamation property' which also appears in \cite[Definition 3.3]{TGT}, naturally generalizes (after dualizing) to the stability condition for $\Tcal$-morphisms. Another generalization which one might consider is the property that every cospan whose legs are $\Tcal$-morphisms can be completed to a square of $\Tcal$-morphisms. However, this version fails in general in categories of principal actions of a monoid, even in the discrete case. For example, in $\Setswith{\Nbb}$, the cospan,
\[\begin{tikzcd}
\Nbb \ar[dr,"0 \mapsto 0"', two heads] & & \Nbb \ar[dl,"0 \mapsto 1", two heads]\\
& \Zbb/2\Zbb, &
\end{tikzcd}\]
cannot be completed to a square of epimorphisms whose final corner is a supercompact object, since the only such epimorphism available over $\Nbb$ is the identity morphism. Similarly, the dual of the amalgamation property fails to hold in general: recall the right Ore property of Definition \ref{dfn:rOre} in Chapter \ref{chap:mpatti}; by definition, in any monoid which is not right Ore, there exist principal ideals $m_1M,m_2M$ with empty intersection, which corresponds to the existence of a cospan of principal $M$-sets which cannot be completed to a commutative square.
\end{rmk}

\subsection{Interlude: factorization systems}
\label{ssec:wfs}

Observe that, in the second condition of Corollary \ref{crly:psitescompact}, the principal sieve generated by the morphism $h$ need not be $J_{\Tcal}$-closed for a general site $(\Ccal,J_{\Tcal})$. This is unfortunate, since it would be convenient to be able to identify supercompact subobjects of $F(C)$ with suitable individual morphisms in $\Ccal$. In order to understand conditions under which this might be achieved, let us characterize $J_{\Tcal}$-closed sieves.

\begin{lemma}
\label{lem:Tcalclosed}
A sieve $S$ on an object $D$ of $\Ccal$ is $J_{\Tcal}$-closed if and only if it is closed under right cancellation of $\Tcal$-morphisms, in the sense that if $f \circ t \in S$ with $t \in \Tcal$, then $f \in S$. In particular, a principal sieve generated by a morphism $h$ is $J_{\Tcal}$-closed if and only if $h$ has the \textbf{right lifting property} (RLP) relative to $\Tcal$-morphisms, in the usual sense that for any commuting square
\[\begin{tikzcd}
A \ar[r] \ar[d, "t"'] & C \ar[d, "h"] \\
B \ar[r] & D
\end{tikzcd}\]
in $\Ccal$ with $t \in \Tcal$, there exists a diagonal filler morphism $k: B \to C$ making both triangles commute.
\end{lemma}
\begin{proof}
By definition, a sieve $S$ on $D$ is $J_{\Tcal}$-closed if and only if for every morphism $f:B \to D$ with $f^*(S) \in J_{\Tcal}$, we have $f \in S$. But $f^*(S) = \{t:A \to B \mid f \circ t \in S\}$, which lies in $J_{\Tcal}$ if and only if some $t \in f^*(S)$ lies in $\Tcal$, whence the given condition emerges.

If $h$ has the RLP relative to $\Tcal$, the same argument shows that any $f$ with $f^*(S) \in J_{\Tcal}$ has $f$ in the sieve generated by $h$, and conversely.
\end{proof}

The dual of the right lifting property is the \textit{left lifting property} (LLP).
\begin{dfn}[{\cite[Definition 1.2]{WFS}}]
A \textbf{weak factorization system} (WFS) on a category $\Ccal$ is a pair $(\Tcal,\Mcal)$ of morphism classes in $\Ccal$ such that
\begin{enumerate}
	\item each morphism can be factored as a $\Tcal$-morphism followed by an $\Mcal$-morphism, and
	\item every member of $\Tcal$ has the LLP relative to $\Mcal$ and every member of $\Mcal$ has the RLP relative to $\Tcal$.
\end{enumerate}
\end{dfn}

\begin{crly}
\label{crly:wfs}
Suppose a stable class $\Tcal$ of morphisms in $\Ccal$ forms the left class in a WFS $(\Tcal,\Mcal)$ on $\Ccal$. Given an object $D$ in $\Ccal$, consider the full subcategory $\Mcal/D$ of $\Ccal/D$ whose objects are those morphisms $m:X \to D$ lying in $\Mcal$. The $J_{\Tcal}$-closed sieves on $C$ correspond to the ideals in $\Mcal/D$: the collections of objects $K \subseteq \Mcal/D$ such that given any morphism
\[\begin{tikzcd}
Y \ar[dr, "n"'] \ar[rr] & & X \ar[dl, "m"] \\
& D &
\end{tikzcd}\]
in $\Mcal/D$, we have $n \in K$ whenever $m \in K$. In particular, principal sieves generated by $\Mcal$-morphisms are $J_{\Tcal}$-closed sieves.
\end{crly}
\begin{proof}
By Lemma \ref{lem:Tcalclosed}, the $J_{\Tcal}$-closed sieves are precisely sieves closed under right cancellation of $\Tcal$-morphisms. Given the WFS, such a sieve is determined by the $\Mcal$-morphisms it contains, and these form an ideal in $\Mcal/D$. Conversely, given an ideal in $\Mcal/D$, its extension to a sieve over $D$ by precomposition is automatically closed under right cancellation of $\Tcal$-morphisms, so is $J_{\Tcal}$-closed, as required.
\end{proof}

\begin{xmpl}
\label{xmpl:wfsnogo}
In spite of the statement regarding principal sieves in Corollary \ref{crly:wfs}, $\Mcal$-morphisms need not be mapped to monomorphisms by $\ell$ in general. For example, consider the factorization system $(\Tcal,\Mcal) = {}$(identity, all) on the category
\begin{equation}
\label{eq:coeq2}
A \rightrightarrows B \xrightarrow{c} C	,
\end{equation}
where $c$ coequalizes the parallel morphisms. The principal Grothendieck topology $J_{\Tcal}$ is the trivial one, so $\ell$ is just the Yoneda embedding. Here, $c$ is not a monomorphism, and nor is its image under the Yoneda embedding, in spite of the fact that $c$ generates an $\Mcal$-closed sieve which determines a subobject of the terminal object $\yon(C)$, namely the subterminal object which is the support of $\yon(B)$.

On the other hand, if all morphisms in $\Mcal$ are monomorphisms, then these are automatically preserved by $\ell$, so the subobjects corresponding to principal $J_{\Tcal}$-closed sieves are representable; moreover, the category $\Mcal/D$ described in Corollary \ref{crly:wfs} collapses into a preorder in this case.
\end{xmpl}

Note that if $\Tcal$-morphisms are all epic or $\Mcal$-morphisms are all monic then the diagonal filler for the lifting property squares are automatically \textit{unique}, so the weak factorization system is upgraded to an \textbf{orthogonal factorization system} (OFS), which are widely studied (cf.\ \cite[\S5.5]{HCA}). Moving from a WFS to an OFS simplifies the statement of Corollary \ref{crly:wfs} in another way, since it forces the morphisms in $\Mcal/D$ to be $\Mcal$-morphisms too. This occurs in both of the special cases of interest to be discussed below, so we shall immediately dive into the assumption that $(\Tcal,\Mcal)$ is an OFS on $\Ccal$.

We now dualize everything. Given a WFS $(\Tcal,\Mcal)$ on $\Ccal$, we obtain a dual factorization system $(\Mcal\op,\Tcal\op)$ on $\Ccal\op$; the latter is orthogonal if and only if the former is. Supposing that $\Tcal$ is a stable class of morphisms in $\Ccal$, $\Tcal\op$ is a `costable class' in the evident dual sense; we saw some examples of costable classes in the computations of Chapter \ref{chap:logic}.

Let us consider how a factorization system on $\Ccal\op$ can be extended to one on $\Ind(\Ccal\op)$. There is a standard way of constructing (weak) factorization systems on accessible categories starting from a class of morphisms in the separating subcategory, namely the \textit{small object argument}. However, this construction typically requires the presence of pushouts; the variant appearing in \cite[Theorem 2.3]{InjAccess} merely requires the possibility of completing a span to a commuting square, but even this is not possible in $\Ind(\Ccal\op)$ for a general $\Ccal$, or even for the cases we are interested in (as a consequence of Remark \ref{rmk:TGTextend}, above). For similar reasons, the construction of a model structure on $\Ind(\Ccal\op)$ from a model category $\Ccal$ such as that found in \cite{promodelstruct} cannot be specialized to our situation, since the presence of small (co)limits in $\Ccal$ is frequently used in the constructions there. As such, we prove a version from scratch here for the case of an OFS. To simplify constructions, we shall need some preliminary results regarding (finitely) accessible categories.

\begin{fact}[{\cite[Theorem 1.5]{Accessible}}]
\label{fact:directedfinal}
Every filtered category admits a final functor from a directed poset.
\end{fact}

A particular consequence of Fact \ref{fact:directedfinal} is that we can simplify the construction of $\Ind(\Ccal\op)$. Indeed, $\Ind(\Ccal\op)$ can be identified as the full subcategory of $[\Ccal,\Set]$ on the filtered colimits of representables, but the above shows that we may instead identify it with the full subcategory on the \textit{directed} colimits of representables. In a directed poset $\Ibb$, we shall write $i,i',j,j'$ for generic elements and whenever $i \leq i'$, we shall denote by $k_i^{i'}:i \to i'$ the unique morphism in $\Ibb$.

\begin{fact}[{\cite[Theorem 5.1]{Makkai}}]
\label{fact:fpmorphism}
Let $\Fbb$ be any finite category. Then $\Ind(\Ccal\op)^{\Fbb} \simeq \Ind\left((\Ccal\op)^\Fbb\right)$. In particular, for $\Fbb = 2$ the walking arrow category, the finitely presentable objects in $\Ind(\Ccal\op)^{2}$ are precisely the morphisms between objects of $\Ccal\op$.
\end{fact}

This fact makes arguments a lot neater, since it circumvents the need to reason about morphisms of $\Ind(\Ccal\op)$ in terms of domain and codomain objects.

\begin{prop}
\label{prop:ofs}
Let $\Ccal$ be a category equipped with an OFS $(\Tcal,\Mcal)$. Then there is an OFS $(\Lcal,\Rcal)$ on $\Ind(\Ccal\op)$ such that:
\begin{enumerate}
	\item $\Rcal$ is the set of morphisms having the unique RLP relative to $\Mcal\op$;
	\item $\Lcal$ is closed under directed colimits in the arrow category $\Ind(\Ccal\op)^2$;
	\item $\Rcal$ is closed under directed colimits in the arrow category $\Ind(\Ccal\op)^2$; and
	\item the restriction of $(\Lcal,\Rcal)$ to $\Ccal\op$, viewed as a full subcategory of $\Ind(\Ccal\op)$, is precisely $(\Mcal\op,\Tcal\op)$.
\end{enumerate}
Moreover,
\begin{enumerate}[label = ({\alph*})]
	\item morphisms in $\Lcal$ and $\Rcal$ can respectively be presented as directed colimits of morphisms in $\Mcal\op$ and $\Tcal\op$;
	\item if $\Mcal$-morphisms are monomorphisms (resp. regular monomorphisms, strict monomorphisms) then $\Lcal$-morphisms are epimorphisms (resp. regular, strict epimorphisms), and
	\item if $\Tcal$-morphisms are epimorphisms (resp. regular epimorphisms) then $\Rcal$-morphisms are monomorphisms (resp. regular monomorphisms).
\end{enumerate} 
\end{prop}
\begin{proof}
The first property provides the definition of $\Rcal$, and hence the definition of $\Lcal$ as those morphisms having the (unique) LLP with respect to $\Rcal$.

For the second property, let $\Ibb$ be a directed indexing poset and $D:\Ical \to \Ind(\Ccal\op)^2$ a diagram such that $D(i):d_D(i) \to c_D(i)$ lies in $\Lcal$ for every $i$ in $\Ibb$. Write $l:K \to K'$ for the colimit and $\delta_i,\gamma_i$ for the legs of the colimit cocone at $D(i)$. Given a square
\[\begin{tikzcd}
K \ar[r, "p"] \ar[d, "l"'] & P \ar[d, "t"] \\
K' \ar[r, "q"'] & Q
\end{tikzcd}\]
with $t \in \Rcal$, for any $i$ we may compose with the colimit legs to obtain a square with a diagonal filler,
\[\begin{tikzcd}
d_D(i) \ar[r, "p \circ \delta_i"] \ar[d, "D(i)"'] & P \ar[d, "t"] \\
c_D(i) \ar[r, "q \circ \gamma_i"'] \ar[ur, dashed, "\lambda_i"] & Q.
\end{tikzcd}\]
By uniqueness of lifts, the $\lambda_i$ form a cocone under the functor $c_D:\Ibb \to \Ind(\Ccal\op)$, whence we have a unique morphism $K' \to P$ which is immediately a diagonal filler for the original square; uniqueness follows from the universal property of $c_D(i)$ as a colimit and a second application of the uniqueness of the $\lambda_i$. Thus $l \in \Lcal$ as claimed.

For the third property, suppose we are instead given a directed indexing poset $\Ibb$ and a diagram $D:\Ibb \to \Ind(\Ccal\op)^2$ such $D(i)$ is an $\Rcal$-morphism for every $i$. Denote its colimit by $r:K \to K'$ and suppose we are given a square,
\[\begin{tikzcd}
A \ar[r, "p"] \ar[d, "m"'] & K \ar[d, "r"] \\
B \ar[r, "q"'] & K'
\end{tikzcd}\]
with $m$ an $\Mcal\op$-morphism. By Fact \ref{fact:fpmorphism}, $m$ is finitely presentable as an object of $\Ind(\Ccal\op)^2$, so this square factorizes through one of the morphisms $D(i)$. Taking the final subcategory of $\Ibb$ of elements above $i$, we express $r$ as a colimit of morphisms forming a square with $m$ having a diagonal filler, and hence obtain a unique diagonal filler of the colimit square by a similar argument to that above, whence $r \in \Rcal$.

The fourth property is a direct consequence of the fact that $\Ccal\op$ is a full subcategory of $\Ind(\Ccal\op)$.

All that remains is to verify that $(\Lcal,\Rcal)$ does in fact form a factorization system. Given a morphism $q:X \to Y$ viewed as an object of $\Ind(\Ccal\op)^2$, we may express $q$ as a directed colimit of finitely presentable morphisms, say $D:\Ibb \to (\Ccal\op)^2$. By the assumption that $(\Mcal\op,\Tcal\op)$ is an OFS, we may factor each $D(i)$ uniquely up to isomorphism as $t_{i} \circ m_{i}$ with $t_{i} \in \Tcal\op$ and $m_{i} \in \Mcal\op$; let $b_D(i)$ be the codomain of $m_{i}$.\footnote{Having to make many choices seems inevitable here, but the end result is well-defined up to unique isomorphism.}We can extend $b_D$ to a functor $\Ibb \to \Ccal\op$ by defining $Fk_i^{i'}$ as the filler in the following rectangle:
\[\begin{tikzcd}
{d_D(i)} \ar[d, "m_{i}"'] \ar[r, "d_Dk_i^{i'}"] & {d_D(i')} \ar[d, "m_{i'}"] \\
{b_D(i)} \ar[d, "t_{i}"'] \ar[r, dashed, "b_Dk_i^{i'}"] & {b_D(i')} \ar[d, "t_{i'}"] \\
{c_D(i)} \ar[r, "c_Dk_i^{i'}"] & {c_D(i')}.
\end{tikzcd}\]
Uniqueness of lifts guarantees functoriality. By the second and third properties above, the colimit of the resulting upper and lower diagrams $\Ibb$-indexed diagram in $\Ind(\Ccal\op)^2$ provide the desired $(\Lcal,\Rcal)$-factorization. This construction also demonstrates point (a).

Point (b) follows from the fact that colimits commute with colimits, while point (c) follows from the fact that directed colimits commute with finite limits.
\end{proof}

\begin{rmk}
In conversation about Proposition \ref{prop:ofs} with Ivan di Liberti, he suggested a shorter proof using the fact that a category equipped with an orthogonal factorization system is an algebra for the monad on $\Cat$ defined by $\Dcal \mapsto \Dcal^2$. Applying the $\Ind$construction to an algebra $\Dcal^2 \to \Dcal$ and employing Fact \ref{fact:fpmorphism} should provide an algebra structure on $\Ind(\Dcal)$ compatible with the one on $\Dcal$ via the full embedding $\Dcal \to \Ind(\Dcal)$. We shall not make this idea precise here.
\end{rmk}

Since we do not want to assume the presence of coequalizers, we also need the following.

\begin{lemma}[Asymptotic identity lemma]
\label{lem:asym}
Suppose we have a directed diagram $V:\Ibb \to \Ccal\op$ with colimit $U$ in $\Ind(\Ccal\op)$. Suppose that for each $i \in \Ibb$, we are given an element $j_i \geq i$. Then:
\begin{enumerate}
\item We may construct a final subposet $\Ibb' \hookrightarrow \Ibb$ such that $j_i \leq j_{i'}$ whenever $i \leq i'$.
\item The diagram $V':\Ibb' \to (\Ccal\op)^2$ sending $i$ to $Vk_i^{j_i}$,
\[\begin{tikzcd}
V(i) \ar[d, "Vk_i^{j_i}"'] \ar[r, "Vk_i^{i'}"] &
V(i') \ar[d, "Vk_{i'}^{j_{i'}}"] \ar[r] & \cdots \\
V(j_i) \ar[r, "Vk_{j_i}^{j_{i'}}"'] & V(j_{i'}) \ar[r] & \cdots
\end{tikzcd}\]
has as colimit in $\Ind(\Ccal\op)^2$ the identity on $U$.
\end{enumerate}
\end{lemma}
\begin{proof}
We may define $\Ibb'$ be the poset with the same elements as $\Ibb$ but with order relation defined by $i \leq' i'$ if and only if $i \leq i'$ and $j_i \leq j_{i'}$. This is a directed poset because for any $i,i'$ there exists an element $j$ larger than $j_i$ and $j_{i'}$.

Now let $\Jbb$ be the directed set $\{(i,j) \in \Ibb \times \Ibb \mid i \leq j\}$ with the product ordering (so $(i,j) \leq (i',j')$ if and only if $i \leq i'$ and $j \leq j'$), and define $W:\Jbb \to (\Ccal\op)^2$ by $(i,j) \mapsto Vk_{i}^{j}$. This has two pertinent final subdiagrams. The first is the diagonal, which clearly converges to the identity on $U$ since all of its components are identities. The second is the diagram $V'$ described above; thus the colimit of $V'$ is the identity on $U$ too, as claimed.
\end{proof}

With all of this theory regarding $\Ind(\Ccal\op)$ under our belts, we can return to the principal sites which motivated this excursion.

\begin{lemma}
\label{lem:Tinj}
Let $(\Ccal,J_{\Tcal})$ be a principal site. Let $U$ be an object of $\Ind(\Ccal\op)$, and consider the flat functor $F := \Hom_{\Ind(\Ccal\op)}(-,U)$. $F$ is $J_{\Tcal}$-continuous if and only if $U$ is \textbf{$\Tcal\op$-injective}, in the sense that given a span
\[\begin{tikzcd}
A \ar[r, "p"] \ar[d, "t"'] & U \\
B, \ar[ur, dashed]
\end{tikzcd}\]
with $t \in \Tcal\op$, there exists a morphism $B \to U$ in $\Ind(\Ccal\op)$ making the triangle commute.
\end{lemma}
\begin{proof}
This is a direct translation of what it means for $F$ to send $\Tcal$-morphisms to epimorphisms.
\end{proof}

\begin{thm}
\label{thm:qhomogen}
Let $(\Ccal,J_{\Tcal})$, $U$ and $F$ be as in Lemma \ref{lem:Tinj}. Suppose that $\Tcal$ is the left class in an orthogonal factorization system $(\Tcal,\Mcal)$ on $\Ccal$, such that $\Mcal$ is contained in the class of monomorphisms. Let $(\Lcal,\Rcal)$ be the extension of $(\Mcal\op,\Tcal\op)$ to a factorization system on $\Ind(\Ccal\op)$. Then $F$ meets the criteria of Corollary \ref{crly:principalcondition} if and only if the following conditions are met:
\begin{enumerate}
	\item $U$ is \textbf{$\Tcal\op$-injective};
	\item $U$ is \textbf{$\Rcal$-universal}, in the sense that every object $X$ of $\Ccal\op$ admits an $\Rcal$-morphism to $U$ in $\Ind(\Ccal\op)$, and
	\item $U$ can be expressed as a directed colimit of finitely presentable objects such that the legs of the colimit cone lie in $\Rcal$, or equivalently such that the morphisms in the diagram lie in $\Tcal\op$.
\end{enumerate}
\end{thm}

\begin{rmk}
\label{rmk:qhomogen}
The first two conditions in Theorem \ref{thm:qhomogen} are natural extensions of (the duals of) the definitions appearing in \cite[Definition 3.3]{TGT} of \textit{$\Ccal$-homogeneous} object and \textit{$\Ccal$-universal} object, respectively. This terminology originates in model theory (cf.\ \cite[Chapter 7]{Model}). We can also generalize the notion of \textit{$\Ccal$-ultrahomogeneous} object; we will need this definition in the proof, so we give it here.

Given the set-up of Theorem \ref{thm:qhomogen}, we say an object $U$ of $\Ind(\Ccal\op)$ is \textbf{$\Ccal$-quasi-homogeneous} if every object $D$ of $\Ccal\op$ admits a morphism $g_0:D \to U$ such that for all morphisms $d:D \to E$ in $\Ccal\op$ and $h: E \to U$ there exists an endomorphism $u:U \to U$ satisfying $h \circ d = u \circ g_0$:
\[\begin{tikzcd}
D \ar[r, "g_0"] \ar[d, "d"'] & U \ar[d, dashed, "\exists u"]\\
E \ar[r, "h"'] & U.
\end{tikzcd}\]
Clearly we may take $d$ to be the identity on $\id_D$ without loss of generality, but the principle we want to express is that `every $\Ccal\op$-morphism extends to an endomorphism of $U$'.
\end{rmk}

\begin{proof}
Suppose $F$ meets the conditions of Corollary \ref{crly:principalcondition}, and let $L := \End_{\Ind(\Ccal\op)}(U)\op$ as usual. We know from Lemma \ref{lem:Tinj} that $\Tcal\op$-injectivity is necessary. By Corollary \ref{crly:psitescompact}, for each object $D$ of $\Ccal$, $F(D)$ must be a principal $L$-set, which is equivalent to $U$ being $\Ccal$-quasi-homogeneous in the sense of Remark \ref{rmk:qhomogen}.

We claim that any generating morphism $g_0 \in \Hom_{\Ind(\Ccal\op)}(D,U)$ is an $\Rcal$-morphism. Indeed, expressing $U$ as a directed colimit of a diagram $V: \Ibb \to \Ccal\op$, $g_0$ corresponds to an equivalence class of morphisms of the form $q: D \to V(i)$. Consider the $(\Lcal,\Rcal)$-factorization of $q$, say $q = q' \circ m$. Let $h:D \to U$ be any other morphism and take its $(\Lcal,\Rcal)$-factorization, say $h = h' \circ n$. By assumption, there exists an endomorphism $u$ of $U$ making this rectangle commute:
\[\begin{tikzcd}
D \ar[r, "m"] \ar[rrr, bend left, "g_0"] \ar[d, equal] & D' \ar[r, "q'"] \ar[d, dashed] & V(i) \ar[r, "\lambda_i"] & U \ar[d, "u"] \\
D \ar[r, "n"'] \ar[rrr, bend right, "h"'] & D'' \ar[rr,"h'"'] & & U.
\end{tikzcd}\]
By orthogonality, we obtain a morphism $D' \to D''$ as indicated making both squares commute. But consequently,
\[F(m) = (- \circ m) : \Hom_{\Ind(\Ccal\op)}(D',U) \to \Hom_{\Ind(\Ccal\op)}(D,U)\]
is a surjection. Since $F$ reflects epimorphisms to $\Tcal$-morphisms, it follows that $m$ lies in both $\Lcal$ and $\Rcal$ and hence is an isomorphism. Thus $q$ is a $\Tcal\op$-morphism. Fixing any such $q$, we may express $g_0$ as the colimit of the following $i/\Ibb$-indexed diagram in $\Ind(\Ccal\op)^2$,
\[\begin{tikzcd}
D \ar[r, equal] \ar[d, "q"'] & D \ar[d, "Vk_i^{i'} \circ q"] \\
V(i) \ar[r, "Vk_i^{i'}"'] & V(i'),
\end{tikzcd}\]
where each of the vertical morphisms is in $\Tcal\op$ by the argument above. Thus, by the third point of Proposition \ref{prop:ofs}, $g_0$ is a member of $\Rcal$, as claimed. Thus $U$ is $\Rcal$-universal.

Consider again the directed diagram $V:\Ibb \to \Ccal\op$ with colimit $U$ in $\Ind(\Ccal\op)$. Clearly, each $\lambda_i:V(i) \to U$ generates a principal sub-$L$-set of $F(V(i))$, whence there is an $\Mcal$-morphism\footnote{We are abusing notation a little here by not distinguishing between the morphism $m_i$ in $\Ccal$ and its dual in $\Ccal\op$.} (which we index for subsequent reference) $m_i:X(i) \to V(i)$ such that
\[Fm_i = (- \circ m_i): F(X(i)) \to F(V(i))\]
has image this sub-$L$-set; since $\Mcal$ consists of monomorphisms by assumption, there exists some generator $g_i$ of $F(X(i))$ which is mapped to $\lambda_i$ under composition with $m_i\op$, whence $g_i \circ m_i$ is the $(\Lcal,\Rcal)$-factorization of $\lambda_i$.

By finite presentability of $X(i)$, $g_i$ factors through some object $V(j_i)$ as $\lambda_{j_i} \circ h_i$; without loss of generality, $j_i \geq i$. Then we have $\lambda_{j_i} \circ h_i \circ m_i = \lambda_{i} = \lambda_{j_i} \circ Vk_i^{j_i}$; by finite presentability, this identity holds when $\lambda_{j_i}$ is replaced by $Vk_{j_i}^{j'_i}$ for sufficiently large $j'_i$. In particular, substituting $h_i$ for $Vk_{j_i}^{j'_i} \circ h_i$ and $j_i$ for $j'_i$, we can without loss of generality assume that $h_i \circ m_i = Vk_i^{j_i}$, which makes this the $(\Lcal,\Rcal)$-factorization of $Vk_i^{j_i}$. We can now apply Lemma \ref{lem:asym} to deduce that after restricting to a suitable final subdiagram $\Ibb'$ of $\Ibb$, the $Vk_i^{j_i}$ assemble into an $\Ibb'$-indexed diagram in $\Ccal\op$ whose colimit in $\Ind(\Ccal\op)$ is the identity on $U$. Moreover, by the proof of Proposition \ref{prop:ofs}, the $h_i,m_i$ assemble into respective diagrams (by orthogonality) converging to the $(\Lcal,\Rcal)$-factorization of the identity, which is trivial. In particular, the colimit of the diagram $X$ is (isomorphic to) $U$, and the colimit cone consists of morphisms in $\Rcal$ since the $g_i$ are in $\Rcal$, as required.

Conversely, suppose we are given $U$ with the stated properties. It is immediate that $F$ sends $\Tcal$-morphisms to epimorphisms by $\Tcal$-injectivity.

Suppose $F$ sends some collection of morphisms $\{f_i:C_i \to D\}$ to a jointly epimorphic family, and let $g_0:D \to U$ be an $\Rcal$-morphism. This means in particular that there exists some $g: C_i \to U$ in $\Ind(\Ccal\op)$ such that $g \circ f_i\op = g_0$. But $g_0$ is a member of $\Rcal$, whence so is $f_i\op$, which is to say that $f_i \in \Tcal$, as required for the first point of Corollary \ref{crly:principalcondition}.

We claim that $g_0$ is also a generator of $F(D)$ as a right $L$-set. Indeed, letting $V$ be a directed diagram as above but now assuming that the morphisms $Vk_i^{i'}$ lie in $\Tcal\op$, it must be that $g_0$ factors through one of the $\lambda_i:V(i) \to U$ via a $\Tcal\op$-morphism $t_i$, whence for any other morphism $h:D \to U$, $\Tcal\op$-injectivity gives,
\[\begin{tikzcd}
D \ar[r, "t_i"] \ar[drr, "h"'] \ar[rr, bend left, "g_0"] & V(i) \ar[r, "\lambda_i"] \ar[dr, dashed, "\exists u_i"] & U\\
& & U;
\end{tikzcd}\]
the sequence $u_i$ thus converges to the desired endomorphism of $U$; this gives $\Ccal$-quasi-homogeneity.

Now, a principal sub-$L$-set of $\Hom_{\Ind(\Ccal\op)}(D,U)$ has some generating morphism $x:D \to U$. By finite presentability, $x$ factors through some $\Rcal$-morphism $V(i) \to U$. In particular, the intermediate object in the $(\Lcal,\Rcal)$-factorization $x = r \circ l$ is finitely presented, and hence the image of $Fl\op$ is precisely the sub-$L$-set generated by $x$. It follows that for any sub-$L$-set $A \hookrightarrow F(D)$ we can construct a $J_{\Tcal}$-closed sieve (consisting of $\Mcal$-morphisms derived as $l\op$ was above) on $D$ in $\Ccal$ whose images jointly cover $A$, as required.
\end{proof}

\begin{rmk}
We only used the condition that $\Mcal$ is contained in the class of monomorphisms in a single place in this proof; we suspect it may be possible to remove this condition, but it is not such a strong condition as to be obstructive for the purposes of the present chapter.
\end{rmk}

\subsection{Supercompactly generated theories}
\label{ssec:sgtheory}

Let $\Tbb$ be a supercompactly generated theory, and let
\[\Phi := \left\{\{\vec{x}_i \cdot \phi_i\} \mid i \in I \right\}\]
be the set of $\Tbb$-supercompact formulae-in-context (up to $\alpha$-equivalence); let $\Ccal_{\Phi}$ be the corresponding full subcategory of the syntactic category $\Ccal_{\Tbb}$. We include all of the $\Tbb$-supercompact formulae because this gives us an (epi,mono) factorization system on $\Ccal_\Phi$. Indeed, since $\Ccal_{\Tbb}$ is closed under subobjects in $\Sh(\Ccal_{\Tbb},J_{\Tbb})$, the intermediate object in the image factorization of any morphism between formulae in $\Phi$ is represented by some formula in $\Phi$. Moreover, by the same proof as Corollary \ref{crly:strict} in Chapter \ref{chap:sgt}, all epimorphisms in $\Ccal_{\Phi}$ are strict.

\begin{rmk}
Including all supercompact formulae-in-context in $\Phi$ is slight overkill: we could simply have closed any generating set of $\Tbb$-supercompact objects under image factorizations; we leave consideration of this generalization to the interested reader. 
\end{rmk}

For the remainder of this section, let $\Tbb$, $\Phi$, $\Ccal_{\Phi}$ and $\Tcal$ be as described above.

The fact that the formulae in $\Phi$ are $\Tbb$-supercompact means that the restriction of the canonical topology to $\Ccal_{\Phi}$ will be a principal topology generated by the (stable) class $\Tcal$ of strict epimorphisms, so that $\Set[\Tbb] \simeq \Sh(\Ccal_{\Phi},J_{\Tcal})$ (by the Comparison Lemma, say). In general there will be supercompact objects of $\Set[\Tbb]$ which are not representable, so $\Ccal_{\Phi}$ need not be a reductive category, but it still has all the features needed to apply the results of the last two sections. First, we refine the presentation of morphisms in $\Ccal_{\Phi}$ and then give syntactic characterizations of the necessary condition derived in Scholium \ref{schl:jointT}.

\begin{lemma}
\label{lem:scompactreduce}
Let $\theta$ be a $\Tbb$-provably functional formula presenting a morphism $\{\vec{x} \cdot \phi\} \to \{\vec{y}\cdot \psi\}$ in $\Ccal_{\Phi}$. Then $\theta$ is ($\Tbb$-provably equivalent to) a $\Tbb$-supercompact formula.
\end{lemma}
\begin{proof}
Express $\theta$ as a disjunction of supercompact formulae $\bigvee_{i \in I} \chi_i$. Recalling that existential quantification commutes with disjunctions, the following sequents are provably in $\Tbb$:
\begin{align*}
\phi & \vdash_{\vec{x}} \bigvee_{i \in I} \left(\exists \vec{y}\right) \chi_i &
\bigvee_{i \in I} \chi_i & \vdash_{\vec{x},\vec{y}} \phi \wedge \psi &
\left(\bigvee_{i \in I} \chi_i\right) \wedge \left(\bigvee_{i \in I} \chi_i [\vec{y}'/\vec{y}] \right) & \vdash_{\vec{x},\vec{y},\vec{y}'} \vec{y} = \vec{y}'
\end{align*}
Since $\phi$ is supercompact, there exists an index $i$ such that $\phi \vdash_{\vec{x}} \left(\exists \vec{y}\right) \chi_i$. By inspection, $\chi_i$ satisfies the requirements to be a $\Tbb$-provably functional formula and $\chi_i \vdash_{\vec{x},\vec{y}} \theta$. This means that the relation presented by $\chi_i$ is contained in that presented by $\theta$, but both are functional relations in $\Ccal_{\Tbb}$, and the ordering on such relations (which corresponds to $\Tbb$-provability) is discrete, so $\chi_i$ is $\Tbb$-provably equivalent to $\theta$, as required.
\end{proof}

\begin{rmk}
\label{rmk:indexmorph}
In particular, Lemma \ref{lem:scompactreduce} means that we can index the morphisms of $\Ccal_{\Phi}$ by elements of $\Phi$ modulo $\Tbb$-provable equivalence. This is important because it means that when the signature of $\Tbb$ is countable, then there are only countably many morphisms between the objects of $\Ccal_{\Phi}$, a fact which we shall be able to take advantage of in Section \ref{sec:Fraisse}. The composition operation, which involves taking a $\Tbb$-supercompact formula $\Tbb$-provably equivalent $(\exists \vec{y})(\gamma \wedge \theta)$, is a little abstract, but the following Lemma makes it somewhat easier.
\end{rmk}

\begin{lemma}
\label{lem:Tcovering}
The site $(\Ccal_{\Phi},J_{\Tcal})$ has the joint-$\Tcal$-covering property if and only if for every pair of formulae-in-context $\{\vec{x} \cdot \phi\}$ and $\{\vec{y}\cdot \psi\}$ in $\Phi$ there exists a third formula $\{\vec{x},\vec{y} \cdot \chi\}$ in $\Phi$ such that the sequents
\[ \chi \vdash_{\vec{x},\vec{y}} \phi \wedge \psi
\hspace{10pt} \text{ and } \hspace{10pt}
\phi \vdash_{\vec{x}} \left(\exists \vec{y}\right) \chi
\hspace{10pt} \text{ and } \hspace{10pt}
\psi \vdash_{\vec{y}} \left(\exists \vec{x}\right) \chi \]
are provable in $\Tbb$. We say $\Phi$ has the \textbf{joint-$\Tbb$-covering property} if this holds.
\end{lemma}
\begin{proof}
Given an arbitrary joint cover of the given objects, it must admit a morphism in $\Ccal_{\Tbb}$ to their product, and taking the image of that morphism we obtain a joint cover in the context $\vec{x},\vec{y}$, whence an object $\{\vec{x},\vec{y} \cdot \chi\}$ for which the first identity holds. The remaining sequents are translations of what it means for the projections to $\{\vec{x} \cdot \phi\}$ and $\{\vec{y}\cdot \psi\}$ to be epimorphisms.
\end{proof}

\begin{rmk}
We can deduce completeness of $\Tbb$ (in the sense of Definition \ref{dfn:completetheory} above) directly from the joint-$\Tbb$-covering property on $\Phi$: applying existential quantification to the formulae in Lemma \ref{lem:Tcovering}, we see that all supercompact propositional formulae are $\Tbb$-provably equivalent, whence (since all propositional formulae are unions of supercompact ones) every proposition is either $\Tbb$-provably equivalent to $\top$ or $\bot$, and either $\Tbb$ is contradictory or $\top$ is a $\Tbb$-supercompact formula.
\end{rmk}

It should be possible to translate each of the conditions of Theorem \ref{thm:qhomogen} into syntactic properties, with only the mild obstacle that we lack a characterization of $\Rcal$-morphisms (in the sense of Proposition \ref{prop:ofs}) in order to describe $\Rcal$-universality. However, our attempts at this translation were no more enlightening than the necessary and sufficient conditions already extracted in Section \ref{ssec:syntactic2}, so we move on immediately. 

\subsection{Reductive sites}
\label{ssec:redsite2}

Finally, we arrive in our analysis to the canonical sites of definition for supercompactly generated toposes.

We saw in Corollary \ref{crly:orthog} of Chapter \ref{chap:sgt} that any reductive category has a (strict epi,mono) factorization system whose left class is stable by definition. Thus we can directly apply the results of Section \ref{ssec:wfs}.

\begin{crly}
\label{crly:reductivecondition}
Let $(\Ccal,J_r)$ be a reductive site. Then $\Sh(\Ccal,J_r)$ is equivalent to a topos of actions of a topological monoid if and only if there exists a flat functor $F:\Ccal \to \Set$ preserving strict epimorphisms such that,
\begin{itemize}
	\item whenever $\{f_i:C_i \to C\}$ is sent by $F$ to a jointly epimorphic family, one of the $f_i$ is a strict epimorphism, and
	\item letting $L$ be the opposite of the monoid of natural endomorphisms of $F$ and equipping each $F(C)$ with its canonical right $L$-action, for each sub-$L$-set $A$ of $F(C)$ there exists a subobject $B \hookrightarrow C$ in $\Ccal$ such that $A \cong F(B)$.
\end{itemize}
Equivalently, using Theorem \ref{thm:qhomogen}, this occurs if and only if there exists an object $U$ in $\Ind(\Ccal\op)$ such that,
\begin{enumerate}
	\item $U$ is injective with respect to strict monomorphisms in $\Ccal\op$;
	\item every object $D$ of $\Ccal\op$ admits a monomorphism to $U$ whose left factors lying in $\Ccal\op$ are strict;
	\item $U$ can be expressed as a colimit of a directed diagram whose constituent morphisms are strict monomorphisms.
\end{enumerate}
\end{crly}
\begin{proof}
The first characterization is a direct translation of Theorem \ref{thm:basic}. For the latter, just as in the last section we lack a characterization of the extension $\Rcal$ of the class $\Tcal\op$ of strict monomorphisms in $\Ccal\op$ to $\Ind(\Ccal\op)$, but we at least know that $\Rcal$-morphisms are monic by Proposition \ref{prop:ofs}(c), so the conditions are necessary.

Conversely, given the stated conditions and an expression of $U$ as a colimit of a directed diagram $E:\Ibb \to \Ccal\op$ consisting of strict epimorphisms, a monomorphism $X \hookrightarrow U$ must factor as a monomorphism $X \hookrightarrow E(i)$ followed by the colimit leg $\lambda_i: E(i) \to U$ which is an $\Rcal$-morphism, and since all monomorphisms in $\Ccal\op$ are strict monomorphisms, it follows that the monomorphism $X \hookrightarrow U$ is an $\Rcal$-morphism, as required.
\end{proof}

\section{An extension of the \Fraisse construction}
\label{sec:Fraisse}

While the necessary and sufficient conditions presented at various levels of generality in Section \ref{sec:setup} appear to be reasonable generalizations of those obtained by Caramello in \cite{TGT}, they are \textit{a priori} useless, since we have not provided any way to \textit{build} an object $U$ with the relevant qualities.

Obtaining such a point, after imposing some mild hypotheses, is the purpose of the \textit{\Fraisse construction}, whose relation to topos theory Caramello studied in her own doctoral thesis (published in \cite{Fraisse}, relying in places on the work of Kubi\'{s} \cite{Fraisse0}). The origin of this construction is recounted in \cite[\S 7.1]{Model}: Roland \Fraisse demonstrated that the rationals (viewed as a countable linearly ordered set) can be reconstructed as the filtered colimit of its finite linear suborders, in such a way that any other countable linear order embeds into the rationals. This construction extends to a demonstration that for any class of finite structures satisfying the joint embedding property and the amalgamation property, there exists a unique countable structure which is universal and ultrahomogeneous with respect to this class, where the categorical interpretation of these terms is that employed by Caramello, as described in Remarks \ref{rmk:TGTextend} and \ref{rmk:qhomogen} above.

\subsection{Existence theorem}
\label{ssec:exists}

We shall not produce such general results here; we simply give a special case analogous to \Fraisse's original construction, following \cite[Theorem 7.1.4]{Model}.

\begin{thm}
\label{thm:Fraisse1}
Let $\Ccal$ be a (finite or) \textbf{countable} category equipped with an OFS $(\Tcal,\Mcal)$ such that $\Tcal$ is a stable class and $\Mcal$ is contained in the class of monomorphisms. Suppose that $\Ccal$ has the joint $\Tcal$-covering property. Then $\Ind(\Ccal\op)$ contains an object $U$ satisfying the conditions of Theorem \ref{thm:qhomogen}.
\end{thm}
\begin{proof}
We shall construct an $\omega$-indexed sequence $\{U_i: i < \omega\}$ in $\Ind(\Ccal\op)$ such that each $u_{k}^{k+1} : U_k \to U_{k+1}$ lies in $\Tcal\op$, such that for any object $A$ in $\Ccal$, $U_k$ admits a $\Tcal\op$ morphism from $A$ for all sufficiently large $k$, and such that the colimit is $\Tcal\op$-injective.

Enumerate $\Tcal\op$ as $\{t_{i}:A_i \to B_i \mid i < \omega\}$;\footnote{This can be adapted to the finite case either by repetition or by choosing a suitable finite indexing set throughout the proof.} since every object appears as a domain of a $\Tcal\op$-morphism (the identity on itself) there is no harm in using this as an enumeration of the objects too. Let $\pi:\omega \times \omega \times \omega \to \omega$ be a bijection such that $\pi(i,j,k) \geq k$. We define $U_k$ inductively as follows:

Let $U_0 := A_0$.

Given $U_k$, enumerate spans,
\[\left\{
\begin{tikzcd}[row sep = tiny]
A_i \ar[d, "t_i"'] \ar[r, "f_{i,j,k}"] & U_k \\
B_i
\end{tikzcd}
\middle| \; j < \omega \right\},\]
If there exists $i',j',k'$ with $\pi(i',j',k') = k$, then define $U'_{k+1}$ by (co)stability of $t_{i'}$ along $u_{k'}^k \circ f_{i',j',k'}$; otherwise set $U'_{k+1} = U_k$. Now define $U_{k+1}$ by applying the joint $\Tcal$-covering property to $A_{k+1}$ and $U'_{k+1}$; we illustrate this (with hooked arrows for $\Tcal\op$-morphisms) as follows:
\begin{equation}
\label{eq:construct}
\begin{tikzcd}
A_{i'} \ar[d, "t_{i'}"', hook] \ar[r, "f_{i',j',k'}"] &
U_{k'} \ar[r, "u_{k'}^k", hook] &
U_k \ar[d, "u'_{k}"', hook] \ar[dd, bend left, "u_{k}^{k+1}"]\\
B_{i'} \ar[rr] & &
U'_{k+1} \ar[d, "u''_k"', hook] \\
& A_{k+1} \ar[r, hook] & U_{k+1}
\end{tikzcd}	
\end{equation}

That every object in $\Ccal\op$ admits a $\Tcal\op$-morphism to some $U_k$ and hence an $\Rcal$-morphism to $U$ is immediate. To verify $\Tcal\op$-injectivity, observe that any $f:A \to U$ must factor through $U_k$ by finite presentability, and given any $\Tcal\op$-morphism $t:A \to B$, we can choose $k$ large enough that there is a morphism completing the triangle, as required.
\end{proof}

Applying this to the category of $\Ccal_{\Phi}$ of $\Tbb$-supercompact formulae discussed in Section \ref{ssec:sgtheory} produces the following result:

\begin{crly}
\label{crly:countabletheory}
Suppose that $\Tbb$ is a supercompactly generated theory over a countable signature, and let $\Phi$ be the class of $\Tbb$-supercompact formulae. If $\Phi$ has the joint-$\Tbb$-covering property, then there exists a model $\Mbb$ of $\Tbb$ satisfying all of the necessary conditions highlighted in Section \ref{ssec:syntactic2}. In particular, $\Tbb$ is classified by a topos of topological monoid actions.
\end{crly}

\subsection{Minimality theorem}
\label{sssec:unique}

We know, from Example \ref{xmpl:Schanuel} for example, that complete topological monoids representing a given topos are far from unique. However, the classical \Fraisse construction, which applies to a suitable class of finitely generated models of a theory over a countable signature, produces a model which is \textit{countably categorical}, in the model-theoretic sense that it is the only countable model with the required universal properties.

Since we saw in Chapter \ref{chap:TDMA} that the Morita-equivalence problem is non-trivial even for discrete monoids, we cannot hope to recover uniqueness. Nonetheless, we can at least generalize a weaker minimality result by emulating Caramello and Kubi\'{s}.

\begin{lemma}
\label{lem:buildamor}
Consider the set-up from Theorem \ref{thm:qhomogen}. Let $X$ be an object of $\Ind(\Ccal\op)$ expressible as a colimit of a diagram $Y:\omega \to \Ccal\op$ whose morphisms lie in $\Tcal\op$. Let $U$ be any $\Tcal\op$-injective object which admits a morphism from some object in the image of $Y$ (for example, suppose $U$ is $\Rcal$-universal). Then there exists a morphism $X \to U$.
\end{lemma}
\begin{proof}
We must construct a cocone under the diagram $Y$ with nadir $U$. We are given a morphism $Y(n) \to U$. For $n' < n$ we simply compose with the morphism $Y(n') \to Y(n)$. Given $Y(n)$, we construct $Y(n + 1)$ by $\Tcal\op$-injectivity of $U$ and so on, inductively.
\end{proof}

\begin{prop}
\label{prop:Fraisse2}
Let $\Ccal$ be a countable category equipped with the data specified in Theorem \ref{thm:Fraisse1}. Then the object $U$ constructed in Theorem \ref{thm:Fraisse1} is weakly initial amongst $\Tcal\op$-injective, $\Rcal$-universal objects of $\Ind(\Ccal\op)$.
\end{prop}
\begin{proof}
This is immediate from Lemma \ref{lem:buildamor}.
\end{proof}

Considering the Schanuel topos of Example \ref{xmpl:nonpresh}, for example, the countable set $\Nbb$ is weakly initial in the category of infinite decidable sets, having (many) embeddings into any other infinite set.

\begin{xmpl}
We have already mentioned that non-triviality of Morita-equivalence for monoids which we saw in Chapter \ref{chap:TDMA} obstructs a generalization of the uniqueness up to isomorphism assertion regarding the \Fraisse construction. One might hope based on this that one could at least demonstrate that any pair of objects satisfying the conditions of Theorem \ref{thm:qhomogen} and constructed as countable directed colimits of $\Tcal\op$-morphisms might be retracts of one another, or embed into one another. Thanks to Jens Hemelaer, we can exhibit a counterexample to these suggestions.

Consider the bicyclic monoid, which has the following presentation:
\[ B := \langle u,v \mid uv = 1 \rangle. \]
Each element of $B$ can be uniquely expressed in the form $v^iu^j$ with $i,j \geq 0$. Consider $B$ as a one-object category, and let
\[\Ucal := \{u^j \mid j \geq 0\} \text{ and } \Vcal := \{v^i \mid i \geq 0\}.\]
By definition, $\Ucal$ and $\Vcal$ are the classes of split epimorphisms and split monomorphisms of $B$ respectively; in fact, one can check that these form an orthogonal factorization system on $B$ and that $\Ucal$ is a stable class, so that $(B,J_{\Ucal})$ is a principal site satisfying the conditions of Theorem \ref{thm:Fraisse1}. Since all of the $\Ucal$-morphisms split, $J_{\Ucal}$ coincides with the trivial topology, so the topos $\Sh(B,J_{\Ucal})$ is simply $\Setswith{B}$. As such, this topos has a canonical point (with the properties required by Theorem \ref{thm:qhomogen}) corresponding to $B$ as a left $B$-set, or equivalently as an object of the category $B\op$. On the other hand, the object constructed in the proof of Theorem \ref{thm:Fraisse1} is the colimit of the diagram,
\[\begin{tikzcd}
B \ar[r, "\cdot u"] & B \ar[r, "\cdot u"] & B \ar[r, "\cdot u"] & \cdots ;
\end{tikzcd}\]
call this colimit $A$. Of course, the specific diagram constructed according to the proof may differ from this one, but considering stability of $\Ucal$ morphisms along identity morphisms we see that the sequence of morphisms $u_k^{k+1}$ in the proof must contain infinitely many non-identity members of $\Ucal$, whence we can construct a final map $\omega \to \omega$ mapping the constructed sequence to the one above.

As a left $B$-set, we can identify $A$ with the set having elements $\{b^pa^q \mid p \in \Nbb, \, q \in \Zbb\}$, and action determined by:
\begin{align*}
u \cdot b^pa^q & =
\begin{cases}
b^{p-1}a^q &  \text{ if }p>0 \\
a^{q+1} & \text{ if }p=0
\end{cases} \\
v \cdot b^p a^q & = b^{p+1}a^q.
\end{align*}

There is an epimorphism $A \to B$, which can be constructed as the colimit in $\Ind(B\op)^2$ of the following diagram in $(B\op)^2$:
\[\begin{tikzcd}
B \ar[r, "\cdot u"] \ar[d, "\cdot 1", two heads] & B \ar[r, "\cdot u"] \ar[d, "\cdot v", two heads] & B \ar[r, "\cdot u"] \ar[d, "\cdot v^2", two heads] & \cdots \\
B \ar[r, equal] & B \ar[r, equal] & B \ar[r, equal] & \cdots ;
\end{tikzcd}\]
concretely, it sends $b^p a^q$ to $v^p u^q$ if $q \leq 0$ and to $v^{p+q}$ if $q < 0$. Since $B$ is projective as a left $B$-set, this epimorphism splits; the most obvious splitting sends $v^i u^j$ back to $b^i a^j$. Thus $B$ is a retract of $A$. However, $A$ is not a retract of $B$ since it is not principal: the element $b^p a^q$ generates a sub-$B$-set on elements of the form $b^{p'} a^{q'}$ with $q' \geq q$; there is no monomorphism $A \to B$, either.

Let $L := \End(A)\op$. This monoid can be presented as,
\[ L \cong \langle a, a^{-1}, w \mid aa^{-1} = 1 = a^{-1}a, \, wa^kw = wa^k, wa^{-k}w = a^{-k}w  \; (k\geq 0) \rangle,\]
where $a$ acts on $A$ by sending $b^p a^q$ to $b^p a^{q+1}$, $a^{-1}$ is its inverse, and $w$ is the idempotent endomorphism obtained as the composite $A \to B \to A$ of the morphisms in the retraction described above. To verify the presentation, we observe that the generators are left $B$-set homomorphisms validating the given relations and that these relations allow us to write any element of the presented monoid uniquely as either $a^{i}wa^j$ or $a^{k}$ with $i,j,k \in \mathbb{Z}$ (by an inductive argument which we omit). To check that these relations are exhaustive, consider the fact that a $B$-set endomorphism $\phi$ of $A$ is determined by its value on the generating set $\{a^q \mid q \in \mathbb{Z} \}$. If there exists some $q$ with $\phi(a^q) = b^p a^{q'}$ for $p > 0$ then for $l \geq q$ we have,
\[\phi(a^l) = u^{l-q}\phi(a^q) =
\begin{cases}
b^{p-(l-q)}a^{q'} &  \text{ if } p - (l - q) \geq 0 \\
a^{q' + (l - q) - p} & \text{ if } p - (l - q) \leq 0,
\end{cases}\]
while for $l \leq q$ we have $\phi(a^q) = u^{q-l}\phi(a^l)$, which has the unique solution $\phi(a^l) = b^{p+(q-l)} a^{q'}$. As such, $\phi$ decomposes in $L$ as $a^{-p-q}wa^{q'}$. On the other hand, if $\phi(a^q) = a^{q'}$ is a power of $a$ for all $q \in \mathbb{Z}$, $\phi$ is expressible as $a^{q'-q}$. This validates the presentation.

To compute the topology on $L$, we use Theorem \ref{thm:factor} from Chapter \ref{chap:TTMA}: the counit of the (hyperconnected) geometric morphism $\Setswith{L} \to \Setswith{B}$ induced by $A$ at $\Pcal(L)$ is
\begin{align*}
	\epsilon_{\Pcal(L)}: \Hom_L(A,\Pcal(L)) \otimes_B A & \to \Pcal(L) \\
	f \otimes a & \mapsto f(a),
\end{align*}
and the base for the topology is the image of this map. Note that $A$ is a principal $L$-set, since $1$ can be mapped to any $b^qa^p$ by the action of $L$; as such, an $L$-set homomorphism $f: A \to \Pcal(L)$ is determined by the image of $1$. A subset $U$ can be the image of $1$ under such an $f$ if and only if $(a^iwa^j)^*(U) = (a^{i+j})^*(U)$ for all $i \geq 0$, since the right congruence on $L$ generated by $(a^iwa^j,a^{i+j})$ produces the quotient map $L \too A$. In other words, $f(1)$ is a union of subsets of $L$ of the form $U_k := \{a^k\} \cup \{a^iwa^{k-i} \mid i \geq 0\}$ and singletons $\{a^iwa^j\}$ with $i \leq 0$, with no further requirements.

Considering the image of such $U$ under the inverse image action of $L$ to complete the picture, we find that for any value of $i$ and $j$, the singleton $\{a^iwa^j\}$ is open, while all basic neighbourhoods of elements of the form $a^k$ are infinite and of the form $(a^{k'})^*(U_k)$. Letting $\rho$ be the topology generated by these subsets, the conclusion is that $\Setswith{B} \simeq \Cont(L,\rho)$. Incidentally, we note that by Proposition \ref{prop:denseisom}, there cannot be a monoid homomorphism $L \to B$ inducing this equivalence, since $L$ is not isomorphic to $B$. This also demonstrates that \textit{discreteness is not a Morita-invariant property} of topological monoids. 
\end{xmpl}

\section{Semi-Galois theories}
\label{sec:sgal}

At last, we shall extract semi-Galois theories from the theory developed so far. We shall impose all of the conditions from the last section, since those are what we shall be using in order to construct examples. As such, we throughout let $(\Ccal,J_{\Tcal})$ be a principal site such that $(\Tcal,\Mcal)$ is a factorization system on $\Ccal$ with $\Mcal$ consisting of monomorphisms. We also take $(\Lcal,\Rcal)$ to be the extension of $(\Mcal\op,\Tcal\op)$ to $\Ind(\Ccal\op)$ as described in Proposition \ref{prop:ofs}.

\subsection{Expanding the site}
\label{ssec:expand}

Let $U$ be a $\Tcal\op$-injective, $\Rcal$-universal object of $\Ind(\Ccal\op)$ which is a colimit of $\Tcal\op$-morphisms, and let $L = \End(U)_{\Ind(\Ccal\op)}\op$. The topology $\rho$ on $L$ from Theorem \ref{thm:basic} is generated by the necessary clopens for the $L$-sets $\Hom_{\Ind(\Ccal\op)}(D,U)$ ranging over objects $D$ of $\Ccal$. It is worth noting that given an expression of $U$ as the colimit of a directed diagram $E:\Ibb \to \Ccal\op$ (whose morphisms lie in $\Tcal\op$, say), we have a basis consisting of the necessary clopens for the $L$-sets $\Hom_{\Ind(\Ccal\op)}(E(i),U)$.

Let $\Rfrakk_{\rho}$ be the category of open congruences of $(L,\rho)$ (see Definition \ref{dfn:Rfrak} and Scholium \ref{schl:Morita2} of the last chapter). Recall that $\Rfrakk_{\rho}$ is a canonical effectual reductive site for $\Cont(L,\rho)$. However, there is not necessarily a canonical functor $\Ccal \to \Rfrakk$, since there are typically several generators of the principal $L$-set corresponding to an object $D$ of $\Ccal$, and each of these can give a different choice of open congruence whose quotient is isomorphic to $\Hom_{\Ind(\Ccal\op)}(D,U)$; this problem does not arise in the group case, since the open congruence is entirely determined by the isomorphism type.

As such, we augment $\Ccal$ to the category $\Ccal_U$ consisting of pairs $(D,g)$, where $D$ is an object of $\Ccal$ and $g:D \to U$ is an $\Rcal$-morphism in $\Ind(\Ccal\op)$, but whose morphisms are simply those of $\Ccal$. The forgetful functor $\Ccal_U \to \Ccal$ is clearly a weak equivalence, so we can lift the morphism $J_{\Tcal}$ to $\Ccal_U$ to obtain an enlarged site for $\Sh(\Ccal,J_{\Tcal})$. Now we \textit{do} have a canonical functor $\Ccal_U \to \Rfrakk_{\rho}$ mapping $(D,g)$ to the right congruence $\rfrak_g$. This functor is:
\begin{enumerate}[label = ({\roman*})]
	\item faithful if and only if $\Tcal$ is contained in the class of epimorphisms;
	\item full and faithful if and only if $\Tcal$ is contained in the class of strict epimorphisms; and
	\item an equivalence if and only if $(\Ccal,J_{\Tcal})$ is a reductive, effectual site.
\end{enumerate}
Recall that an \textit{anafunctor} consists of a span of functors where the left leg is a (weak) equivalence, so we can equivalently say that $\Ccal$ admits an anafunctor to $\Rfrakk_{\rho}$ with the stated properties.

In particular, the syntactic site of any supercompactly generated theory $\Tbb$ satisfying the conditions of \ref{crly:countabletheory} can be identified up to equivalence with a full subcategory of the category of right congruences on any monoid classifying $\Tbb$.

\subsection{Examples}
\label{ssec:xmpls2}

\begin{xmpl}\label{xmpl:classic}
The conditions of Theorem \ref{thm:qhomogen} apply to any atomic site whose underlying category has the dual of (the amalgamation property and) the joint embedding property, with $\Tcal$ the class of all morphisms and $\Mcal$ the class of isomorphisms. Thus we recover Caramello's topological Galois theory \cite{TGT} as a special case, although we only extract the complete monoid of endomorphisms of the object $U$, rather than its dense subgroup of automorphisms; Proposition \ref{prop:densegroup} in the Conclusion is enough to bridge this discrepancy. More work is also needed to recover the uniqueness up to isomorphism of a point constructed as a countable colimit, if one exists.
\end{xmpl}

This in particular means that all of the examples provided by Caramello in \cite{TGT} can be recovered from this approach. As a particular case, consider the following.

\begin{xmpl}
The fact that the category of finite groups and embeddings between them has the joint embedding property and the amalgamation property means that one can apply the \Fraisse construction to obtain Philip Hall's universal locally finite group (cf. \cite[Chapter 7]{Model} again). But this also means that the dual of the full category of finite groups has the joint-$\Tcal$-covering property, for $\Tcal$ the dual of the class of embeddings of groups. The class $\Tcal\op$ is not costable in the full subcategory; consider the inclusion of the alternating group on $4$ elements into that with $5$ elements, for example. However, $\Tcal\op$ \textit{is} stable in the category of abelian groups, whence employing Theorem \ref{thm:Fraisse1} and Proposition \ref{prop:Fraisse2} we obtain a countable, locally finite abelian group into which all such groups embed, which is injective with respect to embeddings of groups and (since all of these embeddings are regular) such that the category of abelian groups is equivalent to a full subcategory of the open right congruences of the topologized endomorphism monoid of this group. More specifically, we can obtain this group by taking (for example) the sequence of groups,
\[n \mapsto \prod_{i=1}^n \prod_{j = 1}^{n-i+1} C_{p_i^j},\]
with component-wise maps between these.

By further restricting to the class of cyclic groups or the cyclic $p$-groups we obtain groups with related properties. In the latter case we recover the Pr\"{u}fer $p$-groups, thus recovering some known universal properties of these groups.
\end{xmpl}

\begin{xmpl}
Let $\Ccal$ be any countable regular category in which all objects are well-supported, and let $\Tcal$ be the class of regular (=strict) epimorphisms. Stability is satisfied by definition, while the joint $\Tcal$-covering property is guaranteed by considering products. Thus $\Ccal$ embeds (up to equivalence) into the category of right congruences of a topological monoid. Interestingly, this includes the dual of the example above: we can take the category of all finite abelian groups and obtain a universal countable pro-group which admits an epimorphism to all finite abelian groups and which is projective with respect to the class of epimorphisms between these.
\end{xmpl}

\begin{xmpl}
In \cite[Theorem 1.2]{Lambda}, it is shown that the category of $\Lambda$-rings over a number field $K$ is equivalent to the category of continuous actions of a topological monoid, known as the \textit{Deligne-Ribet monoid}, on finite sets. While the definition of $\Lambda$-rings is too involved to present in full here, it follows that we can refine this to an equivalence of a subcategory of $\Lambda$-rings corresponding to the principal actions of the Deligne-Ribet monoid of the type covered in this paper.
\end{xmpl}

\begin{xmpl}
The category of rings admits a factorization system whose left class consists of `integral homomorphisms', and whose right class consists of integrally closed monomorphisms cf. \cite[Example 6.1]{JCL}. Selecting a countable ring $R$ and considering the collection of algebras $r:R \to A$ which are integral over $R$ (there are only countably many, since there are countably many polynomials over $R$), we have a countable category with a factorization system. Just as above, there are two ways that this might be transformed into an analogue of classical Galois theory.

The first is to take $\Ccal$ to be the category above and let $\Tcal$ be the class of integral homomorphisms; $\Mcal$ is contained in the class of monomorphisms by definition. This validates the conditions needed for Theorem \ref{thm:qhomogen} (and hence Theorem \ref{thm:Fraisse1}) if and only if any pair of integral algebras over $R$ admit a integral homomorphisms to a common such algebra and integral homomorphisms are stable; a thorough analysis of these conditions is beyond the scope of this thesis. When this is so, we obtain an ind-object of the dual category and an anafunctor to the category of principle actions.

However, the above has the opposite variance to classical Galois theory, and the comparison functor is unlikely even to be faithful: most interesting examples of integral homomorphisms which we know of are not epimorphisms!
An alternative is to equip the above subcategory with a distinct factorization system, the more straightforward (regular epi,mono) factorization, and to apply our results to the dual of this category. When the relevant joint-embedding and stability properties are verified, we obtain from Theorem \ref{thm:Fraisse1} a `universal injective integral closure' $\overline{R}$ of $R$, as well as a contravariant functor to the category of principal continuous actions of the endomorphism monoid of $\overline{R}$. This functor will be full and faithful if and only if the embeddings of rings in the category are strict monomorphisms, which can be interpreted as a separability condition. We leave the elucidation of specific examples of this construction to future work.
\end{xmpl}

%% file: The_Conclusion.tex
\chapter{Conclusion}
\label{chap:conclusion}

While we have gone a long way in establishing the properties of toposes of the form $\Cont(M,\tau)$ and their canonical representatives in this thesis, it is clear that there are multiple avenues for future exploration of the subject.

\section{Discrete monoids}
\label{sec:disc2}

\subsection{Further properties}
\label{ssec:further}

In Chapter \ref{chap:mpatti}, we saw many instances of how properties of the global sections morphism of a topos of the form $\Setswith{M}$ are reflected as Morita-invariant properties of $M$. Despite the rich variety that arose just from considering the global sections morphism, there are many other sources of such properties; we summarize some ideas and examples of these here:
\begin{itemize}
\item \textit{Diagonal properties}: Since the (2-)category of toposes and geometric morphisms has pullbacks, any geometric morphism $\Fcal \to \Ecal$ induces a diagonal $\Fcal \to \Fcal \times_{\Ecal} \Fcal$. We may in particular apply this to the global sections morphism, to express properties such as \textit{separatedness} of a topos (cf. \cite[Definition C3.2.12(b)]{Ele}). However, a more detailed understanding of geometric morphisms between toposes of the form $\Setswith{M}$ is needed to analyze these.
\item \textit{Relative properties}: Some properties of (Grothendieck) toposes are most succinctly expressed by the existence of geometric morphisms of a particular type to or from toposes with certain properties, as we saw in the example of \'{e}tendues at the end of Chapter \ref{chap:mpatti}.
\item \textit{Categorical properties}: There are some categorical properties of Grothendieck toposes that ostensibly aren't expressible in terms of the global sections morphism in a straightforward way, although they might be expressible in the relative sense above. This includes the property of there being a separating set of objects with a particular property, as we saw in the example of locally decidable toposes, also at the end of Chapter \ref{chap:mpatti}.
\item \textit{Internal logic properties}: As a variant of the preceding point, the internal logic of a topos is determined by the structure of subobject lattices, and so is embodied in the structure of their subobject classifiers. As we saw in Chapter \ref{chap:mpatti}, the structure of the subobject classifier in $\Setswith{M}$ corresponds to the structure of the right ideals of $M$, but we have only tackled the most basic cases in which $\Setswith{M}$ is a Boolean or de Morgan topos. There are surely further algebraic or logical properties of this lattice to investigate.
\end{itemize}

Each of these classes merits a systematic study in its own right. In the other direction, there are some notable elementary properties of monoids which we have not yet found a topos-theoretic equivalent for. The most basic is the left Ore condition, dual to Definition \ref{dfn:rOre}; of course, we could simply examine the category of \textit{left} actions of our monoid, and dualize the results presented in Chapter \ref{chap:mpatti}, but we believe it will be more informative to seek a condition intrinsic to the topos of right actions, given the variety of equivalent conditions we reached in Theorem \ref{thm:deMorgan}.

\subsection{Other presentations}
\label{ssec:other}

A related thread for future investigation is that of extending the dictionary of properties between different \textit{presentations} of toposes. We saw an instance of this realized in the comparison of toposes of the form $\Setswith{M}$ with toposes of sheaves on spaces in Chapter \ref{chap:mpatti}.

There are three types of presentation in particular that spring to mind:
\begin{itemize}
\item \textit{Syntactic presentations}: We saw how toposes can be build from geometric theories in Chapter \ref{chap:logic}. How do syntactic and model-theoretic properties translate into monoid properties?
\item \textit{Site presentations}: We know that we can use sites to present toposes, and we can further restrict to principal or finitely generated sites of Chapter \ref{chap:sgt}. How do properties of the underlying category or the Grothendieck topology translate into monoid properties?
\item \textit{Groupoid presentations}: Any Grothendieck topos can be presented as a topos of equivariant actions of a localic groupoid, and this can be chosen to be a topological monoid when the topos has enough points, cf. the work of Butz and Moerdijk, \cite{Grpoid}. Indeed, Jens Hemelaer showed in \cite{TGRM} that a topos of actions of a discrete monoid can presented as the topos of equivariant sheaves on a \textit{posetal groupoid}. Such groupoids possess both algebraic and spatial properties which might profitably be transferred to monoid properties. 
\end{itemize}

A natural approach to all of these comparisons as well as those above is Caramello's `toposes as bridges' principle \cite{TST}: by finding a way to translate a property of a given presentation into an invariant of the corresponding topos, we may subsequently translate it into an invariant of each of the other presentations.

\subsection{Relativization}
\label{ssec:relativize}

Recall that in Theorem \ref{thm:2equiv0}, we demonstrated an equivalence between the 2-category of monoids, semigroup homomorphisms and conjugations, and the 2-category of their presheaf toposes, essential geometric morphisms between these and geometric transformations. This means that we can just as systematically explore how properties of semigroup or monoid homomorphisms are reflected as properties of essential geometric morphisms between toposes of the form $\Setswith{M}$. This is a direct extension, or `relativization to a different base monoid,' of the work we did in Chapter \ref{chap:mpatti}, since the unique homomorphism $M \to 1$ corresponds under this equivalence to the global sections morphism of $\Setswith{M}$. More generally, we will be able to use the biequivalence of Theorem \ref{thm:2equiv1} to compare not-necessarily-essential geometric morphisms with biactions, as we did in Scholium \ref{schl:tidiness}. Since such tensor-hom expressions exist for geometric morphisms between presheaf toposes more generally (see \cite[Section VII.2]{MLM}), toposes of monoid actions may provide a good context from which to build an algebraic analysis of geometric morphisms.

This idea has already yielded some success: in further joint work with Jens Hemelaer \cite{NonLC}, we obtained an instance of a geometric morphism induced by a monoid homomorphism which is hyperconnected, essential and local but not locally connected, providing a counterexample for an open problem posed by Thomas Streicher.

A final direction for future investigation, related to the above, is relativization in the usual topos-theory sense of considering (elementary) toposes over a base topos other than $\Set$. Amongst internal categories in arbitrary toposes (cf.\ Section \ref{ssec:proper} of Chapter \ref{chap:sgt}), monoids are naturally defined as those whose object of objects is the terminal object. Accordingly, one might be interested in examining toposes of internal right actions of internal monoids relative to a topos other than $\Set$. While many of the results we obtained in Chapter \ref{chap:mpatti} were arrived at constructively or are expressed in a way that relativizes directly, there are some which cannot be transferred directly into an arbitrary topos. For instance, our inductive construction of the submonoid right-weakly generated by $S$ in Lemma \ref{lem:construct} requires the presence of a natural number object, while the application of Proposition \ref{prop:cc2} in Theorem \ref{thm:labsorb} relies on the law of excluded middle. More significantly, the proof that condition \ref{item:rab6} implies condition \ref{item:rab7} in Theorem \ref{thm:rabsorb} explicitly relies on a form of the axiom of choice. This investigation will therefore be non-trivial, and it will be interesting to discover the relative analogues of the results presented in this thesis.

\section{Supercompactly generated toposes}
\label{sec:sgt2}

The theory-laden middle chapters contained results in various possible directions, which already hint at directions for possible future exploration. We highlight two in particular which we spent some time investigating in the course of research for those chapters, but which didn't make it into this thesis.

\subsection{Points}
\label{ssec:pts2}

A result which appeared in the original preprint version (\cite{SGT}) of Chapter \ref{chap:sgt}, but which had to be removed due to an error in its proof, was a generalization of Deligne's classical completeness theorem for coherent toposes to the class of compactly generated toposes. This result was not needed for the purposes of the thesis, since all of our toposes have enough points constructively, but the question remains whether Deligne's original proof, \cite{SGA4Coh}, can be extended to the setting of sites without pullbacks.

Determining this will be no easy task, not least because there are various results of a similar nature, such as \cite[Theorem 6.2.4]{MakkaiReyes} which demonstrates a completeness theorem for toposes of sheaves on sites which are countable in a suitable sense (albeit still with pullbacks), which are proved by totally different means.

\subsection{Reductive logic}
\label{ssec:redlogic}

We discussed the usual account of categorical logic as it applies to toposes in Chapter \ref{chap:logic}. However, this approach does not seem the ideal fit for analyzing toposes whose canonical sites do not have finite limits. We anticipate that there is scope for developing branches of logic which admit interpretations in reductive and coalescent categories.

This too will be more of a challenge than it first appears, because pullbacks and equalizers were fundamental in the interpretations of even the most basic formulae appearing in Section \ref{ssec:geomsem}. The logics for these categories will therefore not have terms in the sense that first order theories do, which already distances it from classical logic.

\subsection{Exactness Properties}
\label{ssec:effects}

In Proposition \ref{prop:effective} and Example \ref{xmpl:countable}, we explored the relationship between the conditions of effectualness on reductive and coalescent categories introduced in Definition \ref{dfn:redeff} and the more familiar notion of effectiveness recalled in Definition \ref{dfn:regeff}. More generally, one might wonder how this relates to the general notion of \textit{exactness}, which refers to the types of interaction between classes of colimits and pullbacks which occurs in Grothendieck toposes. This notion has been explored in situations admitting finite limits (including the analogous concepts for enriched categories) in \cite{GarnerLack}. It is reasonable to hope that an examination of exactness properties resembling effectualness for sites which lack finite limits might enable one to refine this notion into one which makes sense outside of the `lex' context. This may be as simple as replacing Garner and Lack's finite limit conditions with suitable flatness conditions, although the precise nature of these conditions is not clear \textit{a priori}.

\section{Topological monoids}
\label{sec:topmonoid}

There are some natural questions which arose during the developments of Chapter \ref{chap:TTMA} that we were unable to resolve. We record them here and suggest some future directions this research might proceed, independently of the content of Chapter \ref{chap:TSGT}.

\subsection{Pathological powder monoids}
\label{ssec:moremonoid}

The reader may have noticed in Chapter \ref{chap:sgt} that we did not exhibit any examples illustrating the asymmetry in the definition of powder monoids. This is because our main classes of examples, prodiscrete monoids and powder groups, are both blind to this distinction, since their definitions are stable under dualizing. Similarly, any commutative right powder monoid is also a left powder monoid. These cases make constructing examples of right powder monoids which are not left powder monoids difficult. Nonetheless, we posit that:
\begin{conj}
\label{conj:powdery}
There exists a right powder monoid which is not a left powder monoid.
\end{conj}

Scholium \ref{schl:G3'} puts some limits on Conjecture \ref{conj:powdery}, since it says that any right powder monoid is at most one step away from also being a left powder monoid. In particular, we never get an infinite nested sequence of topologies on a monoid by repeatedly computing the associated right and left action topologies. We have not demonstrated comparable results for complete monoids, but we expect them to hold:
\begin{conj}
\label{conj:complete}
The right completion of a left powder monoid or left complete monoid retains the respective property, and dually for left completions of right powder monoids or right complete monoids. However, we expect that there exists a right complete monoid which is not a left powder monoid.
\end{conj}

Another way of expressing Conjectures \ref{conj:powdery} and \ref{conj:complete} is to say that we expect the diagram of monadic full and faithful functors \eqref{eq:monadic2} to extend as follows:
\[\begin{tikzcd}
& & \ar[dl] \mathrm{CMon}_s & & \\
& \ar[dl] \mathrm{PMon}_s & &
\ar[ul] \ar[dl] \mathrm{CP'Mon}_s & \\
\mathrm{T_0Mon}_s & &
\ar[dl] \ar[ul] \mathrm{P''Mon}_s & &
\ar[ul] \ar[dl] \mathrm{C''Mon}_s,\\
& \ar[ul] \mathrm{P'Mon}_s & &
\ar[dl] \ar[ul] \mathrm{C'PMon}_s \\
& & \ar[ul] \mathrm{C'Mon}_s & & 
\end{tikzcd}\]
where the notation is the intuitive extension of that employed in \eqref{eq:monadic2} and each inclusion represented is non-trivial.

\subsection{Finitely generated complete monoids}
\label{ssec:fgcompmon}

Besides these conjectures characterizing pathological examples, there is plenty of ground still to cover in understanding these classes of monoids. What does a `generic' complete monoid look like, beyond what was shown in Chapter \ref{chap:TTMA}? Is it possible to classify them?

For example, given an element $x$ in a complete monoid, we may consider the closure of the submonoid generated by $x$, which by Corollary \ref{crly:closed} is a complete monoid. One might consider this an instance of a `complete monoid generated by one element'\footnote{We include the quote marks to emphasize that this submonoid is not generated by $x$ in an algebraic sense.}. We can identify such monoids as the canonical representatives of toposes admitting a hyperconnected morphism from $\Setswith{\Nbb}$. By Corollary \ref{crly:prodisc} these are commutative prodiscrete monoids. Analogously, `finitely generated complete monoids' would correspond to complete monoids representing toposes admitting hyperconnected geometric morphisms from the toposes $\Setswith{F_{n}}$, where $F_n$ is the free (discrete) monoid on $n$ elements. By Proposition \ref{prop:prince2} they correspond to filters of right congruences on $F_n$, which we expect to have a tame classification. Is it possible to identify the `finitely presented complete monoids' amongst these? One could go on to investigate the properties of various ($2$-)categories of such monoids, taking advantage of results such as such as those in Section \ref{ssec:monads}. Future applications of the theory developed in Chapter \ref{chap:TSGT} may rely on understanding these answers to these questions.

\subsection{Invariant properties}
\label{ssec:mitopmon}

In investigating complete monoids, it will be desirable to extend the results of Chapter \ref{chap:mpatti} to the topological case. In this regard, we can already glean some positive results. Whilst we saw in Example \ref{xmpl:notpro2} that a complete monoid Morita-equivalent to a topological group need not be a group, we have the next best result.
\begin{prop}
\label{prop:densegroup}
Let $(M,\tau)$ be a topological monoid. The following are equivalent:
\begin{enumerate}
	\item $\Cont(M,\tau)$ is an atomic topos;
	\item The completion of $(M,\tau)$ has a dense subgroup;
	\item The group of units in the completion of $(M,\tau)$ is dense;
	\item For each open relation $r \in \underline{\Rcal}_{\tau}$ and $m \in M$, there exists $m' \in M$ with $(mm',1) \in r$.
\end{enumerate}
\end{prop}
\begin{proof}
($3 \Rightarrow 2 \Rightarrow 1$) If the group of units of (the completion of) $(M,\tau)$ is dense, then clearly this provides a dense subgroup. If the $(G,\tau|_G)$ is a dense subgroup of (the completion of) $(M,\tau)$, then $\Cont(M,\tau)$ admits a hyperconnected morphism from $\Setswith{G}$, whence the former is an atomic topos, by Scholium \ref{schl:descend}.

($1 \Rightarrow 3$) If $\Cont(M,\tau)$ is atomic, all of the supercompact objects are necessarily atoms. The opposite of $\underline{\Rcal}_{\tau}$ is easily verified to satisfy the amalgamation property and joint embedding property of \cite[Definition 3.3]{TGT}, and the canonical point of $(M,\tau)$ provides an $\underline{\Rcal}_{\tau}\op$-universal, $\underline{\Rcal}_{\tau}\op$-ultrahomogeneous object in $\mathrm{Ind}$-$\underline{\Rcal}_{\tau}\op$, namely the completion of $(M,\tau)$ itself. Thus by \cite[Theorem 3.5]{TGT}, there is a topological group $(G,\sigma)$ representing the topos (and having the same canonical point), and by the more detailed description of this construction in \cite[Proposition 5.7]{ATGT}, the group so constructed is precisely the group of units of $(M,\tau)$, and this group is dense in $(M,\tau)$.

($1 \Leftrightarrow 4$) Consider $\underline{\Rcal}_{\tau}$ when $\Cont(M,\tau)$ is atomic. Since all of the $M/r$ are atoms, all of the morphisms in this category are (strict) epimorphisms, which means that in particular the canonical monomorphisms $[m]: m^*(r) \to r$ are isomorphisms, providing $[m']: r \to m^*(r)$ such that $(mm',1) \in r$ and $(m'm,1) \in m^*(r)$, but the latter is implied by the former, so the former suffices. Conversely, if $4$ holds, then all of the morphisms of $\underline{\Rcal}_{\tau}$ are strict epimorphisms, whence the reductive topology coincides with the atomic topology, so $\Cont(M,\tau)$ is atomic as required.
\end{proof}

We anticipate a plethora of results of this nature, where a complete monoid generates a topos having a property $Q$ if and only if it has a dense submonoid having property $P$, where $P$ is the property corresponding to $Q$ for toposes of discrete monoid actions; the above is the case where $Q$ is the property of being atomic and $P$ is the property of being a group from Theorem \ref{thm:atomic} of Chapter \ref{chap:mpatti}. In order to attain these results, some preliminary work will be needed to accumulate the relevant factorization results for properties of geometric morphisms along hyperconnected geometric morphisms.

On the subject of geometric morphisms, we have two further conjectures. In the hope of improving Scholium \ref{schl:factors} to a more elegant result, we begin with the following:
\begin{conj}
\label{conj:characterization}
There exists an intrinsic characterization, independent of the representing monoid $M$, of those hyperconnected geometric morphisms with domain $\Setswith{M}$ identifying toposes of the form $\Cont(M,\tau)$.
\end{conj}
To be more specific, observe that Proposition \ref{prop:intrinsic} provides an intrinsic sufficient condition for a hyperconnected morphism to express its codomain topos in terms of a topology on any monoid representing its domain topos; Conjecture \ref{conj:characterization} posits that it should be possible to refine this to a necessary and sufficient condition.

We also record our expectation that the converse of Scholium \ref{schl:Morita2} fails.
\begin{conj}
\label{conj:inclusion}
There exists a complete monoid $(M',\tau')$ and an idempotent $e \in M'$ such that the semigroup inclusion $M:= eM'e \hookrightarrow M'$ is \textit{not} a Morita equivalence, but the induced geometric inclusion $\Cont(M,\tau'|_{M}) \hookrightarrow \Cont(M',\tau')$ is an equivalence.
\end{conj}

\subsection{Actions on Topological Spaces}

A proof of, or counterexample to, Conjecture \ref{conj:powdery} will establish the extent of the symmetry in the type of Morita equivalence studied in this Chapter \ref{chap:TTMA}. Whichever way this result falls, however, we have shown that the category of right actions of a topological monoid on discrete spaces is a very coarse invariant of such a monoid. Moreover, anyone interested in actions of topological monoids is likely to wish to examine their actions on more general classes of topological space. A solution to this, which is viable in any Grothendieck topos $\Ecal$, is to first consider the topos $[M\op,\Ecal]$, which can be constructed as a pullback in $\TOP$, as in the lower square here:
\begin{equation}
\label{eq:pbEM}
\begin{tikzcd}
\Ecal \ar[r] \ar[d] \ar[dr, phantom, "\lrcorner", very near start] & \Set \ar[d] \\
{[M\op{,}\Ecal]} \ar[r] \ar[d] \ar[dr, phantom, "\lrcorner", very near start] & {[M\op{,}\Set]} \ar[d] \\
\Ecal \ar[r] & \Set,
\end{tikzcd}	
\end{equation}
since $M$ induces an internal monoid in $\Ecal$ by its image under the inverse image functor of the global sections morphism of $\Ecal$. Taking $\Ecal$ to be the topos of sheaves on a space $X$, we can view the objects of $[M\op,\Ecal]$ as right actions of $M$ on spaces which are discrete fibrations over $X$; taking $\Ecal$ to be a more general topos of spaces, we similarly get actions of $M$ on such spaces.

In each case, we can construct the subcategory of $[M\op,\Ecal]$ on the actions which are continuous with respect to a topology $\tau$ on $M$. In the best cases, this will produce a topos hyperconnected under $[M\op,\Ecal]$, and the analysis can proceed analogously to that of Chapter \ref{chap:TTMA}, taking advantage of the $\Ecal$-valued point constructed in the upper square of \eqref{eq:pbEM}. If this can be done with sufficient generality, one will be able to address a host of interesting Morita-equivalence problems in this way.

\subsection{Topological Categories}
\label{ssec:topcat}

Another direction to generalize is to consider topologies on small categories with more than one object. Let $\Ccal$ be a small category with set of objects $C_0$, set of morphisms $C_1$, identity map $i:C_0 \to C_1$, domain and codomain maps $d,c: C_1 \to C_0$, and composition $m:C_2\to C_1$, where $C_2$ is the pullback:
\[\begin{tikzcd}
C_2 \ar[r] \ar[d] \ar[dr, phantom, "\lrcorner", very near start] & C_1 \ar[d, "d"] \\
C_1 \ar[r, "c"'] & C_0;
\end{tikzcd}\]
this matches the presentation of internal categories recapitulated in Section \ref{ssec:proper} of Chapter \ref{chap:sgt}. As such, a presheaf on $\Ccal$ can be expressed as an object $a:F_0 \to C_0$ of $\Set/C_0$, equipped with a morphism $b:F_1 \to F_0$ where $F_1$ is the pullback,
\[\begin{tikzcd}
F_1 \ar[r, "\pi_2"] \ar[d, "\pi_1"'] \ar[dr, phantom, "\lrcorner", very near start] & F_0 \ar[d, "a"] \\
C_1 \ar[r, "c"'] & C_0,
\end{tikzcd}\]
satisfying $a \circ b = d \circ \pi_2$, $b \circ (\id_{F_0} \times_{C_0} i) = \id_{F_0}$ and $b \circ (b \times_{C_0} \id_{C_1}) = b \circ (\id_{F_0} \times_{C_0} m)$.

Equipping $C_0$ and $C_1$ with topologies such that $c$ and $d$ are continuous, we might call a presheaf as above \textit{continuous} if $b$ is continuous when $C_0$ and $F_0$ are equipped with the discrete topology and $F_1$ is equipped with the pullback topology\footnote{The map $a$ is automatically continuous when the topology on $F_0$ is discrete.}. Yet again, we can consider the full subcategory of $\Setswith{\Ccal}$ on these presheaves, and we expect it to be coreflective. In good cases, we will have the analogue of Proposition \ref{prop:hyper}, and the analysis can proceed as in Chapter \ref{chap:TTMA}, leading to a class of genuine topological categories representing these toposes.

\subsection{Localic monoids and constructiveness}
\label{ssec:localicmon}

Topos theorists tend to try to work constructively wherever possible, since doing so ensures that all results can be applied over an arbitrary topos. In this light, our frequent reliance on complementation in the underlying sets of our monoids in Chapter \ref{chap:TTMA} is quite restrictive, since \textit{a priori} it means our results are applicable only over Boolean toposes, and we have not formally demonstrated here that they apply even to this level of generality.

From a constructive perspective, more suitable objects of study than topological monoids would be \textbf{localic monoids}, which are monoids in the category of locales over a given base topos, typically $\Set$. Early on in the research for this thesis, Steve Vickers suggested that we consider pursuing this direction. However, while the category of actions of a localic monoid on sets (again viewed as discrete spaces) is easy to define, it is much harder to show that such a category is a topos. In \cite[Example B3.4.14(b)]{Ele}, we see that the more powerful results of \textit{descent theory} are required to show that categories of actions of localic groups are toposes. While descent theory is an important tool, it is far more abstract than the comonadicity theorem we used in Corollary \ref{crly:topos}, making concrete characterization results for these toposes more challenging to prove.

While we did not end up treating localic monoids in this thesis, we anticipate that the present work will be valuable in that analysis. Indeed, the functor sending a locale to its topological space of points preserves limits, so that it provides a canonical `forgetful' functor from a category of actions of a localic monoid to a category of actions of a topological monoid. We anticipate that, just as in Section \ref{sec:properties}, this functor can be used to constrain the properties of a category of actions of a localic monoid. 

We should mention that another obstacle in our study of localic monoids is a lack of easily tractable examples, especially examples of localic monoids (or even localic groups) which one can show are \textit{not} Morita equivalent to topological monoids in their actions on discrete spaces. While the construction of the localic group $\mathrm{Perm}(A)$ of permutations of a locale $A$, described by Wraith in \cite{LocGrp}, is used as a basis for the Localic Galois Theory of Dubuc \cite{LGT}, the latter author provides no specific examples of instances of these. We expect that the construction of such examples will further illuminate the appropriate approach to studying categories of actions of localic monoids.

\section{Applying toposes of topological monoid actions}
\label{ssec:TSGT}

Orthogonal factorization systems abound in category theory, and we anticipate many applications of the results of Chapter \ref{chap:TSGT}; already there is much work to do in elucidating the algebraic examples sketched there. It is worth noting that in a category with pullbacks, the conditions needed for those results coincides with the concept of \textit{stable factorization system}, which is an orthogonal factorization system in which the left class is stable under pullback.

We know from Theorem \ref{thm:characterization} of Chapter \ref{chap:TTMA} that any hyperconnected morphism $\Setswith{M} \to \Ecal$ is enough to ensure the existence of a topological monoid presentation for $\Ecal$. As such, we can weaken the statement of Theorem \ref{thm:basic} and its derivatives by replacing $L$ with any monoid $M$ which acts by endomorphisms on a point of $\Ecal$. However, extra conditions along the lines of those discussed in Section \ref{ssec:factor} of the Chapter \ref{chap:TTMA} would be required to guarantee that $\Ecal$ can be identified with the actions of $M$ which are continuous with respect to some topology.

A final direction which would tie together many of the themes of the present thesis would be to complement the translation of topos-theoretic invariants into properties of topological monoid suggested in Section \ref{ssec:mitopmon} above with a translation of the same properties into site-theoretic ones described in Section \ref{ssec:other}, since this would yield concrete results about the presenting monoids without needing to explicitly calculate them.

We hope that some of these ideas will eventually lead to fruitful research uniting diverse areas of mathematics.